\documentclass[onefignum,onetabnum]{siamonline250211}

\usepackage[english]{babel}

\usepackage{amssymb}
\usepackage{graphicx}
\usepackage{subfig}
\usepackage{cleveref}

\title{Inverse  scattering beyond Born approximation via rotation-equivariance-aware neural network and low-rank structure     
}
\author{Yuyuan Zhou\thanks{Academy of Mathematics and Systems Science, Chinese Academy of Sciences, Beijing, 100190, 
China and School of Mathematical Sciences, University of Chinese Academy of Sciences, Beijing 100049, China
(\texttt{zhouyuyuan@amss.ac.cn})}
\and
Shixu Meng\thanks{Department of Mathematical Sciences, The University of Texas at Dallas, 75025 Richardson,  USA. Corresponding author. (\texttt{smeng@utdallas.edu})} }

\begin{document}
\maketitle

\begin{abstract}
This work proposes a hybrid method (ULR) which integrates a rotation-equivariance-aware neural network and a low-rank structure to solve the two dimensional inverse medium scattering problem. 
The neural network is to model the data corrector which maps the full data to the Born data, and the   low-rank structure is to design an inverse Born  solver that finds a regularized solution from the perturbed Born data. The proposed rotation-equivariance-aware neural network naturally incorporates the  reciprocity relation and the rotation-equivariance  in inverse scattering, while the low-rank structure effectively filters high-frequency noise in the output of the neural network and leads to a regularized method supported by theoretical stability  in the Born region. For a comparative study, we  replace the low-rank inverse Born solver by another rotation-equvariance-aware neural network to propose a two-step neural network (UU). Furthermore, we extend the proposed methods (ULR and UU) to tackle the more challenging case with only limited aperture data.  A variety of numerical experiments are conducted to compare the proposed ULR, UU, and a black-box neural network.
\end{abstract}

\section{Introduction}

This work is concerned with the inverse scattering problem where the goal is to determine a penetrable object from far field measurements. It merits potential applications to medical imaging, geophysical exploration, ocean acoustics, and non-destructive testing. For an introduction to the field of inverse scattering, we refer to \cite{2014A,cakoni2022inverse,cakoni2011linear,colton2012inverse,kirsch2007factorization}. 

Inverse scattering is a challenging problem due to its nonlinear and ill-posedness nature.  The well-known Born approximation leads to a Born unknown-to-data map given by a restricted Fourier transform. A new explicit a priori estimate in the spirit of increasing stability was obtained  in \cite{meng23data} based on the dual property of the disk prolate spheroidal wave functions (disk PSWFs), which are eigenfunctions of both the restricted Fourier transform and a particular Sturm-Liouville differential operator. One salient feature of the restricted Fourier transform is that its dominating eigenvalues have numerically the same magnitude and the eigenvalues decay to zeros exponentially fast, which poses challenges for the conventional eigensolver solely based on the restricted Fourier transform. Based on the dual property, a low-rank-structure-assisted computational framework was further developed to solve the inverse scattering problem for weak scatterers in \cite{zhou2024exploring}. The above development allows a deep understanding of the ill-posedness, however the strong non-linearity remains a bottleneck to address, particularly for high contrast media. In recent years, deep learning  has achieved remarkable success in  diverse fields such as computer vision, pattern recognition, natural language processing and others \cite{goodfellow2016deep}. In particular, machine learning has been applied to  the electric impedance tomography and inverse scattering problem \cite{chen2026deep,chen2026data,2020A,desai2025neural,gao2022artificial,Khoo2019,li2024reconstruction,meng2024kernel,ning2024direct,2022Fast,zhang2025back} recently.  In this work, we propose to tackle the nonlinearity and the ill-posedness separately. This idea shares a similar spirit to the decomposition methods  \cite{kirsch1986integral,kirsch1987numerical,kirsch1987optimization}. The recent work \cite{desai2025neural} proposes a neural network to model the data corrector which maps the full far field data to the Born far field data, and an inverse Born solver is then applied to find the unknown from the Born data. Fueled by this idea, our work proposes a rotation-equivariance-aware neural network for the data corrector, and integrates this neural network with an intrinsic low-rank structure associated with the inverse Born solver. This work also tackles the more challenging case when only limited aperture data are available.

To explore neural networks that are suitable for inverse scattering, we propose to take advantage of the mathematical properties in wave scattering. Particularly, the reciprocity relation and rotation-equivariance, though fundamental in inverse scattering theory, haven't received enough attention in computational algorithms. We take advantage of the reciprocity relation to introduce a processed dataset
that is equivalent to the original far field dataset. With this data processing, each processed datum is defined on a unit disk and can be regarded as a double-channel (i.e., real and imaginary) image which enjoys a rotation-equivariance property: if the unknown rotates a certain angle, the corresponding datum (viewed as an image) rotates the same angle. To incorporate this rotation-equivariance property, we propose a rotation-equivariance-aware neural network based on a U-Net with periodic padding in polar coordinates. However, unnecessary high-frequency noise may still be present in the output Born data, we further propose to apply an intrinsic low-rank structure to filter the high-frequency noise in the imperfect Born data (i.e., output of the data corrector) and to find a regularized solution in a low-rank space.
In particular, the low-rank structure is based on the above-mentioned disk PSWFs  which forms an eigenbasis for the Born unknown-to-data map. By projecting the solution onto a low-rank space spanned by finitely many disk PSWFs, the low-rank structure leads to a robust inverse Born solver that is capable of filtering high-frequency noise and is supported by theoretical stability in the Born region. For a comparative study, we also  propose a two-step neural network (UU) which replaces the low-rank inverse Born solver by another rotation-equvariance-aware neural network for comparison. We further compare the proposed ULR and UU, against a black-box neural network. 

We also extend our proposed methods to tackle the limited aperture problem. The limited aperture problem is more challenging since less information is known (cf. \cite{meng23data}); specifically the limited processed data become only a subset of the full processed data.  In this limited apeture,  the rotational-equivariance property still holds and we propose to learn a corresponding data corrector which maps the limited processed (nonlinear) data to the full processed Born data. This allows to extend the proposed ULR and UU to the limited aperture case.

The remaining of the paper is organized as follows. We introduce the mathematical formulation of the inverse scattering problem in Section \ref{sec: intro} and discuss the intrinsic mathematical properties on reciprocity relation and rotation-equivariance in Section \ref{sec: math property}. We then propose the hybrid method ULR and a two-step neural network UU in Section \ref{sec: neural network}.  We conduct a variety of numerical examples to illustrate the potential of our proposed methods in Section \ref{sec: numerical example}. Finally we extend the proposed methods to the limited aperture case in Section \ref{sec: limited aperture}.

{  
\section{Mathematical formulation of the inverse scattering problem} \label{sec: intro}
We consider a two dimensional inverse scattering problem due to a  plane incident wave given by
$u^i(x)=e^{ik\hat{\theta}\cdot x}$, where the wavenumber is $k>0$ and the propagation direction is $\hat{\theta}\in \mathbb{S}:=\{x\in \mathbb{R}^2:|x|=1\}$. In the model of transverse magnetic or electric waves propagation due to an infinitely long cylindrical scatterer, one can introduce a contrast function $q\in L^\infty(\mathbb{R}^2)$  given by the magnetic permeability and electric permittivity. The support of the medium $\Omega:=\mbox{ supp } q$ is an open and bounded set
with Lipschitz boundary $\partial \Omega$ such that $\mathbb{R}^2\setminus \Omega$ is connected. Without the loss of generality, we assume that $\Omega \subset B$ where $B:=\{x\in\mathbb{R}^2:|x|\leq1 \}$ denotes the unit disk. The direct scattering problem is to find the scattered field $u^s\in H^1_{loc}(\mathbb{R}^2)$ that satisfies
    \begin{eqnarray}\label{eq: Helmholtz equation}
        \Delta u^s+k^2(1+q)u^s&=&-k^2qu^i~\mbox{ in } \mathbb{R}^2,\\
        \lim_{|x|\to\infty}|x|^{\frac{1}{2}}\left(\nabla u^s\cdot \frac{x}{|x|}-ik u^s\right)&=&0. \label{eq: sommerfeld radiation condition}
    \end{eqnarray}
    where the last equation is the sommerfeld radiation condition. It can be proved that there exists a unique solution to \eqref{eq: Helmholtz equation}-\eqref{eq: sommerfeld radiation condition} (cf. \cite{2014A} or \cite{colton2012inverse}) and the solution can be obtained by the Lippmann-Schwinger integral equation
$$ u^s(x)=k^2\int_\Omega \frac{i}{4}H_0^{(1)}\left(k|x-y|)(u^i(y)+u^s(y)\right)q(y){\rm d}y$$
where $i$ denotes the imaginary unit and $H_0^{(1)}$  is the Hankel function of first kind of order zero (cf. \cite{colton2012inverse}).

The radiating scattered field $u^s$ has the following asymptotic behaviour 
$$u^s(x)=\frac{e^{ik\frac{\pi}{4}}}{\sqrt{8k\pi}}\frac{e^{ik|x|}}{\sqrt{|x|}}\left(u^\infty(\hat{x};\hat{\theta};k
)+\mathcal{O}( |x|^{-1})\right),~\mbox{ as }|x|\to \infty$$
where $\hat{x}=x/|x|$ denotes the observation direction and $u^\infty(\hat{x};\hat{\theta};k)$ is called the far-field pattern. 
The inverse scattering problem aims to determine the contrast $q$ from the multi-static far-field data
$$
\{u^\infty(\hat{x};\hat{\theta};k):\hat{x},\hat{\theta}\in \mathbb{S}\}.
$$
The inverse scattering problem is  challenging  due to its  non-linearity and ill-posedness nature. Indeed, from the Lippmann-Schwinger integral equation and the asymptotic of the Hankel function,
one can obtain the following representation of the far-field pattern  
\begin{equation} \label{eqn: far field integral representation}
u^\infty(\hat{x};\hat{\theta};k)=k^2\int_\Omega e^{-ik\hat{x}\cdot y}(u^s(y)+e^{ik\hat{\theta}\cdot y})q(y) {\rm d}y,
\end{equation}
and the Born far field pattern $u_b^\infty$ (by neglecting the scattered wave field on the right hand side) 
$$
u_b^\infty(\hat{x};\hat{\theta};k)=k^2\int_\Omega e^{ik(\hat{\theta}-\hat{x})\cdot y}q(y) {\rm d}y.
$$
Since $\hat{x}$ and $\hat{\theta}$ belong to the unit circle $\mathbb{S}$, then  $p= \frac{\hat{\theta} - \hat{x}}{2} \in \overline{B}$ where $B$ denotes the unit disk $B(0,1)$; correspondingly the knowledge of the   Born far field data gives the knowledge of 
$$
u_b(p;c) := \frac{1}{k^2} u_b^\infty (\hat{x};\hat{\theta};k), \qquad p := \frac{\hat{\theta}- \hat{x}}{2}, \qquad c := 2k,
$$
where $u_b(p;c)$ is given by the restricted Fourier transform of the unknown $q$, 
\begin{equation}\label{Born model: inverse problem}
    u_b(p;c)=\int_{B} e^{icp\cdot y}  q(y) {\rm d} y, \quad p\in B.
\end{equation}
The above Born unknown-to-data map provides insights on how to explore unique scattering properties, particularly the reciprocity relation and rotation-equivariance, which will be used later on to propose a suitable  neural network for the inverse scattering problem. 
}

{  
\section{Intrinsic mathematical properties in inverse scattering: reciprocity relation and rotation-equivariance } \label{sec: math property}

\subsection{Reciprocity relation and reformulation}
For the Born far field pattern given by
$
u_b^\infty(\hat{x};\hat{\theta};k)=k^2\int_\Omega e^{ik(\hat{\theta}-\hat{x})\cdot y}q(y) {\rm d}y
$,
one can directly see that  
$$
u_b^\infty(\hat{x};\hat{\theta};k) = u_b^\infty(-\hat{\theta};-\hat{x};k), \quad \forall \hat{x},\hat{\theta} \in \mathbb{S}.
$$
Indeed, this property also holds for the full far field pattern due to the following reciprocity relation  (cf. \cite{2014A}).
\begin{lemma}
Let $u^s(x;\hat{\theta};k)$ be the unique radiating solution to \eqref{eq: Helmholtz equation}-\eqref{eq: sommerfeld radiation condition} and  $u^\infty(\hat{x};\hat{\theta};k)$ be its far-field pattern. Then the following reciprocity relation holds
$$
u^\infty(\hat{x};\hat{\theta};k) = u^\infty(-\hat{\theta};-\hat{x};k), \quad \forall \hat{x},\hat{\theta} \in \mathbb{S}.
$$
\end{lemma}
The above reciprocity relation allows to define   a processed dataset uniquely, cf. \cite{zhou2024exploring}; more precisely, for any point $p \in B$ and $p\not=0$, there exist only two incident-observation pairs $(\hat{x}_j,\hat{\theta}_j)_{j=1}^2$ such that
    $
    p = \frac{\hat{\theta}_j - \hat{x}_j}{2}$ for $j =1,2$,
    where $\hat{\theta}_2 = - \hat{x}_1$ and $\hat{x}_2 = - \hat{\theta}_1$. Note the reciprocity relation, it follows that
    $$
    u^\infty(\hat{x}_2;\hat{\theta}_2;k) = u^\infty(-\hat{\theta}_1;-\hat{x}_1;k) = u^\infty(\hat{x}_1;\hat{\theta}_1;k).
    $$
   Let $c=2k$, then one can introduce  the following unique processed datum
\begin{equation} \label{eqn: far field to processed}
 u(p;c) = \frac{1}{k^2} u^\infty(\hat{x};\hat{\theta};k), \, p \not= 0,  \mbox{ where }   p = \frac{\hat{\theta} - \hat{x}}{2}   \mbox{ for some } (\hat{\theta},\hat{x}).   
\end{equation}
The processed data set $\{u(p): p \in B\}$ is uniquely defined almost everywhere. For convenience, we define the unknown-to-data operator $\mathcal{F}:L^\infty(B) \to L^2(B)$ by
\begin{equation} \label{eqn: full unknown-to-data}
 u =  \mathcal{F} q   
\end{equation}
for any $q \in L^\infty(B)$, where the processed data $\{u(p): p \in B\}$ is uniquely defined via \eqref{eqn: far field to processed} by the far field patterns $\{u^\infty(\hat{x};\hat{\theta};k):\hat{x},\hat{\theta}\in \mathbb{S}\}$ due to the unknown contrast $q$. In the Born region, the above unknown-to-data operator becomes explicitly
\begin{equation}\label{eqn: Born unknown-to-data}
    (\mathcal{F}_b\,q)(x)=\int_{B} e^{icp\cdot y}  q(y) {\rm d} y, \quad p\in B.
\end{equation}
The above formulation takes advantage of the reciprocity relation explicitly, and additionally, allows to explore the equivariance properties in the next section.

\subsection{Property of rotation-equivariance}
We first introduce the following rotation group 
\begin{equation*}
    \mathcal{R}:=\left\{R_\phi:R_\phi=
    \begin{bmatrix}
    \cos\phi & -\sin\phi\\
    \sin\phi & \cos\phi
    \end{bmatrix}, \forall \phi\in [0,2\pi] \right\},
\end{equation*}
where $x'=R_\phi x$ denotes the  rotation of $x\in \mathbb{R}^2$ by angle $\phi$ counterclockwise. The rotated function of $f\in L^2(B)$ is defined by $R_\phi f(x):=f(R_{-\phi}x)$.

Now we are ready to present the following rotation-equivariance property in inverse scattering.
\begin{theorem}\label{thm: equivariance}
For any $R_\phi\in \mathcal{R}$, it holds the property of rotation-equivariance where
$$
\mathcal{F} (R_\phi q)   = R_\phi(\mathcal{F} q)
$$ 
for all $q \in L^\infty (B)$.
\end{theorem}
\begin{proof}
In this proof, we adapt the notation where the Born far-field data for contrast $q$   is given by 
    $$
    u_b^{\infty}(\hat{x};\hat{\theta};k;q)=k^2\int_B e^{i2k\frac{\hat{\theta}-\hat{x}}{2}y}q(y){\rm d}y.
    $$
    We first prove the result for the Born model. Consider the Born far-field data for the rotated contrast $R_\phi q$ with some rotation angle $\phi \in [0, 2\pi]$, one can obtain via a change of variable $\eta=R_{-\phi}y$ that
    \begin{eqnarray*}
    u_b^{\infty}(R_\phi\hat{x};R_\phi\hat{\theta};k;R_\phi q) &=& k^2\int_B e^{i2k\frac{R_\phi\hat{\theta}-R_\phi\hat{x}}{2} \cdot y}(R_{\phi}q)(y) {\rm d}y =k^2\int_B e^{i2k\frac{\hat{\theta}-\hat{x}}{2} \cdot (R_{-\phi}y)}q(R_{-\phi}y){\rm d}y\\
    &=&k^2\int_B e^{i2k \frac{\hat{\theta}-\hat{x}}{2}\cdot \eta }q(\eta) {\rm d}\eta = u_b^{\infty}(\hat{x};\hat{\theta}; k; q),
    \end{eqnarray*}
and the corresponding processed data obey the relation $u_b(R_\phi p;c;R_\phi q)=u_b(p;c;q)$ with $p = \frac{\hat{\theta}-\hat{x}}{2}$. This  yields that $[\mathcal{F}_b (R_\phi q)](R_\phi p) = [\mathcal{F}_b q](p)$ for almost every $p\in B$, which can be written as $[\mathcal{F}_b (R_\phi q)] (p)  = (\mathcal{F}_b q) (R_{-\phi}p )$ for almost every $p\in B$. This shows that $\mathcal{F}_b (R_\phi q)   = R_\phi(\mathcal{F}_b q)$.
    
     Now we prove the result for the full far-field data. Let $u^s(\cdot;\hat{\theta};k;q)$ denote the scattered field for the contrast $q$ due to the incident plane wave field $e^{ik x\cdot \hat{\theta}}$. One can obtain from the integral representation \eqref{eqn: far field integral representation} of the far field pattern that
     \begin{eqnarray*}
    &&u^{\infty}(\hat{x};\hat{\theta};k;q)=k^2\int_B e^{-ik\hat{x}\cdot y}(u^s(y;\hat{\theta};k;q)+e^{ik\hat{\theta}\cdot y})q(y) {\rm d}y\\
    &=& k^2\int_B e^{i2k\frac{\hat{\theta}-\hat{x}}{2}\cdot y}q(y){\rm d}y+ k^2\int_B e^{-ik \hat{x}\cdot y}u^s(y;\hat{\theta};k;q)q(y){\rm d}y := I_1(\hat{x};\hat{\theta};q)+I_2(\hat{x};\hat{\theta};q),
 \end{eqnarray*}
 where the first and second quantity is denoted by $I_1$ and $I_2$, respectively. It was already shown in the Born case that  $I_1(R_\phi\hat{x};R_\phi\hat{\theta};R_\phi q) = I_1(\hat{x};\hat{\theta};q)$ and  we  now focus on $I_2$.
Note that $u^s(\cdot;\hat{\theta};k;q)$  satisfies the following equation
    $$
    \Delta u^s(x;\hat{\theta};k;q)+k^2(1+q(x))u^s(x;\hat{\theta};k;q)=-k^2 q(x)e^{ikx \cdot \hat{\theta}},
    $$
    then one can obtain the equation corresponding to the rotated incident direction $R_\phi\hat{\theta}$ and  rotated contrast $R_\phi q$ by
    $$\Delta u^s(x;R_\phi \hat{\theta};k;R_\phi q)+k^2(1+ (R_\phi q)(x) ))u^s( x;R_\phi \hat{\theta};k;R_\phi q)=-k^2 (R_\phi q)(x)e^{ik x\cdot (R_\phi\hat{\theta})}.
    $$
    By the change of variable $\eta=R_{-\phi} x $ and note that $(R_\phi q)(x) = q(R_{-\phi} x)= q(\eta)$, the above equation can be written as 
    $$
    \left[\Delta_\eta + k^2(1+q(\eta))\right] u^s(R_{\phi}\eta; R_{\phi}\hat{\theta};k; R_{\phi} q) =-k^2 q(\eta)e^{ik \eta \cdot \hat{\theta}},
    $$
    which yields that   $u^s(R_{\phi}\eta;R_\phi \hat{\theta};k; R_{\phi} q) = u^s(\eta;\hat{\theta};k; q)$ due to uniqueness of the scattering problem.
    Together with a change of variable $\eta=R_{-\phi}y$, one can obtain that  
    \begin{eqnarray*}
    &&I_2(R_\phi\hat{x};R_\phi\hat{\theta};R_\phi q)=k^2\int_B e^{-ik (R_\phi \hat{x})\cdot y}u^s(y;R_\phi\hat{\theta};k;R_\phi q)q(R_{-\phi}y) {\rm d } y \\ 
     &=& k^2\int_B e^{-ik \hat{x} \cdot \eta}u^s(R_\phi \eta;R_\phi\hat{\theta};k;R_\phi q)q(\eta){\rm d}\eta =  k^2\int_B e^{-ik \hat{x} \cdot \eta} u^s(  \eta; \hat{\theta};k;q)q(\eta) {\rm d}\eta =  I_2(\hat{x};\hat{\theta};q),
 \end{eqnarray*}
 and thereby $u^{\infty}(R_\phi \hat{x};R_\phi\hat{\theta};k;R_\phi q) = u^{\infty}(\hat{x};\hat{\theta};k;q)$. Following exactly the Born case, this shows that $\mathcal{F} (R_\phi q)   = R_\phi(\mathcal{F} q)$ which completes the proof for the rotation-equivariance.
\end{proof}

\begin{figure}[htbp]
\centering
 \includegraphics[width=0.30\textwidth]{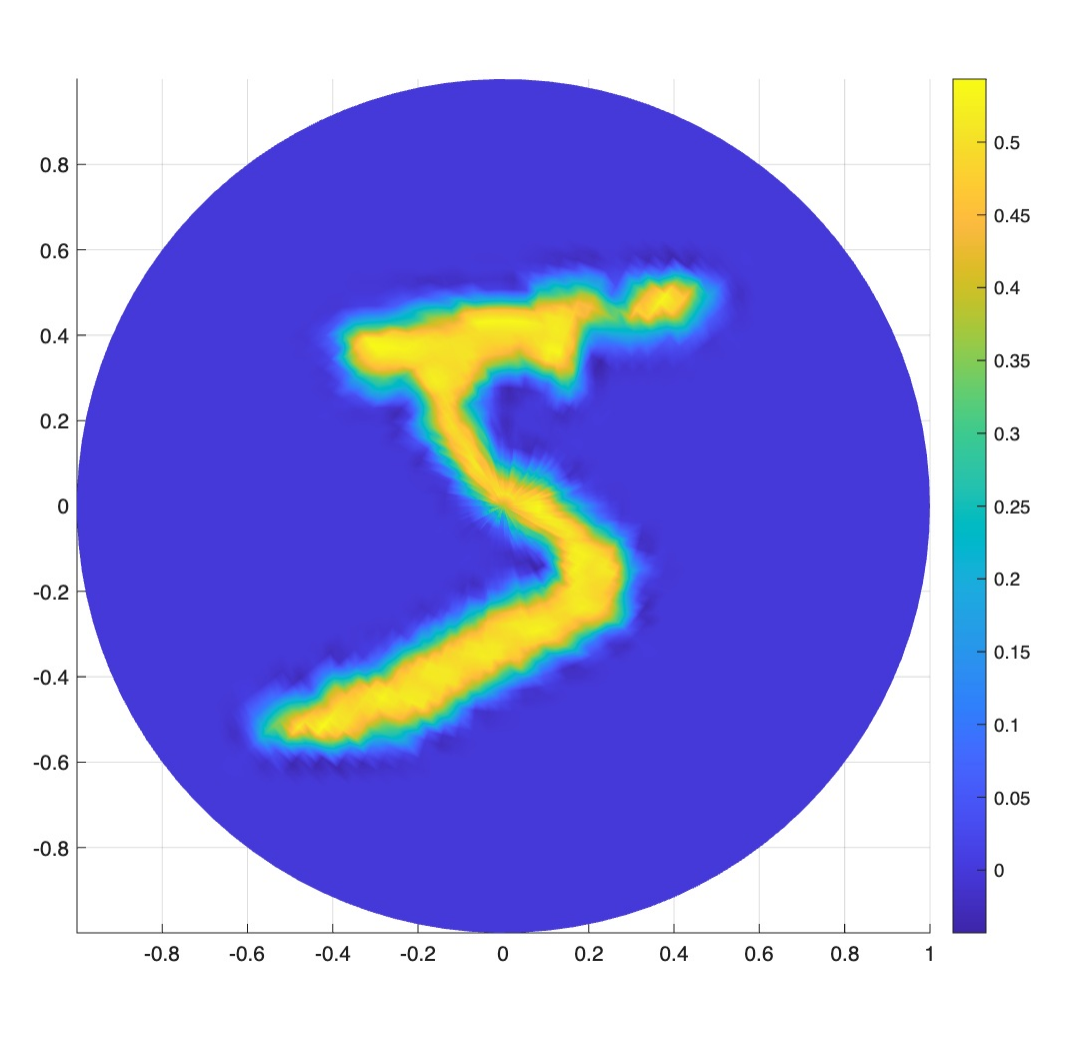} 
 \hspace{0.8em}
\includegraphics[width=0.30\textwidth]{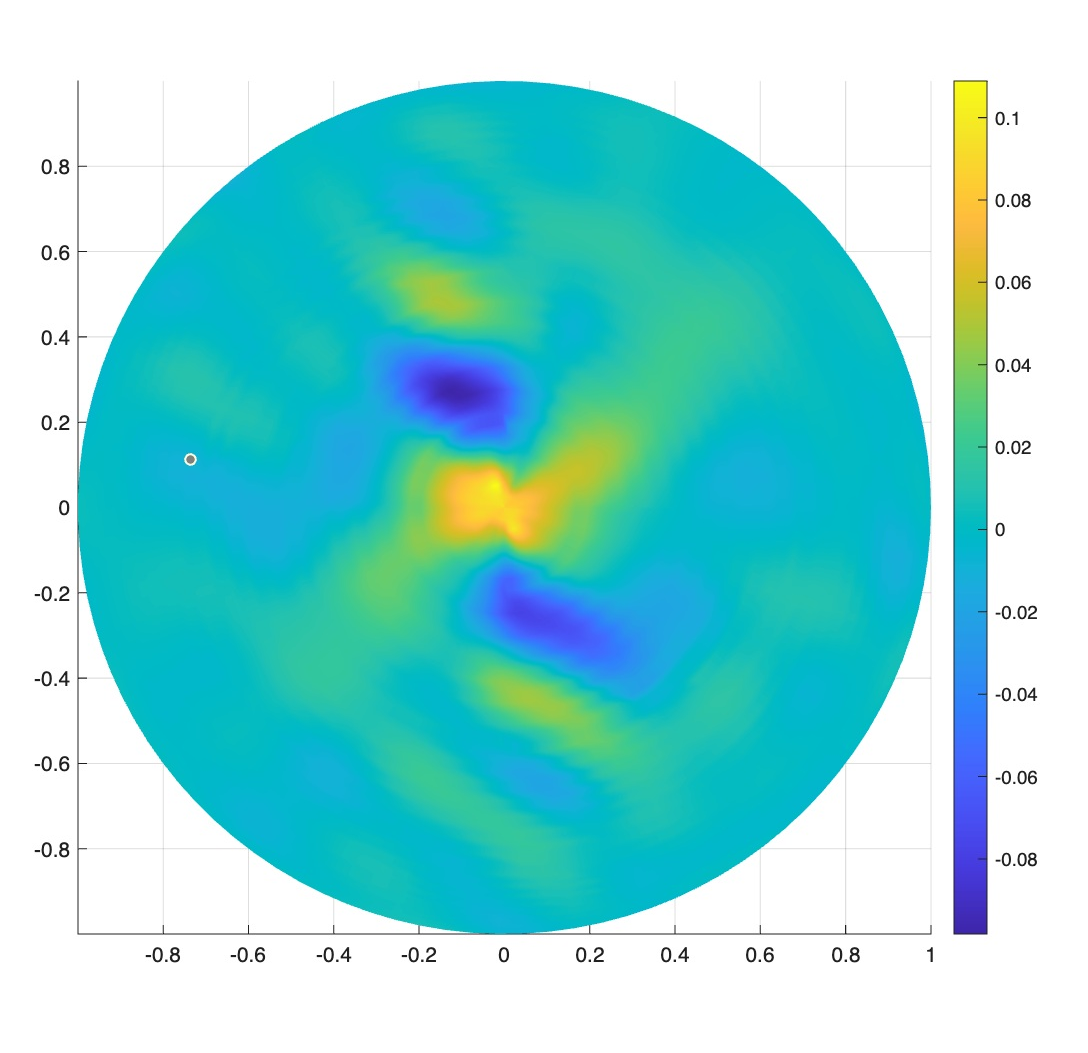} %
 \hspace{0.8em}
\raisebox{0.4cm}{\includegraphics[width=0.27\textwidth]{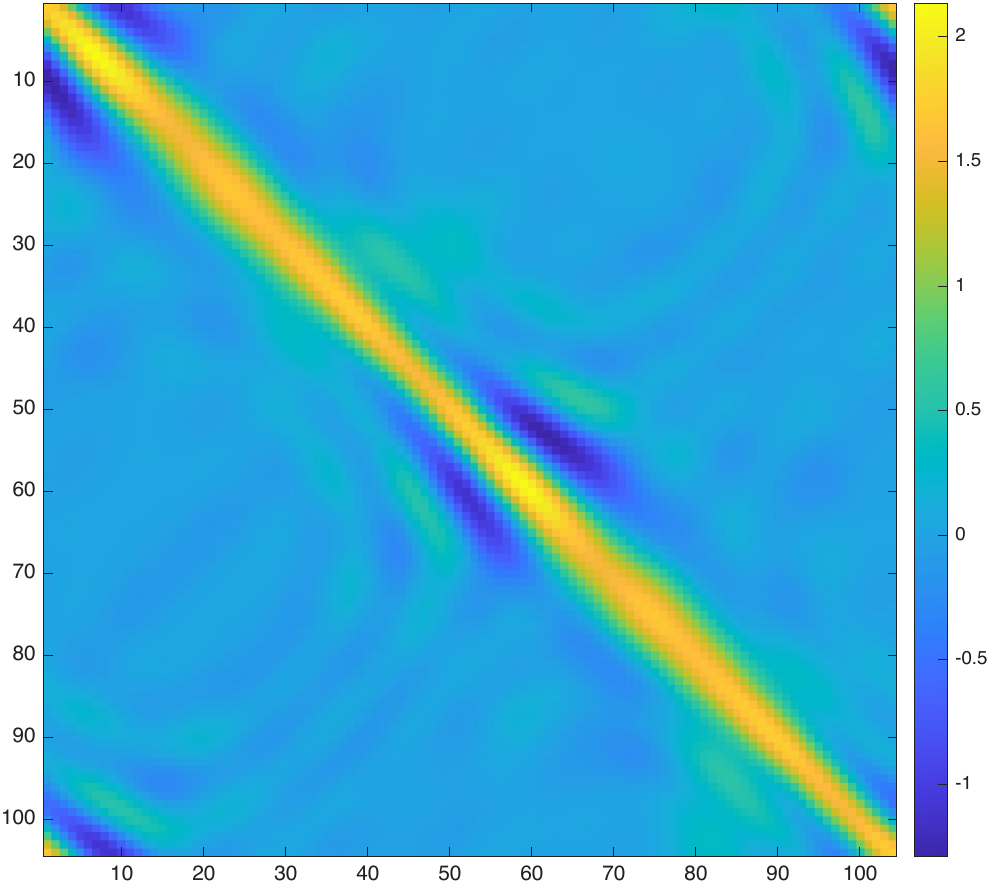}}\\
\includegraphics[width=0.30\textwidth]{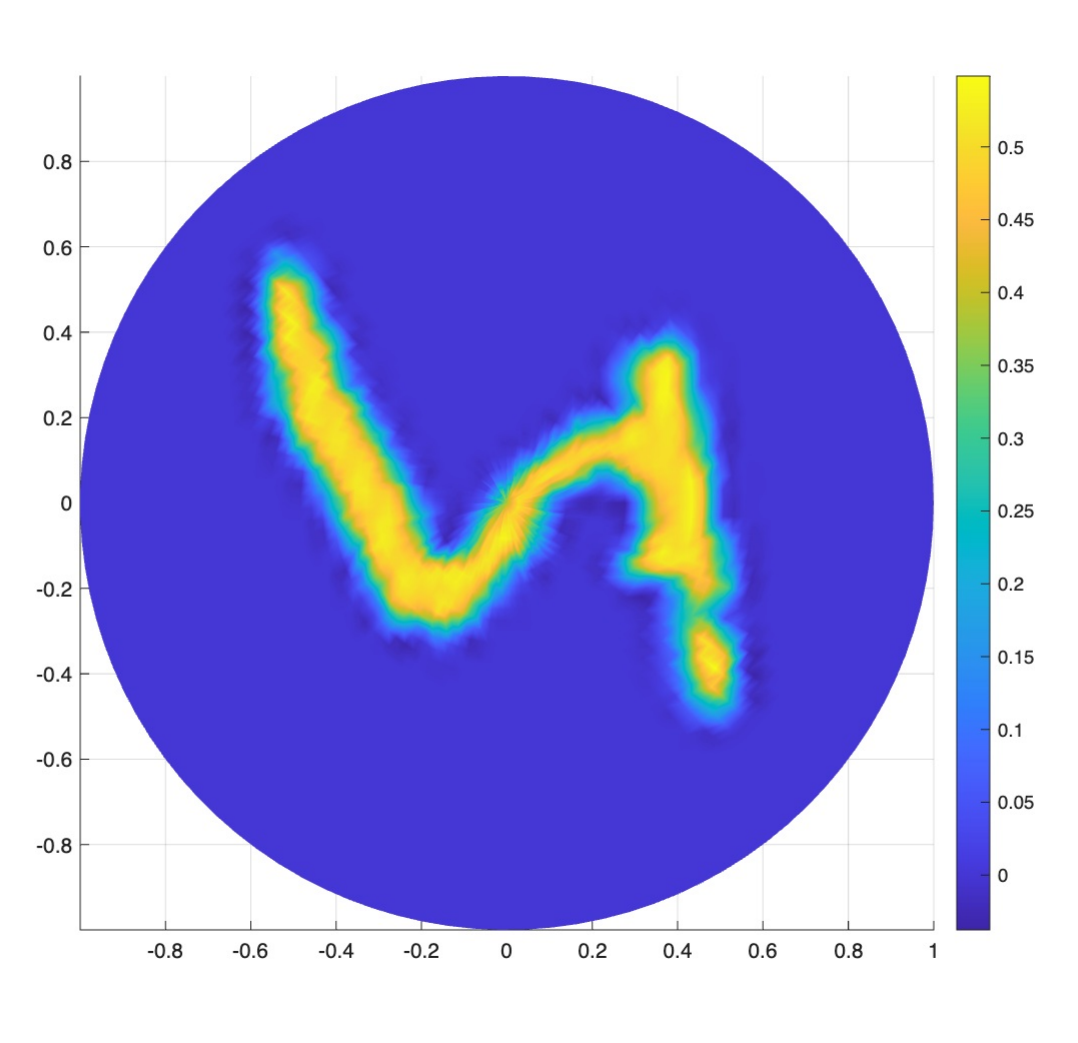} 
 \hspace{0.8em}
\includegraphics[width=0.30\textwidth]{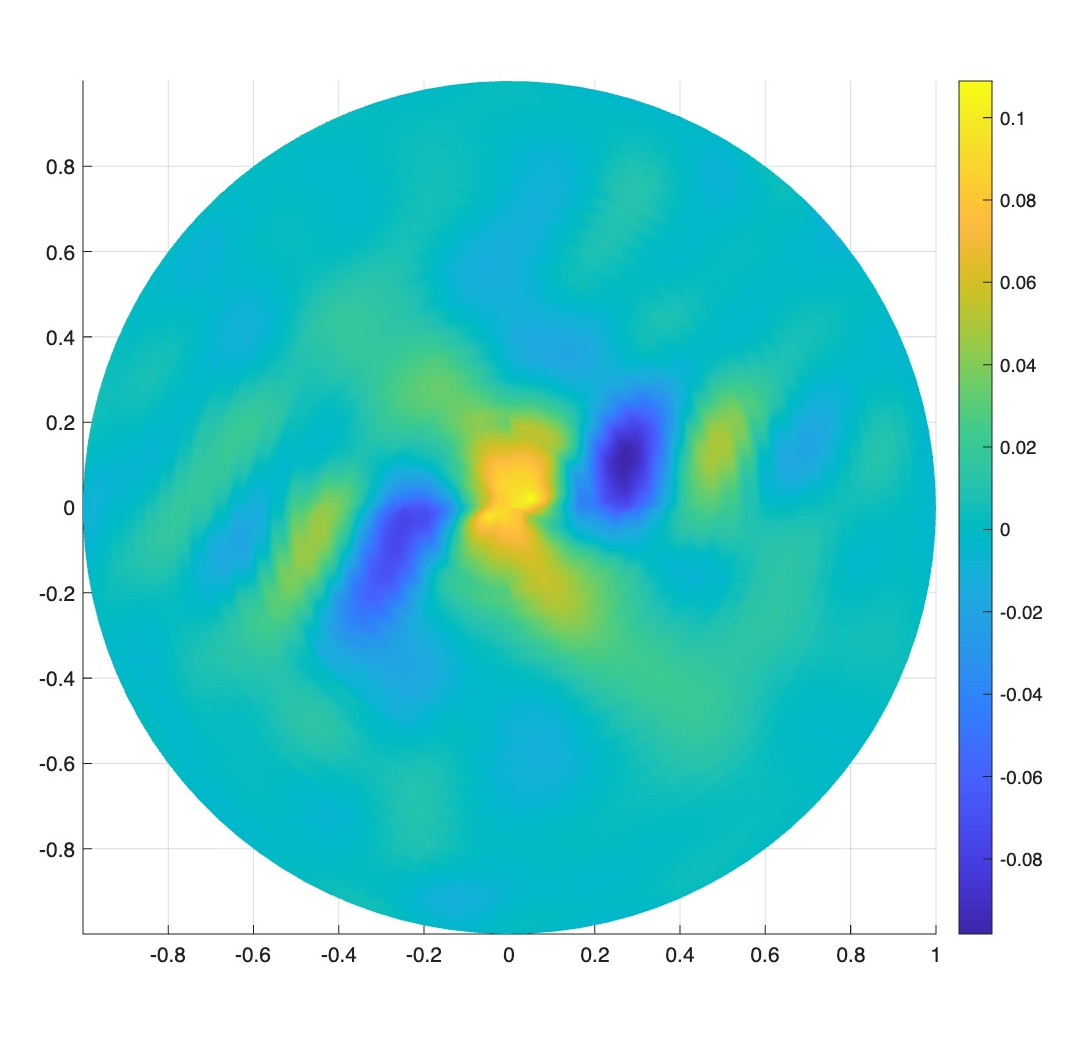} %
 \hspace{0.8em}
\raisebox{0.4cm}{\includegraphics[width=0.27\textwidth]{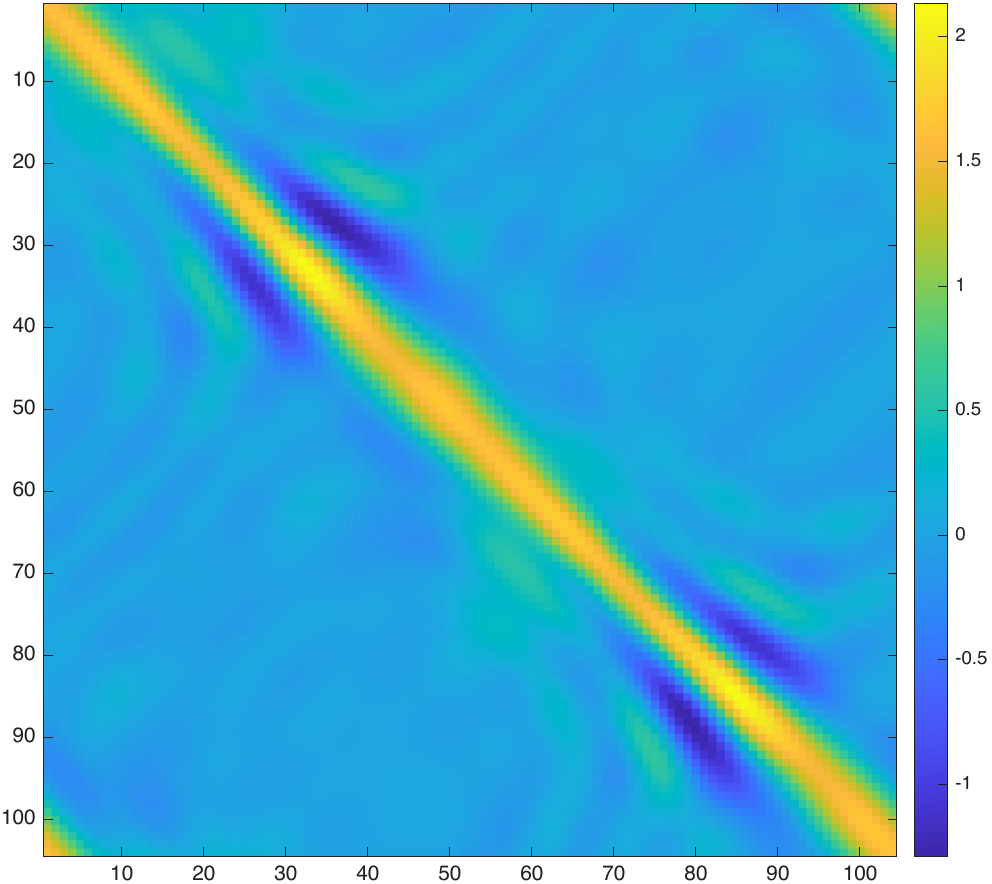}}
   \caption{Illustration of rotation-equivariance.  First column: contrast $q$; second column, imaginary  part of the processed datum \eqref{eqn: far field to processed}; third column, imaginary part of the far field datum. In the first row, we plot the number ``5'' and the corresponding  data. In the second   row, we plot the    data for a $\frac{3\pi}{2}$ rotated contrast.  }  \label{figure: equivariance of data}
\end{figure}

We demonstrate the rotation-equivariance  in \Cref{figure: equivariance of data}. The first row plots the original contrast $q$, the imaginary  part of the processed datum $\Re(u)$, and the imaginary part of the far field datum $\Re(u^\infty)$. The second row plots the rotated contrast and the associated data, it is observed that the processed datum rotates accordingly. It is worth noting that any rotation   of a processed datum corresponds to a physical contrast (i.e., a rotation of the original contrast). In contrast, if one attempts to naively rotate or translate  the far-field pattern  $u^\infty(\hat{x};\hat{\theta};k)$, it becomes unclear whether the transformed datum is a far-field pattern of a physical contrast. Furthermore,  a direct application of a translation-equivariant convolutional neural network seems less justified. Compared with the original far-field patterns, the processed data take explicit advantage of the intrinsic mathematical properties that eventually motivate us to explore appropriate neural networks in the next section.

}

{  
\section{Integrating neural network and   low-rank structure} \label{sec: neural network}
The recent work \cite{desai2025neural} proposes a neural network to model the data corrector which maps the full far field data to the Born far field data, and an inverse Born solver is then applied to find the unknown from the Born data. Fueled by this idea, our work proposes a rotation-equivariance-aware neural network for the data corrector, and integrates this neural network with an intrinsic low-rank structure associated with the inverse Born solver. This proposed method is referred to as the hybrid method (ULR). For comparison, we are also motivated to propose a two-step neural network (UU) which replaces the low-rank inverse Born solver by another rotation-equvariance-aware neural network.  

\subsection{Low-rank structure for the inverse Born solver and encoder-decoder architecture}\label{sec: low-rank inverse Born}
\subsubsection{Intrinsic low-rank structure for the  inverse Born solver}
To find a regularized  solution from the Born unknown-to-data map $\mathcal{F}_b$  \eqref{eqn: Born unknown-to-data}, we follow \cite{meng23data} and \cite{zhou2024exploring} based on  the intrinsic low-rank structure given by the disk  PSWFs. The PSWFs and their generalizations were studied in a series of work  \cite{Slepian64,slepian1961prolate} in the 1960s. We refer to \cite{osipov2013prolate} for a comprehensive introduction to the one dimensional PSWFs and to \cite{greengard2024generalized,GreengardSerkh18,ZLWZ20} for more recent studies on multidimensional generalizations of the PSWFs. 
 It was known \cite{Slepian64} that there exist real-valued  eigenfunctions $\{\psi_{m,n,l}(x;c)\}^{l\in\mathbb{I}(m)}_{m,n\in \mathbb{N}}$ of the restricted Fourier transform with parameter c such that
    \begin{align}\label{eigen_R_Fourier}
        [\mathcal{F}_{b}  \psi_{m,n,l}](x;c)=\int_{B}e^{ic x\cdot y}\psi_{m,n,l}(y;c) {\rm d} y =\alpha_{m,n}(c)\psi_{m,n,l}(x;c),\quad x\in B,
    \end{align}
    where $\mathbb{N}=\{0,1,2,3,\dots\}$ and
    \begin{eqnarray*}
        \mathbb{I}(m)=\left\{
            \begin{array}{cc}
                \{1\} & m=0 \\
                \{1,2\} & m \geq 1
            \end{array}\right..
    \end{eqnarray*}
We refer to $\psi_{m,n,l}(x;c)$   as the disk PSWFs and to $\alpha_{m,n}(c)$  as the prolate eigenvalues.

One of the most important properties of the disk PSWFs is  the so-called dual property. A direct calculation using \cite{Slepian64} shows that the disk PSWFs are also eigenfunctions of a Sturm-Liouville operator, i.e.,
\begin{equation}\label{sturm-liouvill}
    \mathcal{D}  [\psi_{m,n,l}](x)=\chi_{m,n} \psi_{m,n,l}(x),\quad x\in B,
\end{equation}
where
\begin{align*}
    \mathcal{D}  := -(1-r^2)\partial_r^2-\frac{1}{r}\partial_r+3r\partial_r-\frac{1}{r^2}\Delta_0+c^2r^2
\end{align*}
and the Laplace–Beltrami operator $\Delta_0= \partial^2_\theta $ is the spherical part of Laplacian $\Delta$. We further refer to $\chi_{m,n}(c)$  as the Sturm-Liouville eigenvalue.

For a complete paper, we state the following preliminaries and we refer the details to \cite{greengard2024generalized,GreengardSerkh18,ZLWZ20} and to \cite{zhou2024exploring} for its application to inverse scattering.

\begin{lemma} \label{lemma: SL}
Let $c>0$ be a positive real number.
\begin{itemize}
\item[]
(a)~$\{\psi_{m,n,l}(x;c)\}^{l\in\mathbb{I}(m)}_{m,n\in \mathbb{N}}$ forms a complete and orthonormal system of $L^2(B)$, i.e., for any $m,~n,~m',~n'\in\mathbb{N},~l\in\mathbb{I}(m)$ and $l'\in\mathbb{I}(m')$, it holds that
   \begin{equation*}
       \int_{B(0,1)}\psi_{m,n,l}(y;c)\psi_{m',n',l'}(y;c) {\rm d} y=\delta_{m m'}\delta_{n n'}\delta_{l l'},
   \end{equation*}
where $\delta$ denotes the Kronecker delta. 
\item[]
(b)~The corresponding  Sturm-Liouville eigenvalues $\{\chi_{m,n}\}_{m,n\in \mathbb{N}}$ in \eqref{sturm-liouvill} are real positive which  are ordered for fixed $m$ as follows
       \begin{equation*}
       0<\chi_{m,0}(c)<\chi_{m,1}(c)<\chi_{m,2}(c)<\cdots .
   \end{equation*}
\item[]
(c)~Every prolate eigenvalue $\alpha_{m,n}(c)$ is  non-zero, and $\lambda_{m,n} = |\alpha_{m,n}(c)|$ can be arranged for fixed $m$ as
    \begin{equation*}
       \lambda_{m,n_1}(c)>\lambda_{m,n_2}(c)>0, \quad \forall n_1<n_2.
   \end{equation*}
   Moreover $\lambda_{m,n}(c)\longrightarrow 0$  as  $m,n\longrightarrow +\infty$.
\end{itemize}
\end{lemma}
A direct evaluation of the leading disk PSWFs solely based on the restricted Fourier transform is not reliable as the leading prolate eigenvalues have numerically the same amplitude, cf. \cite{zhou2024exploring}. Instead, we follow \cite{Slepian64} to evaluate the disk PSWFs  using the Sturm-Liouville differential operator to ensure stability and efficiency. More precisely in polar coordinates, each disk PSWF  $\psi_{m,n,l} (x;c)$ can be obtained by separation of variables by (cf. \cite{Slepian64} or \cite{meng23data})
\begin{equation*}
    \psi_{m,n,l}(x;c)={r^m}\varphi_{m,n}(2{ r}^2-1;c)Y_{m,l}(\hat{x}),\quad x\in B(0,1),
\end{equation*}
where $x=r\hat{x}=(r\cos{\theta},r\sin{\theta})^T$ and the spherical harmonics $Y_{m,l}(\hat{x})$ are given by
\begin{align}\label{spherical_harmonic}
    Y_{m,l}(\hat{x})=\left\{
            \begin{array}{cc}
                \frac{1}{\sqrt{2\pi}}, & m=0,l=1 \\
                \frac{1}{\sqrt{\pi}}\cos{m\theta}, & m\geq 1,l=1 \\
                \frac{1}{\sqrt{\pi}}\sin{m\theta},& m\geq 1,l=2
            \end{array}\right. .
\end{align}
An efficient method to evaluate the disk PSWFs is to expand $\varphi_{m,n}(\eta;c)$ by normalized Jacobi polynomials $\{P^{(m)}_{j}(\eta)\}^{j\in\mathbb{N}}_{\eta\in (-1,1)}$,
\begin{equation}\label{expansion}
    \varphi_{m,n}(\eta;c)=\sum_{j=0}^{\infty}\beta_j^{m,n}(c) P^{(m)}_j(\eta).
\end{equation}
The  coefficients $\{\beta_j^{m,n}(c)\}$ can be solved via a tridiagonal linear system, see for instance, \cite{ZLWZ20} and \cite{greengard2024generalized}. Here the normalized Jacobi polynomials $\{P_n^{(m)}(x)\}_{x\in (-1,1) }$ can be obtained through the three-term recurrence relation
\begin{align*}
         P^{(m)}_{n+1}(x)&=\frac{1}{a_n} [ (x-b_n) P_n^{(m)}(x)-a_{n -1}P^{(m)}_{n-1}(x) ],\quad n\geq 1\\
    P^{(m)}_{0}(x)&=\frac{1}{h_0 },\quad P^{(m)}_{1}(x)=\frac{1}{2h_1  }[(m+2)x-m],
\end{align*}
where $h_{0}= \frac{1}{\sqrt{2(m+1)}}$, $h_{1}= \frac{1}{\sqrt{2(m+3)}}$, and
\begin{align*}
   \left\{
            \begin{array}{cc}
                a_{n}=&\frac{2(n+1)(n+m+1)}{(2 n+m+2)\sqrt{(2 n+{m+1})(2 n+m+3)}} \\
                b_{n}=&\frac{m^{2}}{(2 n+m)(2 n+m+2)} 
            \end{array}\right. , \qquad n\in\mathbb{N}.
\end{align*}
One of the advantages of the disk PSWFs is that one can obtain the following explicit a priori estimate, in terms of computable prolate and Sturm-Liouville eigenvalues, cf. \cite{meng23data} and \cite{zhou2024exploring}. Let $\langle ,\rangle_{L^2(B)}$ denote the inner product in $L^2(B)$.
\begin{lemma} \label{lemma: priori estimate}
Let  $u_b = \mathcal{F}_b q$ given by \eqref{eqn: Born unknown-to-data} and $u^{\delta}\in L^2(B)$ is such that $\|{   u^{\delta}} - {u_b }\|_{L^2(B)} \le \delta$.
\begin{itemize}
    \item[(a)] If  $q \in  H^s(B)$ with $0<s<1/2$,  let
 \begin{eqnarray*}
\quad q^{\delta,\alpha}
= \sum_{\chi_{m,n}(c)<\alpha^{-1}}   \frac{1}{ \alpha_{m,n}(c) } \left\langle {   u^{\delta}},  \psi_{m,n,\ell}(\cdot;c) \right\rangle_{B}  \psi_{m,n,\ell}(\cdot;c),
\end{eqnarray*}
and $\beta(\alpha)= \min_{ \chi_{m,n}(c)<\alpha^{-1}} \left\{   \left|\alpha_{m,n}(c)\right| \right\}$. Then
\begin{eqnarray*}
\|{q^{\delta,\alpha}} - q\|_{L^2(B)}  \le  { \frac{\delta}{{\beta(\alpha)}}} + { (4\alpha C) ^{s/2} (1+c^2)^{s/2}   \| q \|_{H^s\left(B\right)}},
\end{eqnarray*}
where $C\ge\sqrt{3}$ is a positive constant independent of $\delta$, $\alpha$, $s$ and $c$.
\item[(b)] 
   If  $q \in \mbox{span}\{ \psi_{m,n,\ell}(\cdot;c): {|\alpha_{m,n}(c)|} > \eta\}$, {let
$$
q^\eta = \sum_{|\alpha_{m,n}(c)|>\eta}   \frac{1}{ \alpha_{m,n}(c) } \left\langle {   u^{\delta}},  \psi_{m,n,\ell}(\cdot;c) \right\rangle_{B}  \psi_{m,n,\ell}(\cdot;c),
$$
}
then $\|{ q^\eta} - q\|_{L^2(B)}  \le  { {\delta}/{\eta}}$.
\end{itemize}
\end{lemma}
For a general contrast and setting $\beta(\alpha) = \delta ^{\gamma}$ with $\gamma \in (0,1)$, the a priori estimate of Lemma \ref{lemma: priori estimate}(a) is in the spirit of increasing stability or H\"{o}lder-Logarithmic stability, since $\alpha_{m,n}$ decays exponentially to zero while $\chi_{m,n}$ grow polynomially to infinity; for a contrast in a low-rank space, the stability estimate is of Lipschitz, cf. Lemma \ref{lemma: priori estimate}(b). Intrinsically, only the reconstruction in the low-rank space is stable since otherwise it is always possible to find a contrast (i.e., $\psi_{m,n,\ell}$ with very small prolate eigenvalue $|\alpha_{m,n}|$) whose reconstruction is unstable. As a result, the high-frequency noise in the imperfect data can be filtered out by projecting the data onto a low-rank space $\mbox{span}\{ \psi_{m,n,\ell}(\cdot;c): {|\alpha_{m,n}(c)|} > \eta\}$ with a suitable spectral cutoff  parameter $\eta$.
\begin{figure}[htbp]
{\centering  
\includegraphics[width=1\linewidth]{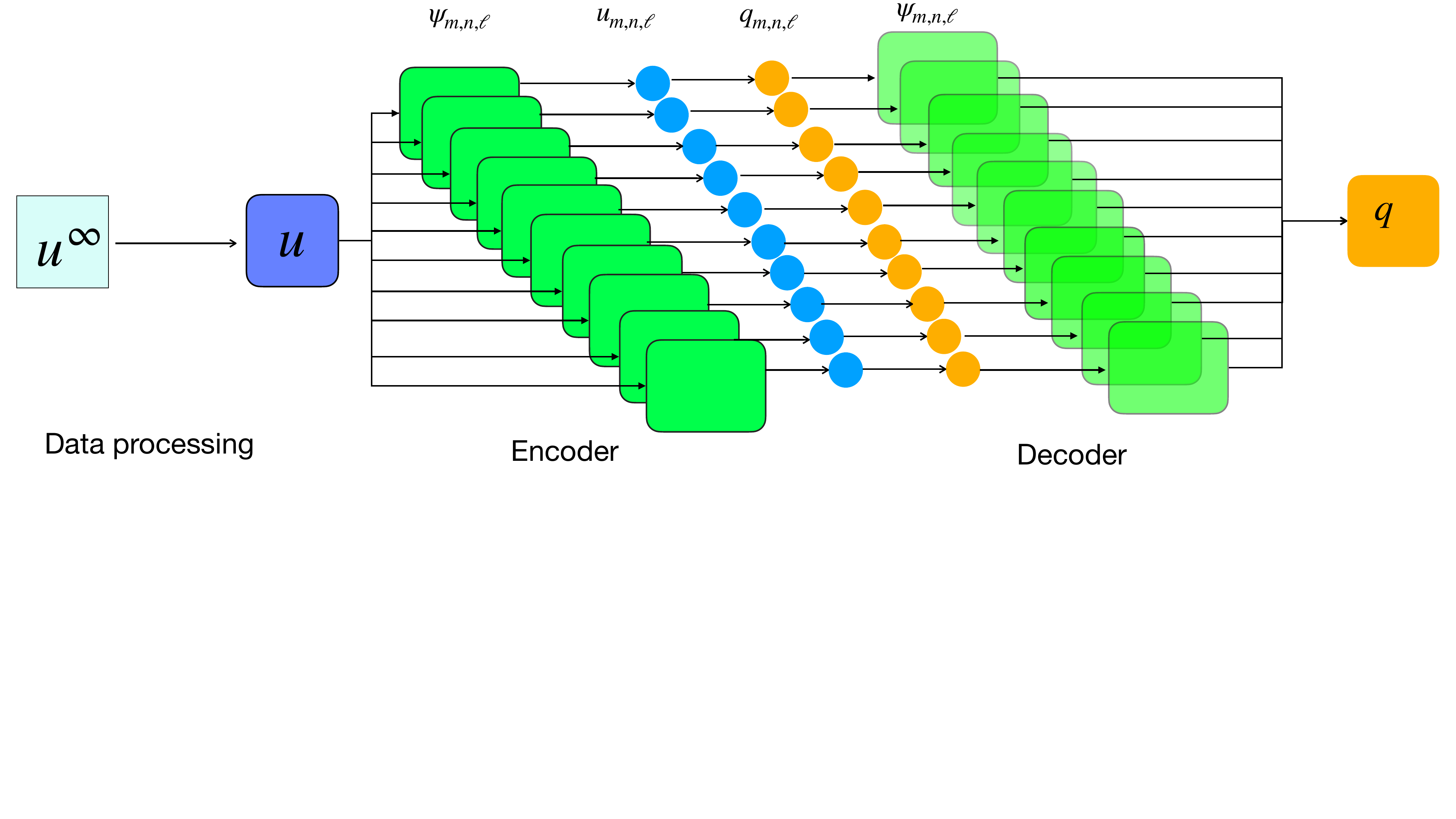} 
} 
\vspace{-8\baselineskip}
\caption{Illustration of low-rank inverse Born solver from the viewpoint of encoder-decoder.}
\label{figure: encoder_decoder}
\end{figure}
\subsubsection{A training-free encoder-decoder}
The above low-rank-structure regularized approach can be reinterpreted as a training-free encoder-decoder. Specifically, the low-rank regularized solution is obtained as follows.
\begin{itemize}
    \item The far-field dataset $\{u^\infty(\hat{x},\hat{\theta};k): \hat{x},\hat{\theta} \in \mathbb{S}\}$ is processed to the processed dataset  $\{u(p;c): p \in B \}$ according to \eqref{eqn: far field to processed}.
    \item Information of   $u$ is encoded into projections $u_{m,n,\ell}=\langle u,\psi_{m,n,\ell}\rangle_{L^2(B)}$ for $\{ (m,n,\ell): |\alpha_{m,n}(c)| > \eta, \,\ell \in \mathbb{I}(m)\}$ with regularization parameter $\eta$.
    \item Information about the contrast is decoded as
    $$q_{m,n,\ell}^\eta = \frac{u_{m,n,\ell}}{\alpha_{m,n}} \mbox{ for }  (m,n,\ell) \mbox{ where }|\alpha_{m,n}(c)| > \eta, \,\ell \in \mathbb{I}(m).
    $$
    \item The contrast is then reconstructed as
$$
q^\eta(x)=\sum_{|\alpha_{m,n}|>\eta}q_{m,n,\ell}^\eta \psi_{m,n,\ell}(x).
$$
\end{itemize}
The  encoder-decoder interpretation of the low-rank structure is illustrated in \Cref{figure: encoder_decoder}. This low-rank structure is intrinsic to the Born unknown-to-data map, leading to a robust method  and high-frequency noise filtering. We propose to integrate this low-rank structure with a suitable neural network for the data corrector in the next section. 
}

\subsection{Rotation-equivariance-aware neural network for the data corrector} \label{sec: neural data corrector}

\begin{figure}[htbp]
\centering 
\subfloat[Ground truth $q$]{\includegraphics[width=0.3\linewidth]{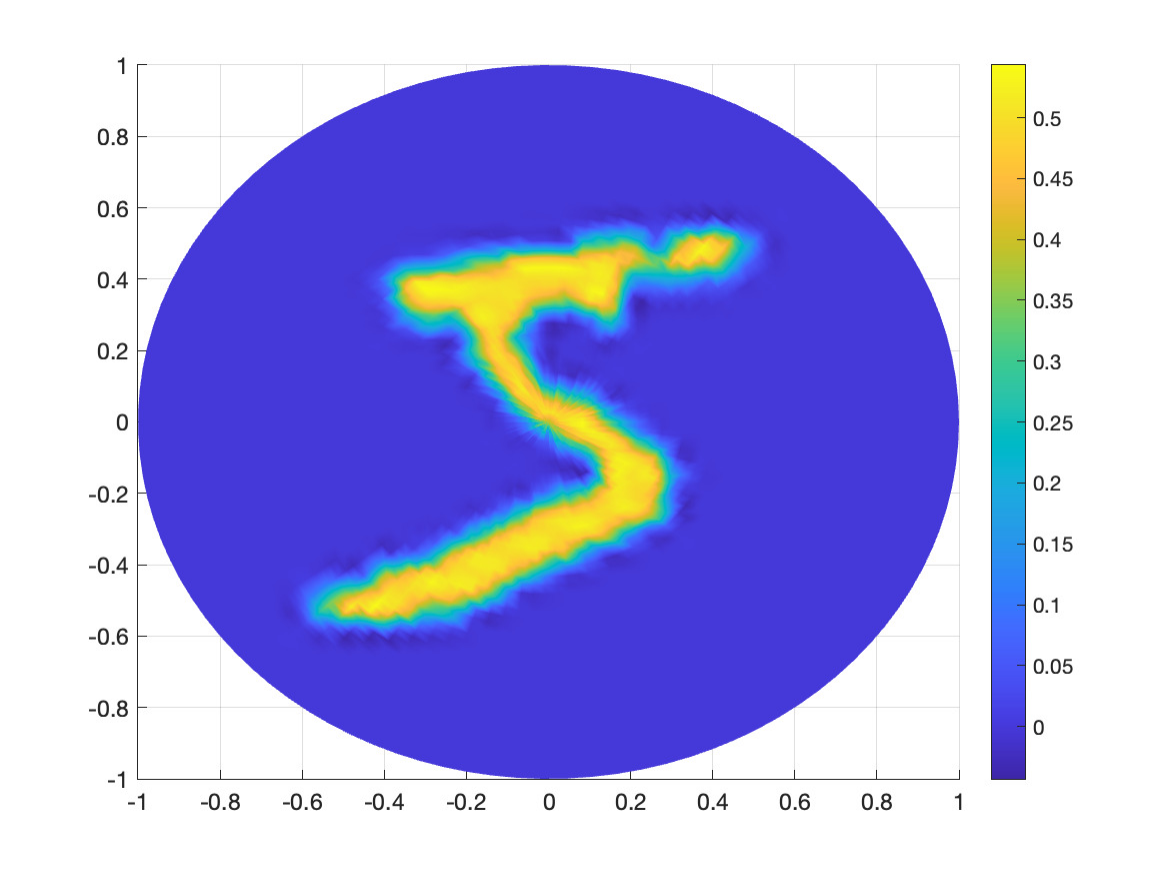}}
\subfloat[$\Re u$]{\includegraphics[width=0.3\linewidth]{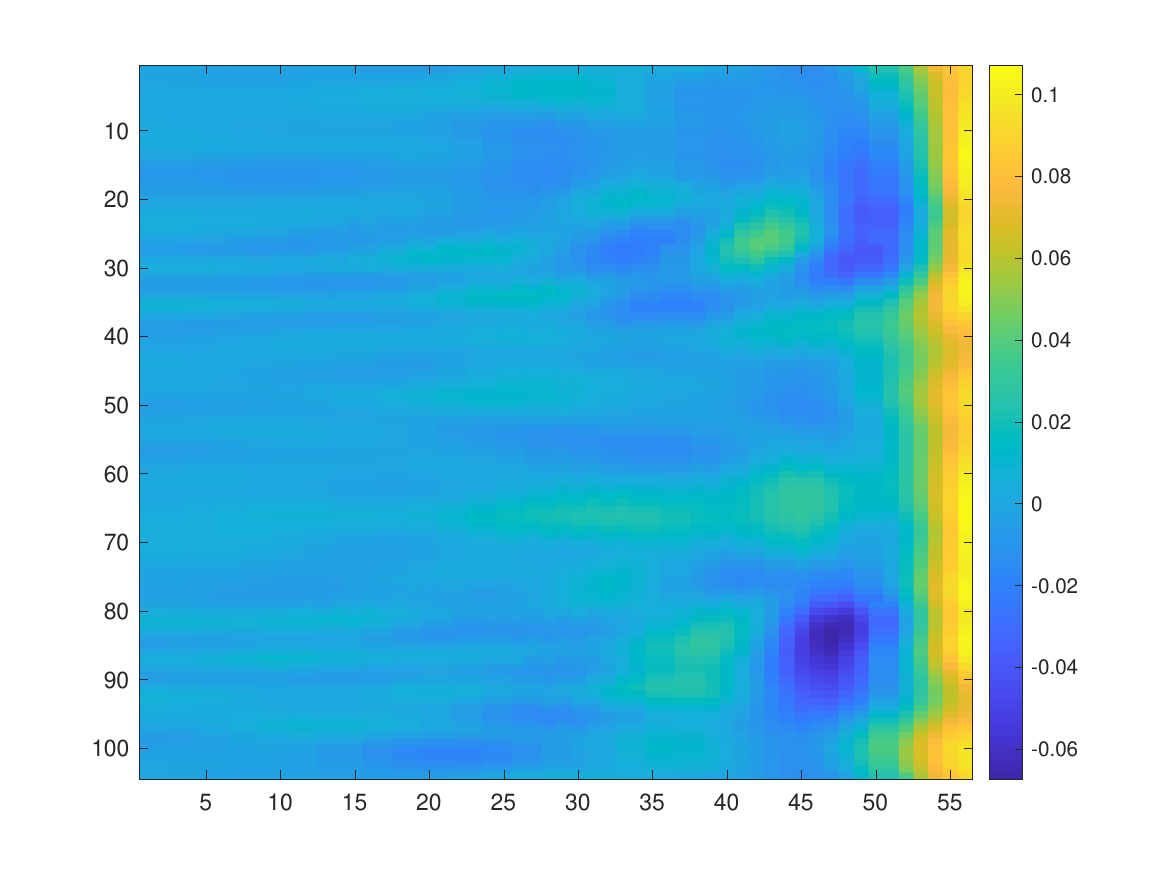}}
\subfloat[$\Im u$]{\includegraphics[width=0.3\linewidth]{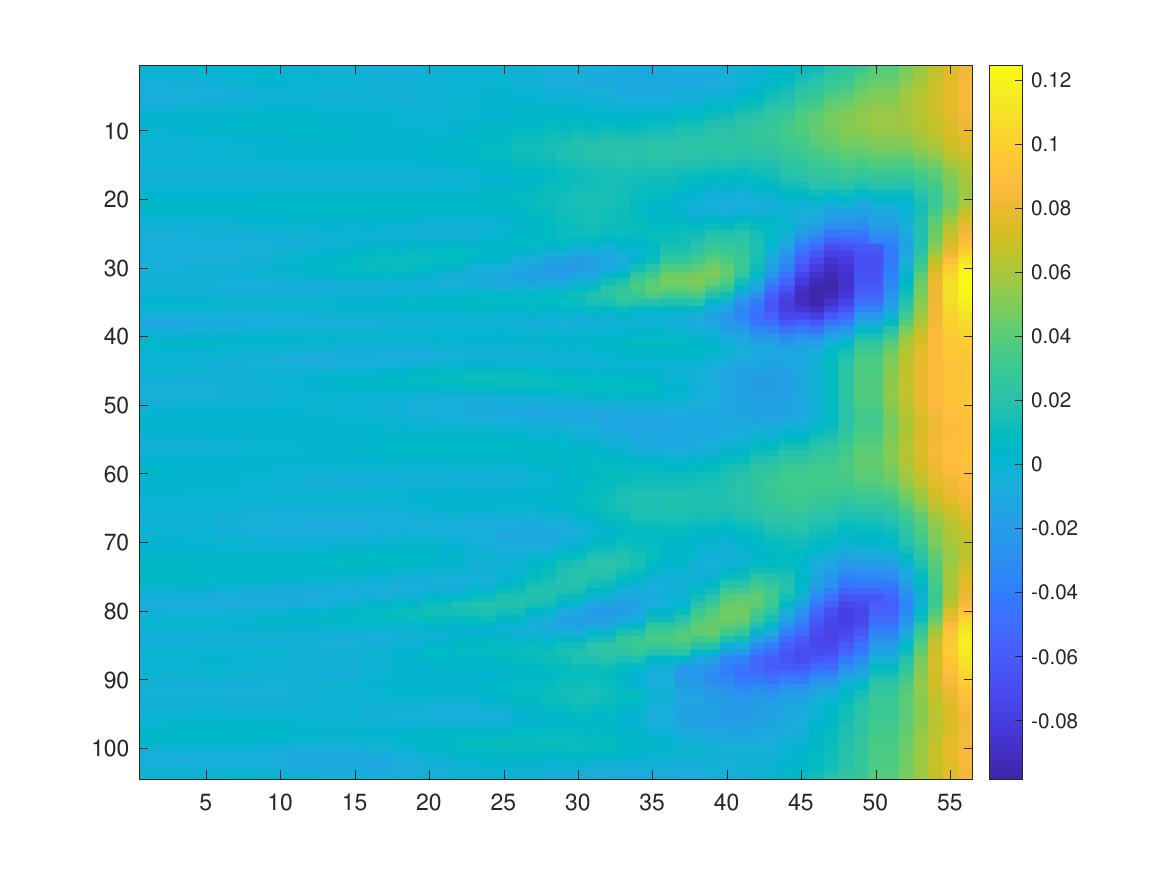}}
\\
\subfloat[New rectangular image of $q$]{\includegraphics[width=0.3\linewidth]{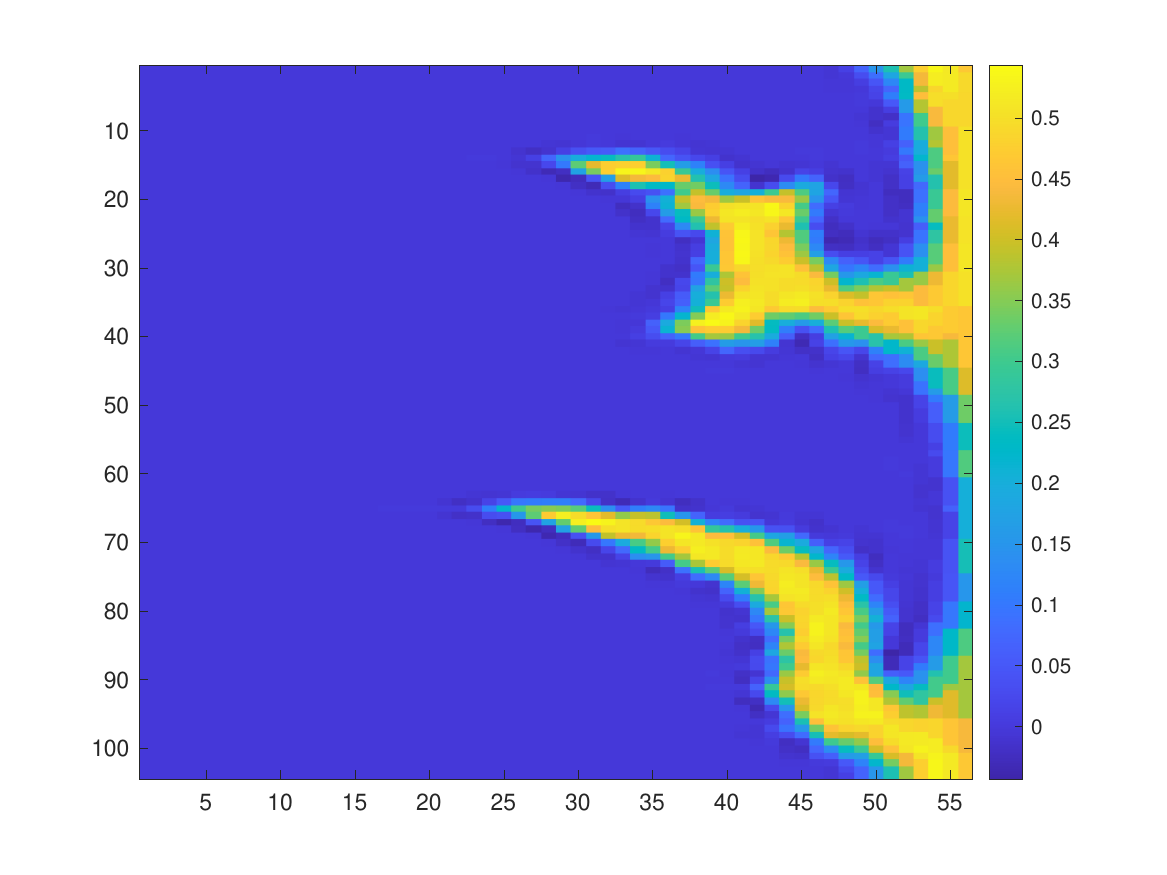}}
\subfloat[$\Re u_b$]{\includegraphics[width=0.3\linewidth]{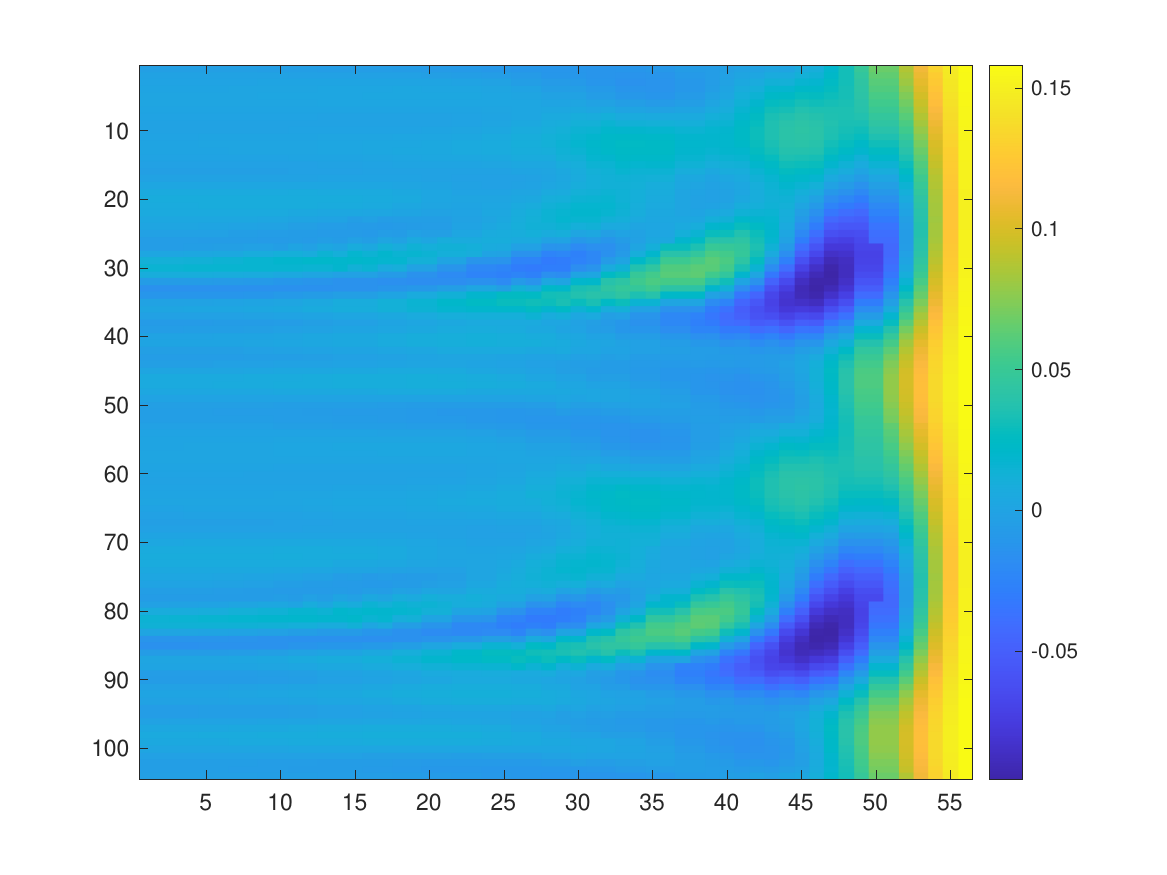}}
\subfloat[$\Im u_b$]{\includegraphics[width=0.3\linewidth]{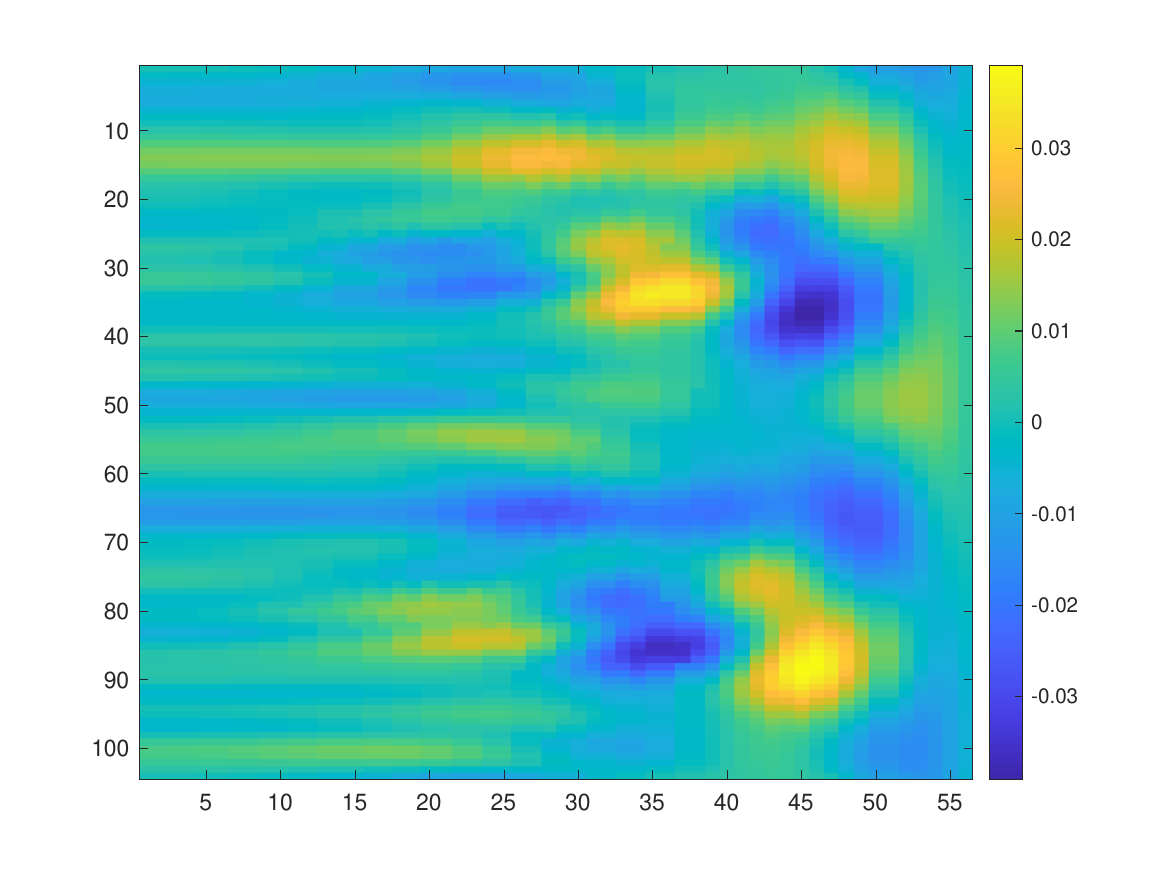}}

  \caption{Illustration of training data.
  Input of neural network $U_1$: (b) and (c); output of neural network $U_1$: (e) and (f). Input of neural network $U_2$: (e) and (f); output of neural network $U_2$: (d).  }  \label{figure: training sample 1}
\end{figure}
{  
We propose a suitable neural network for the data corrector inspired by the reciprocity relation and the rotation-equivariance in Section \ref{sec: math property}.  To begin with, we simply regard each contrast $q$ (or $\Re(u)$ and $\Im(u)$) as an image, and each circular image in $B$ will be transformed to a new (rectangular) image in polar coordinates $(r,\theta)$. We plot in Figure \ref{figure: training sample 1}(a) the circular image of the contrast and in Figure \ref{figure: training sample 1}(d) the new rectangular image (with horizontal axis in $r$ and  vertical axis in $\theta$). Correspondingly, a rotation of the circular image in $B$ corresponds to a translation in $\theta$-variable of the new rectangular image; moreover, the new rectangular image is periodic in $\theta$-variable, which has to be taken consideration of the neural network design. The same principle applies to the new rectangular image of the processed datum in polar coordinates, cf. Figure \ref{figure: training sample 1}(b)(c).

Note that  the rotation-equivariance in \Cref{figure: equivariance of data}  corresponds to translation-equivariance in rectangular images, this motivates us to consider convolution operations that respect translation-equivariance \cite{2015U}.
To further respect the periodicity of the new rectangular image and to design the data corrector as an image-to-image neural network, we propose a U-Net with an appropriate circular padding in the $\theta$-variable.  The main ingredients of U-Net include convolutional layers, batch normalization, ReLU (Leacky ReLU), max pooling, transposed convolution and skip connection, cf. \Cref{figure: procedures}(b) for an illustration of the neural network structure.  For a more  comprehensive description of the U-Net, we refer to \cite{2015U}.

Specifically, in each convolutional layer, we utilize a (3 × 3) convolution with circular-padding defined by \eqref{circular padding}, which adapts  the periodicity in $\theta$-variable. Numerically, this choice of circular padding improves the continuity of the reconstructed contrast in the vicinity of $\theta=0$.

\begin{equation}\label{circular padding}
    \text{Circular Pad}\left(\begin{bmatrix}
 
    a_{1,1}&  a_{1,2} & \cdots & a_{1,M} \\
   a_{2,1}&  a_{2,2} & \cdots & a_{2,M} \\
 \vdots&  \vdots &\ddots  &   \vdots\\
  a_{N,1}&  a_{N,2} & \cdots & a_{N,M} 
  \end{bmatrix} \right)
  =
  \begin{bmatrix}
  0&  a_{N,1}&  a_{N,2} & \cdots & a_{N,M} & 0\\
  0&  a_{1,1}&  a_{1,2} & \cdots & a_{1,M} & 0\\
 0&  a_{2,1}&  a_{2,2} & \cdots & a_{2,M} & 0\\
 \vdots&  \vdots & \vdots & \ddots & \vdots & \vdots\\
 0&  a_{N,1}&  a_{N,2} & \cdots & a_{N,M} & 0\\
  0&  a_{1,1}&  a_{1,2} & \cdots & a_{1,M} & 0\\
\end{bmatrix}.
\end{equation}
}

\subsection{A hybrid method ULR and a two-step neural network UU}
{  
In this work, we propose a hybrid method by integrating the neural network model for the data corrector in Section \ref{sec: neural data corrector} and  a low-rank-structure-assisted  inverse Born solver in Section \ref{sec: low-rank inverse Born}. This hybrid approach accommodates the nonlinearity and ill-posedness of the inverse scattering problem,  incorporates the properties of rotation-equivariance and reciprocity relation, and facilitates high-frequency noise filtering and robustness. This hybrid method is illustrated in Figure \ref{figure: procedures}(a).

Furthermore, we compare our proposed algorithm against two alternative methods. The first method replaces the low-rank-structure-assisted   inverse Born solver by another U-Net, leading to a two-step neural network denoted as UU; particularly, the second neural network $\mbox{U}_2$, referred to as the ``neural inverse Born solver'', is trained to map the processed Born data to the unknown contrast.  The second alternative, denoted as U, is a black-box neural network, consisting of a single U-Net that maps the far-field pattern (instead of processed data) to the contrast directly. For this black-blox neural network, we adopt a ``naive training strategy" -- training without any data processing -- under the hypothetical scenario when mathematical insights on inverse scattering are not taken into consideration. All three methods (our algorithm ULR, UU, and the black-box U) are schematically illustrated in \Cref{figure: procedures}(a). 
We use Figure \ref{figure: training sample 1} to illustrate the input and output of these neural networks, where  a sample contrast (resp. its polar representation) is plotted in (a) (resp. (d)),  the real (resp. imaginary) part of the processed datum is plotted in (b)(resp. (c)),  and the real (resp. imaginary) part of the processed Born datum is plotted in (e)(resp. (f)), respectively. The input of the neural network $\mbox{U}_1$ is a tensor of dimension $N_1 \times N_2 \times 2$ (illustrated by Figure \ref{figure: training sample 1}(b)(c)) and the corresponding output is still a   tensor  of dimension $N_1 \times N_2 \times 2$ (illustrated by Figure \ref{figure: training sample 1}(e)(f)); the input of the neural network $\mbox{U}_2$ a   tensor  of dimension $N_1 \times N_2 \times 2$ (illustrated by Figure \ref{figure: training sample 1}(e)(f)) and the output is a tensor of dimension $N_1 \times N_2$ (illustrated by Figure \ref{figure: training sample 1}(d); the input of the neural network $\mbox{U}$ the far field pattern  which is a   tensor  of dimension $N_{\rm obs} \times N_{\rm inc} \times 2$  and the output is a tensor of dimension $N_1 \times N_2$ (illustrated by Figure \ref{figure: training sample 1}(a); here $N_{\rm obs}$ and $N_{\rm inc}$ denotes the number of observation and incident directions, respectively.
 }

\begin{figure}[htbp]
{\centering  
\includegraphics[width=1\linewidth]{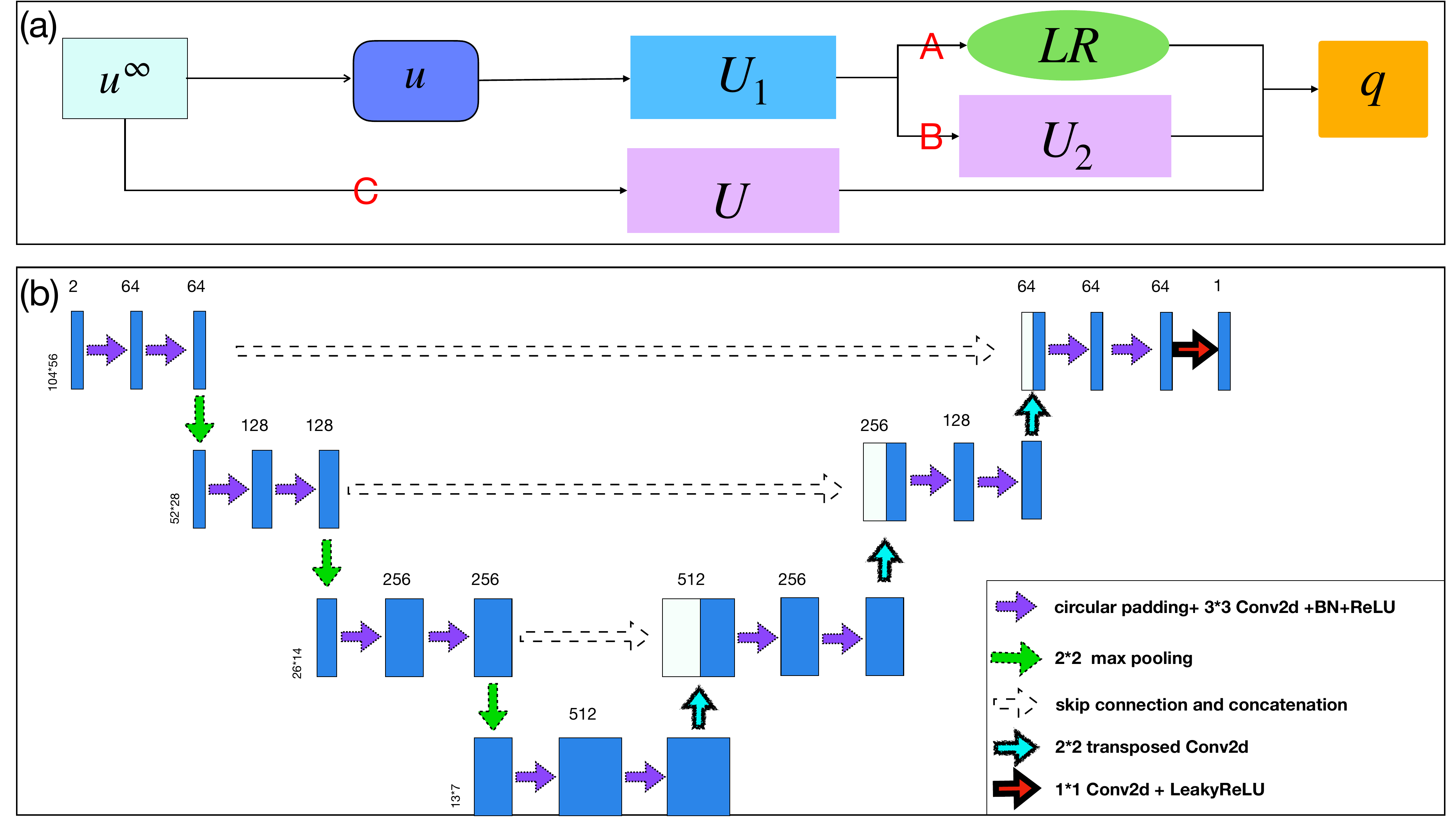} 
}
\caption{Illustration of three algorithms. (a)~A: proposed ULR; B: proposed UU; C: a black-box approach. (b) Structure of  the U-Net: each blue or white volume represents multi-channel feature maps whose width, length, and depth are displayed; the operation of each arrow  is illustrated at the bottom right.}\label{figure: procedures}
\end{figure}

\subsection{Training of Neural networks}
\subsubsection{Preparation of training dataset}
 
{  
A dataset consisting of the input-output $\{(u,u_b)\}$ will be used to train the neural network $\mbox{U}_1$; similarly, a dataset consisting of the input-output $\{(u_b,q)\}$ is to train $\mbox{U}_2$ and a dataset consisting of  the input-output $\{(u^\infty,q)\}$ is to train the black-box neural network $\mbox{U}$. The full far-field data and the Born far-field data are generated by IPscatt \cite{burgel2019algorithm}, an open-source Matlab toolbox, on a desktop computer (Apple M1 Max chip, 64 GB of RAM).  The  wavenumber is $k=16$, and there are $N_{\rm inc}=104$ incident directions and $N_{\rm obs}=104$ observation directions  uniformly distributed on the unit circle. The training data are noise-free and the degree of nonlinearity $\mbox{rel}(k)$ ranges from $20\%$ to $160\%$ approximately. 
Here the degree of nonlinearity is proposed to quantify the degree of nonlinearity  according to
$$
\mbox{rel}(k)=\frac{\|U^s-U_b^s\|_{\rm F}}{\|U^s\|_{\rm F}}$$
where $U^s=(u^s(x_i;\hat{\theta}_j))_{N_p\times N_{\rm inc}}$ and $U_b^s=(u_b^s(x_i;\hat{\theta}_j))_{N_p\times N_{\rm inc}}$ are two matrices: $u^s(x_j;\hat{\theta}_j)$ is the scattered field evaluated at $x_i$ due to the incident wave $e^{ik\hat{\theta}_j\cdot x}$ and $u_b^s(x_j;\hat{\theta}_j)$ is its Born approximation.  Here $\|\cdot\|_{\rm F}$ represents the Frobenius norm and $N_p=104^2$ is the number of equispaced grids in the computational domain $\Omega=[-1,1]\times [-1,1]$. 
The processed data are given by the far field data according to
\begin{eqnarray*}
       \widetilde{u}_b(p_{m,n})&\approx&\dfrac{1}{k^2}u_b^{\infty}(\hat{x}_{i^*};\hat{\theta}_{j^*}) \quad \mbox{and} \quad 
       \widetilde{u}(p_{m,n})\approx\dfrac{1}{k^2}u^{\infty}(\hat{x}_{i^*};\hat{\theta}_{j^*}) \quad \mbox{where}\\
       (i^*,j^*) &=& \arg\min_{1\leq i\leq N_{\rm obs},\, 
       1\leq j\leq N_{\rm inc}} \left\| p_{m,n} - \dfrac{\hat{\theta}_j - \hat{x}_i}{2} \right\|_2
\end{eqnarray*}
 where $p_{m,n}=r_m \big(\cos\theta_n,~\sin \theta_n\big)^T$ and the transformed Clenshaw-Curties quadrature points $r_m$ and trapezoidal quadrature points $\theta_n$ are given by
 \begin{eqnarray*}
     r_m=\sqrt{\frac{\cos(m\pi/N_2)+1}{2}},~0\leq m\leq N_2-1 \mbox{ and }%
 \theta_n=\frac{2\pi(n-1)}{N_1},~1\leq n\leq N_1, \mbox{ respectively}.
 \end{eqnarray*} 
 Here $N_1=104,~N_2=56$. 
 }

\subsubsection{ Training method and settings}
The neural networks $\mbox{U}_1,~\mbox{U}_2,~\mbox{U}$ in this work were  trained using PyTorch, using the AdamW optimizer with a learning rate $10^{-3}$ and weight decay  $10^{-5}$. A mini-batch size of $40$ samples was used for each iteration, and the Mean Squared Error (MSE) loss function was minimized over $100$ training epochs. To improve convergence stability, the ReduceLROnPlateau scheduler was employed to dynamically adjust the learning rate, reducing it by a factor of $0.5$ when the validation loss plateaued for $4$ consecutive epochs. 

\section{Numerical examples} \label{sec: numerical example}
In this section, we conduct a variety of numerical experiments to illustrate the effectiveness of our proposed methods.

\subsection{A first comparison of three methods}\label{sec: A first comparison of three methods}
To begin with, we first generate a dataset of $20000$ samples of the input-output $\{(u,u_b)\}$ to train the neural network $\mbox{U}_1$; similarly, a dataset of $20000$ samples of the input-output $\{(u_b,q)\}$ is generated to train $\mbox{U}_2$ and a dataset of $20000$ samples of the input-output $\{(u^\infty,q)\}$ is generated to train the black-box neural network $\mbox{U}$.  The first $10000$ samples of contrasts are based on the MNIST dataset \cite{lecun2010mnist}, and the other $10000$ contrasts are generated by $n$ piece-wise constant disks where $n$ is randomly selected from $\{1,~2,~3\}$. The radius and origins of the disks are sampled according to uniform distribution $\mathcal{U}(0.1,0.3)$ and $\mathcal{U}(-0.4,0.4)$, respectively; the maximum magnitude is sampled according to $\max_B |q|\sim \mathcal{U}(0.1,0.8)$. Among the $20000$ samples, $16000$ are used for training and $4000$ samples are reserved for validation. Later on, we will discuss other ways to generate efficient training dataset.

Unless otherwise noted, we add  noise to the far-field data for the purpose of testing according to
 $$
 u^{\infty,\delta}(\hat{x};\hat{\theta})=u^\infty(\hat{x};\hat{\theta})(1+\delta \xi)
 $$
 where $\delta=20\%$, $\xi$ is a complex number such that $\Re \xi \sim\mathcal{N}(0,1/2)$ and $\Im \xi\sim\mathcal{N}(0,1/2)$ with independent $\Re \xi$ and $\Im \xi$.

We take a sample, number ``4'', from MNIST dataset that is not included in the training dataset, to illustrate our proposed method. In \Cref{figure: 4_ULR_UU_U}, we plot the reconstructions of the ground truth using ULR, UU, and U, respectively. All three methods give reasonable reconstruction, and  ULR is observed to be more robust as the top of the reconstructed ``4'' were blurred when using UU and U. We further test a sample consisting of three disks in \Cref{figure: 3_disks}. It is observed again that the ULR and UU are more robust than the black-box neural network U.
\begin{figure}[htbp]
\subfloat[Ground truth]{\includegraphics[width=0.25\linewidth]{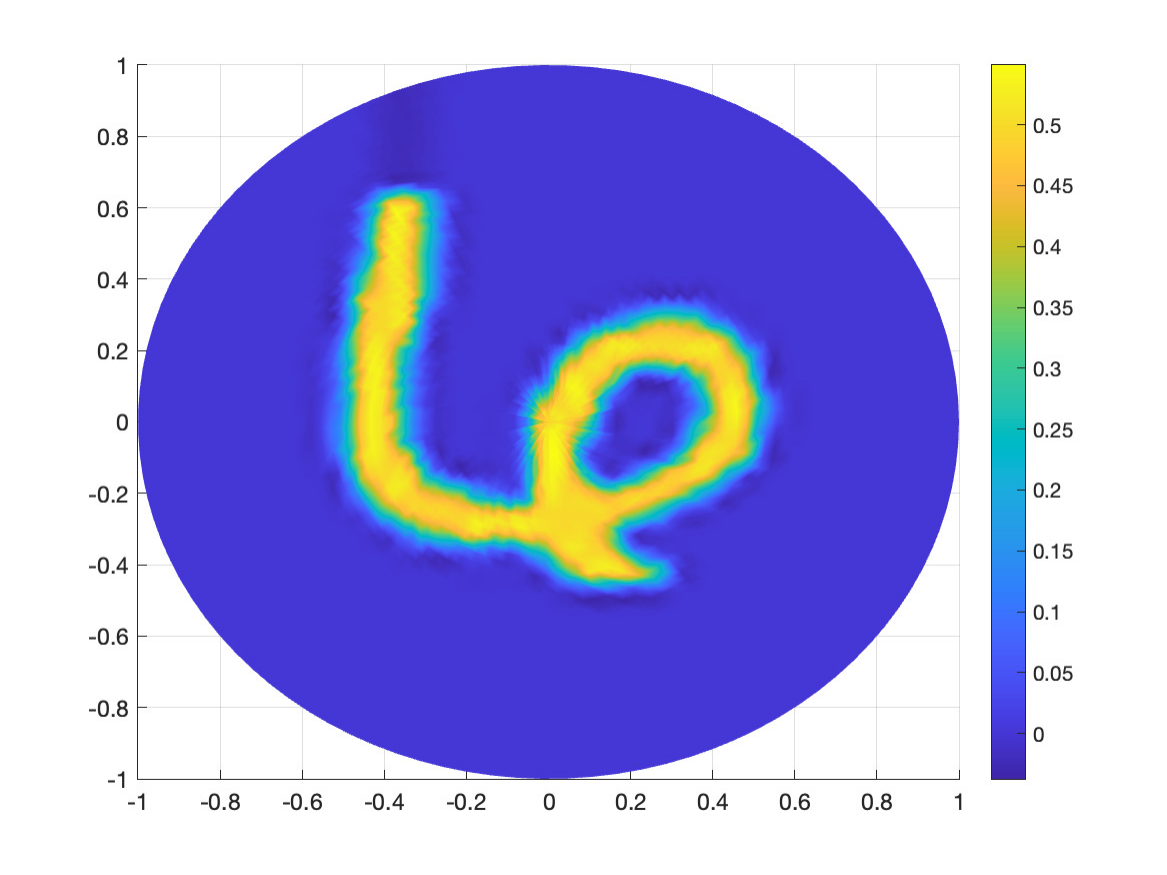}}
\subfloat[ULR]{\includegraphics[width=0.25\linewidth]{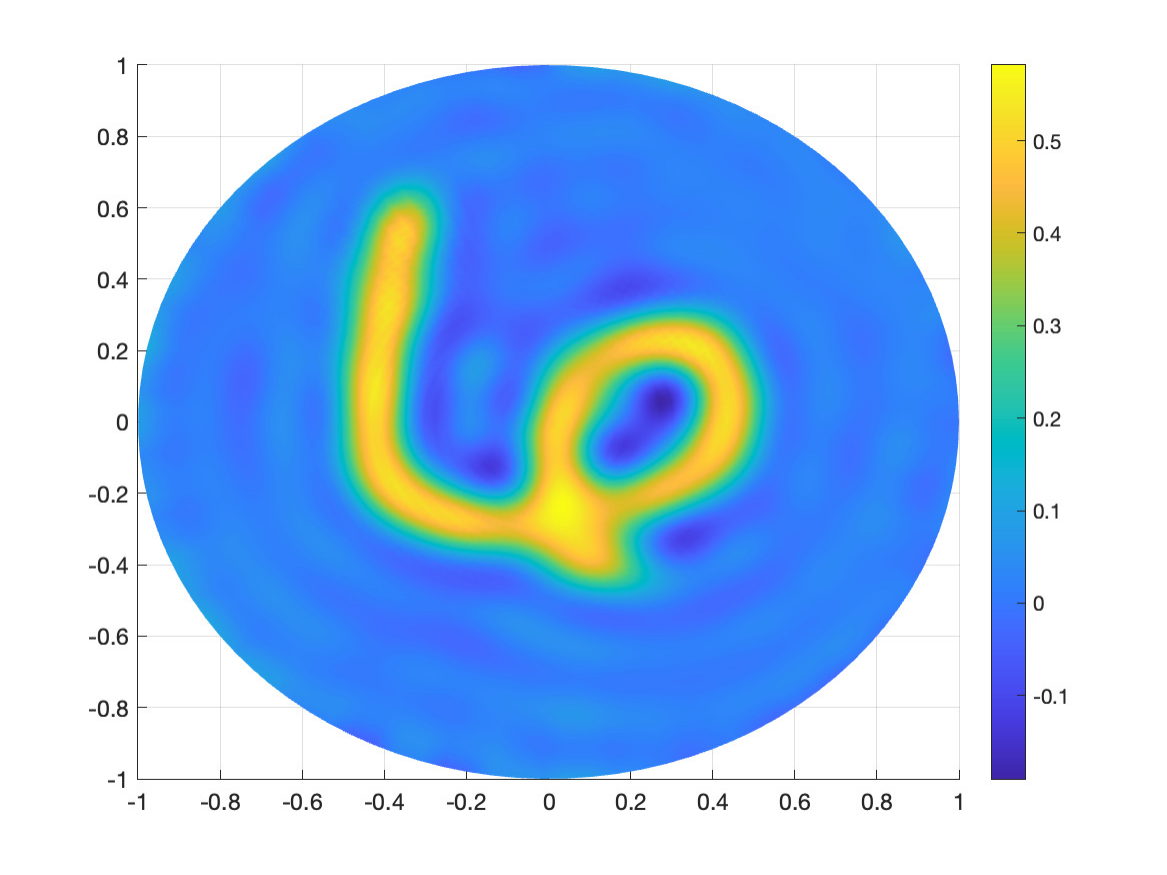}}
\subfloat[UU]{\includegraphics[width=0.25\linewidth]{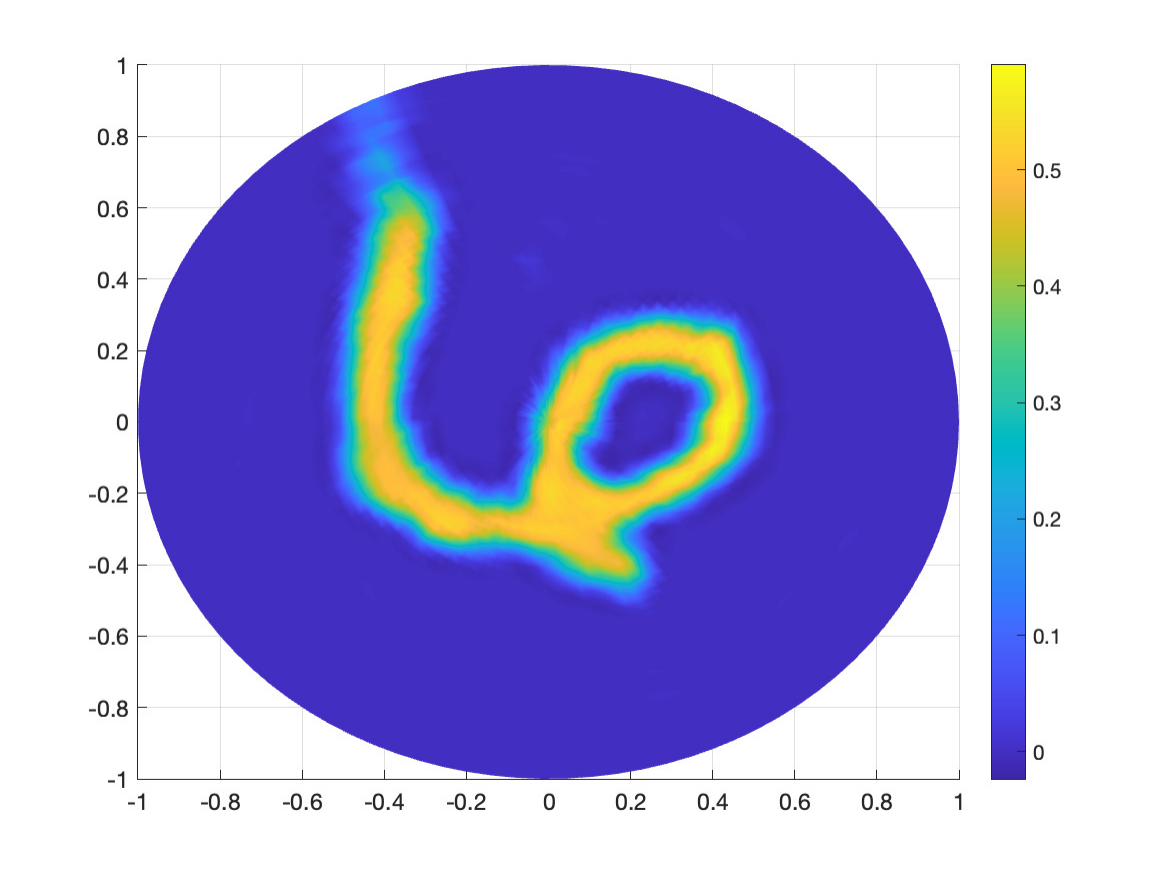}}
\subfloat[U]{\includegraphics[width=0.25\linewidth]{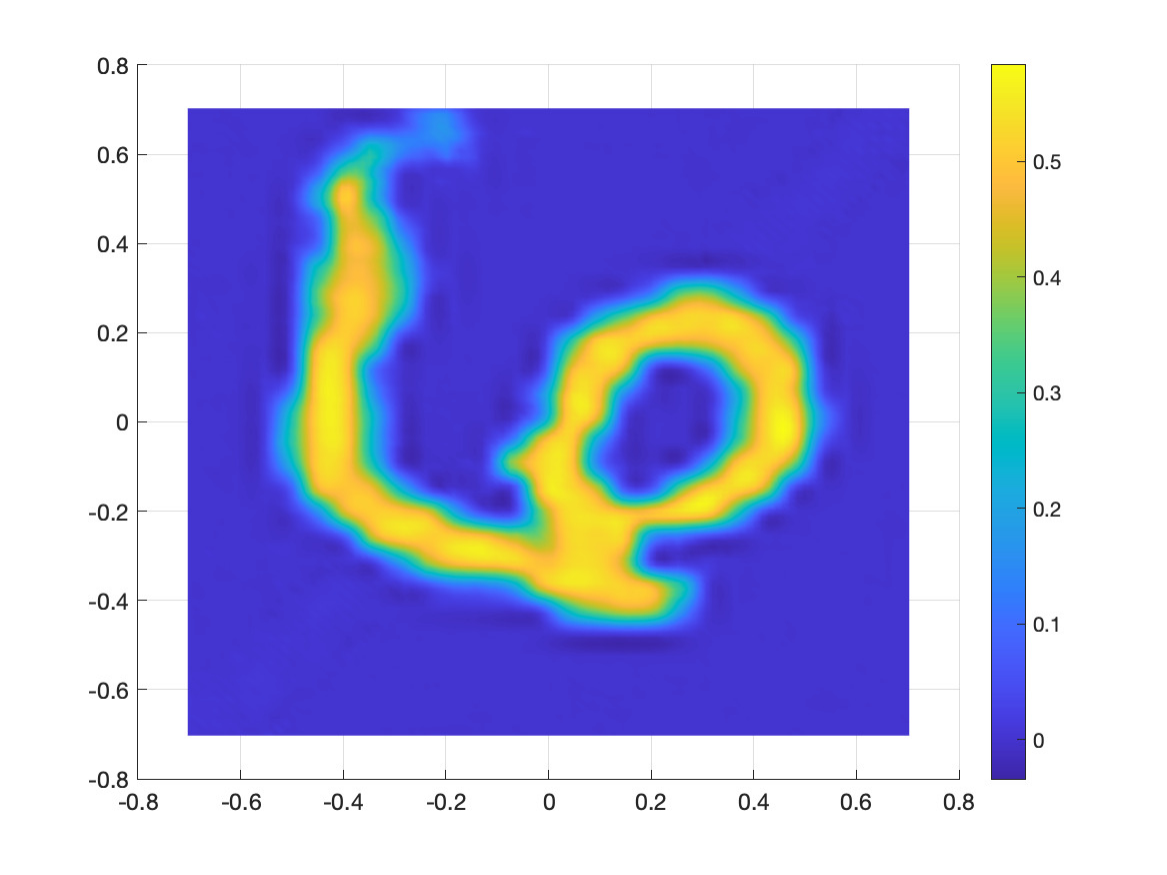}}
  \caption{Reconstruction of number ``4''. Left to right:  ground truth, reconstructions by ULR, UU, and U, respectively.}  \label{figure: 4_ULR_UU_U}
\end{figure}

\begin{figure}[htbp]
\subfloat[Ground truth]{\includegraphics[width=0.25\linewidth]{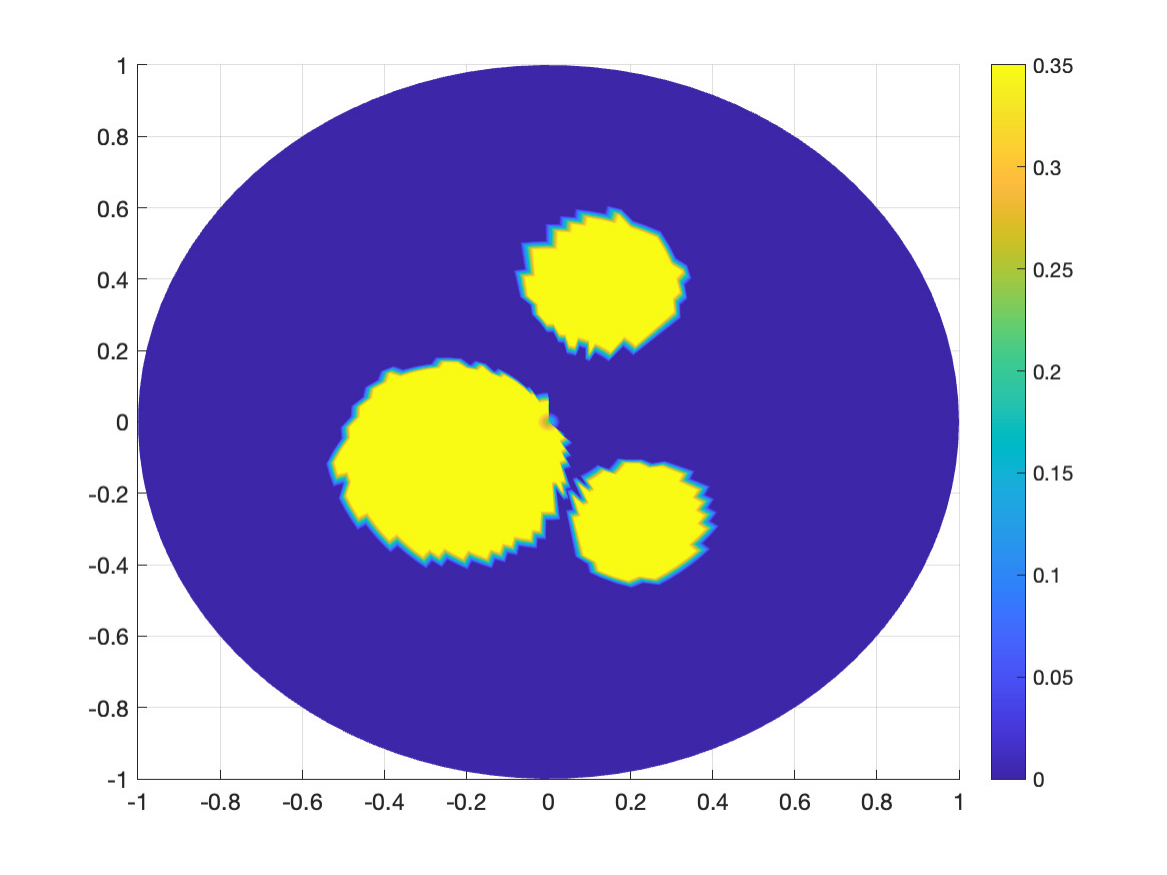}}
\subfloat[ULR]{\includegraphics[width=0.25\linewidth]{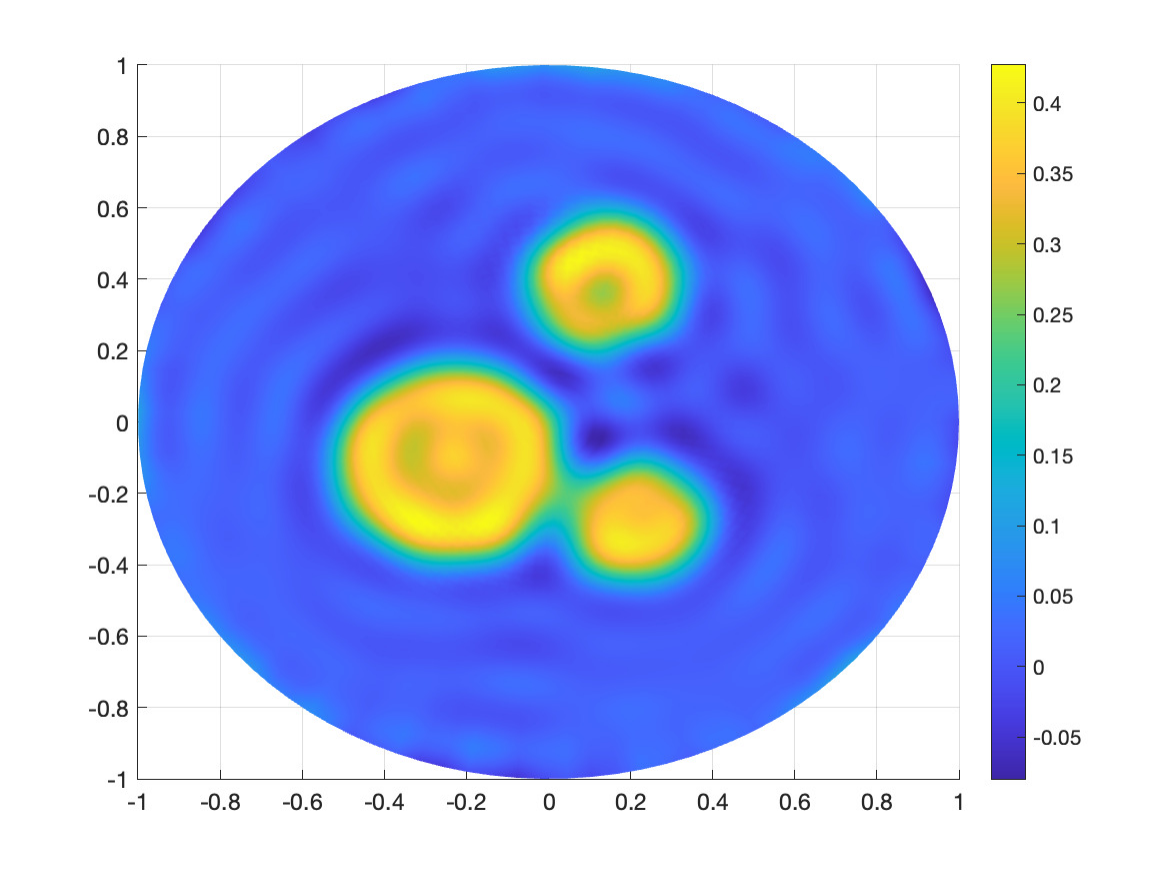}}
\subfloat[UU]{\includegraphics[width=0.25\linewidth]{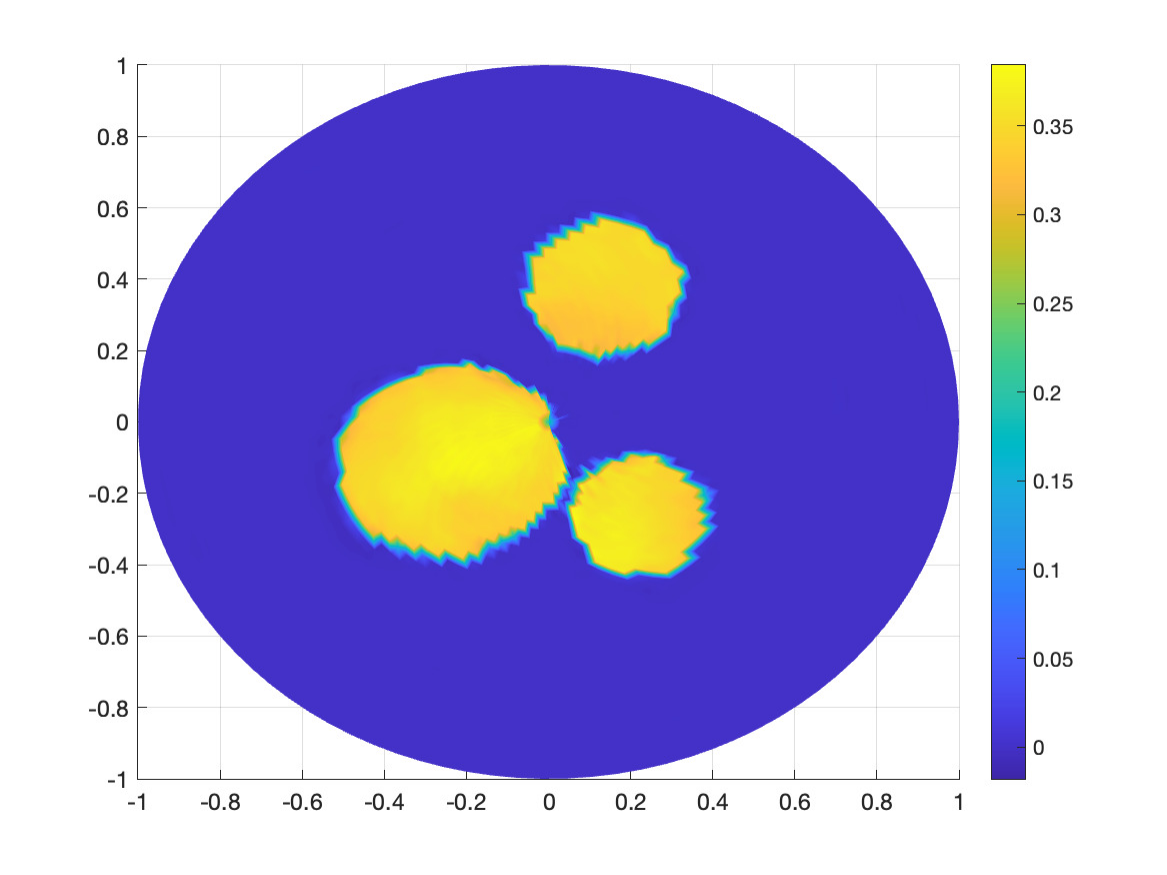}}
\subfloat[U]{\includegraphics[width=0.25\linewidth]{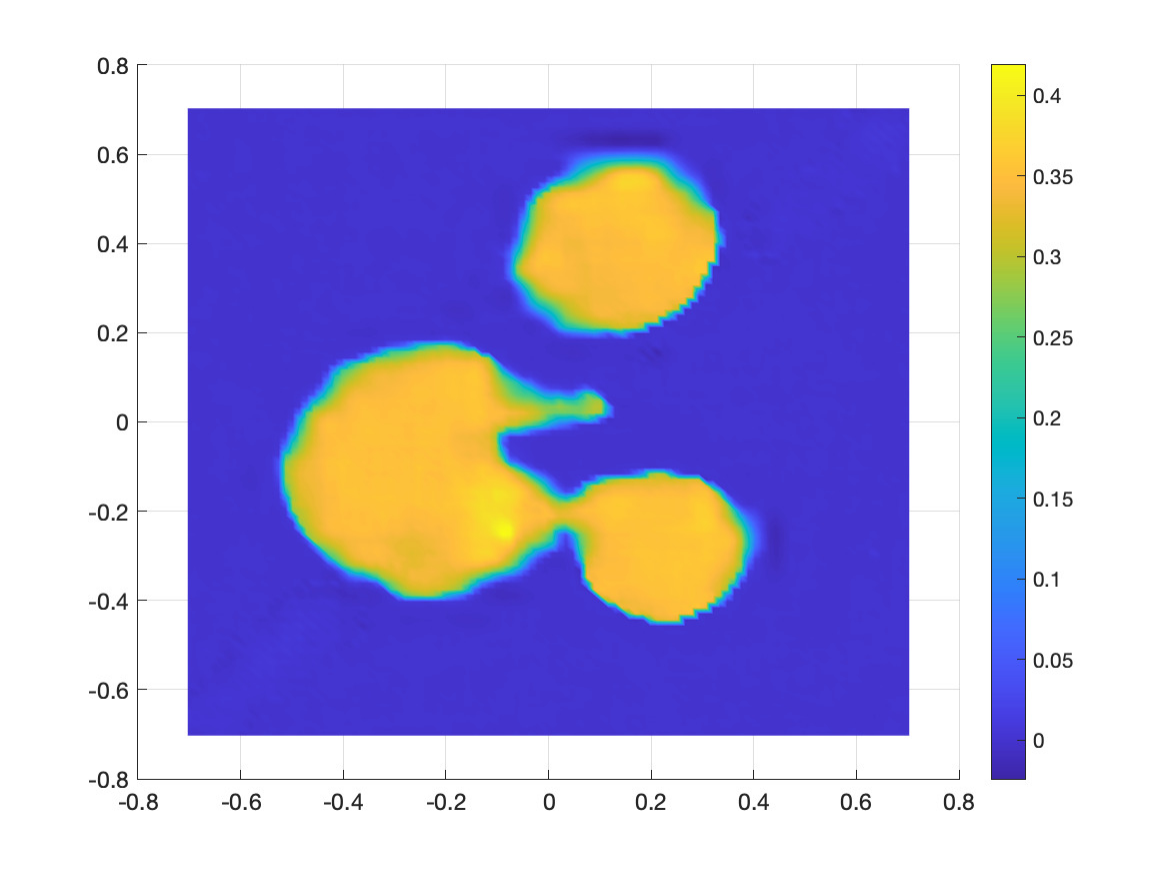}}
  \caption{Reconstruction of three disks. Left to right:  ground truth, reconstructions by ULR, UU, and U, respectively. }  \label{figure: 3_disks}
\end{figure}

We further test the proposed methods to reconstruct the rotations of the number ``4''. We plot in \Cref{figure: full rotation-equivariance} the reconstructions using ULR, UU, and U in the second, third, and fourth column, respectively. We note that ULR and UU are more robust to rotations while the black-box neural network U is more sensitive to rotations. One major reason is that the neural networks ULR and UU take advantage of the rotation-equivariance property while the neural network U does not. This example also indicates that ULR and UU has the potential generalization capability to reconstruct any rotated contrast $\mathcal{R}_\phi q$ provided that it works well for $q$.

\begin{figure}[htbp]
{\includegraphics[width=0.24\linewidth]{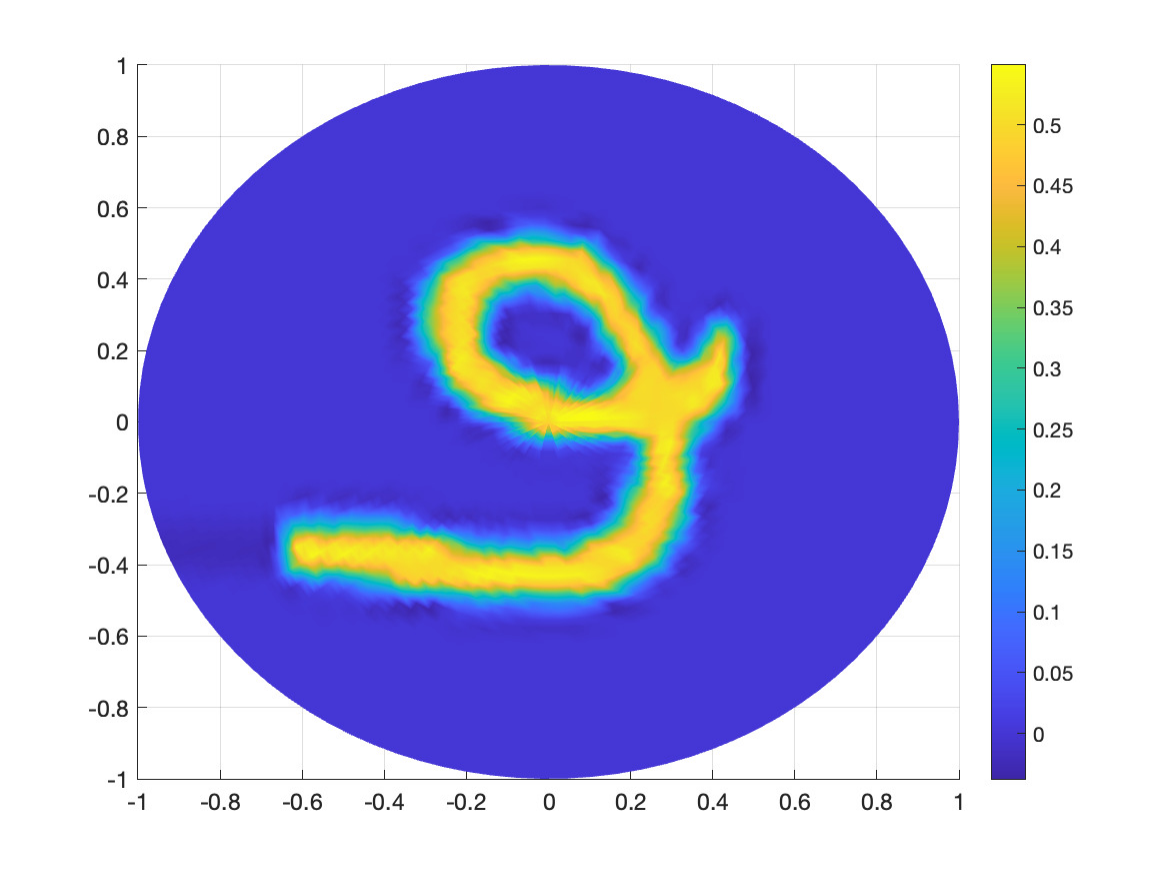}}
{\includegraphics[width=0.24\linewidth]{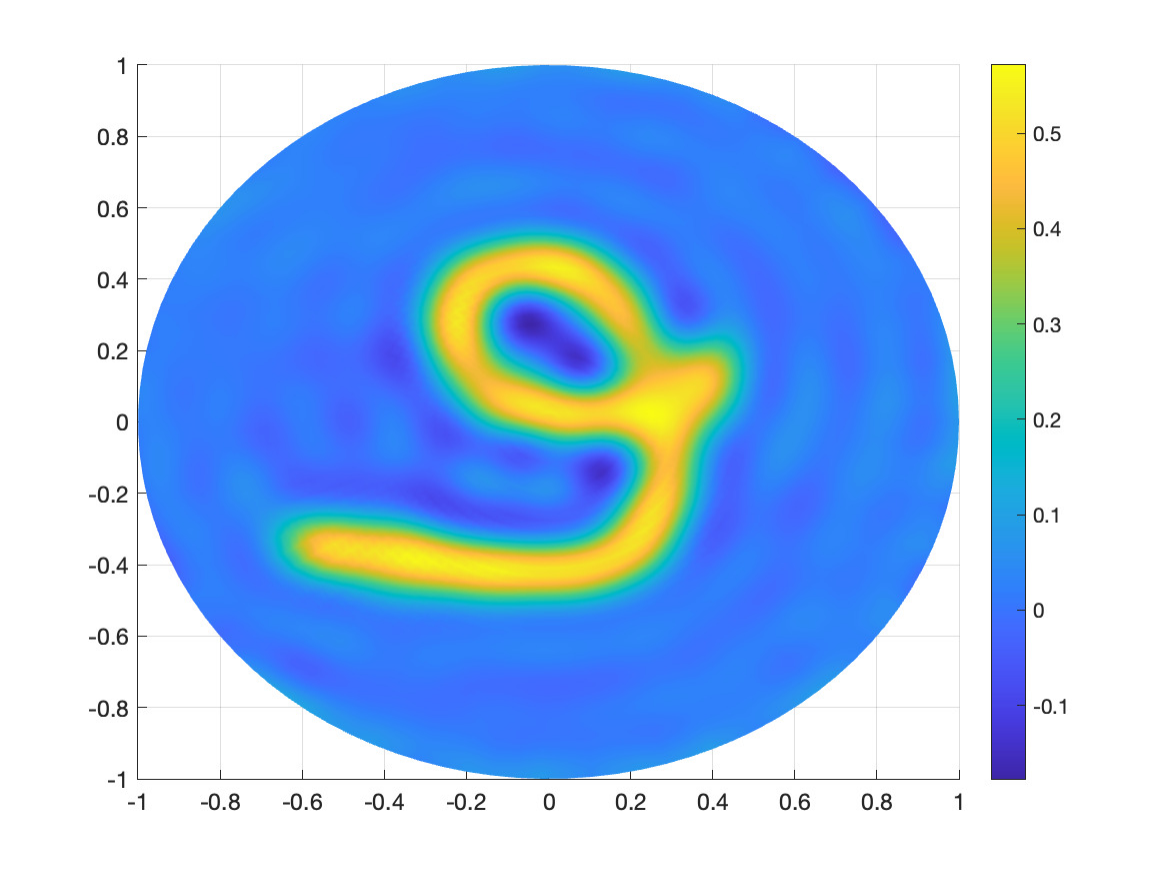}}
{\includegraphics[width=0.24\linewidth]{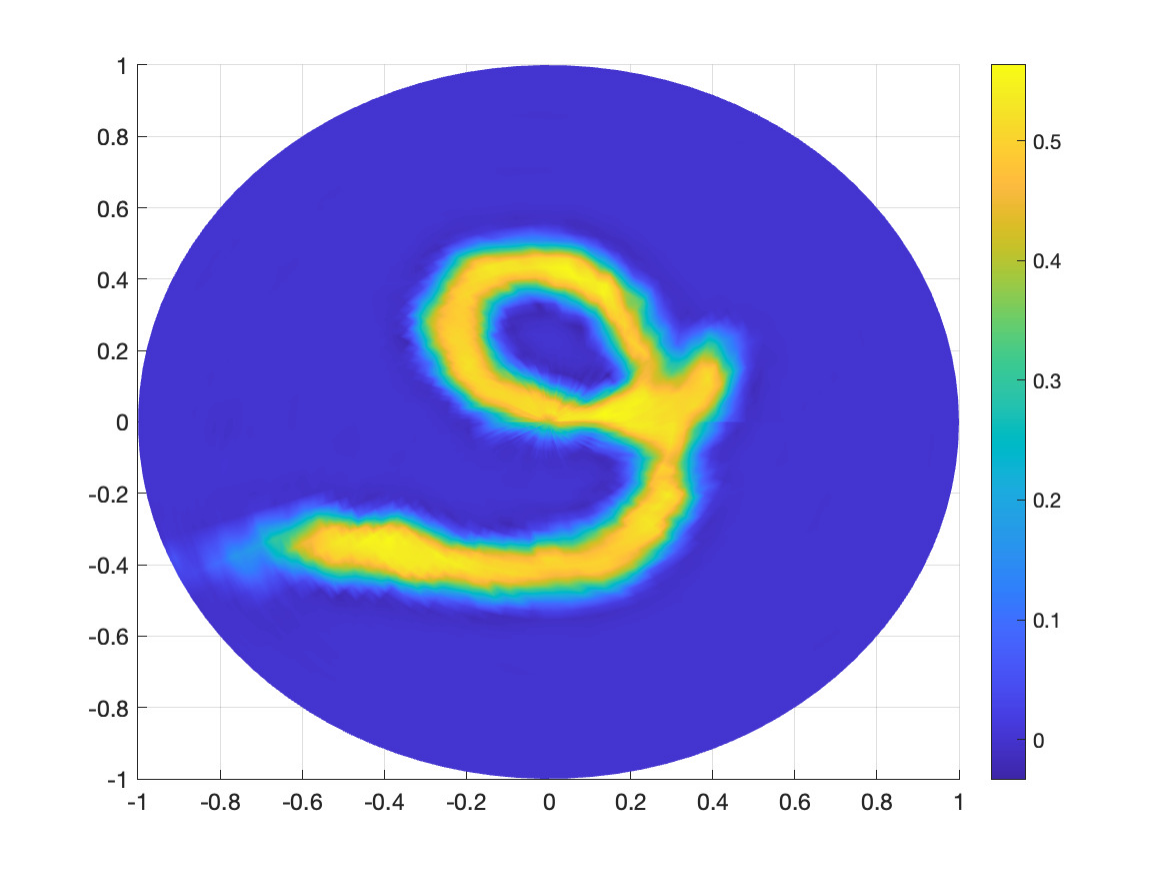}}
{\includegraphics[width=0.24\linewidth]{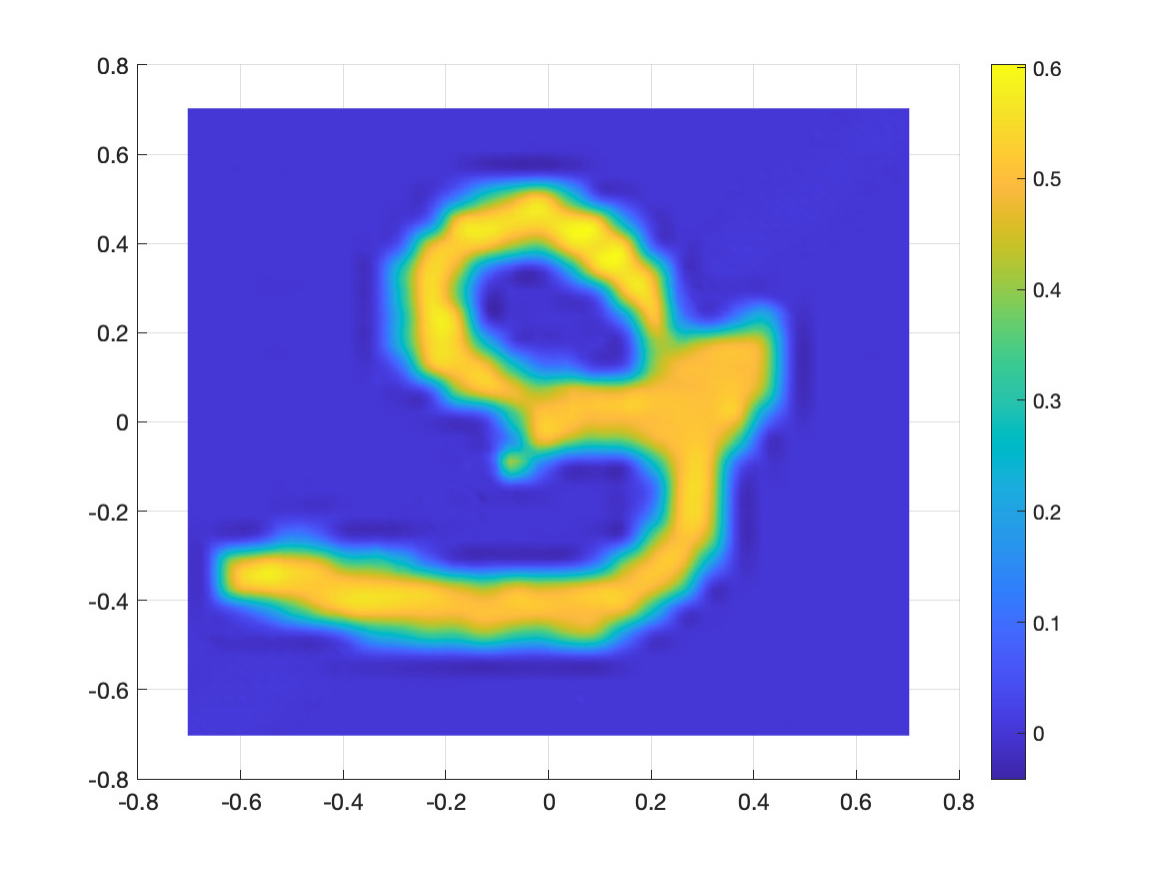}}\\
{\includegraphics[width=0.24\linewidth]{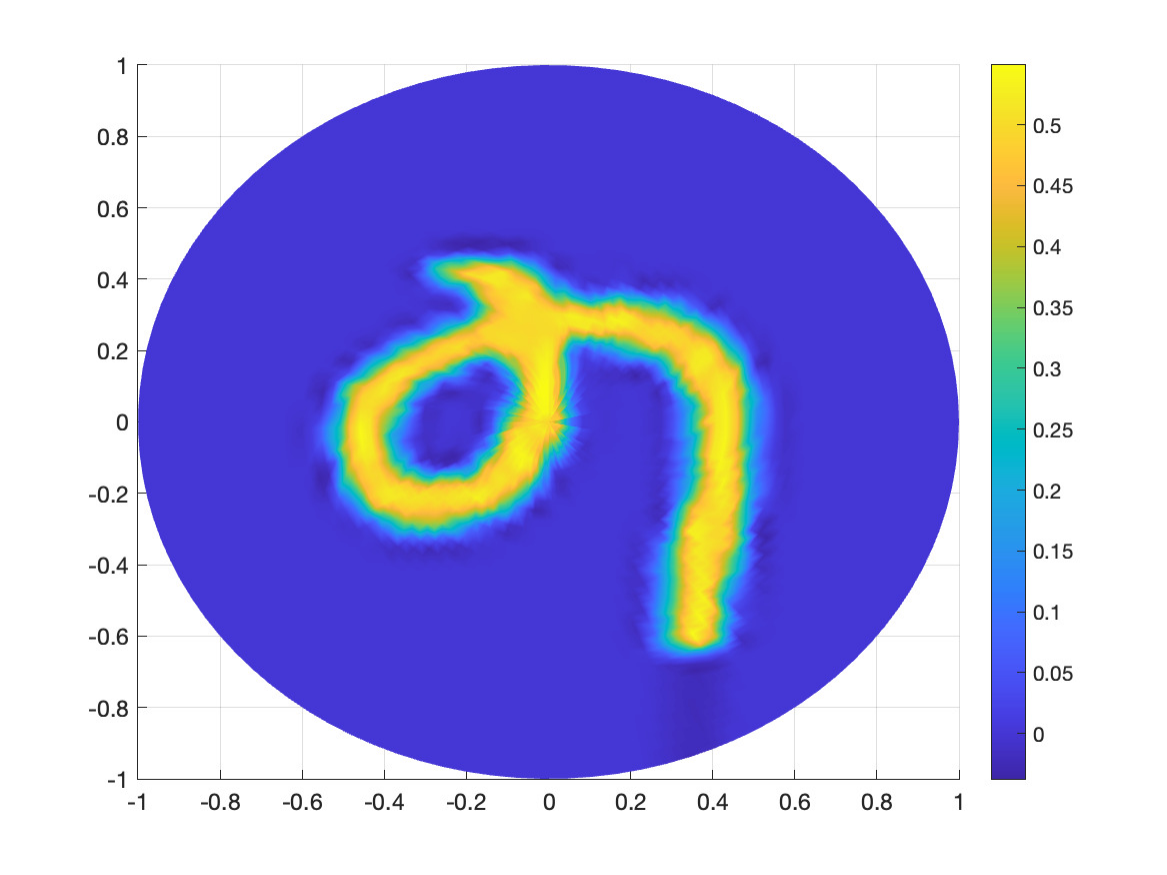}}
{\includegraphics[width=0.24\linewidth]{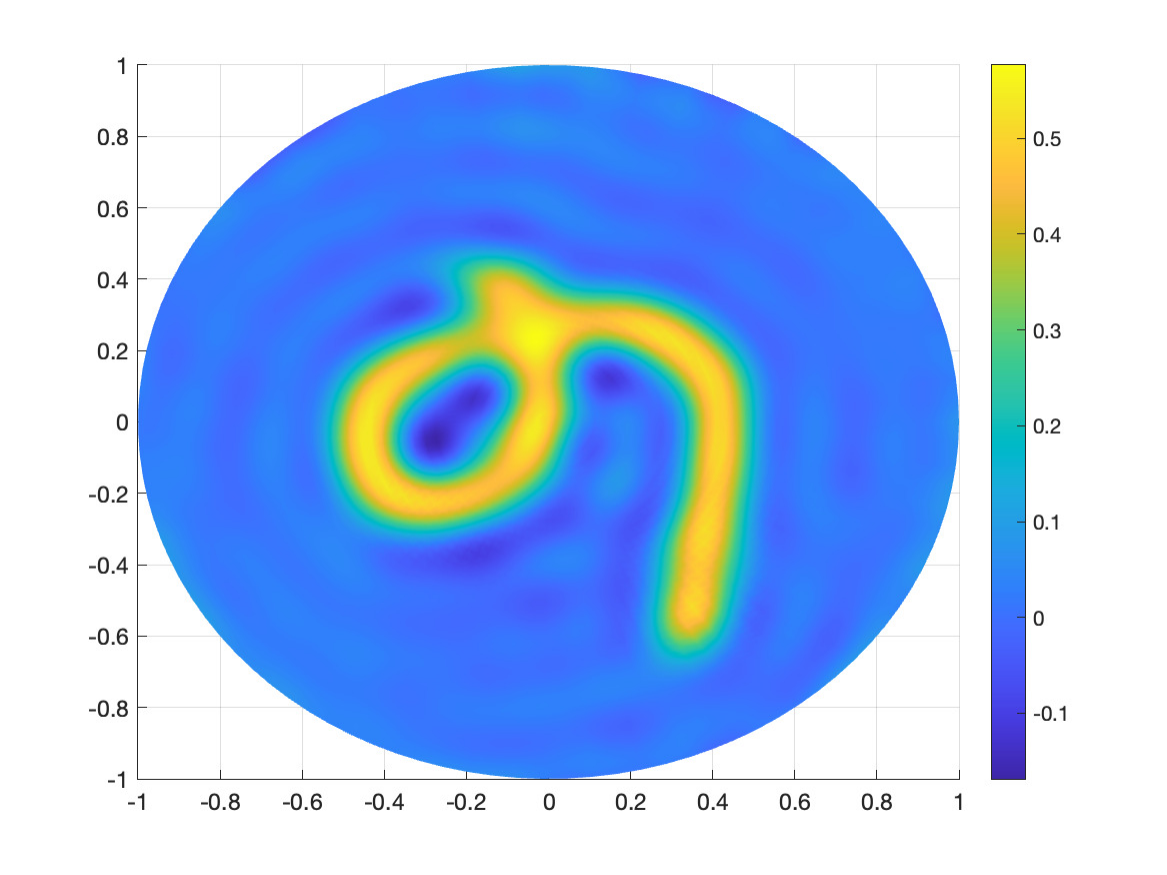}}
{\includegraphics[width=0.24\linewidth]{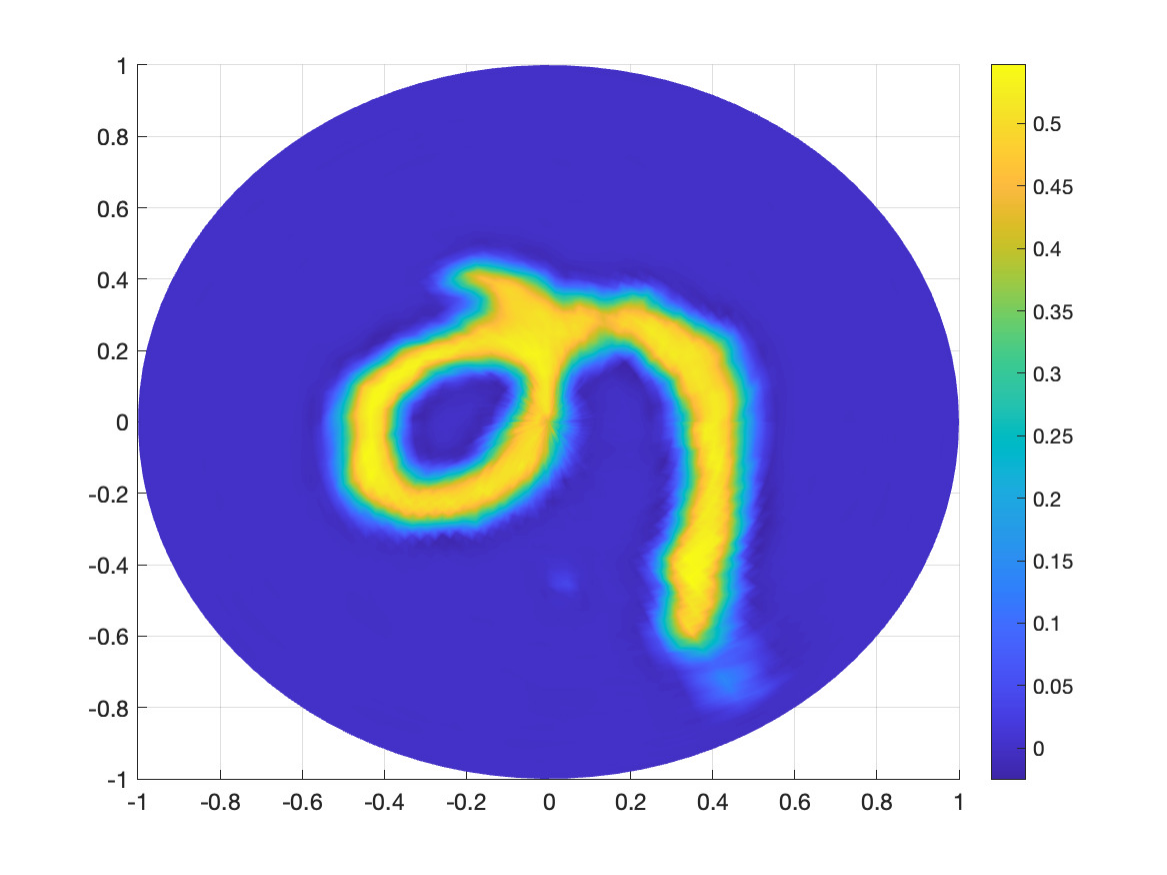}}
{\includegraphics[width=0.24\linewidth]{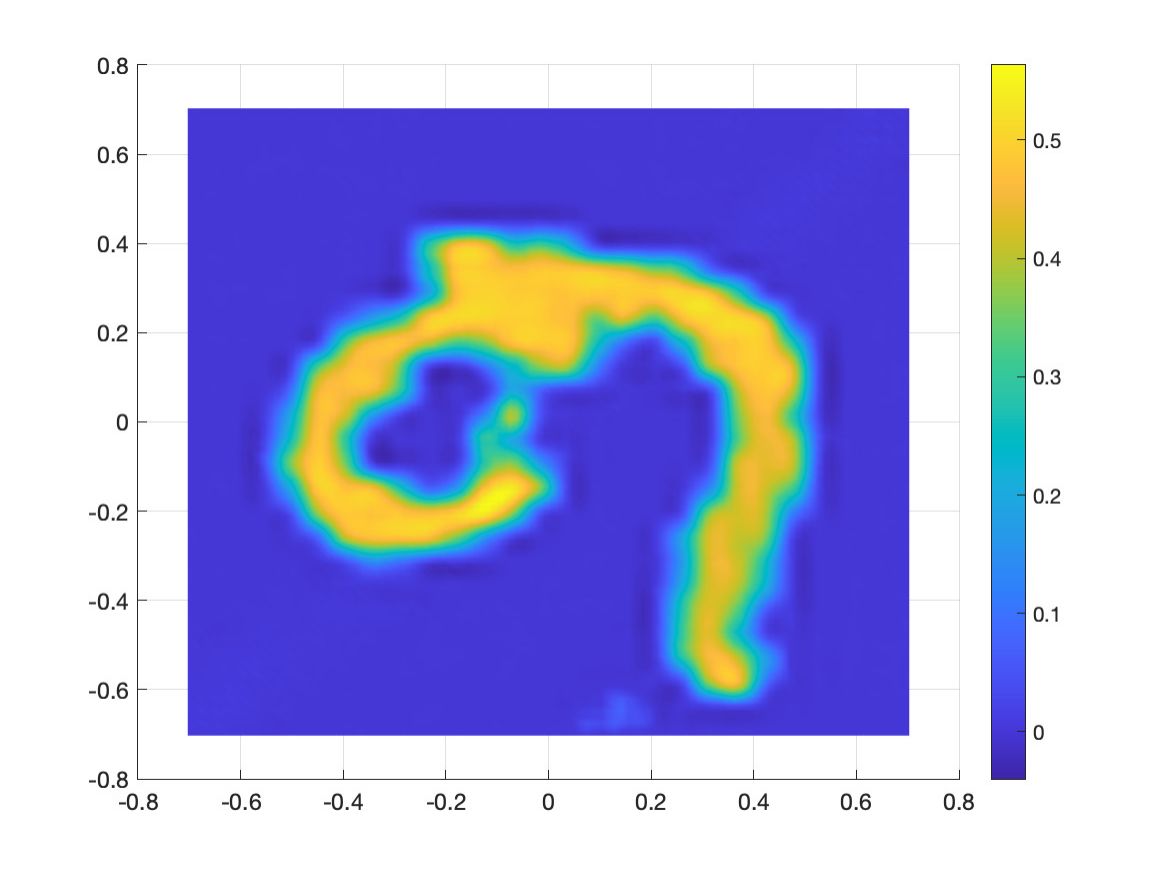}}
\\
{\includegraphics[width=0.24\linewidth]{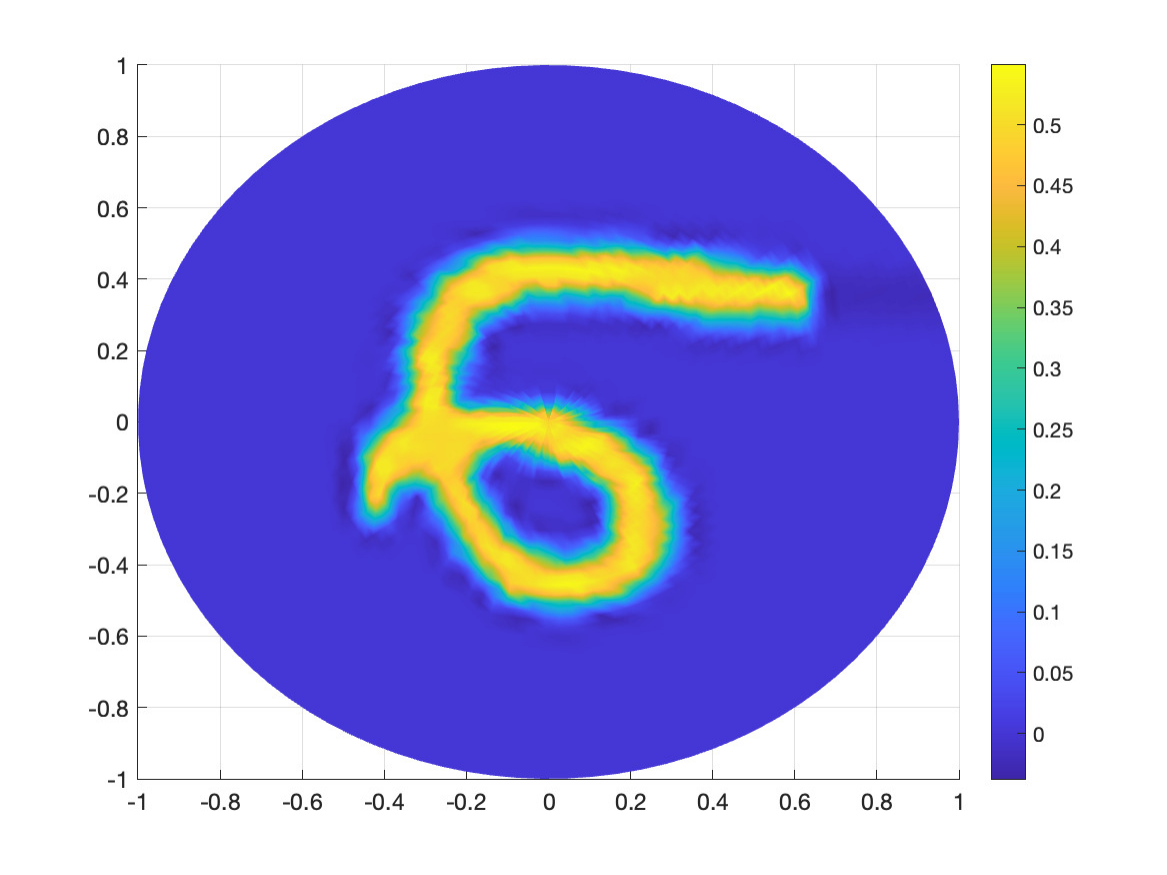}}
{\includegraphics[width=0.24\linewidth]{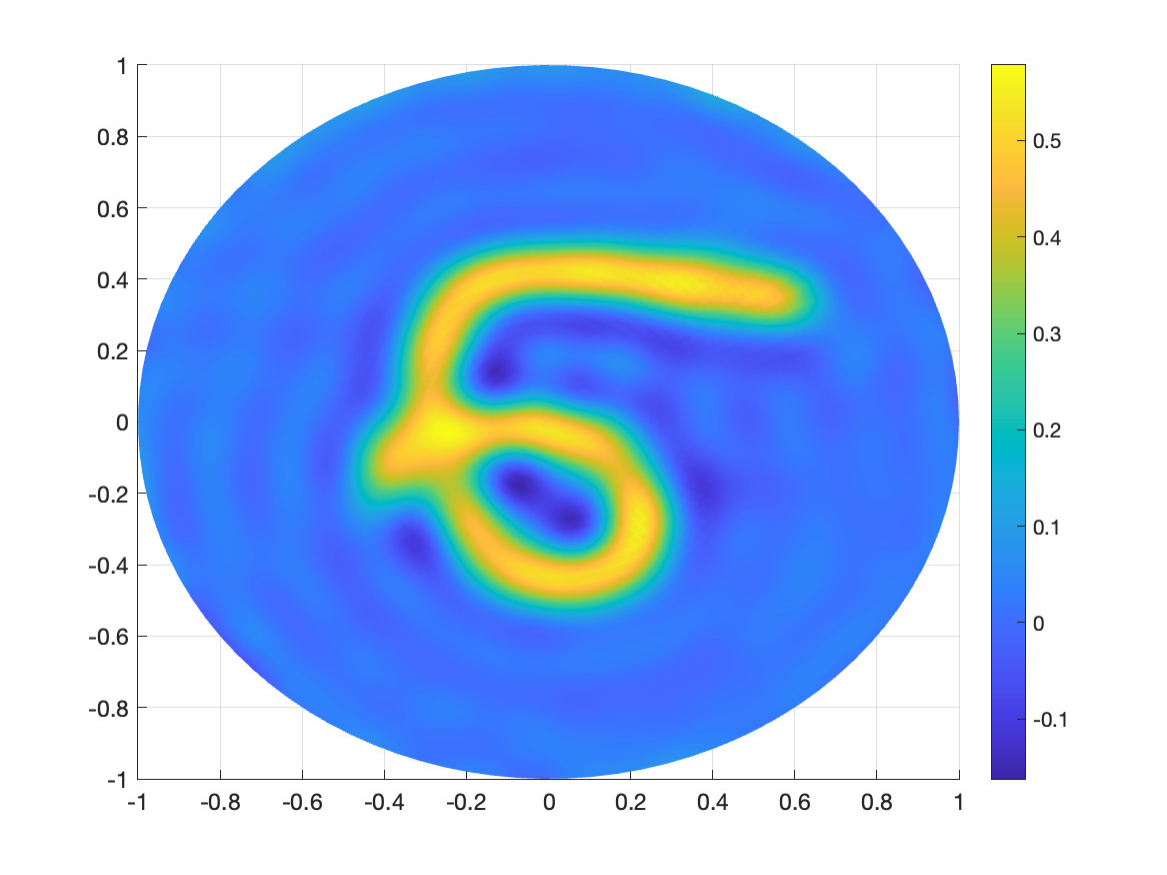}}
{\includegraphics[width=0.24\linewidth]{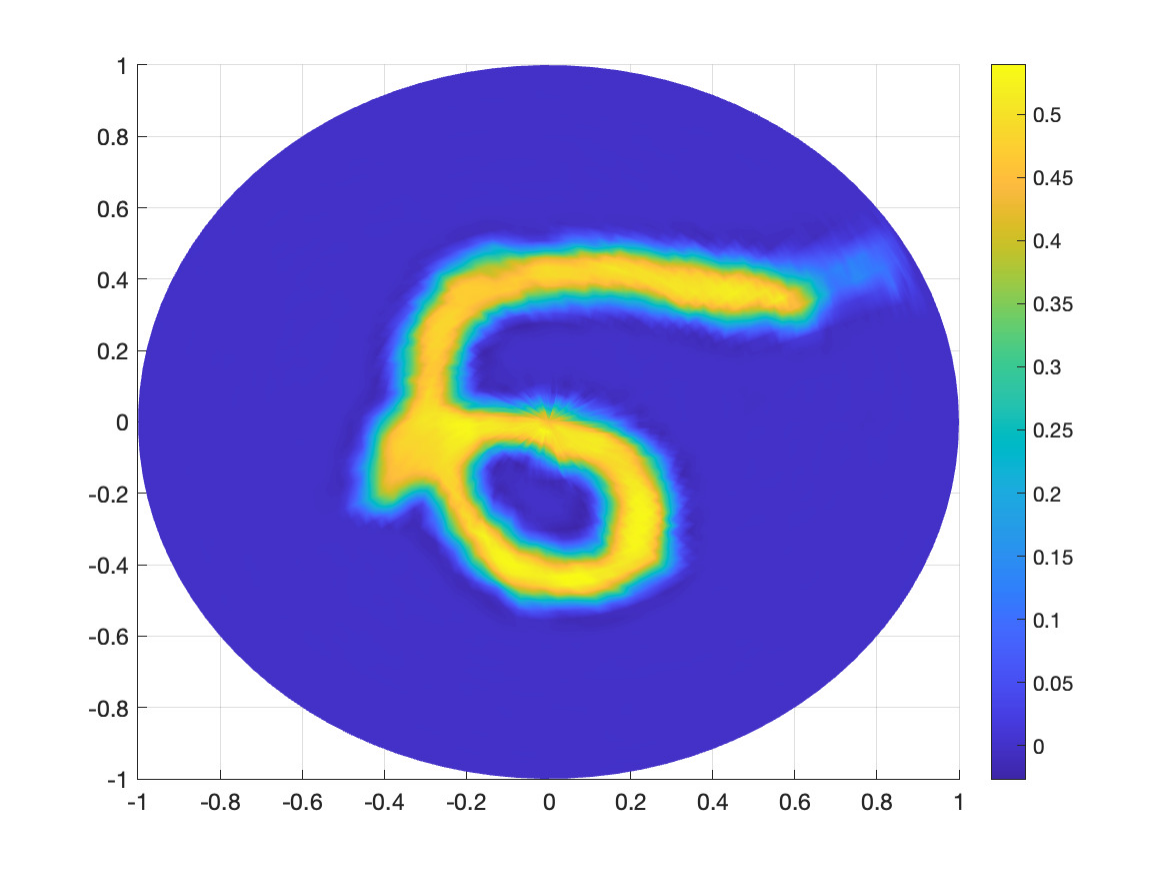}}
{\includegraphics[width=0.24\linewidth]{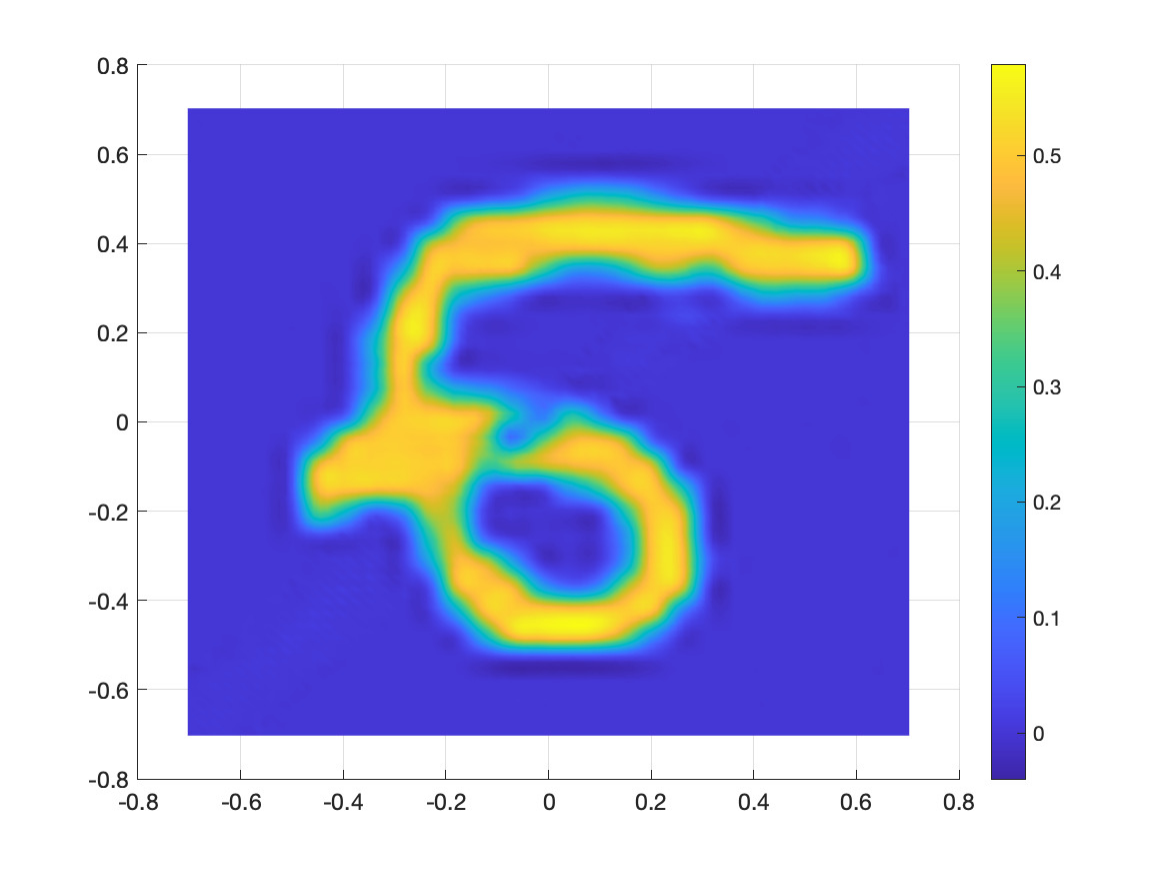}}
  \caption{Reconstruction of number ``4'' and its rotations. Top to bottom: rotation  $\pi/2$, $\pi$, and $3\pi/2$ of the number ``4''; left to right: ground truth, reconstructions by ULR, UU, and U, respectively.}  \label{figure: full rotation-equivariance}
\end{figure}

Note that the value of $q$ in the three disks is $0.35$ in \Cref{figure: 3_disks},  we continue to test the proposed methods for the three disks case by increasing the contrast value to $0.7$ and $1$ (corresponding to degree of nonlinearity $2.3230$ and $3.2166$), respectively.  
With reference to \Cref{figure: generalization 3 balls}, the reconstruction (in second column) by the low-rank regularized inverse Born solver is poor due to the large degree of nonlinearity. The reconstructions using ULR (third column) and UU (fourth column) greatly improve the inverse Born solution, and they are more robust than the black-box neural network (fifth column).  

\begin{figure}[htbp]
\includegraphics[width=0.19\linewidth]{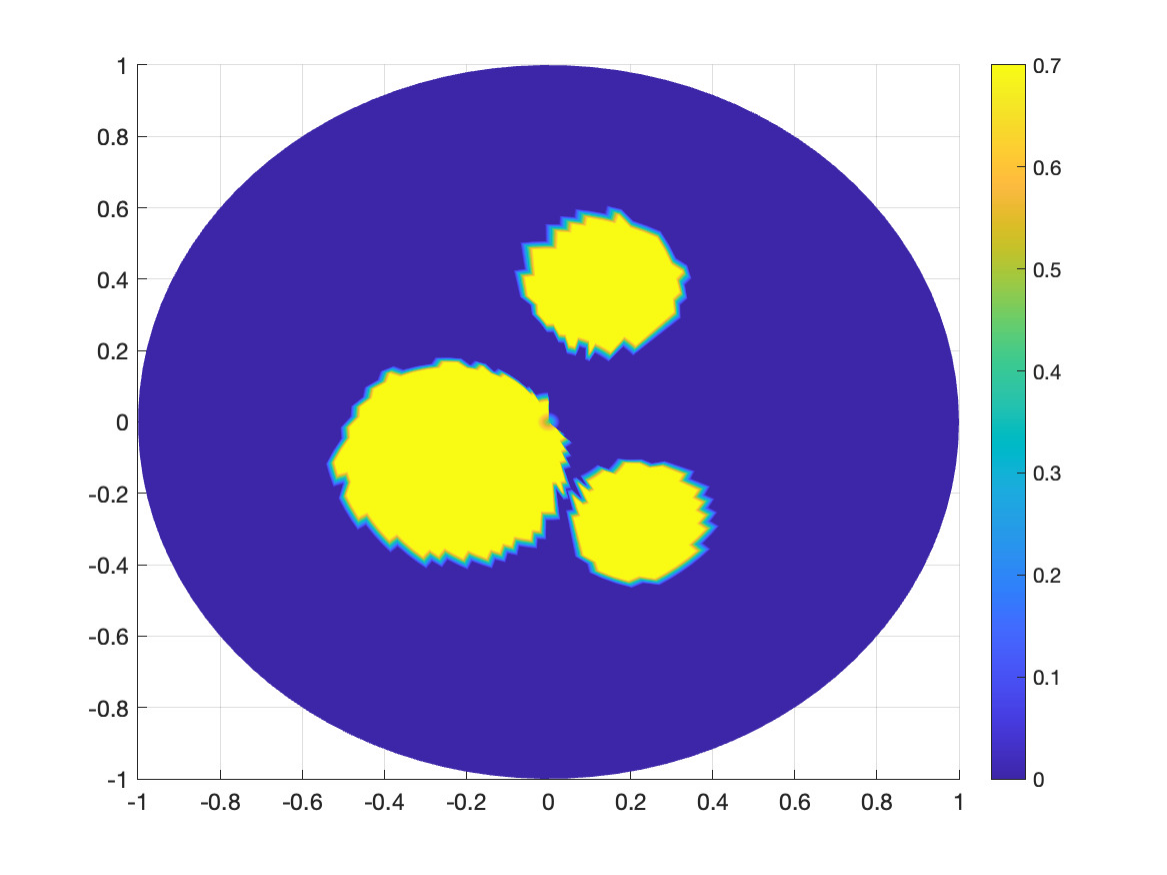}
\includegraphics[width=0.19\linewidth]{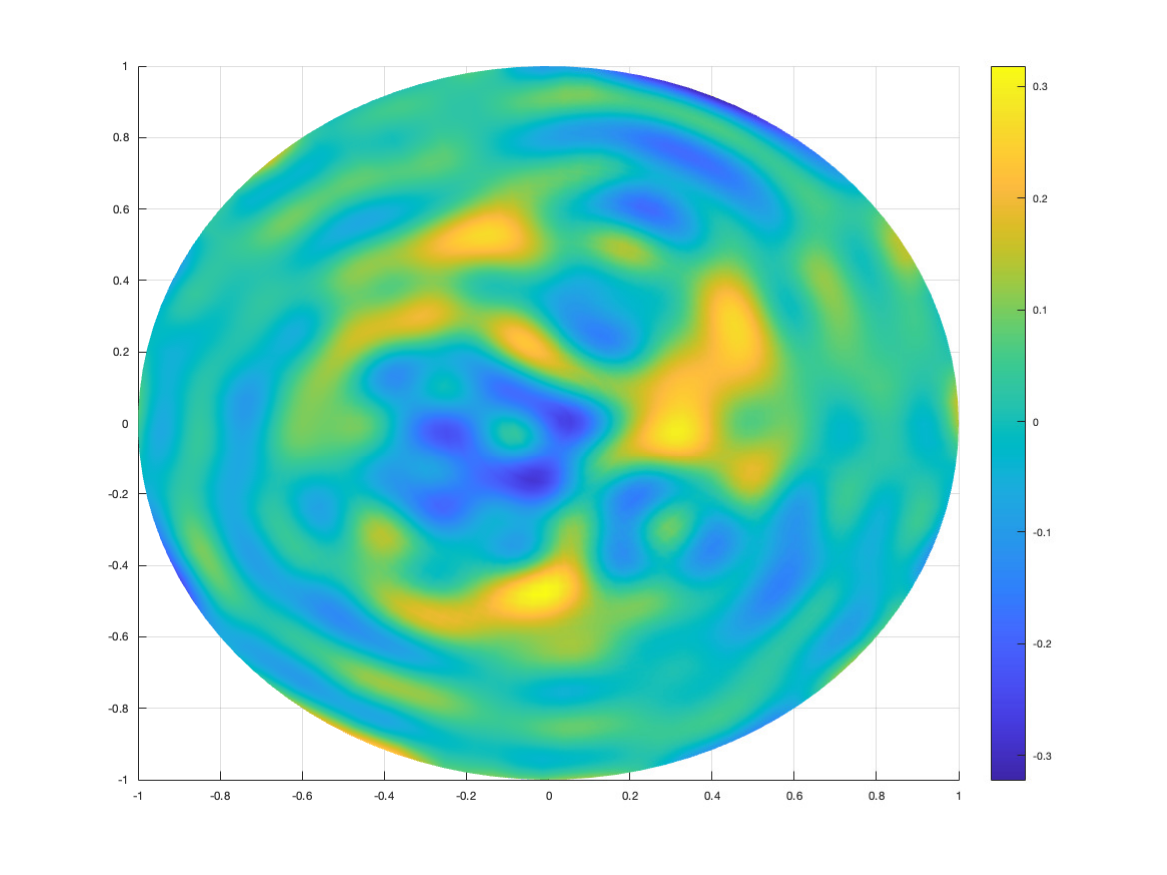}
\includegraphics[width=0.19\linewidth]{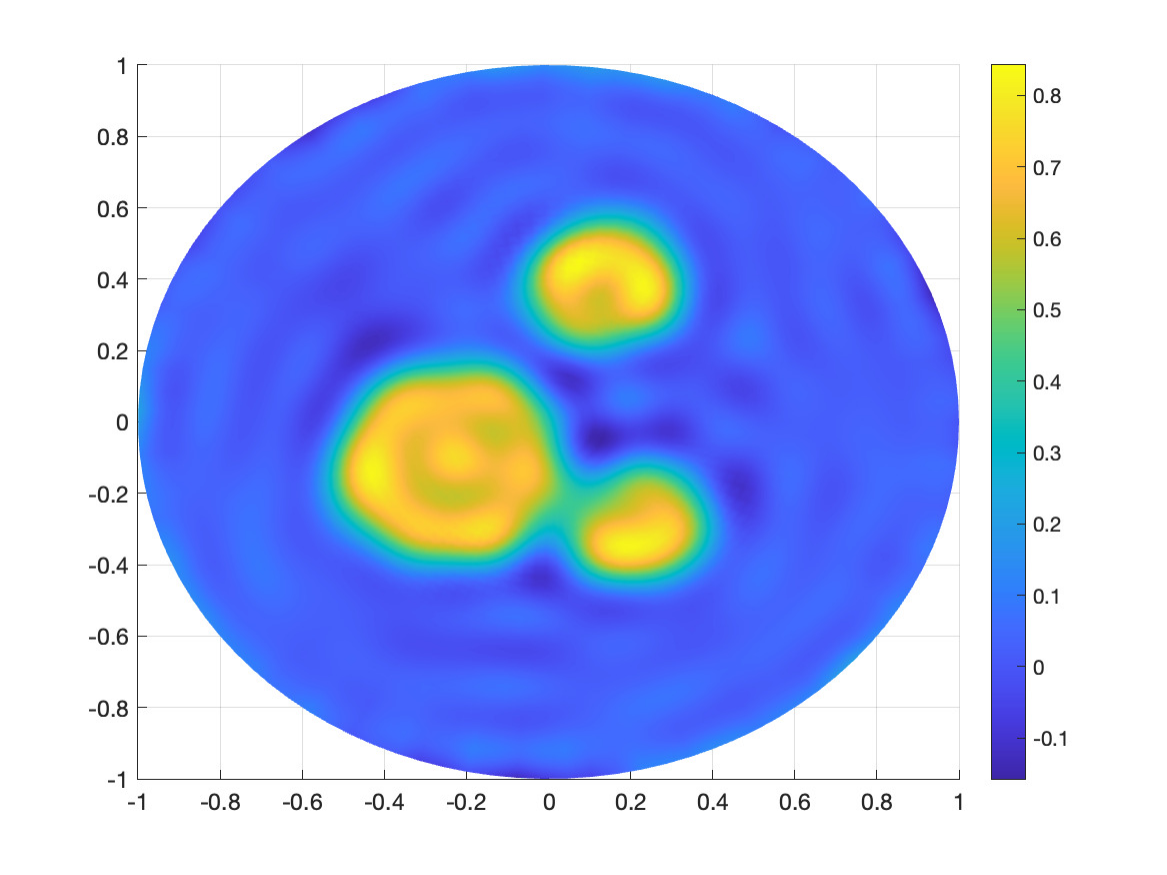}
\includegraphics[width=0.19\linewidth]{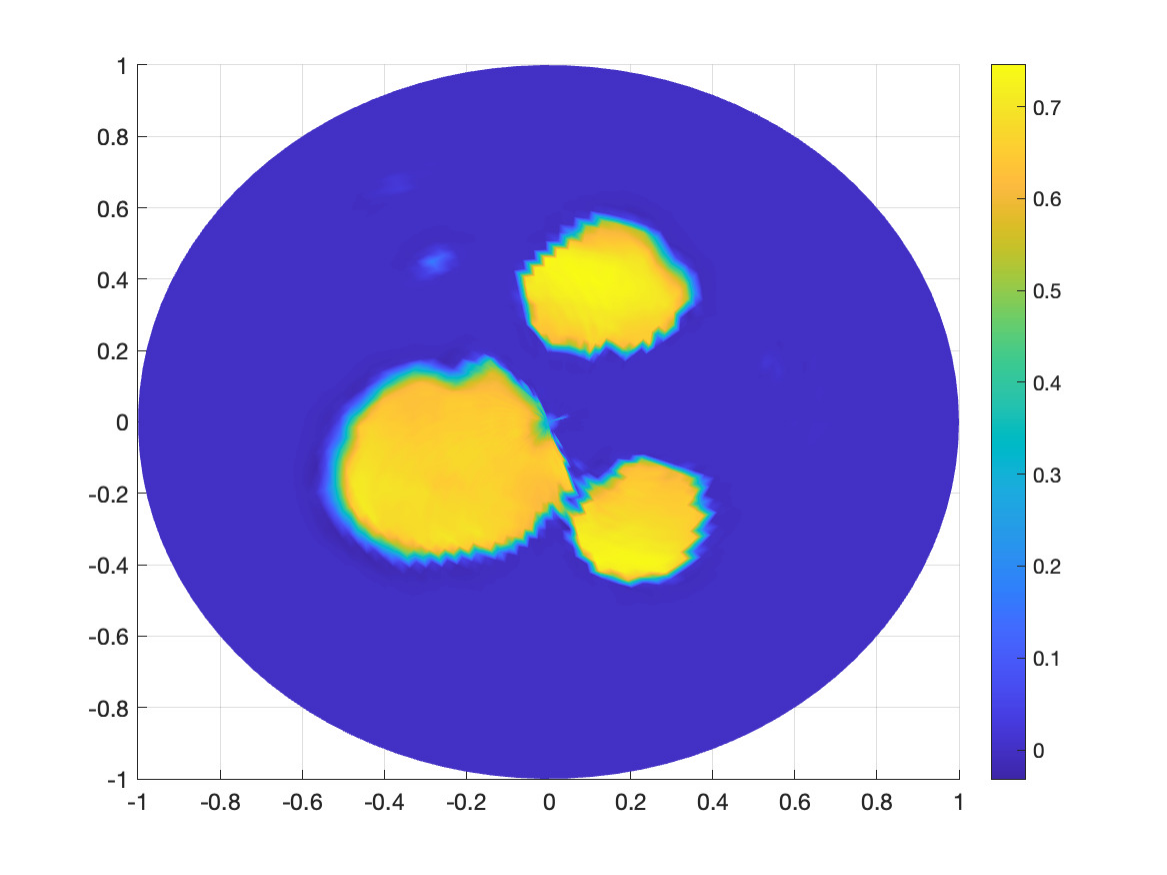}
\includegraphics[width=0.19\linewidth]{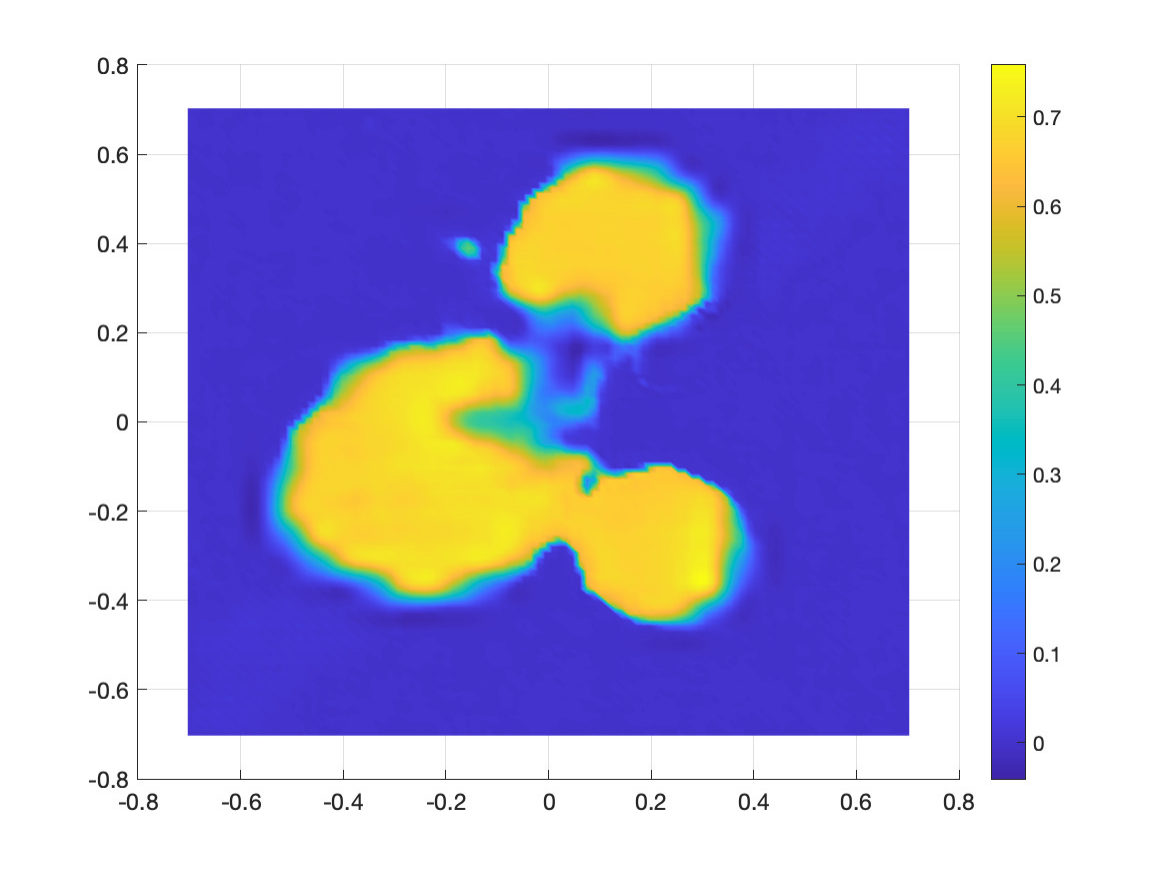}\\
\includegraphics[width=0.19\linewidth]{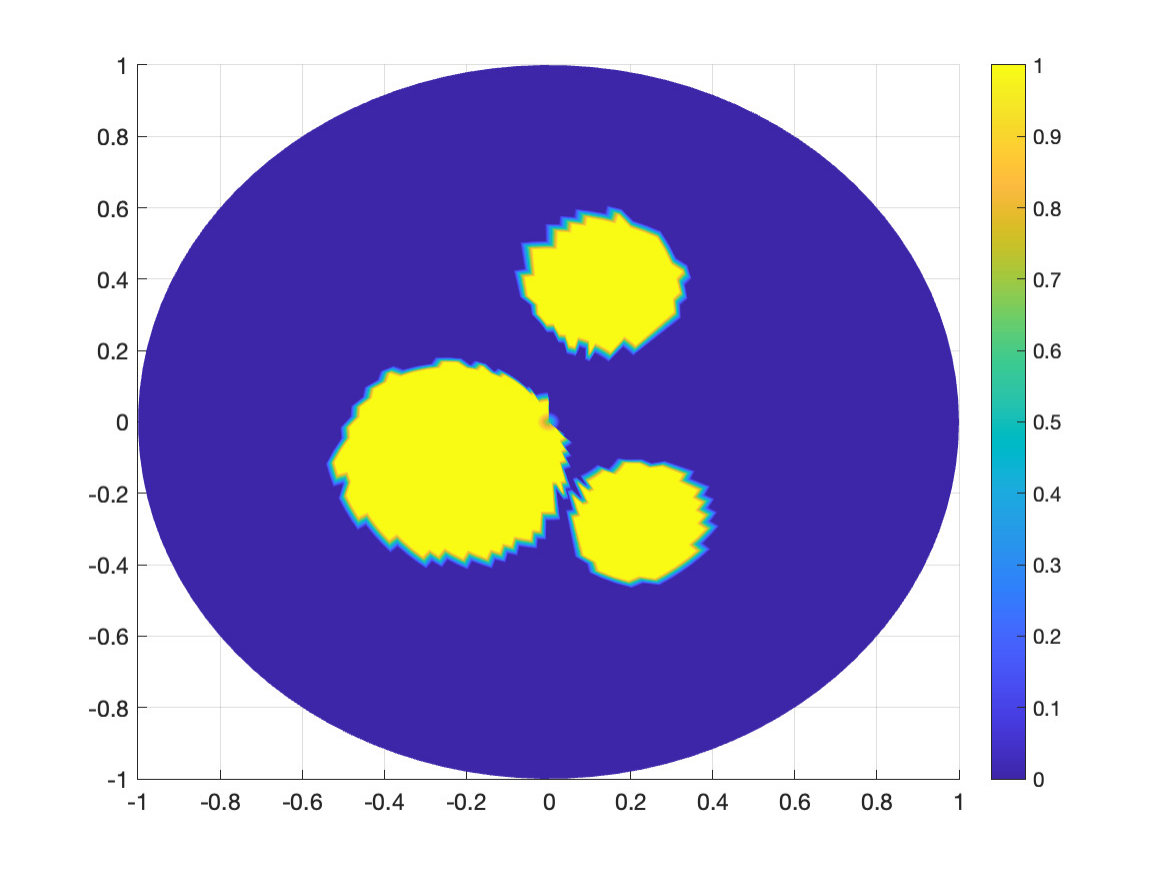}
\includegraphics[width=0.19\linewidth]{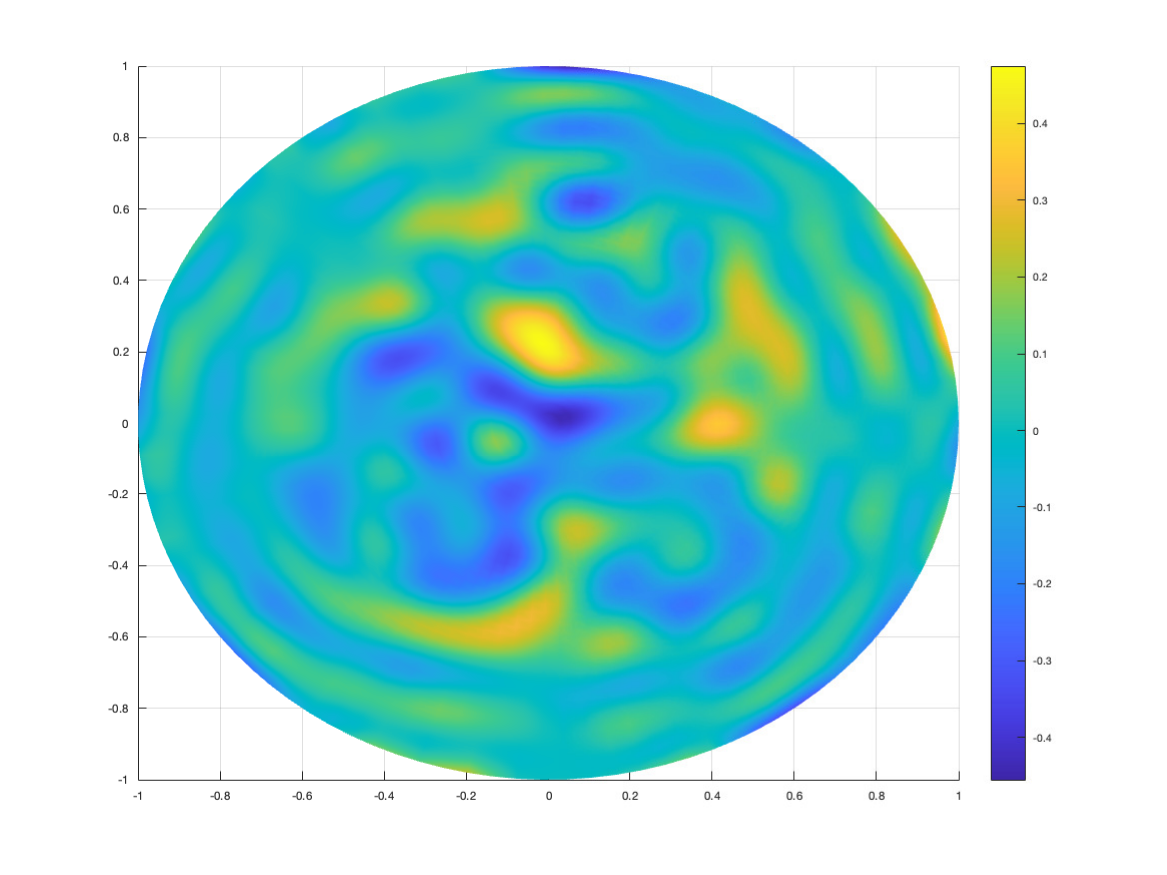}
\includegraphics[width=0.19\linewidth]{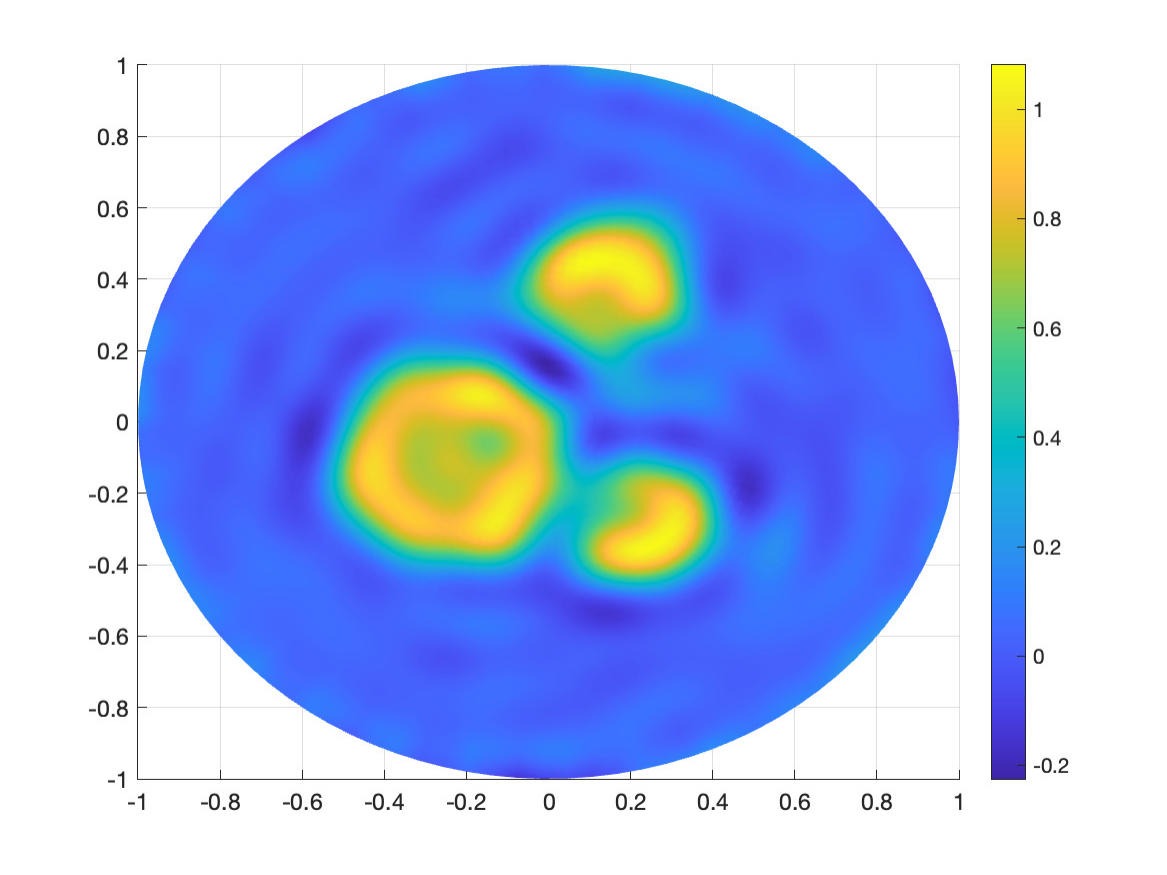}
\includegraphics[width=0.19\linewidth]{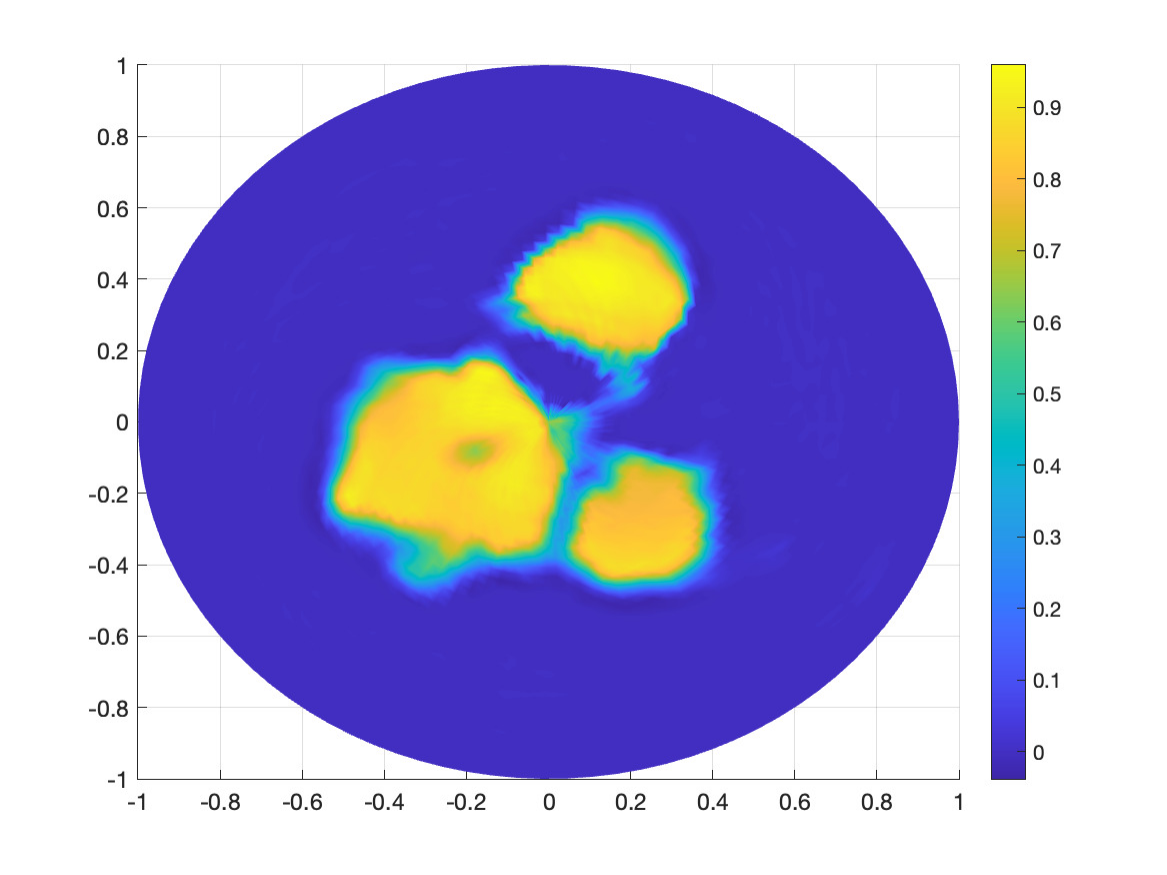}
\includegraphics[width=0.19\linewidth]{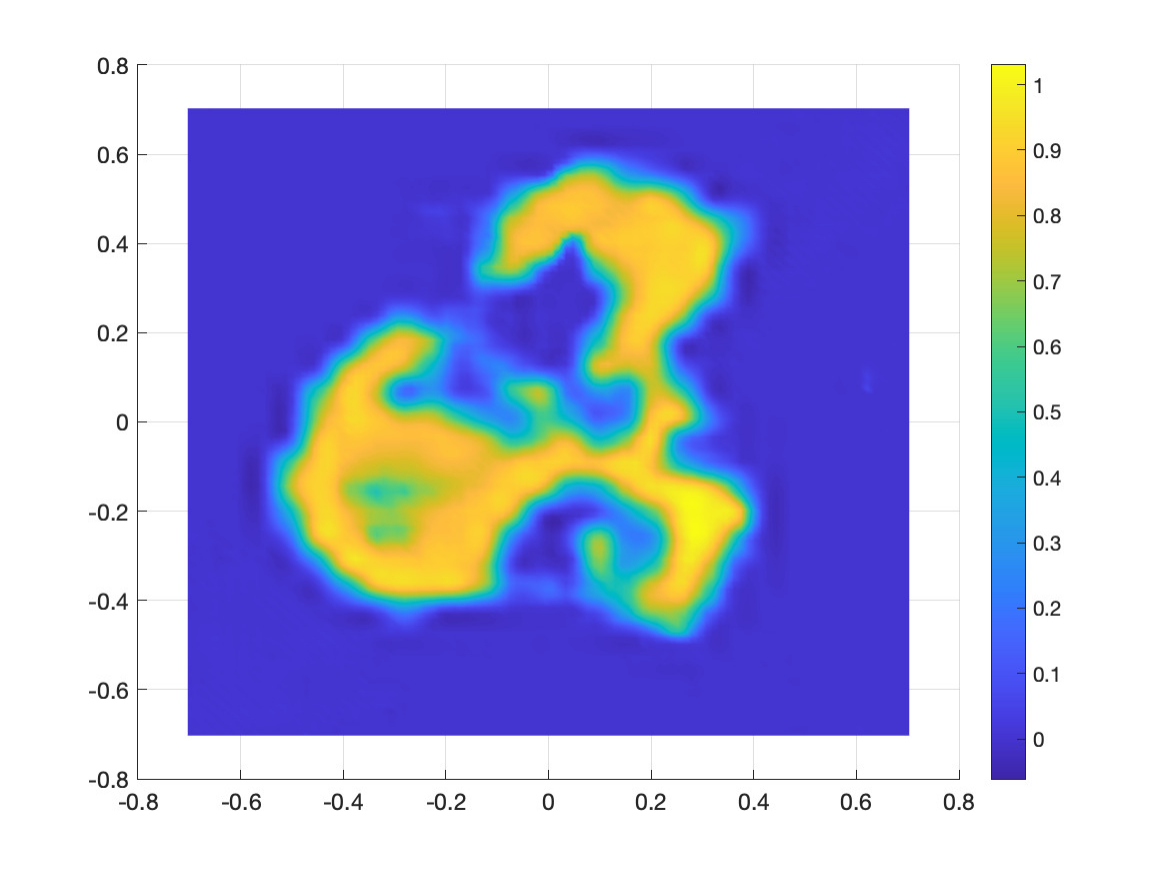}
  \caption{Reconstruction of three disks with   $q=0.7$ (top) and $q=1$ (bottom), corresponding to degree of nonlinearity  $2.3230$ (top) and $3.2166$ (bottom). Left to right: ground truth, reconstruction by the low-rank structure, ULR, UU, and U, respectively. }  \label{figure: generalization 3 balls}
\end{figure}

\subsection{A deeper insight on the differences} \label{sec: comparison of inverse Born solvers}
To get a deeper understanding of the proposed methods, we first focus only on the inverse Born solver by comparing the low-rank regularized method in Section \ref{sec: low-rank inverse Born}, the rotation-equivariance aware neural network $\mbox{U}_2$, and another black-box neural network. The ground truth is plotted in the first row of \Cref{figure: High amplitude reconstruction}, where $\|q\|_\infty=2,5,10,30$ from left to right.  Specifically, given Born far field data $\{u_b^\infty(\hat{x}_i;\hat{\theta}_j): 1\le i \le N_{\rm obs}, \, 1\le j \le N_{\rm inc}\}$, we   obtain the Born processed data and apply  the low-rank regularized method and the rotation-equivariance aware NN; on the other hand, we directly use the Born far field data as the input for the black-box neural network. With reference to \Cref{figure: High amplitude reconstruction}, it is observed that the rotation-equivariance-aware  $\mbox{U}_2$ generalize well to high contrast; however, the black-box neural network fails for these high contrast. Therefore, the proposed methods ULR and UU are expected to generalize well to high contrast if the data corrector $\mbox{U}_1$ has good generalization capability.
\begin{figure}[htbp]
\centering
{\includegraphics[width=0.24\linewidth]{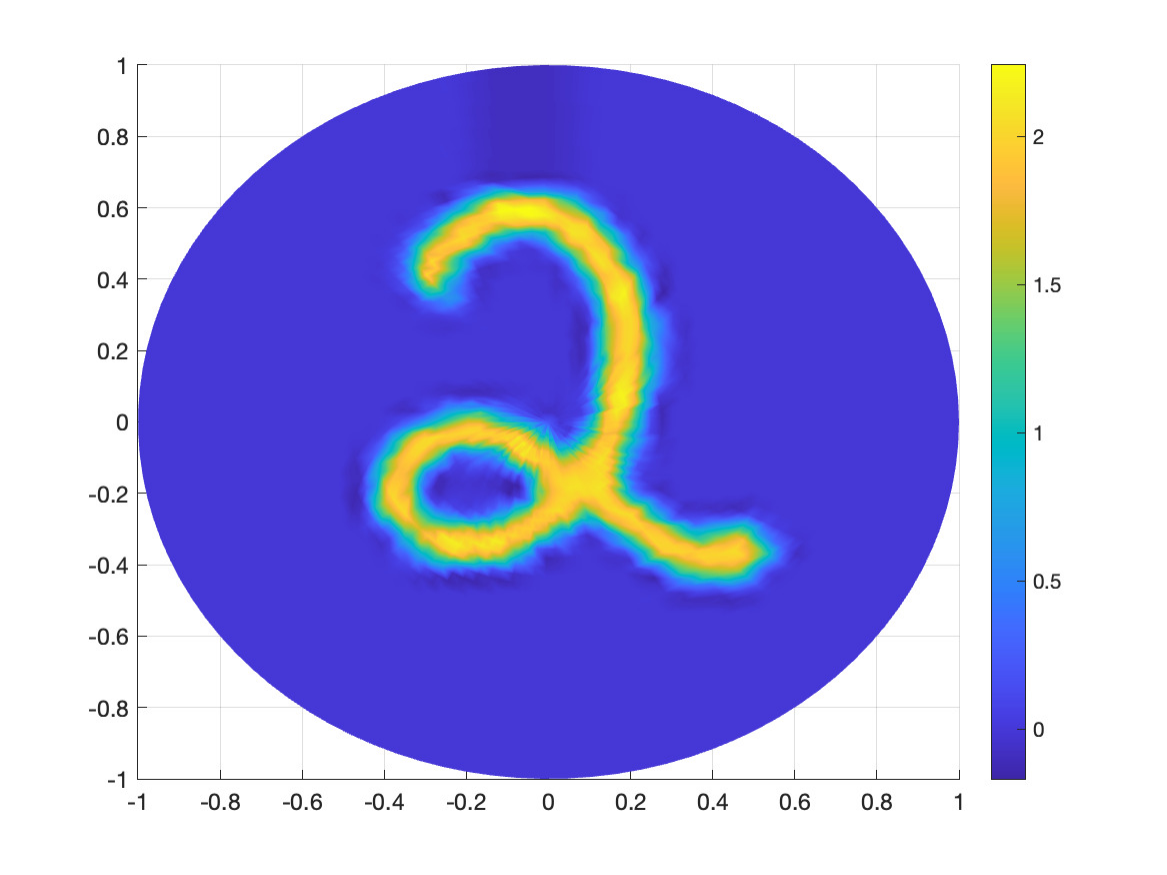}}
{\includegraphics[width=0.24\linewidth]{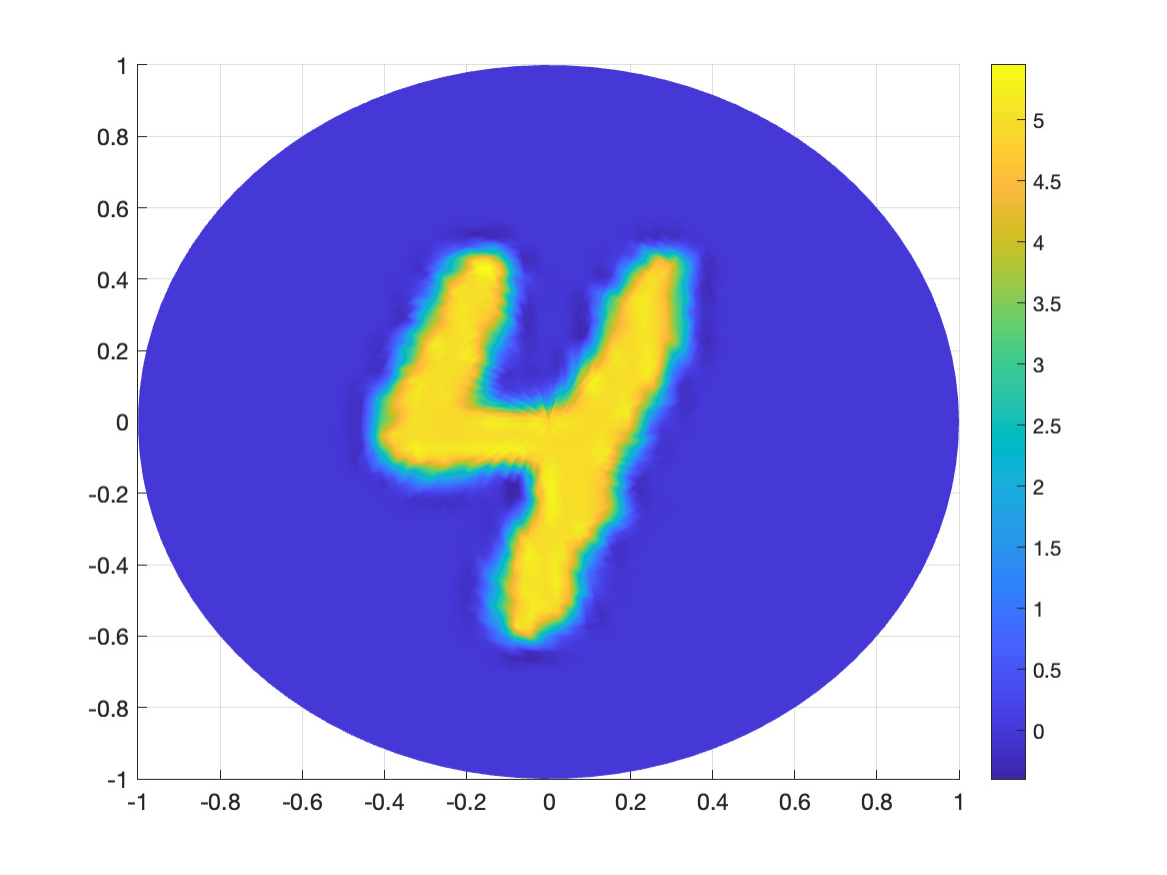}}
{\includegraphics[width=0.24\linewidth]{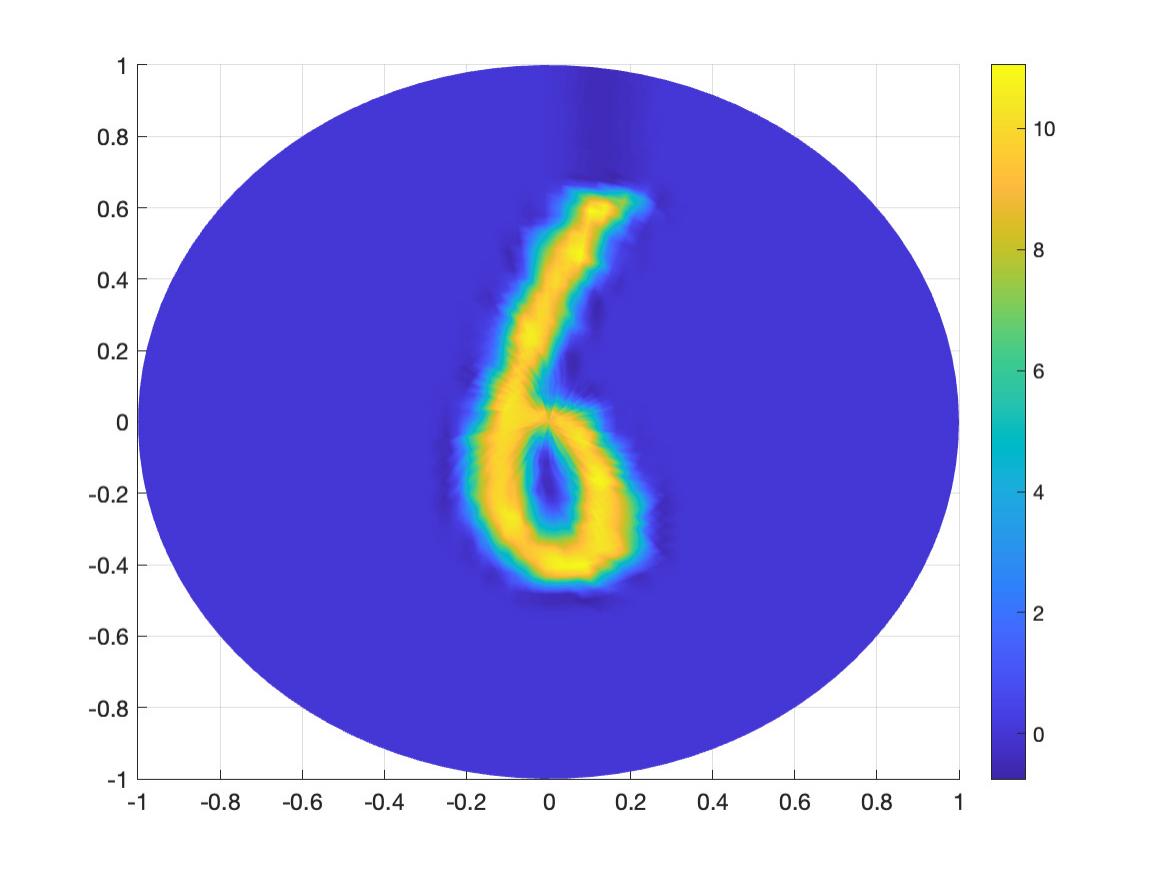}}
{\includegraphics[width=0.24\linewidth]{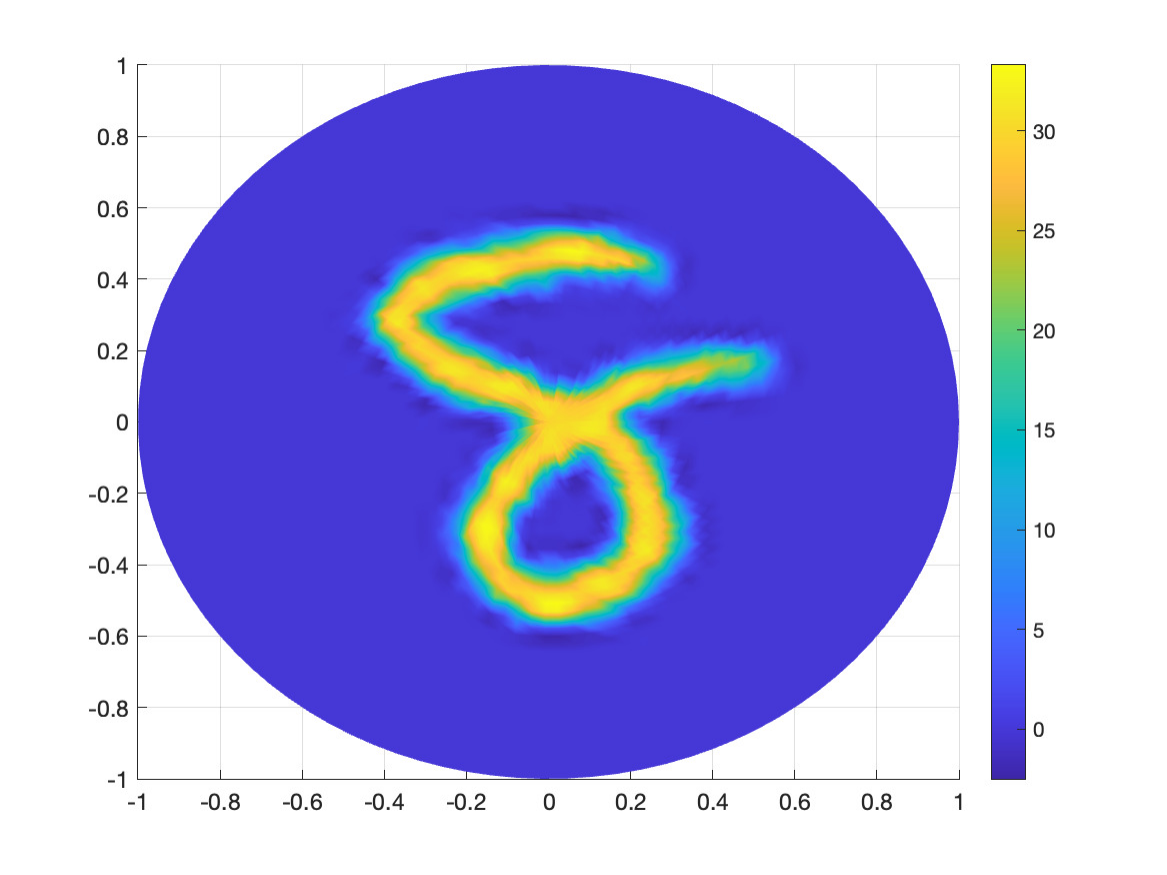}}\\
{\includegraphics[width=0.24\linewidth]{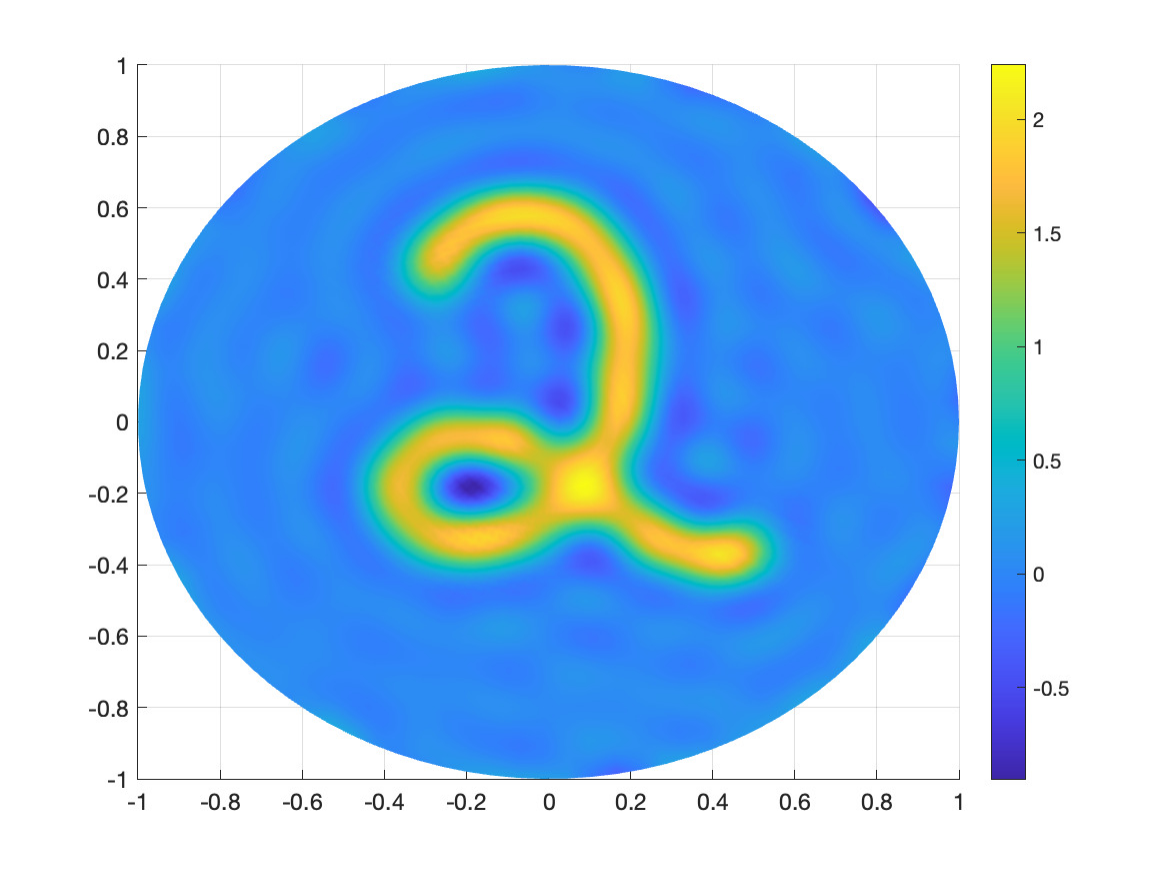}}
{\includegraphics[width=0.24\linewidth]{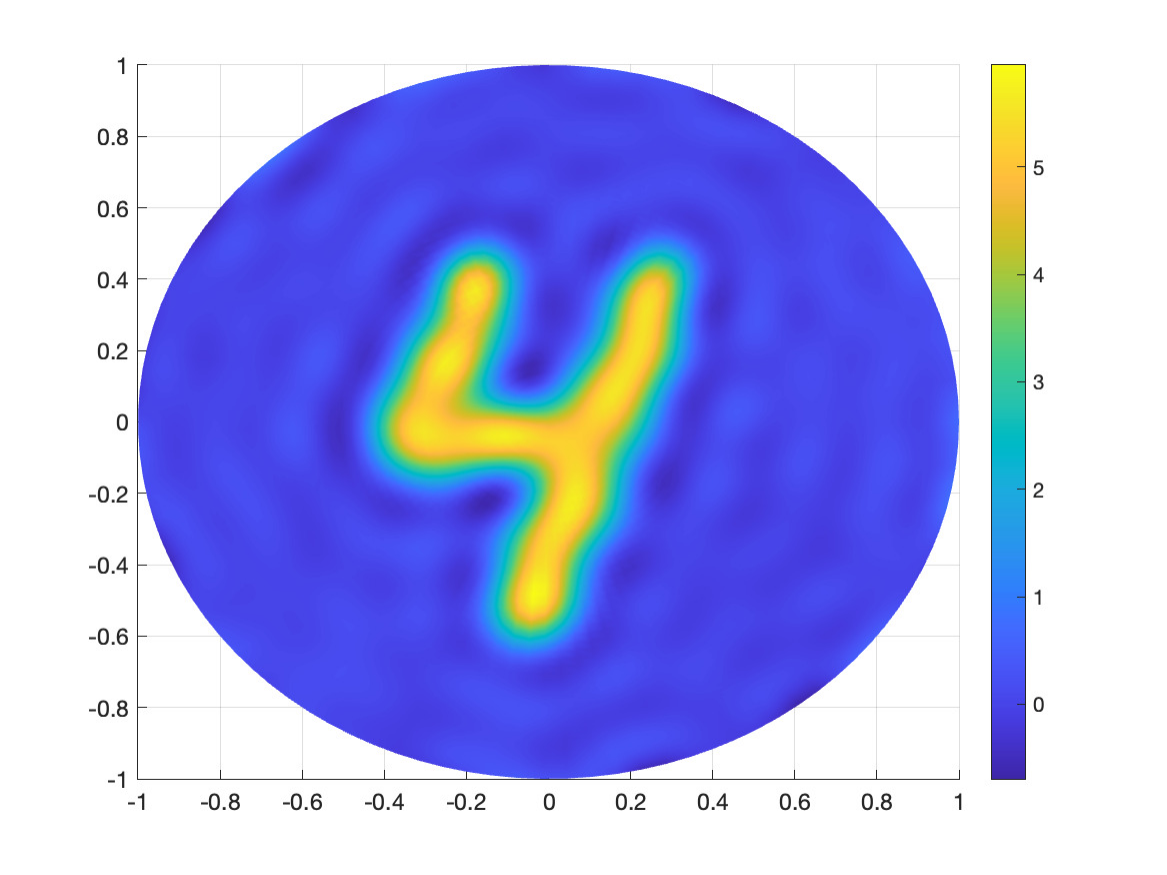}}
{\includegraphics[width=0.24\linewidth]{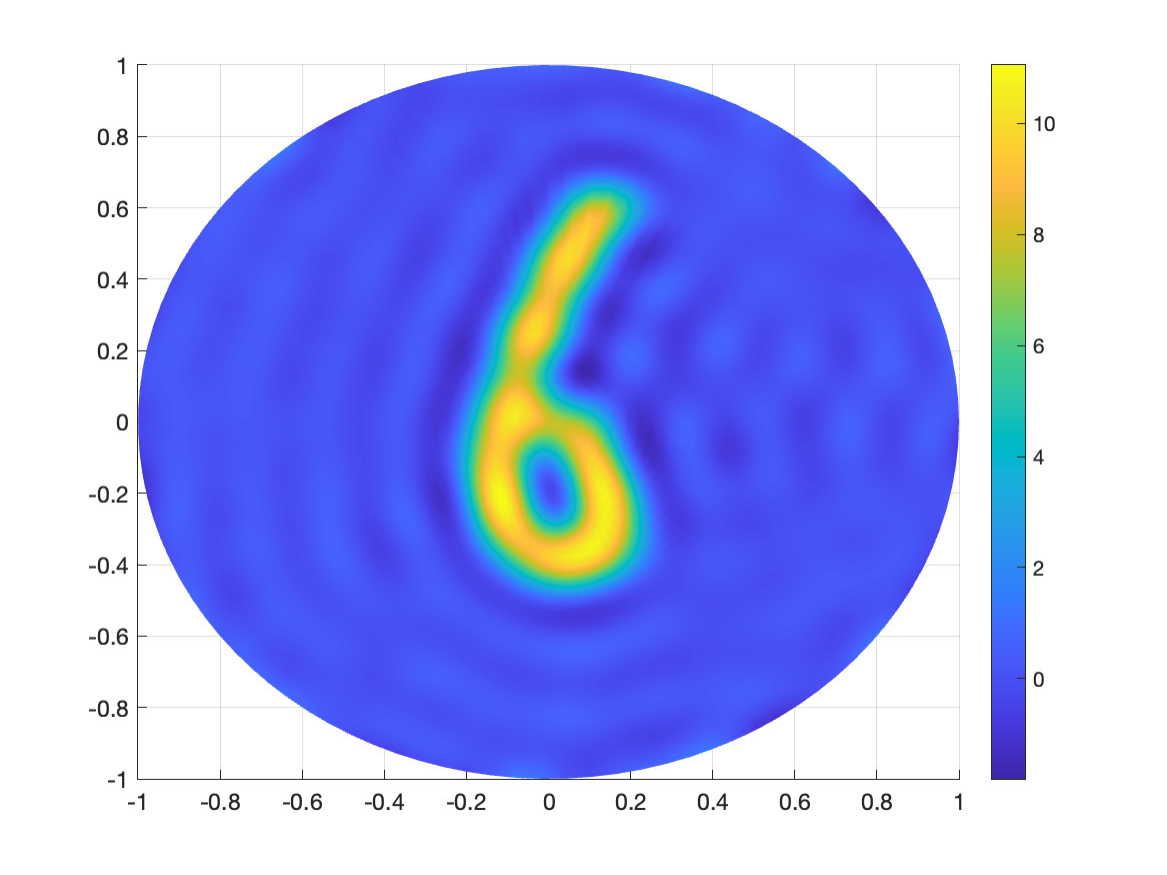}}
{\includegraphics[width=0.24\linewidth]{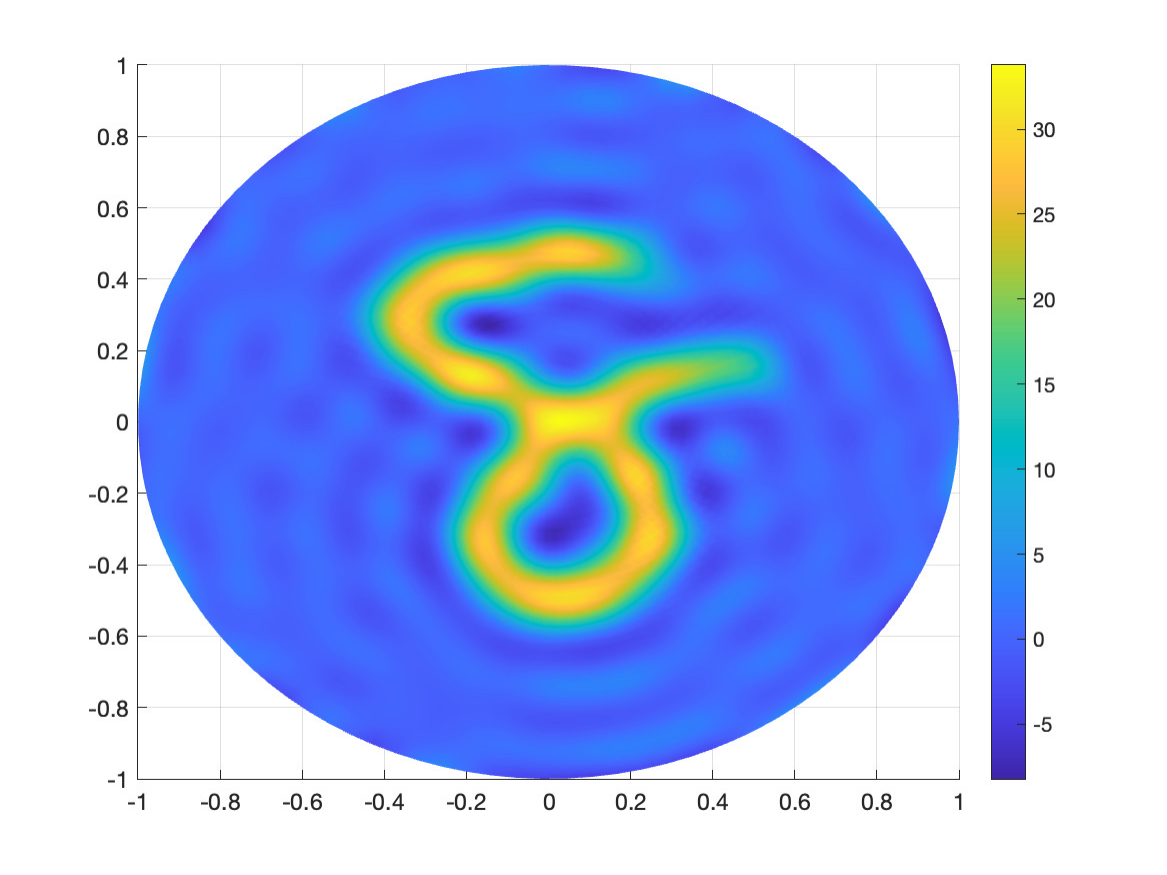}}\\
{\includegraphics[width=0.24\linewidth]{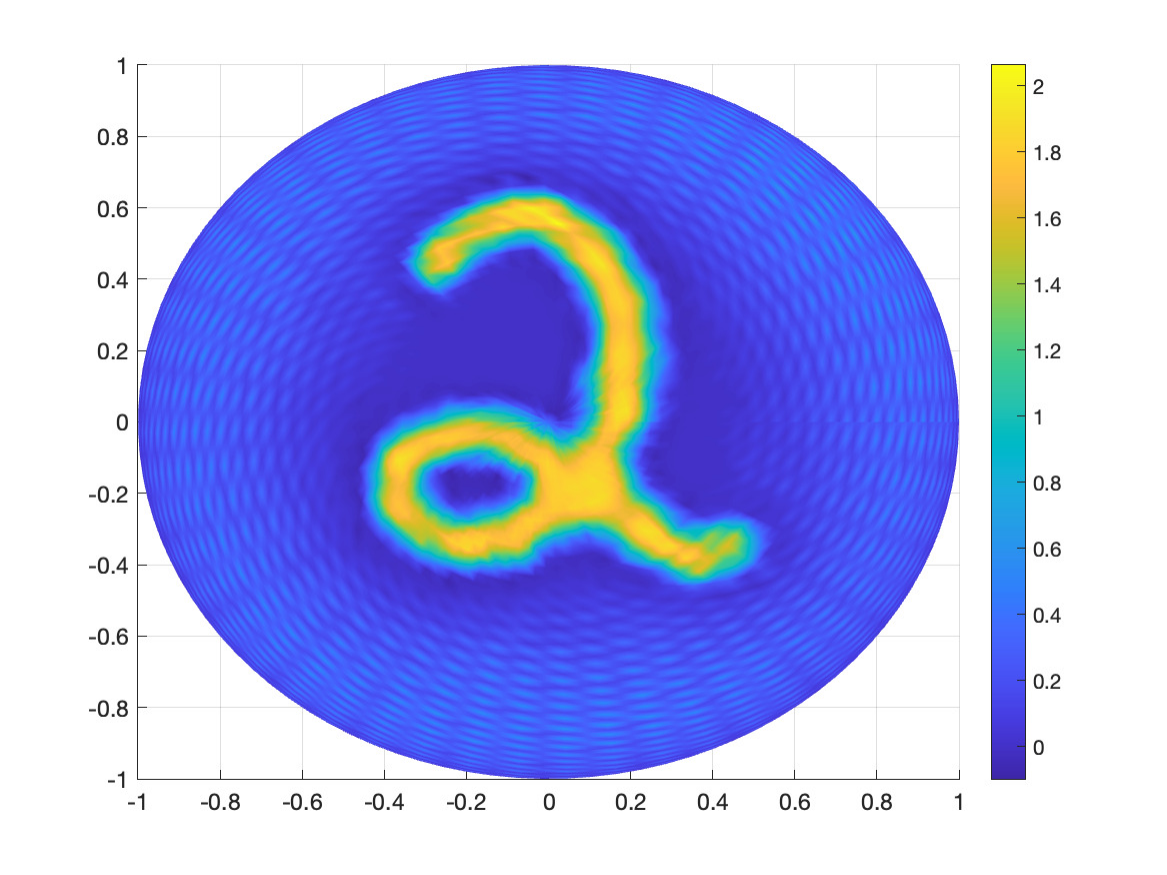}}
{\includegraphics[width=0.24\linewidth]{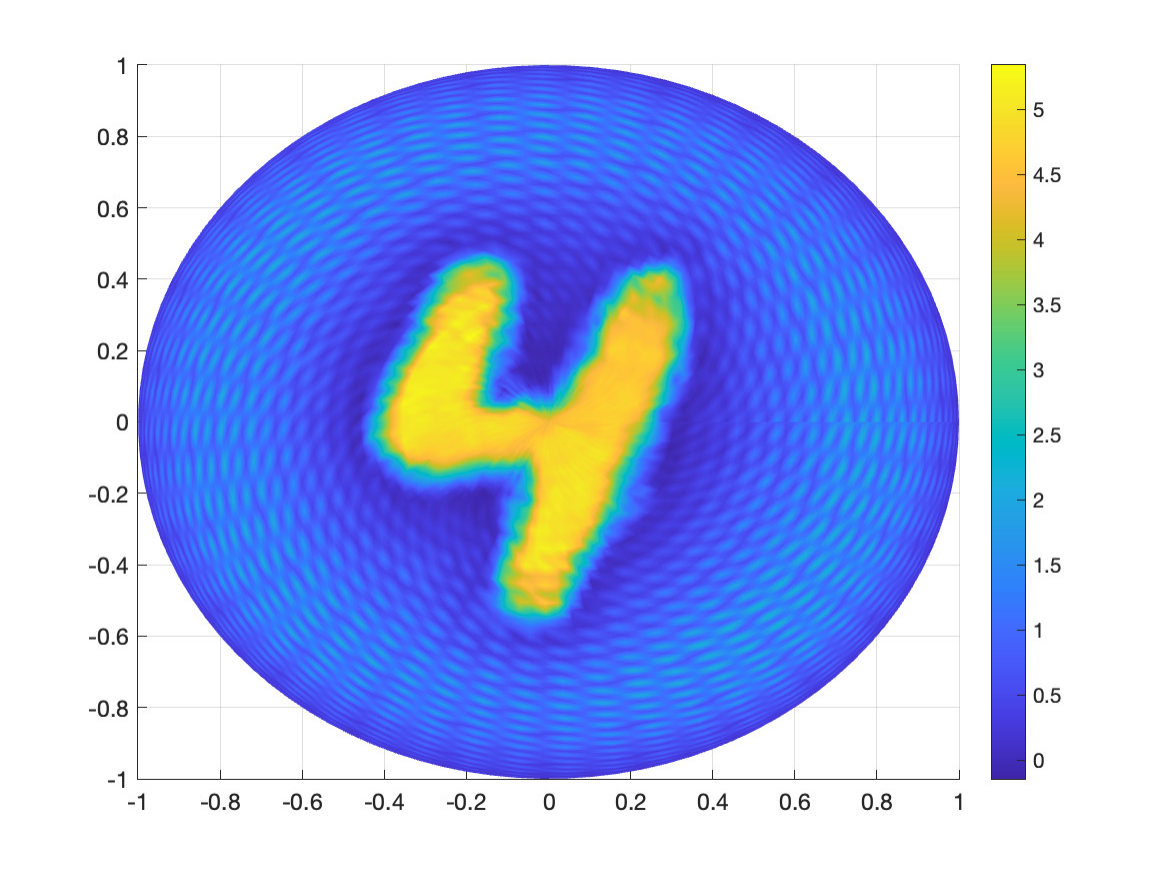}}
{\includegraphics[width=0.24\linewidth]{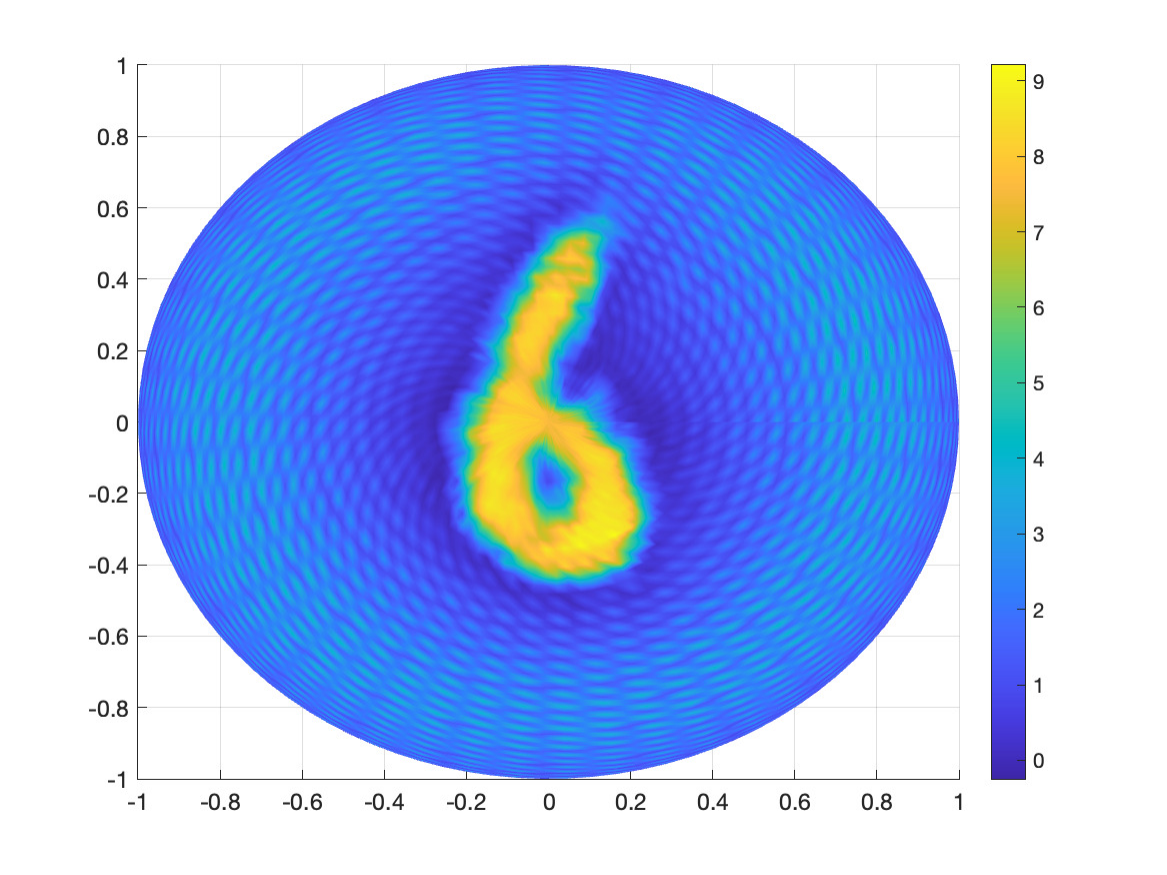}}
{\includegraphics[width=0.24\linewidth]{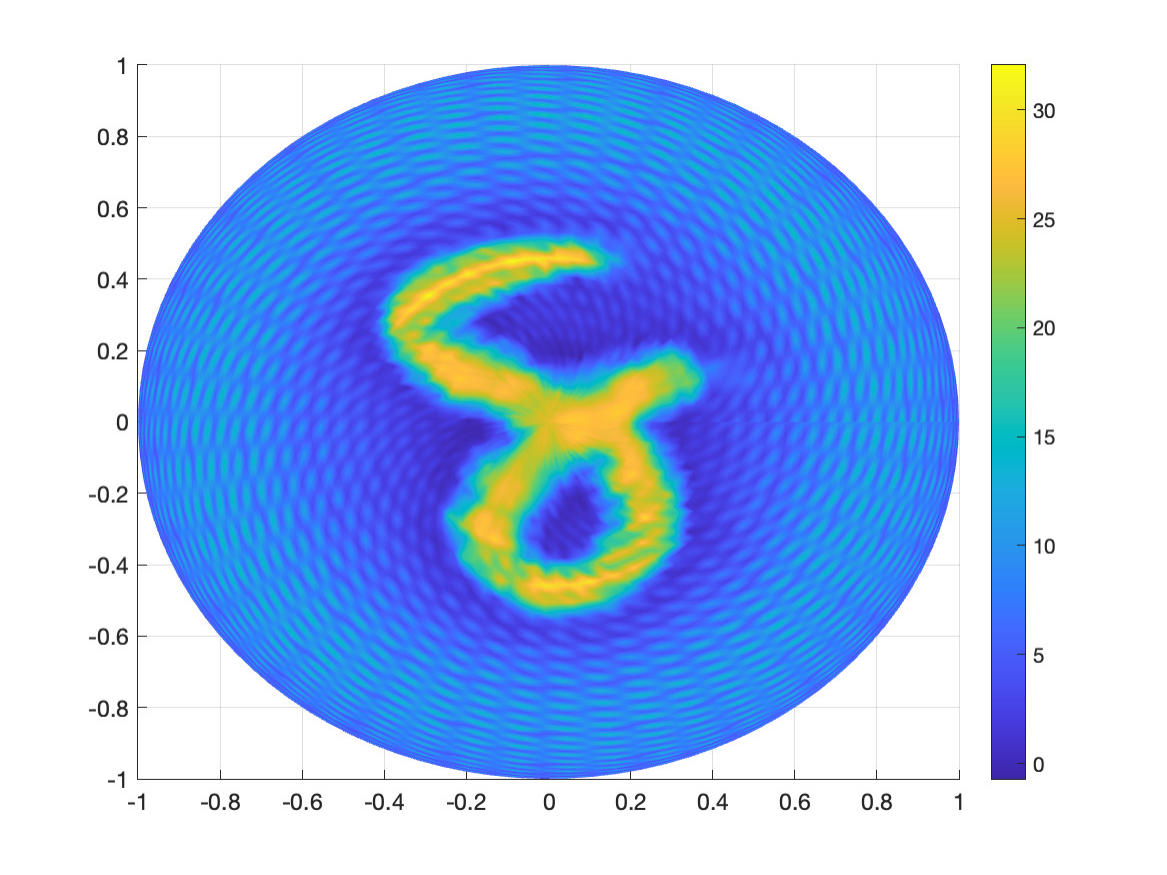}}\\
{\includegraphics[width=0.24\linewidth]{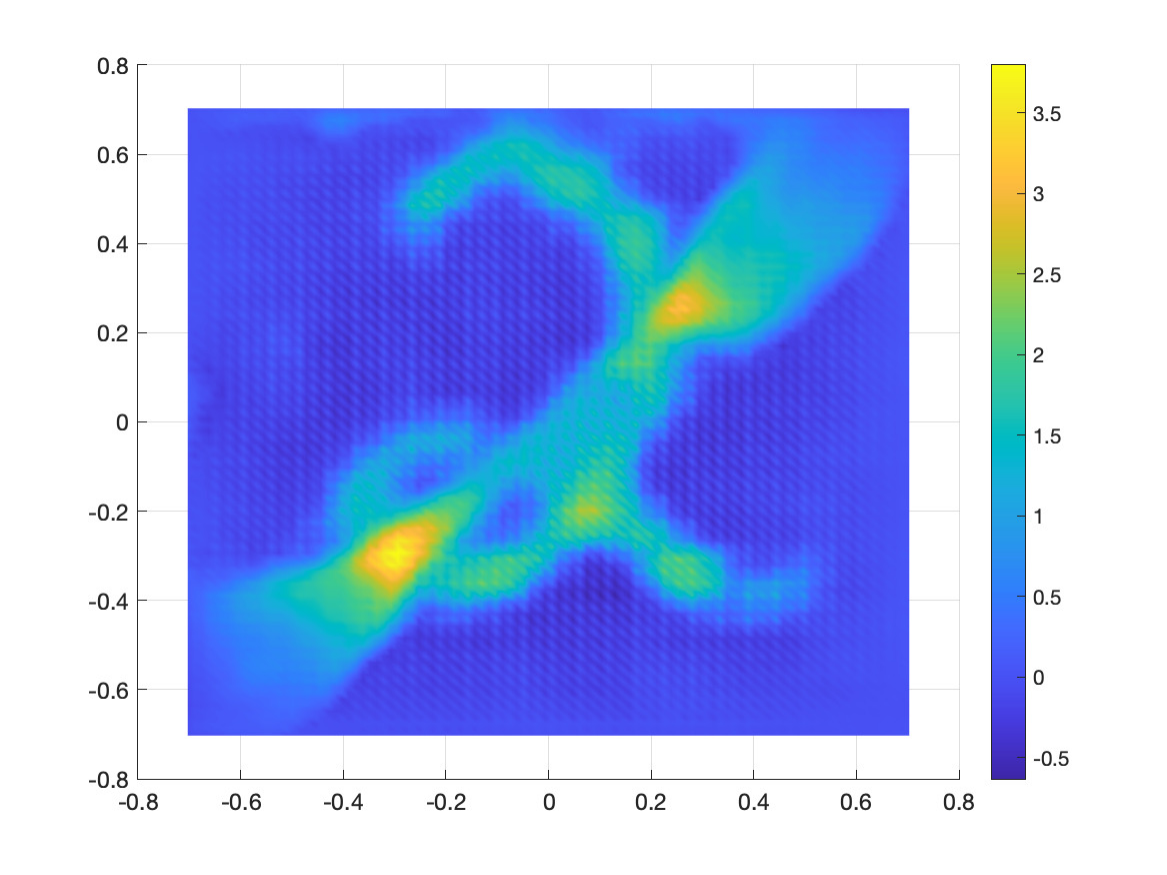}}
{\includegraphics[width=0.24\linewidth]{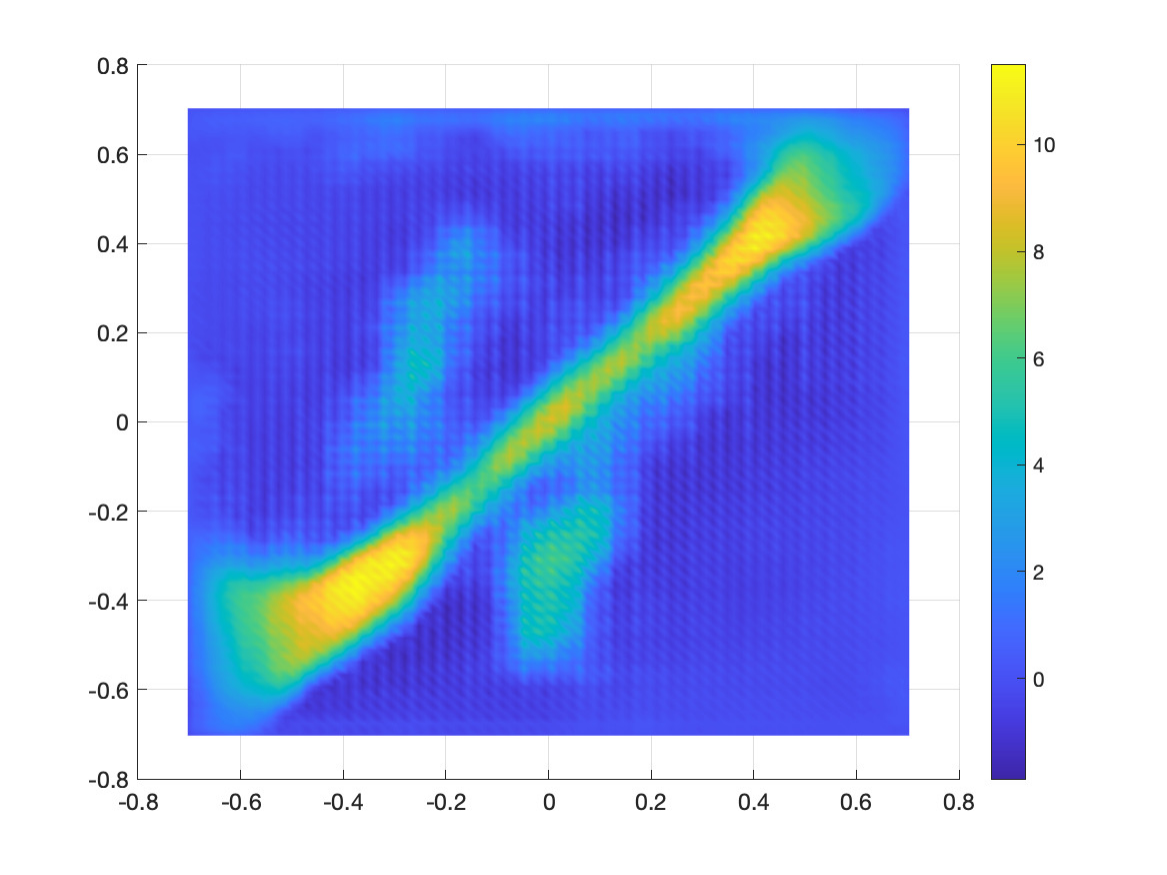}}
{\includegraphics[width=0.24\linewidth]{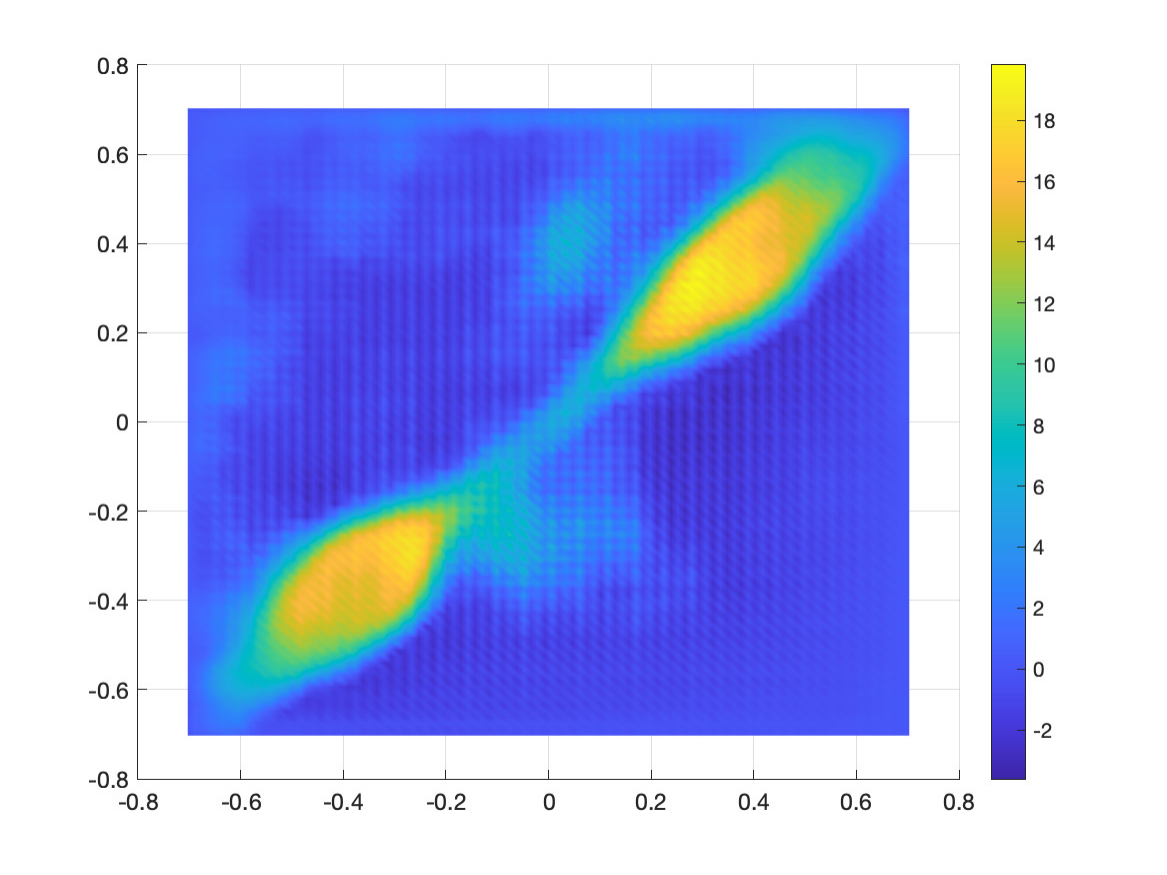}}
{\includegraphics[width=0.24\linewidth]{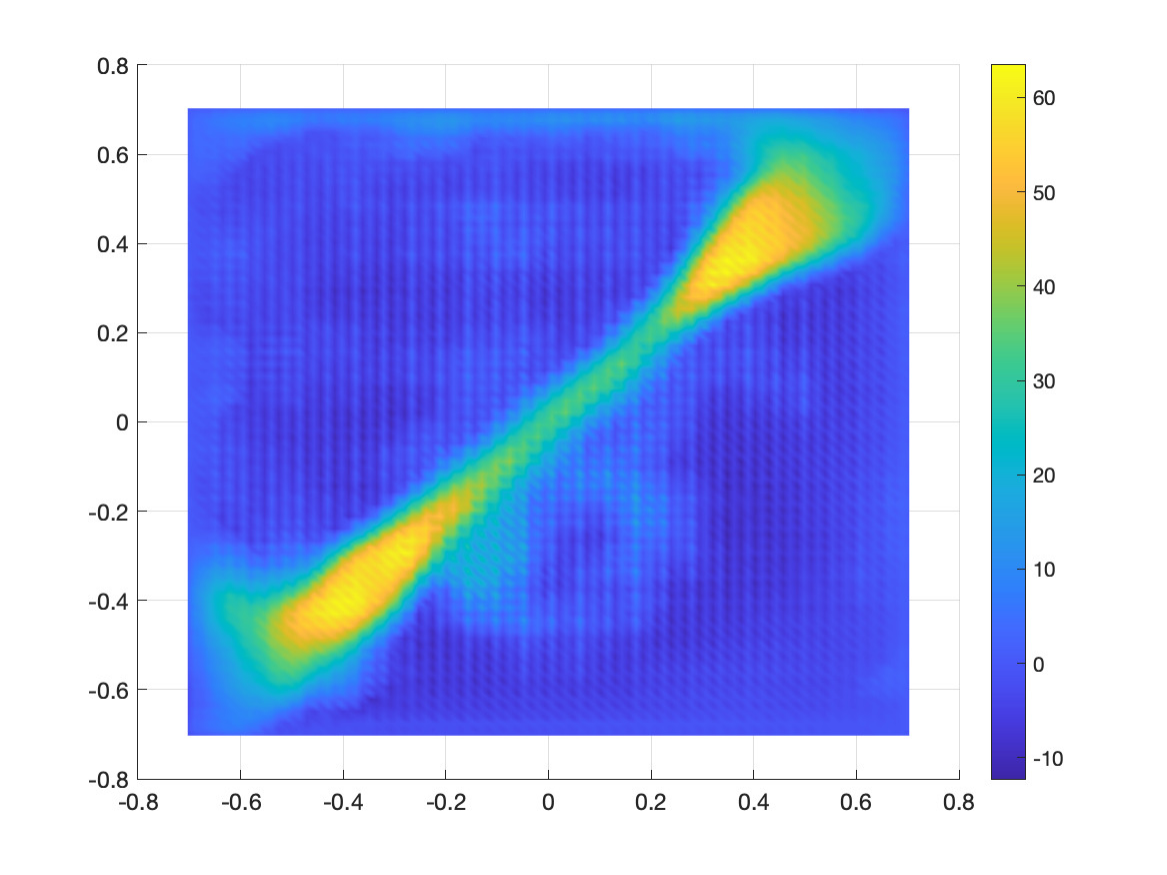}}\\

  \caption{Comparison of inverse Born solvers. We plot the ground truth in the first row, where $\|q\|_\infty=2,5,10,30$ from left to right. Given Born far field data $\{u_b^\infty(\hat{x}_i;\hat{\theta}_j): 1\le i \le N_{\rm obs}, \, 1\le j \le N_{\rm inc}\}$, we compare three methods for reconstructing the contrast. Row 2: process the  Born far field data $\{u_b^\infty(\hat{x}_i;\hat{\theta}_j): 1\le i \le N_{\rm obs}, \, 1\le j \le N_{\rm inc}\}$ to the Born processed data$\{ u(p_{m,n};c): 0\le m \le N_2-1,\, 1 \le n \le N_1\}$ and apply the low-rank regularized method of Section \ref{sec: low-rank inverse Born}. Row 3: process the  Born far field data $\{u_b^\infty(\hat{x}_i;\hat{\theta}_j): 1\le i \le N_{\rm obs}, \, 1\le j \le N_{\rm inc}\}$ to the Born processed data$\{ u(p_{m,n};c): 0\le m \le N_2-1,\, 1 \le n \le N_1\}$ and apply the rotation-equivariance aware neural network $\mbox{U}_2$. Row 4: apply a black-box neural network $U$ to reconstruct the contrast directly from the Born far field data $\{u_b^\infty(\hat{x}_i;\hat{\theta}_j): 1\le i \le N_{\rm obs}, \, 1\le j \le N_{\rm inc}\}$. }  \label{figure: High amplitude reconstruction}
\end{figure}

To further shed light on the difference between ULR and UU, we test a set of out-distribution contrasts in \Cref{figure: generalization}. 
It is clear that both ULR and UU outperform  the black-box neural network U.
Note that in  ULR and UU, the first data corrector  $\mbox{U}_1$ is the same, and the only difference between ULR and UU is  the inverse Born solver. Therefore the difference between the second column (ULR) and third column (UU) is due to how they process the imperfect Born data (i.e., the output of the data corrector). It turns out that the low-rank regularized method merits the additional noise filtering property: the unnecessary high-frequency noise in the imperfect Born data can be filtered out by the low-rank structure.

\subsection{Dataset expansion by low-rank structure}
Recall that the training dataset is based on piece-wise constant contrasts, the poor reconstruction of a smooth contrast is not surprising, cf. second column of \Cref{figure: generalization ability enhancement}. This new experiment in \Cref{figure: generalization ability enhancement} is concerned with smooth contrasts.  The first ground truth contrast is given by a superposition of random Gaussian scatterers $q(x)=R\sum_{j=1}^{25}r_j e^{-45(x_1-a_j-n_j)^2-60(x_2-b_j-m_j)^2}$ with $a_j=\frac{1}{4}[(j ~\mbox{mod } 5)]-0.5$, $b_j=-\frac{1}{4}\lfloor\frac{j}{5}\rfloor +0.5$,   $m_j,~n_j\sim U([-0.005,0.015])$, $r_i\sim U([-1,1])$, and $R$ is such that $\|q\|_\infty=0.7$. The second and third contrasts are given by $q(x)=0.6(1-x_1^2-x_2^2)\cos x_1\sin x_2$ and $q(x)=0.5\psi_{2,3,1}(x)$ where $\psi_{2,3,1}$ is a disk PSWF, respectively. To improve the generalization capability, we explore to expand  the original training dataset by adding three additional datasets, respectively.  
\begin{itemize}
    \item Dataset-(B) based on $20000$ Gaussian samples:  $$q(x,y)=R\sum_{i=1}^n r_i e^{-a_i(x-x_i)^2-b_i(y-y_i)^2}$$ where  $n\in \{1,2,3\}$ is random integer, $a_i,~b_i\sim  U([16,66])$,  $x_i,~y_i\sim  U([-0.4,0.4])$,  $r_i\sim U([-1,1])$ and  $R$ is a   random number  such that $\|q\|_\infty\sim U([0.5,0.8])$.
    \item Dataset-(C) based on $20000$ disk PSWFs samples: $$q(x)=R\sum_{m<5,n<5}\frac{\xi_{m,n,\ell}}{\chi_{m,n}^{1/2}}\psi_{m,n,\ell}(x),$$
    where $\xi_{m,n,\ell}$ is a random number that follows the standard Gaussian distribution, $\chi_{m,n}>0$ is the Sturm-Liouville eigenvalue in \eqref{sturm-liouvill}, and $R$ is a random number such that  $\|q\|_\infty\sim U(0.5,0.8)$.
    \item Dataset-(D) based on $40000$ disk PSWFs samples: 
    $$q(x)=R\sum_{m<10,n<10}\frac{\xi_{m,n,\ell}}{\chi_{m,n}^{1/2}}\psi_{m,n,\ell}(x),$$ where $\xi_{m,n,\ell}$, $\chi_{m,n}$, and $R$ are chosen exactly the same as in dataset-(C). The only difference in dataset-(C) and dataset-(D) is the number of disk PSWFs. We restrict ourself to $m<10$ and $n<10$ due to limited computational resource.
\end{itemize}

With reference to \Cref{figure: generalization ability enhancement}, we plot the reconstructions from left to right  based on:  dataset-(A) (original training dataset consisting of MNIST and circular samples),  dataset-(A)(B),  dataset-(A)(C), and dataset-(A)(D), respectively. It is observed that the reconstructions are rarely improved with dataset-(A)(B), and are largely improved with dataset-(A)(C). Moreover, the dataset-(A)(D) leads to improved reconstruction by adding more samples to the dataset-(A)(C).

\begin{figure}[htbp]
\includegraphics[width=0.24\linewidth]{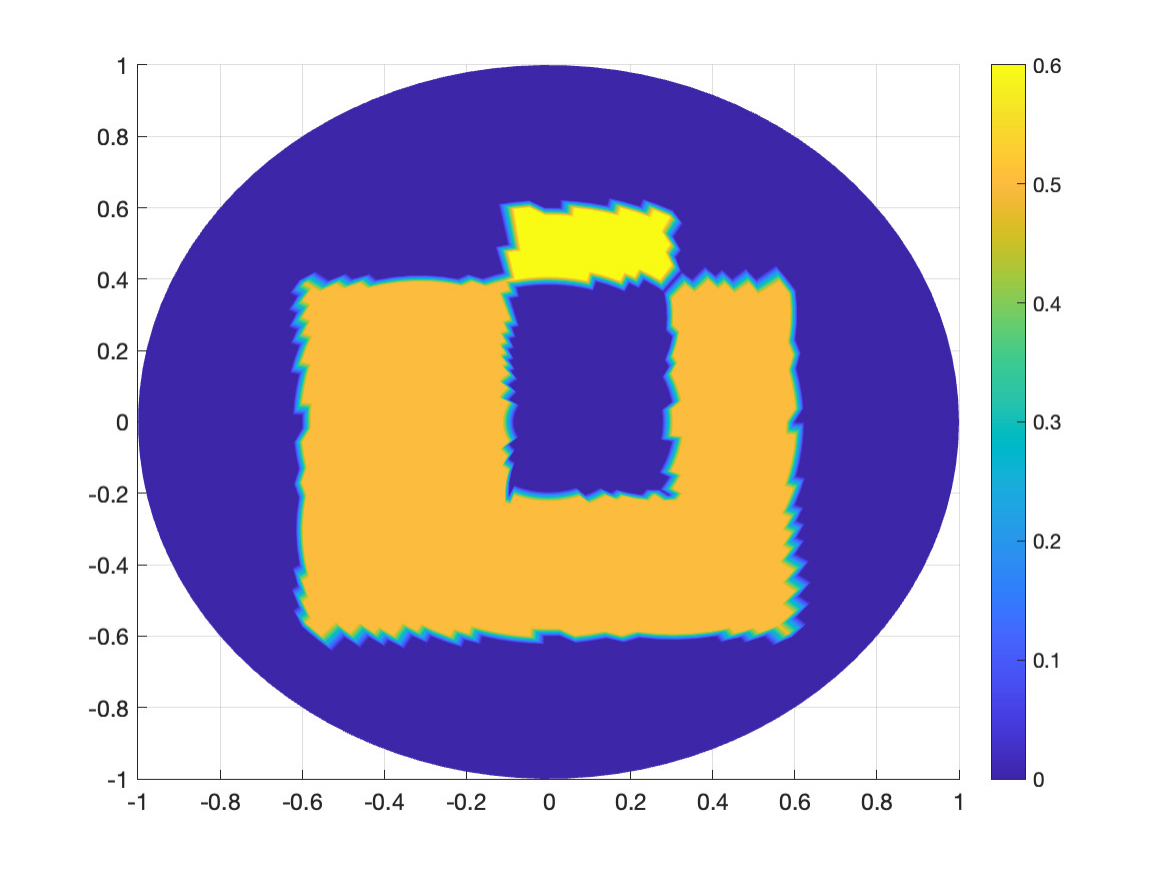}
\includegraphics[width=0.24\linewidth]{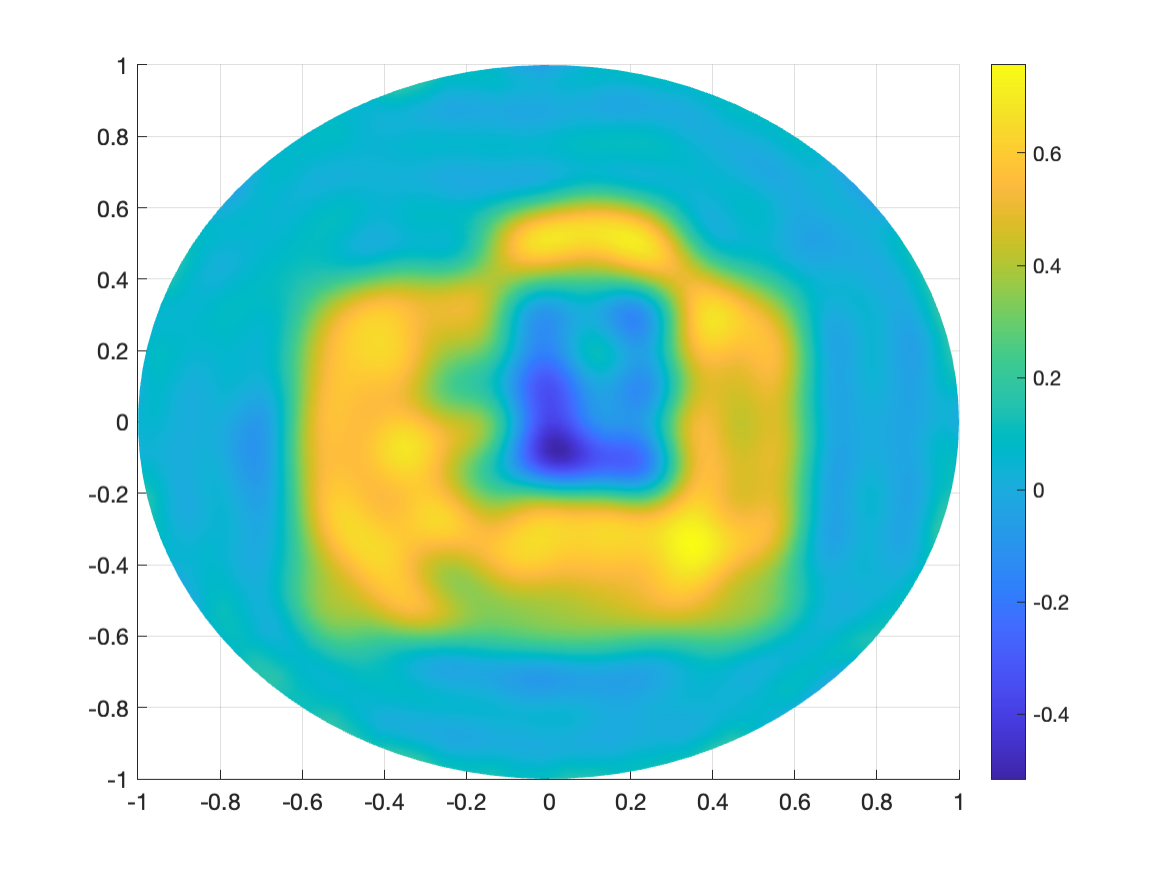}
\includegraphics[width=0.24\linewidth]{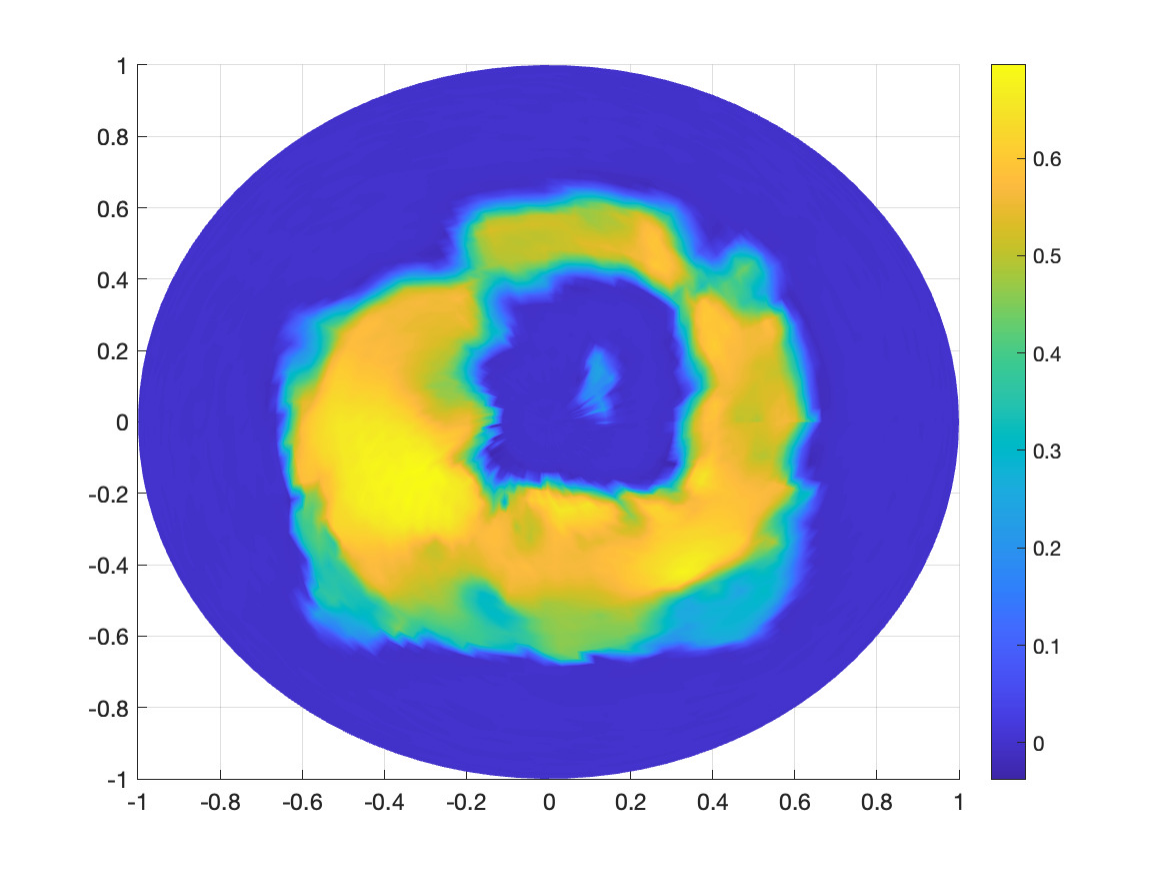}
\includegraphics[width=0.24\linewidth]{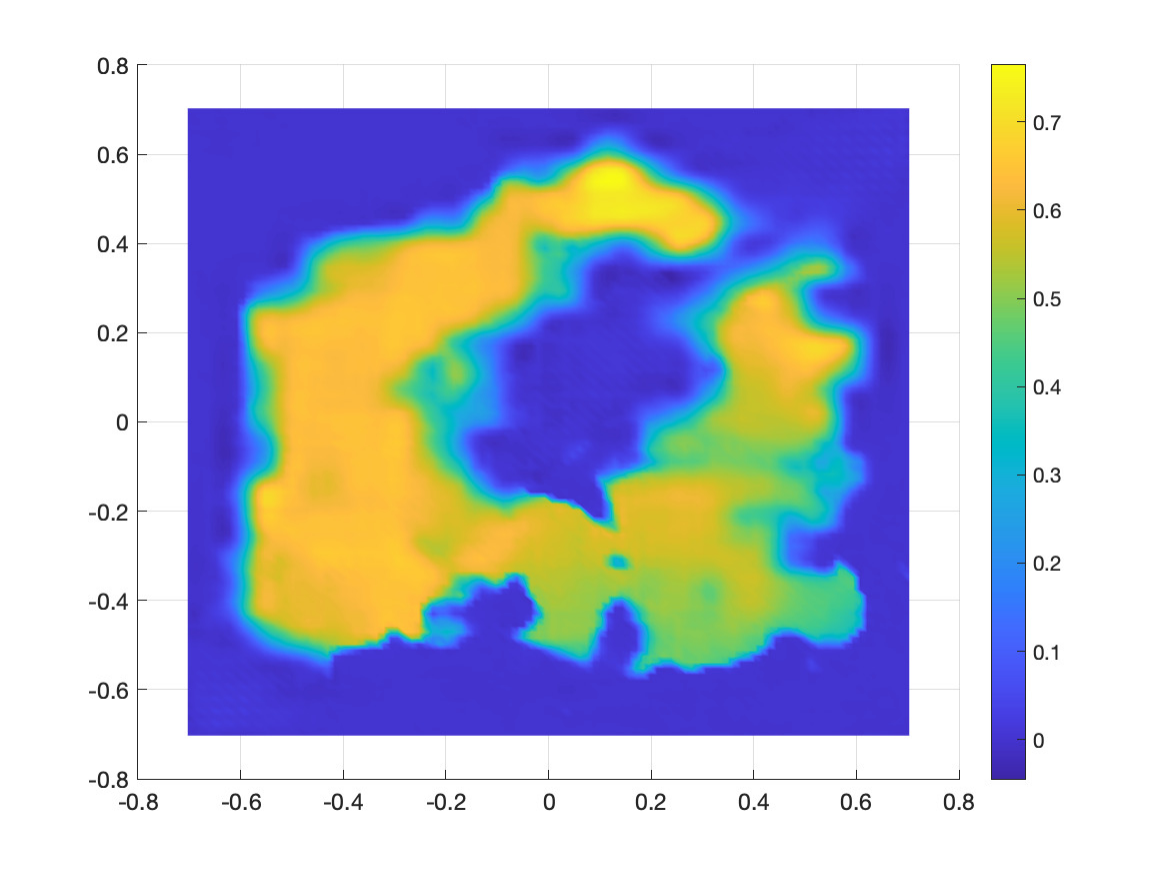}\\
\includegraphics[width=0.24\linewidth]{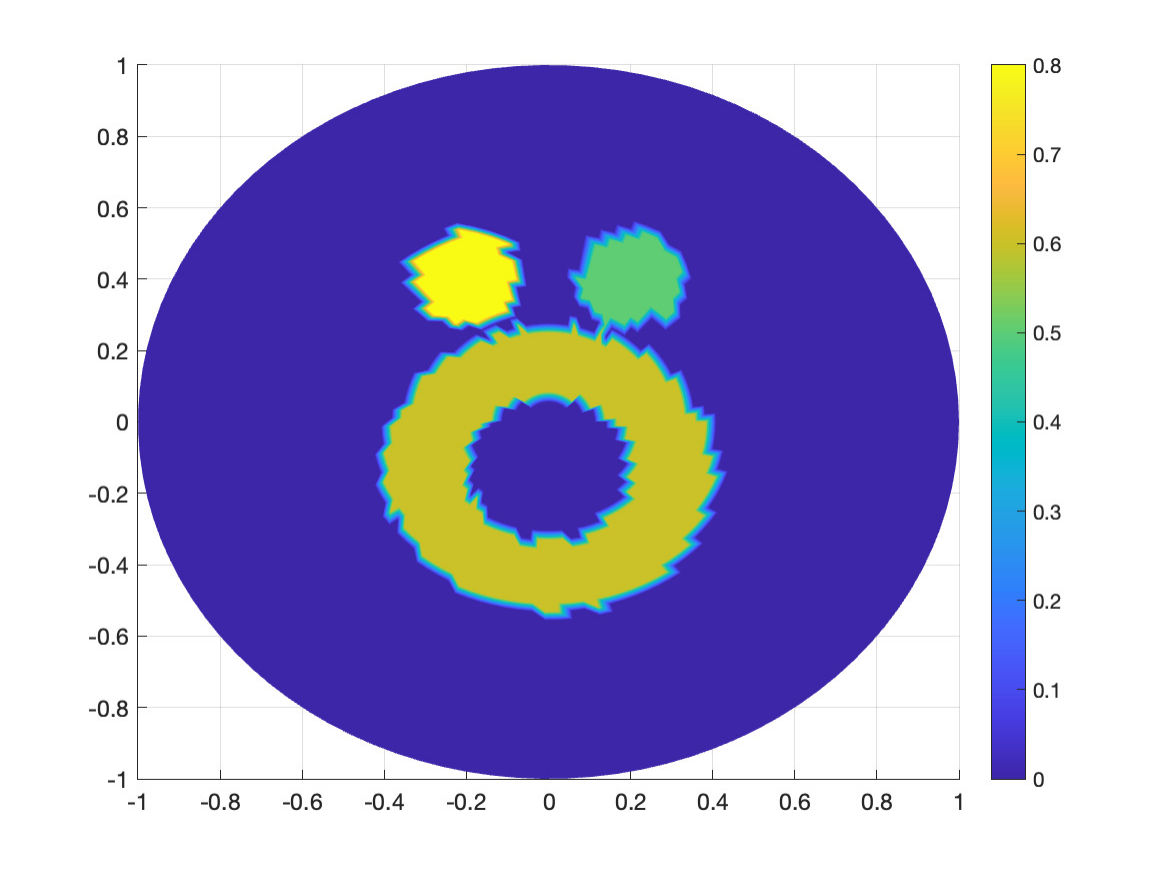}
\includegraphics[width=0.24\linewidth]{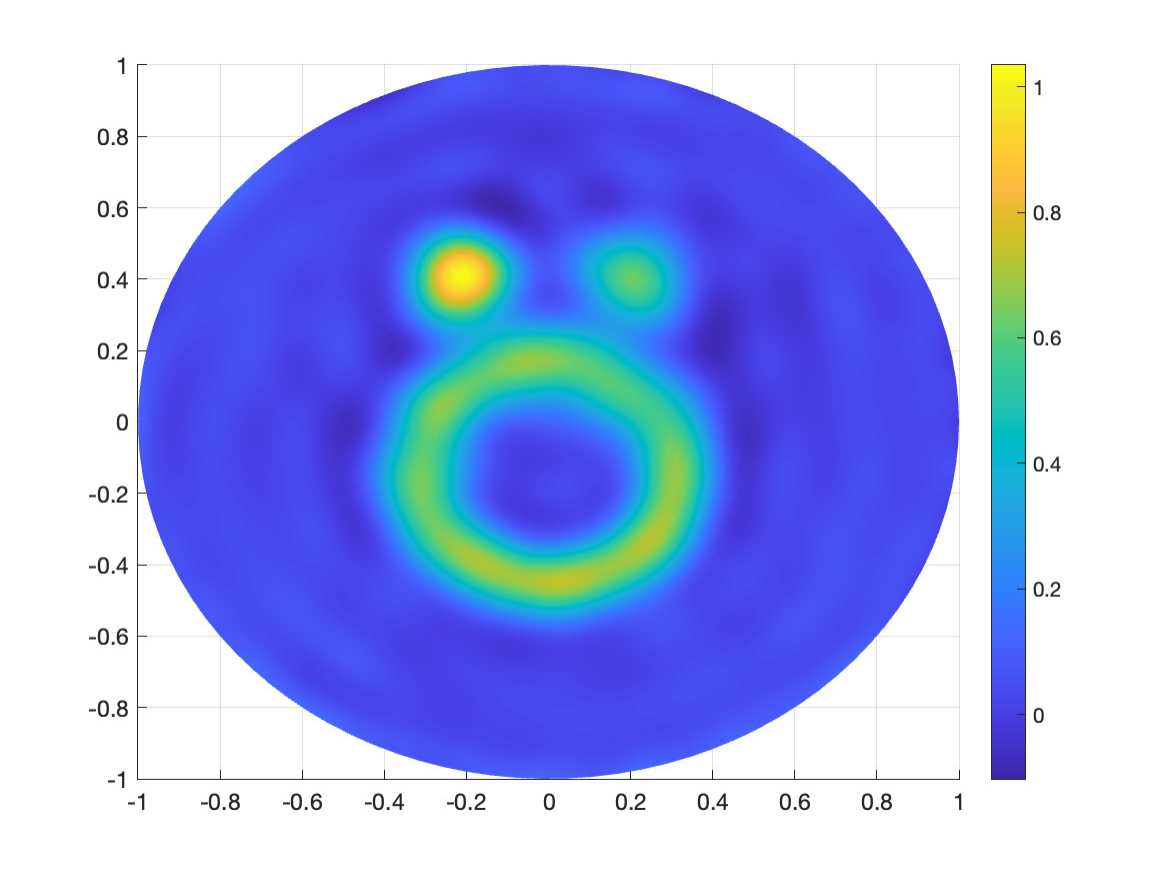}
\includegraphics[width=0.24\linewidth]{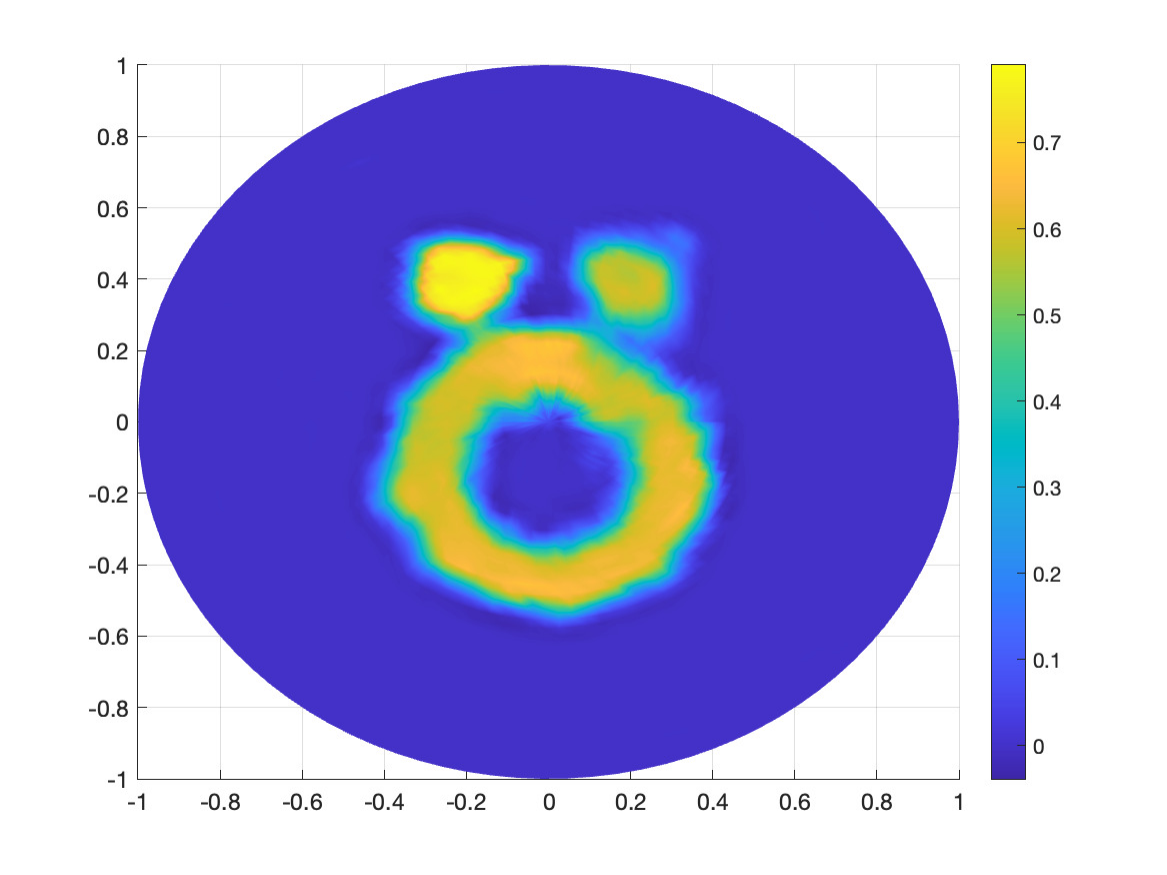}
\includegraphics[width=0.24\linewidth]{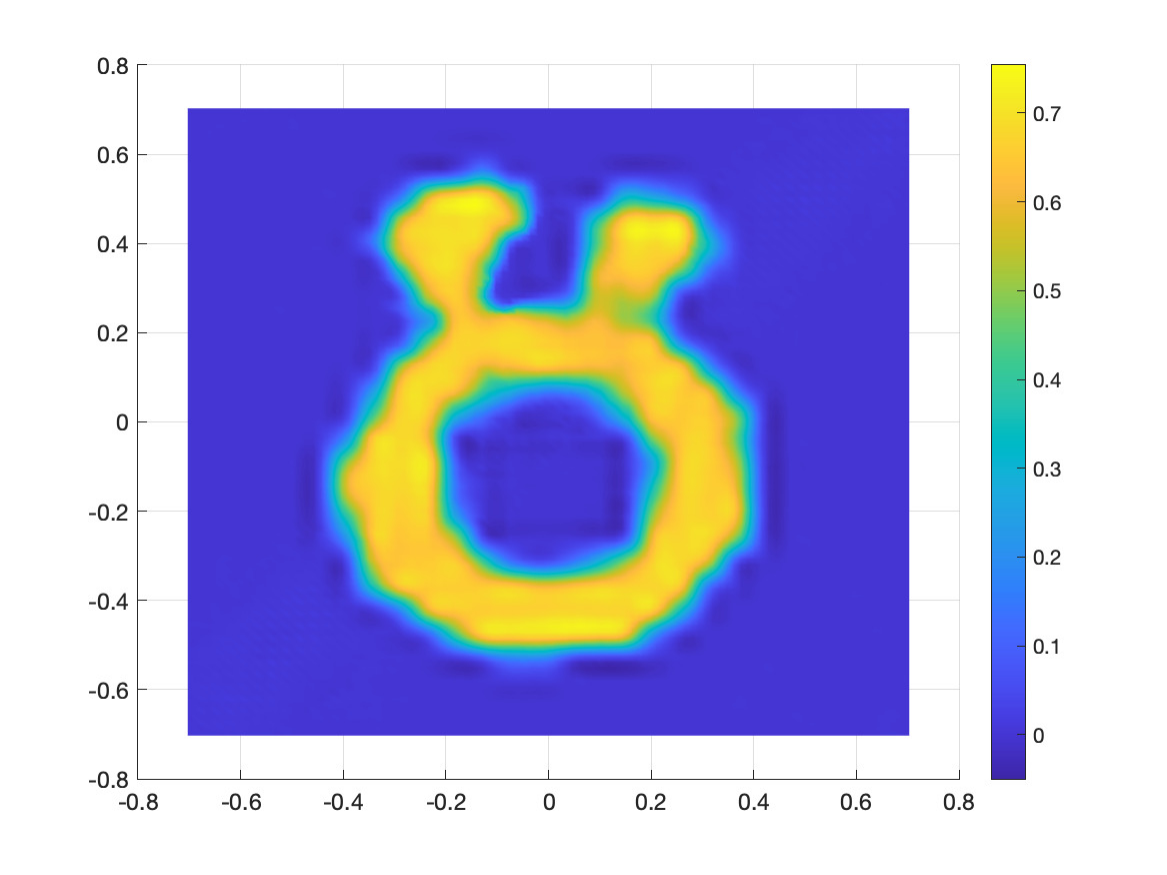}\\
\includegraphics[width=0.24\linewidth]{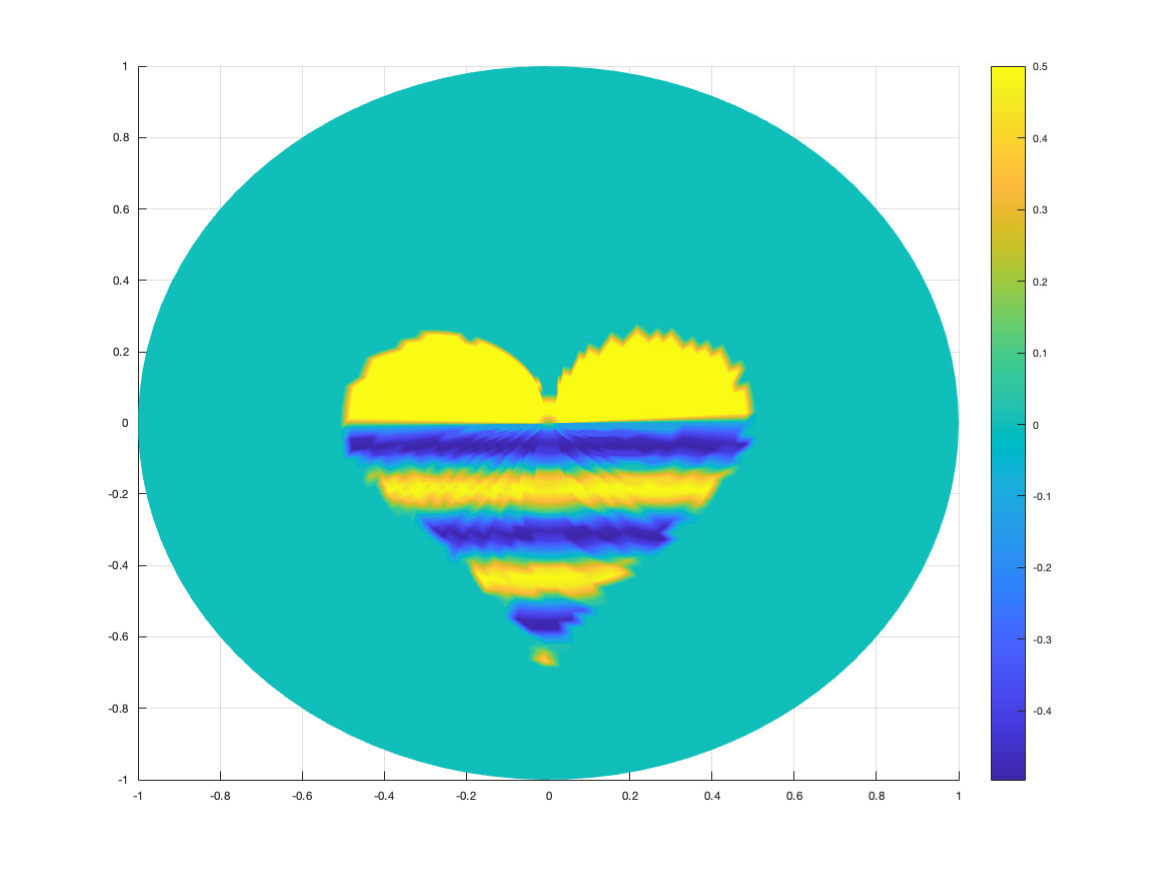}
\includegraphics[width=0.24\linewidth]{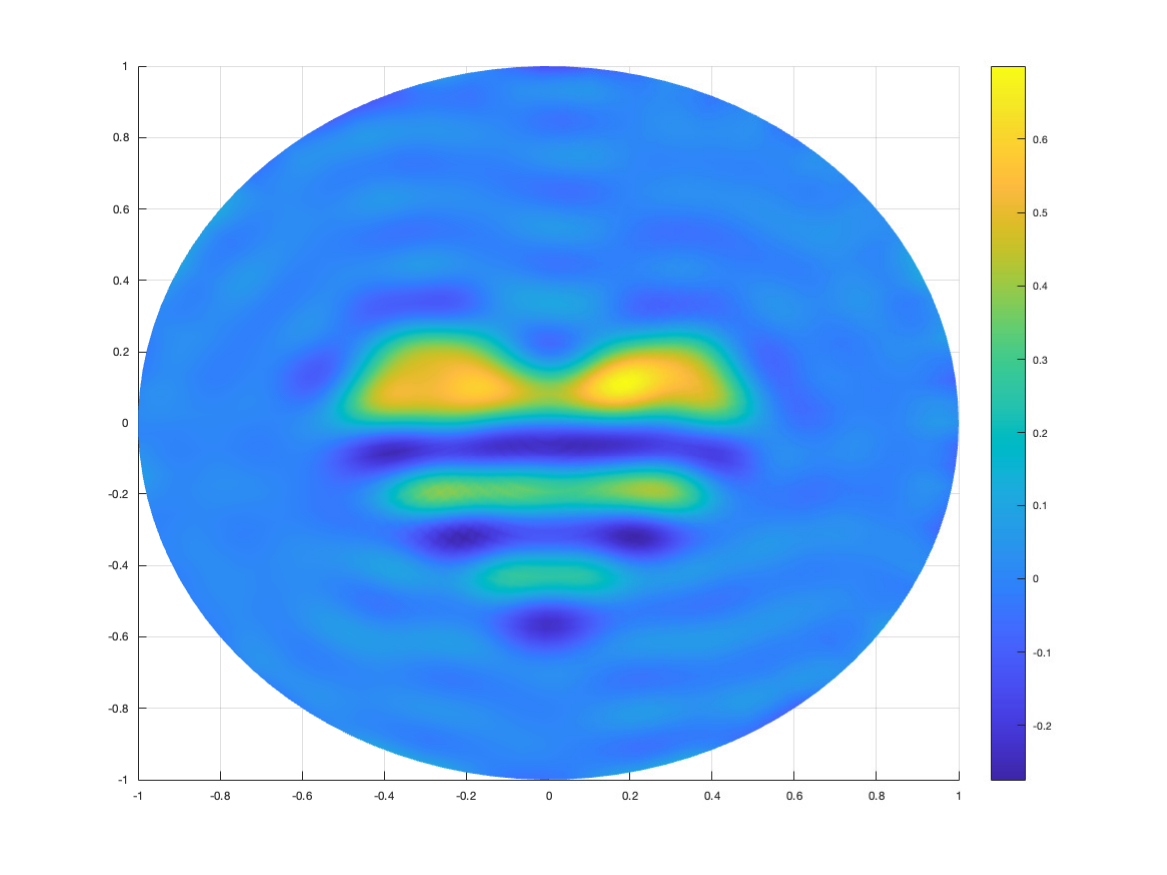}
\includegraphics[width=0.24\linewidth]{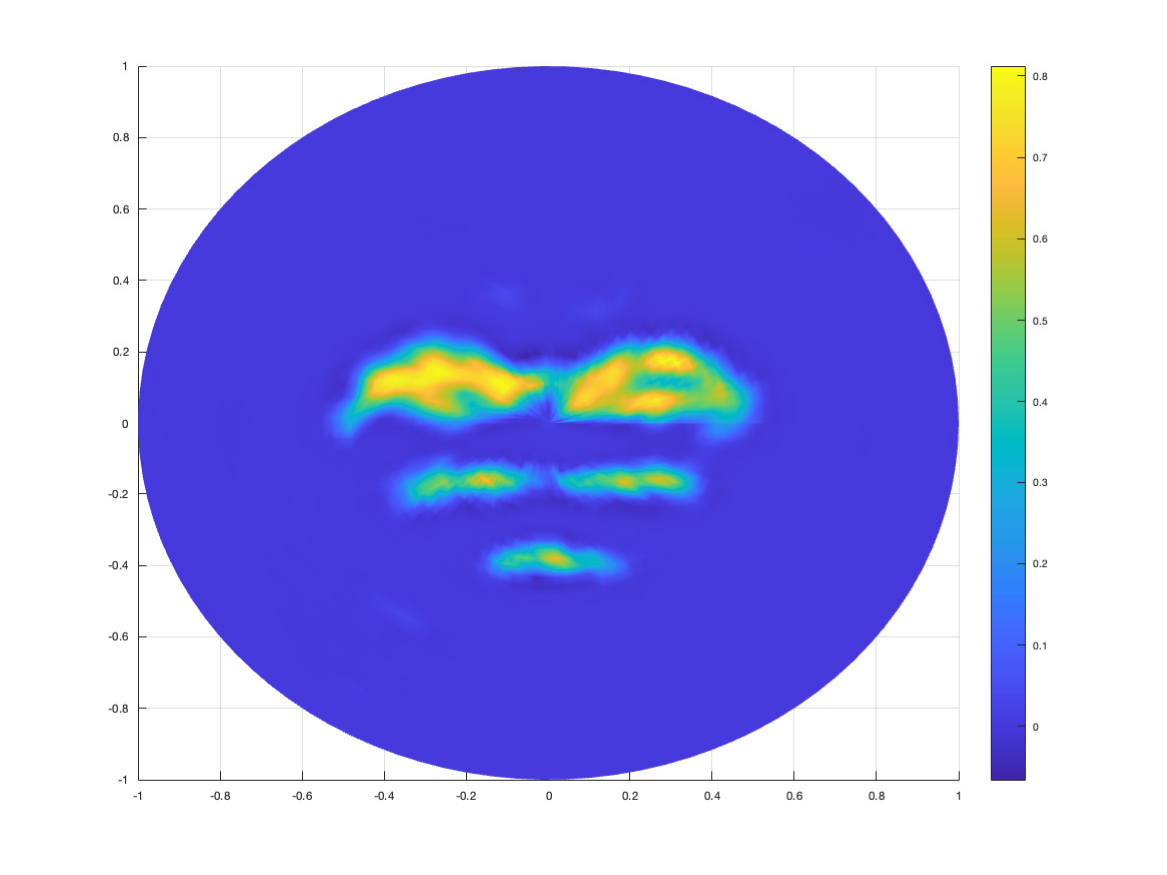}
\includegraphics[width=0.24\linewidth]{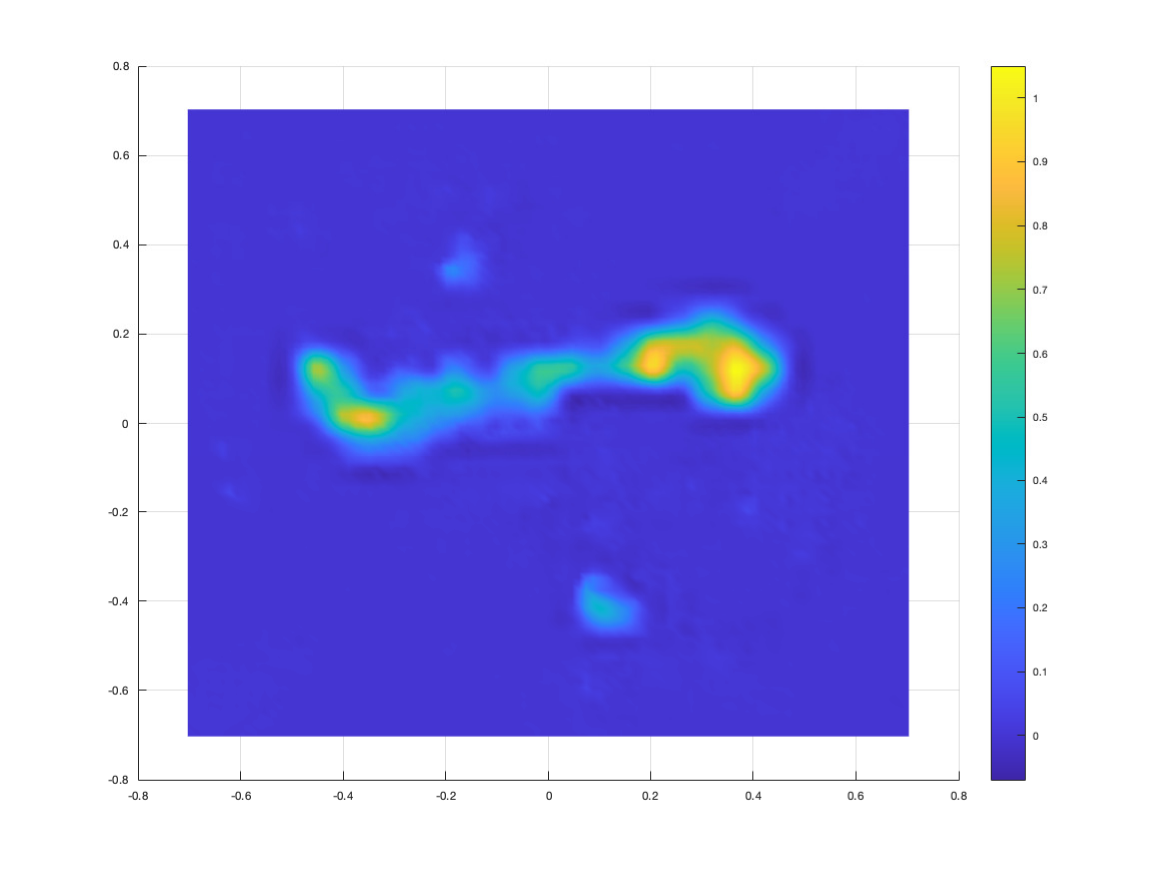}\\
\includegraphics[width=0.24\linewidth]{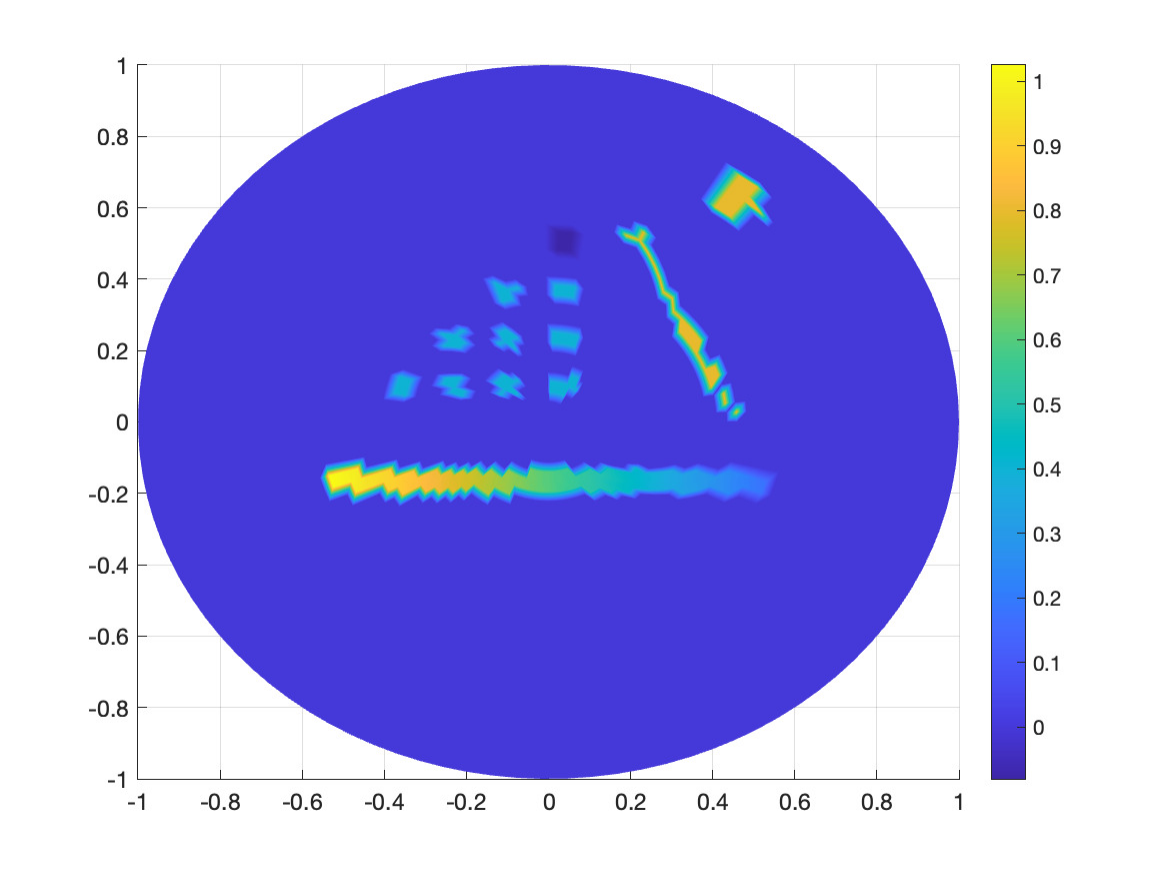}
\includegraphics[width=0.24\linewidth]{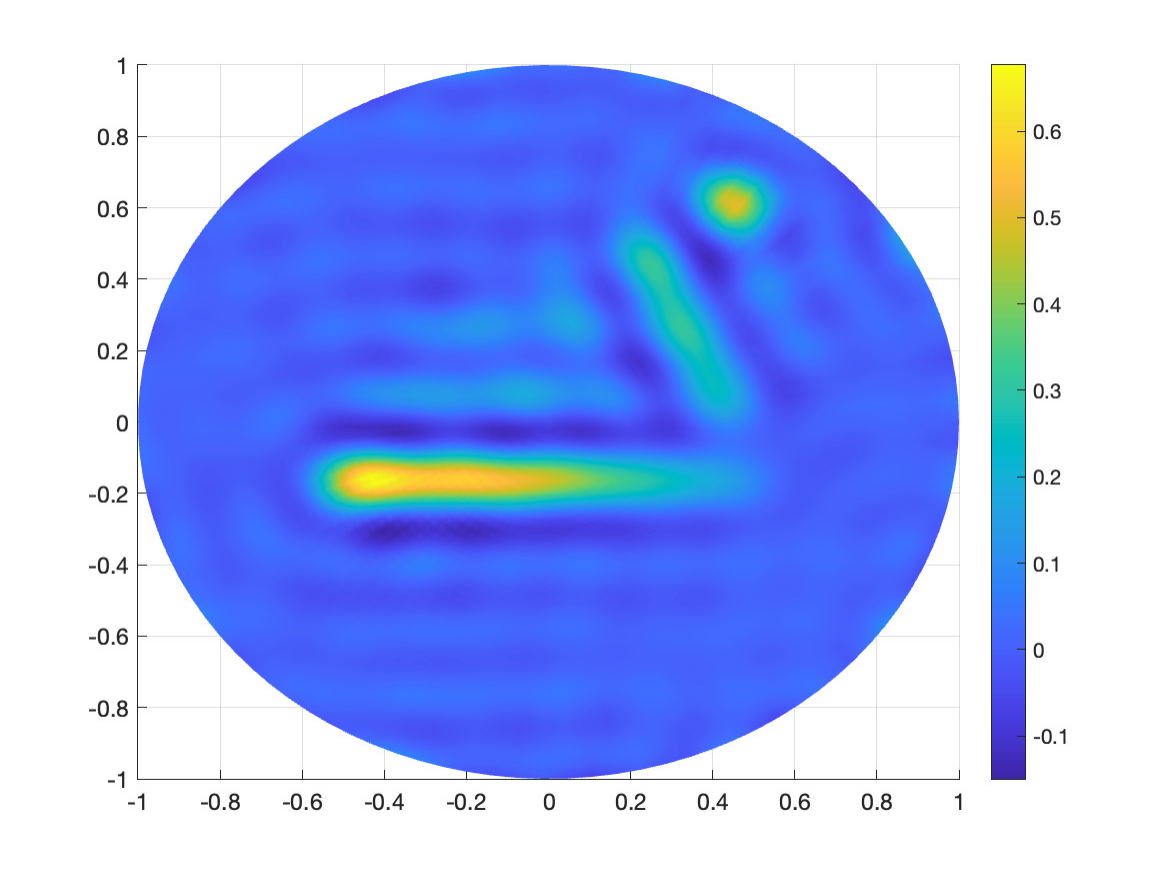}
\includegraphics[width=0.24\linewidth]{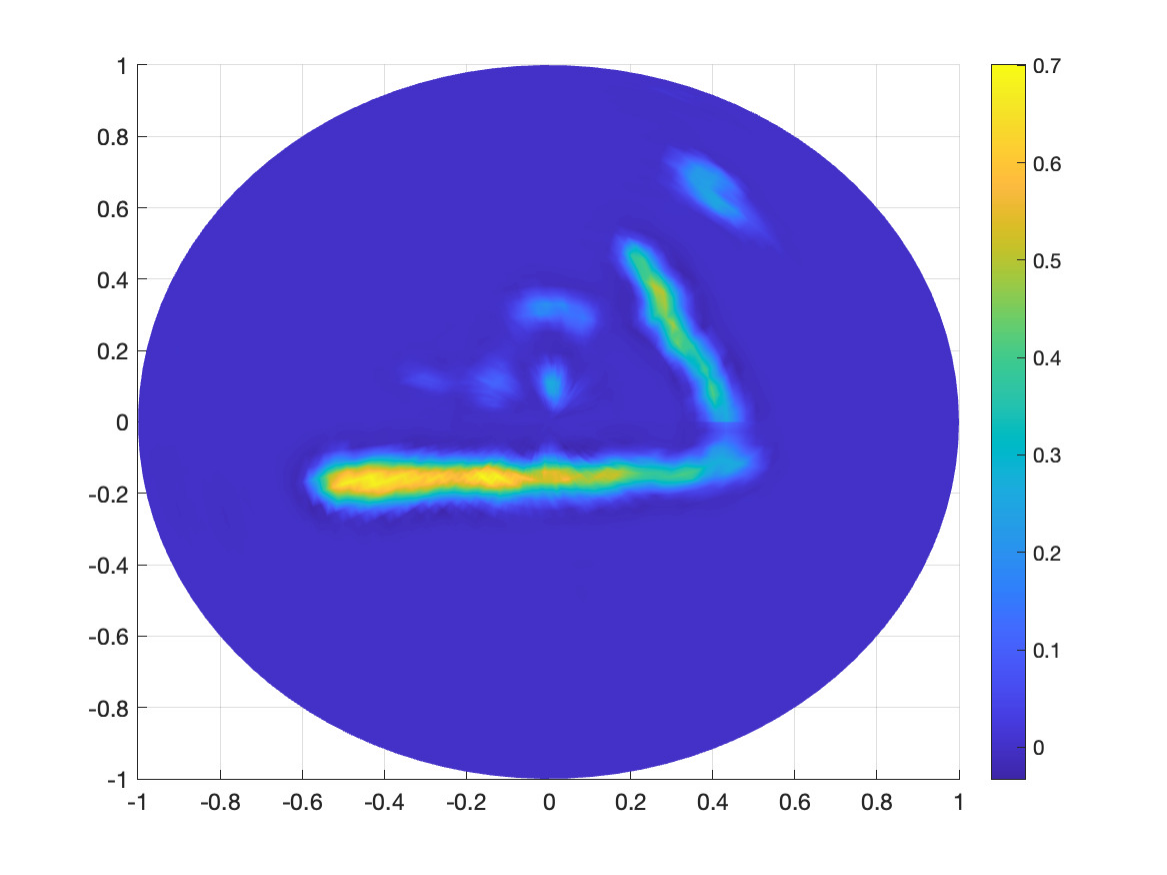}
\includegraphics[width=0.24\linewidth]{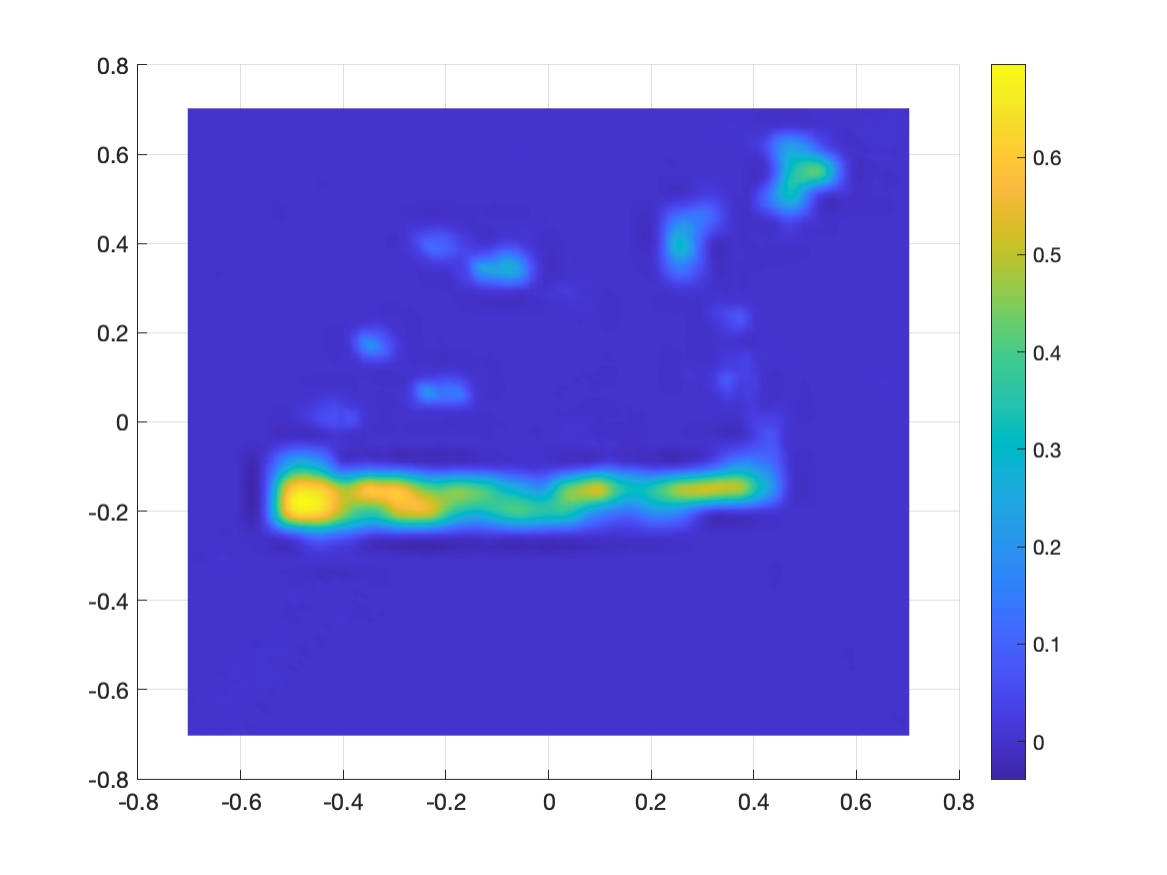}

  \caption{ Reconstruction of out-distribution contrasts. From left to right: ground truth, reconstructions by ULR, UU, and U, respectively. }  \label{figure: generalization}
\end{figure}

\begin{figure}[htbp]
  \centering
\includegraphics[width=0.19\linewidth]{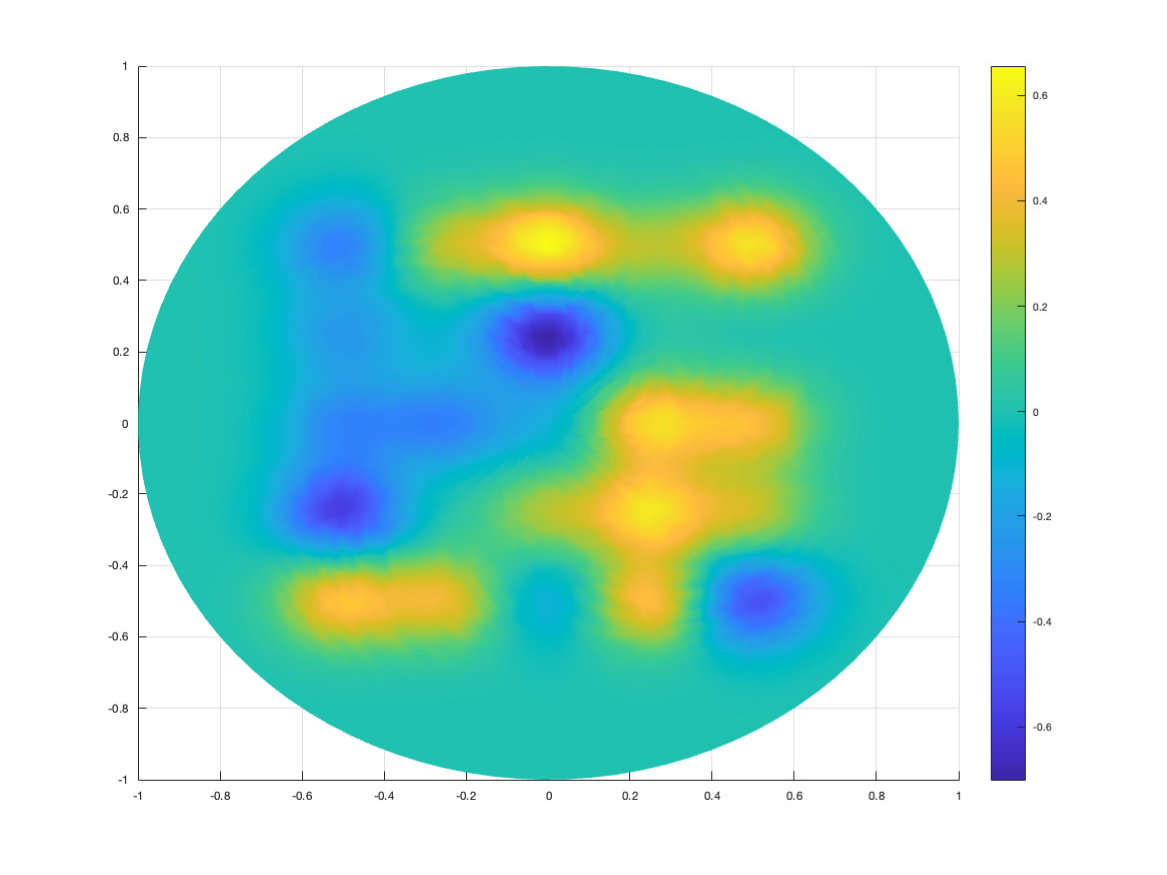}
\includegraphics[width=0.19\linewidth]{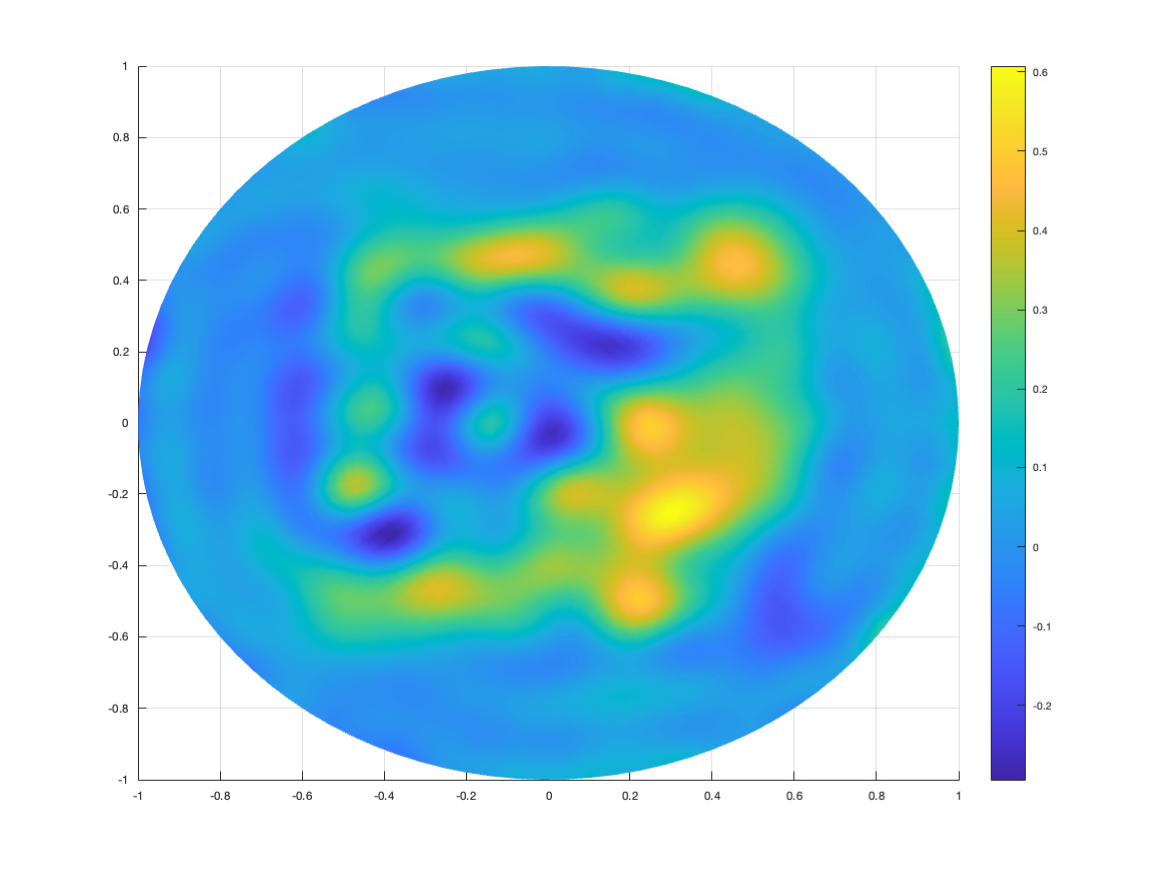}
\includegraphics[width=0.19\linewidth]{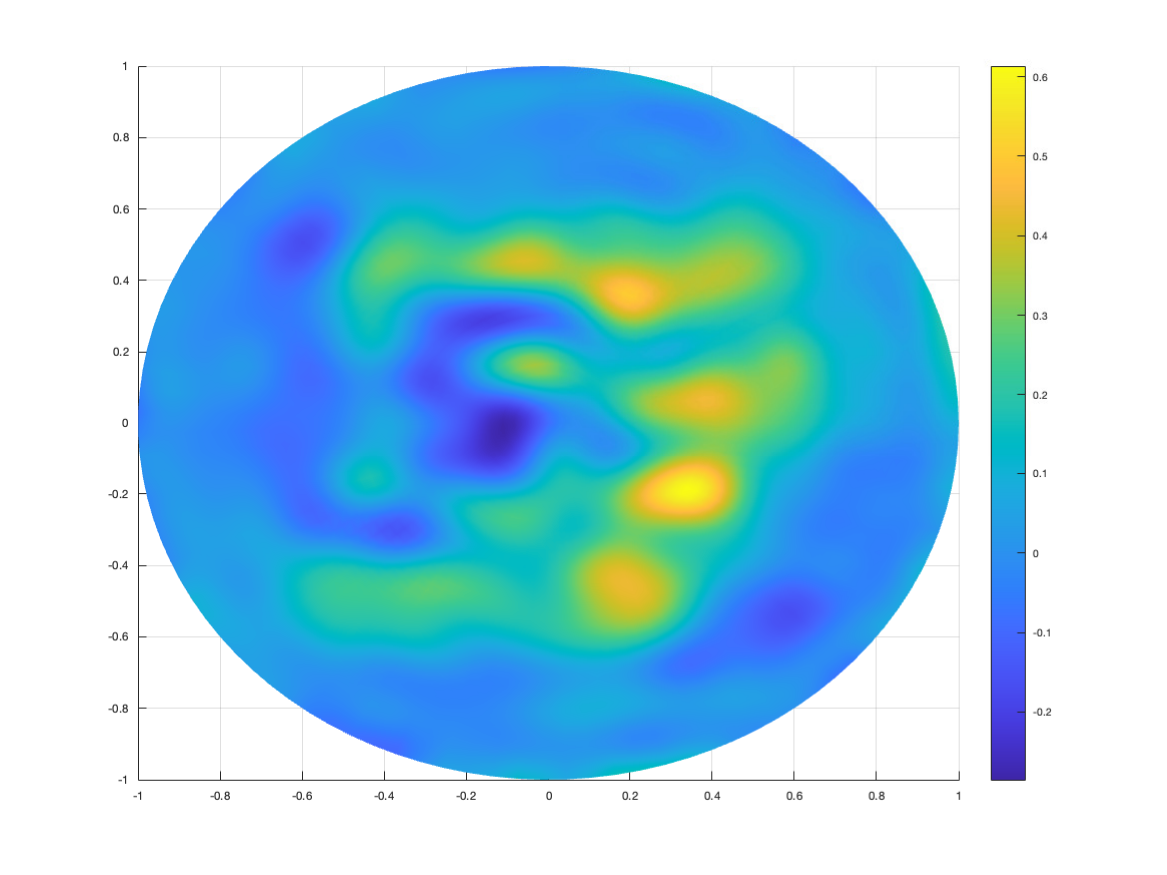}
\includegraphics[width=0.19\linewidth]{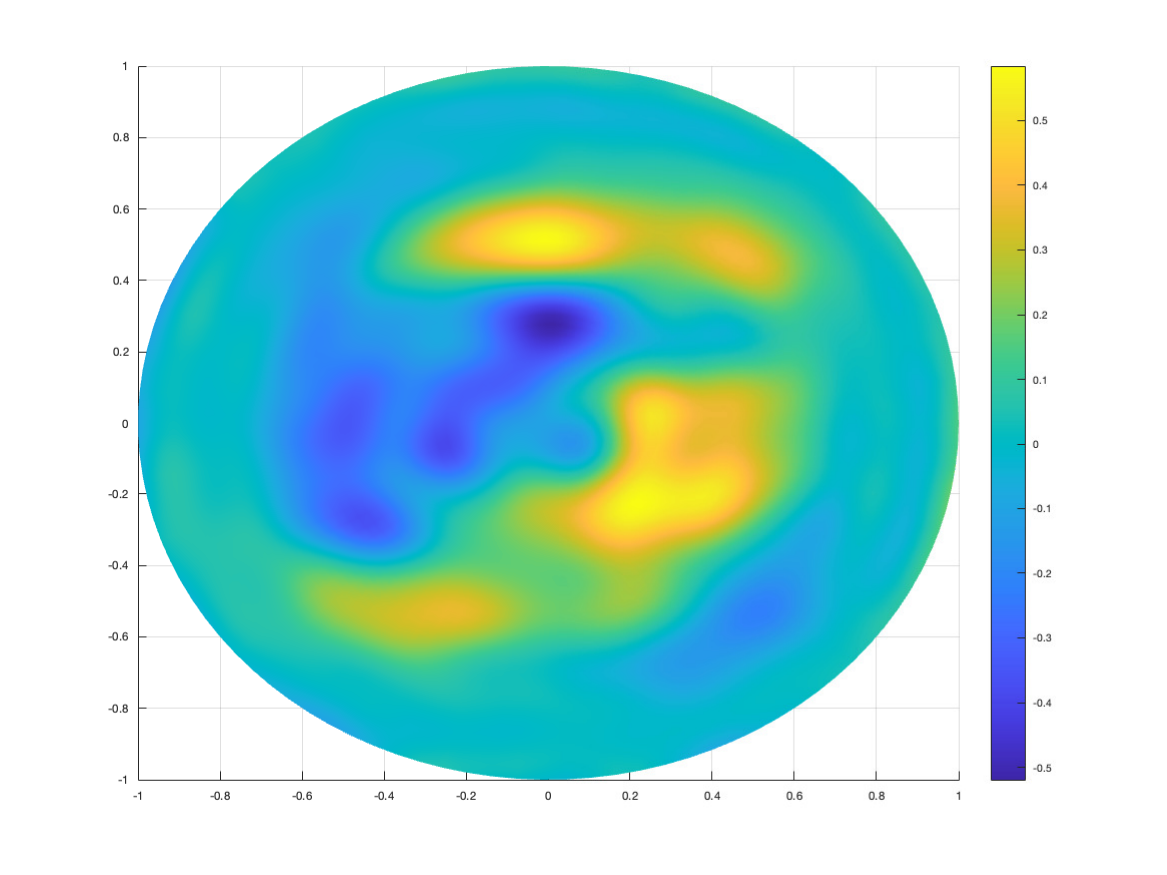}
\includegraphics[width=0.19\linewidth]{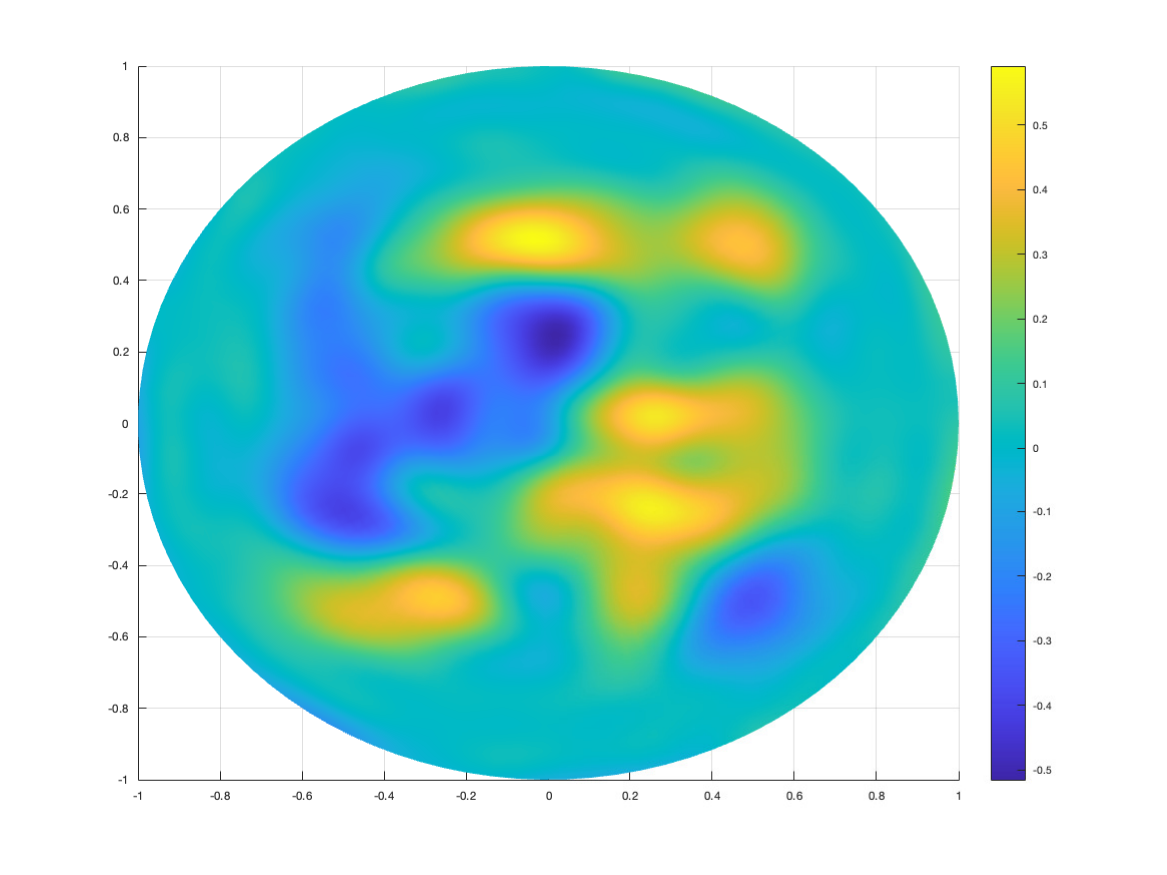}\\
\includegraphics[width=0.19\linewidth]{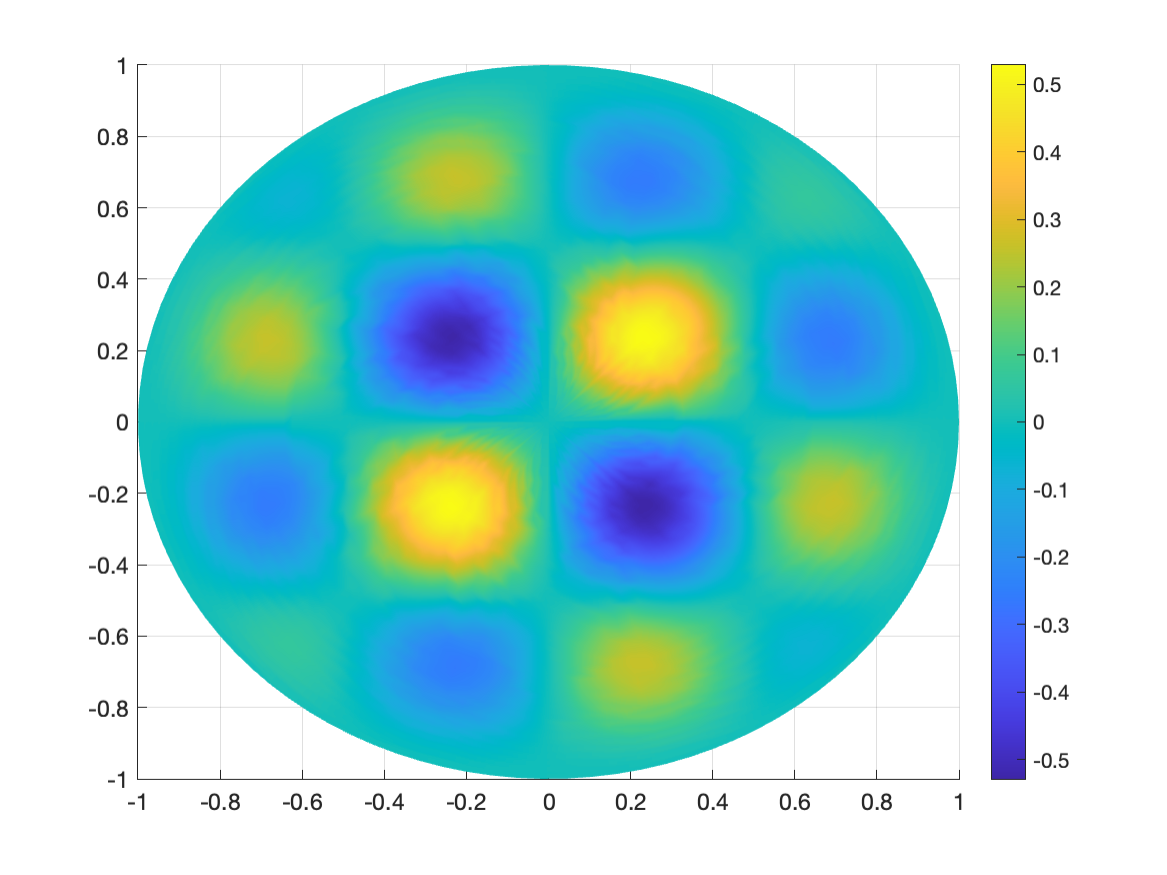}
\includegraphics[width=0.19\linewidth]{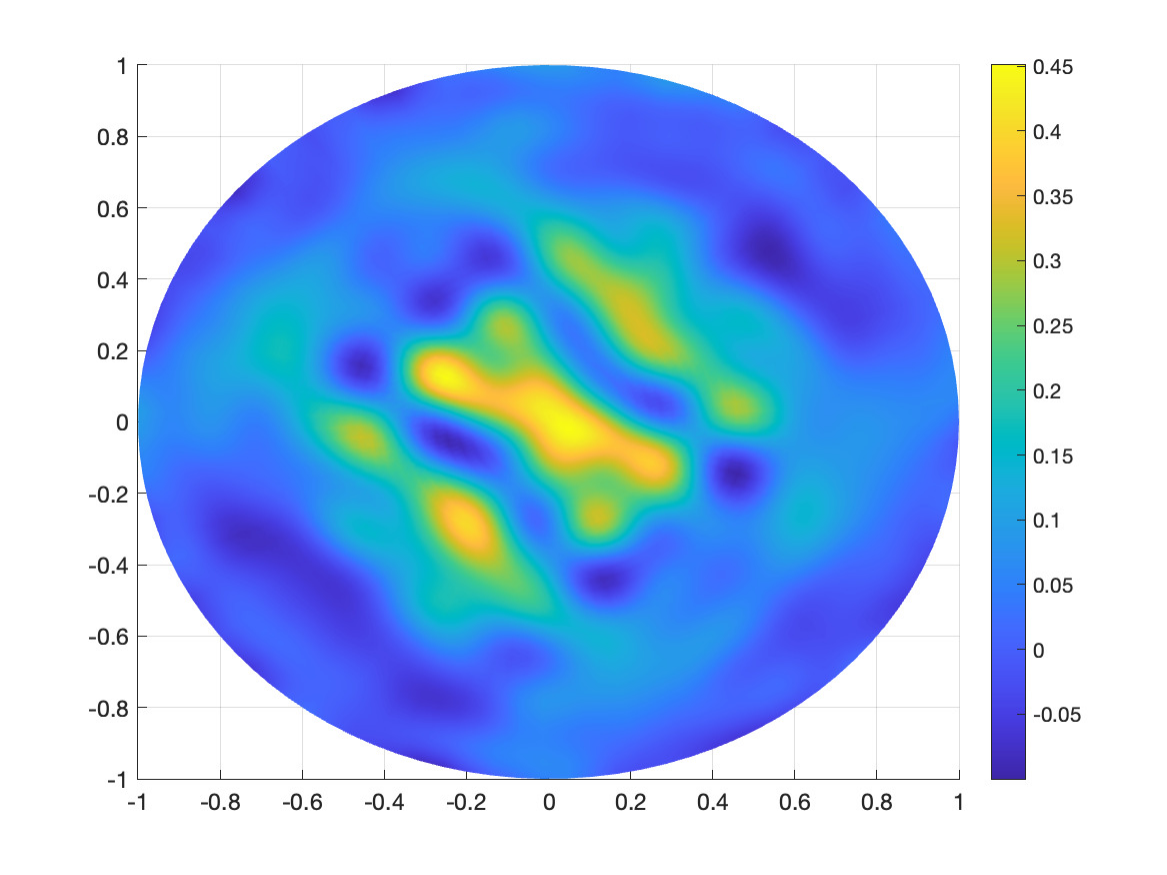}
\includegraphics[width=0.19\linewidth]{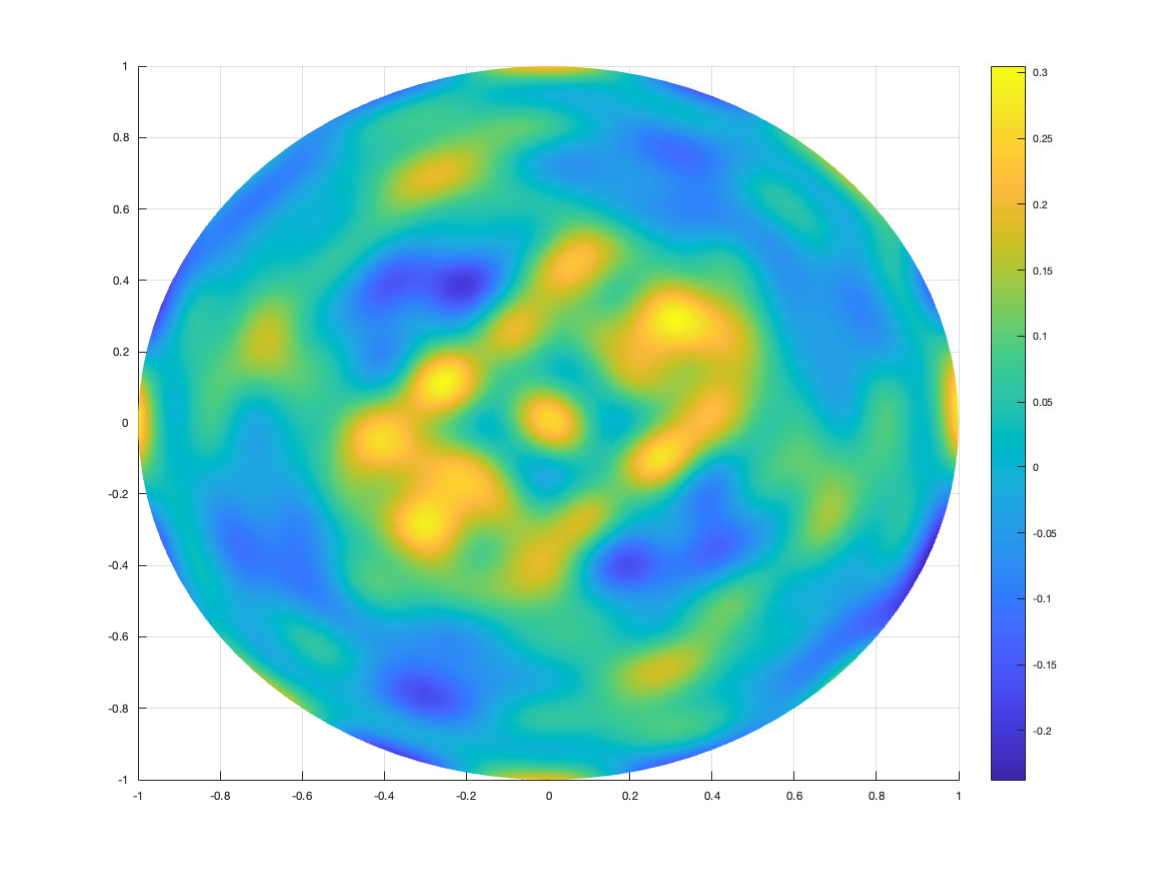}
\includegraphics[width=0.19\linewidth]{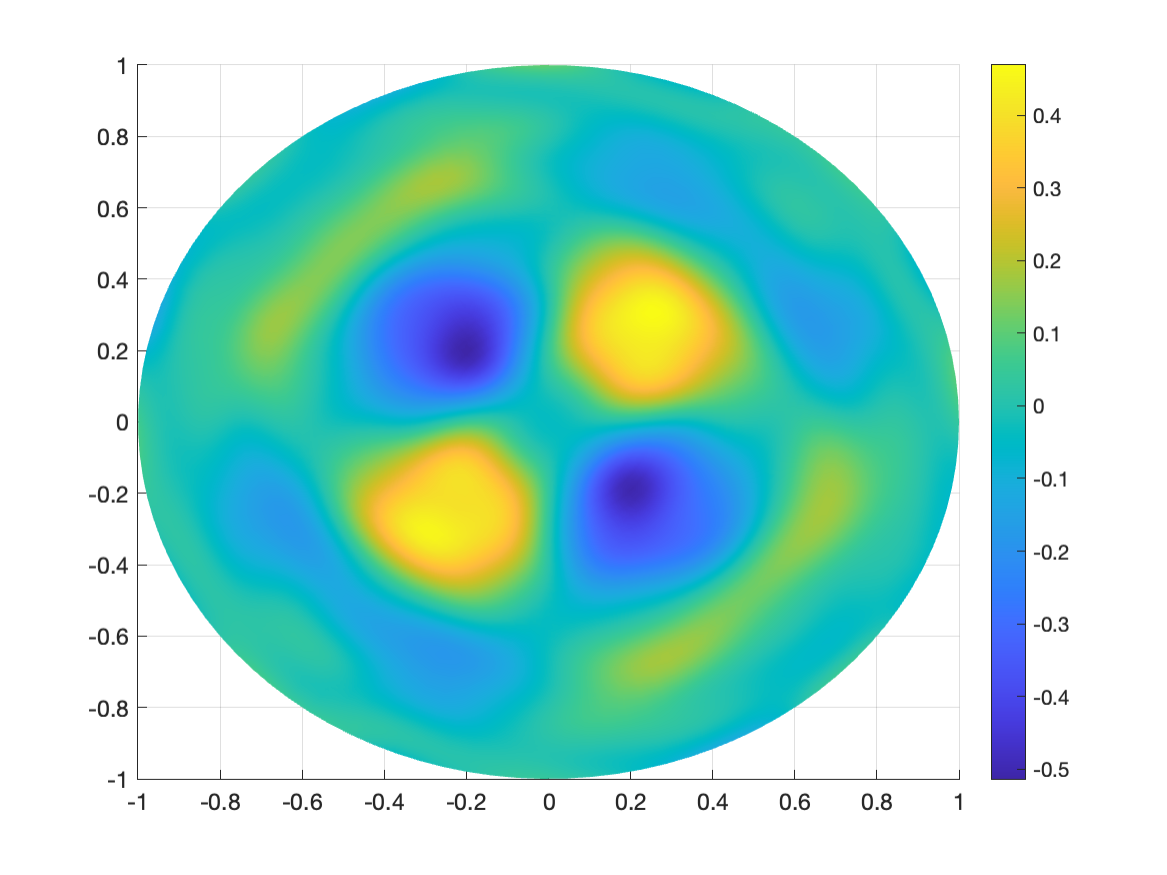}
\includegraphics[width=0.19\linewidth]{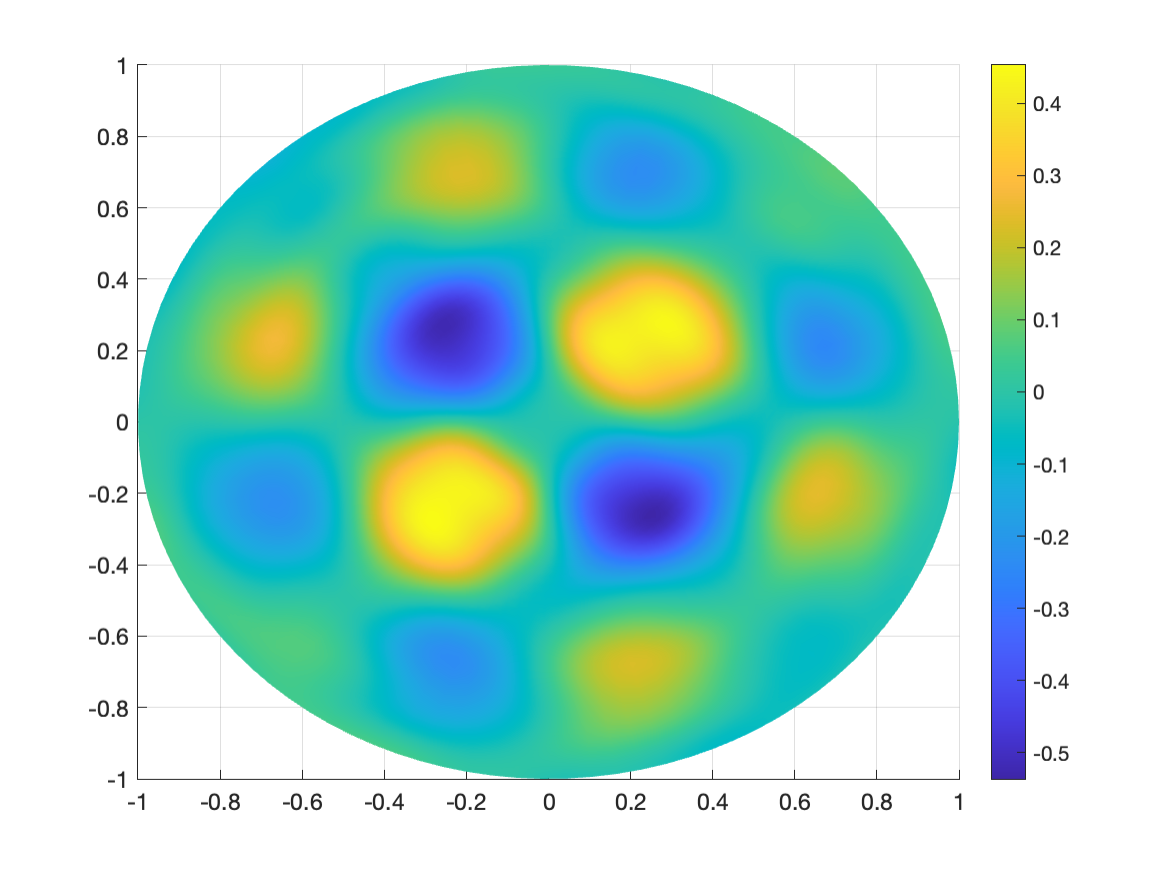}\\
\includegraphics[width=0.19\linewidth]{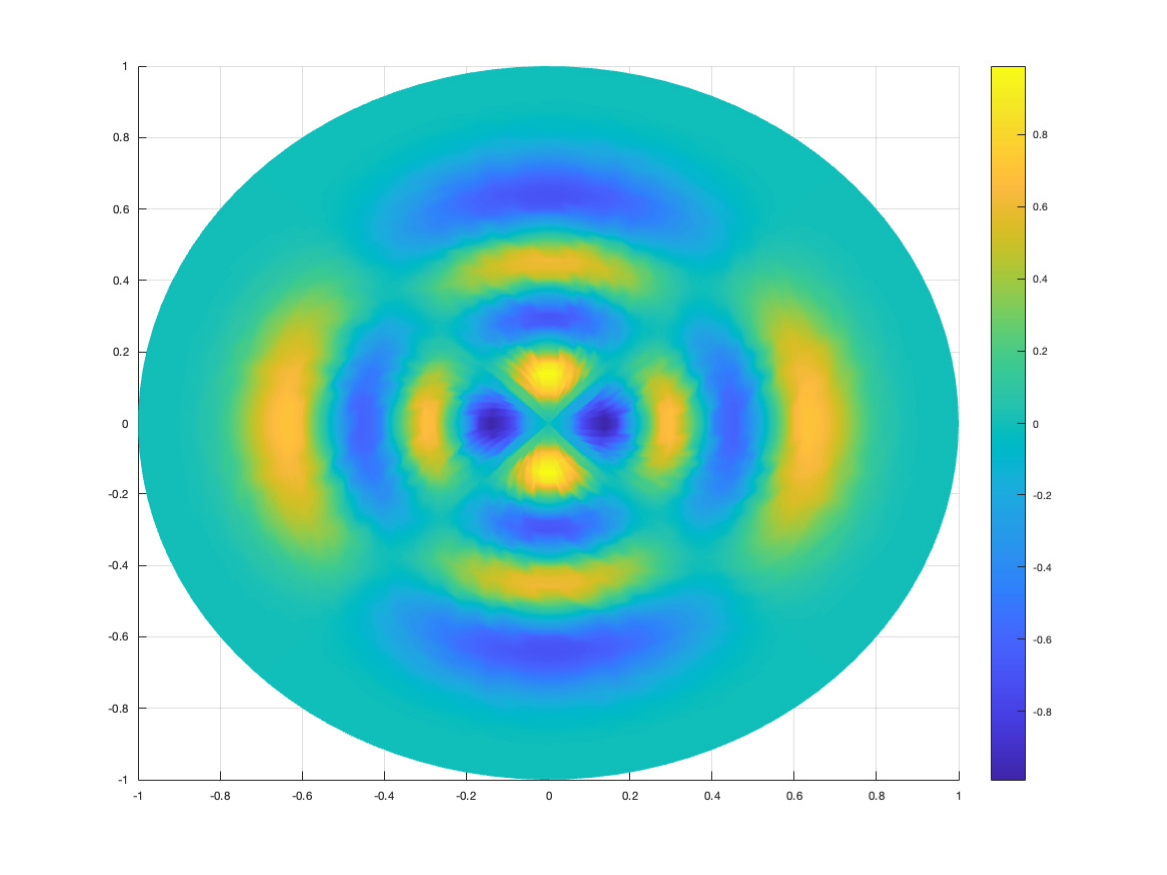}
\includegraphics[width=0.19\linewidth]{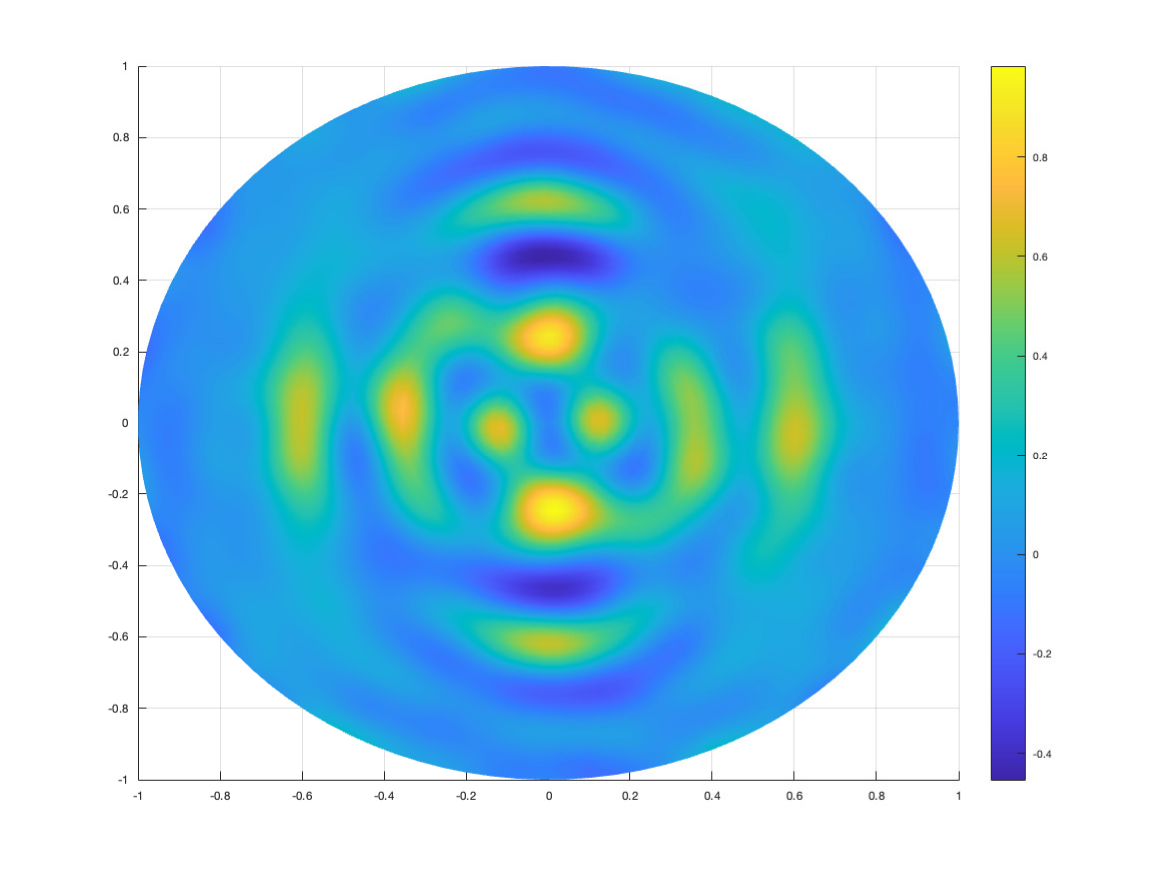}
\includegraphics[width=0.19\linewidth]{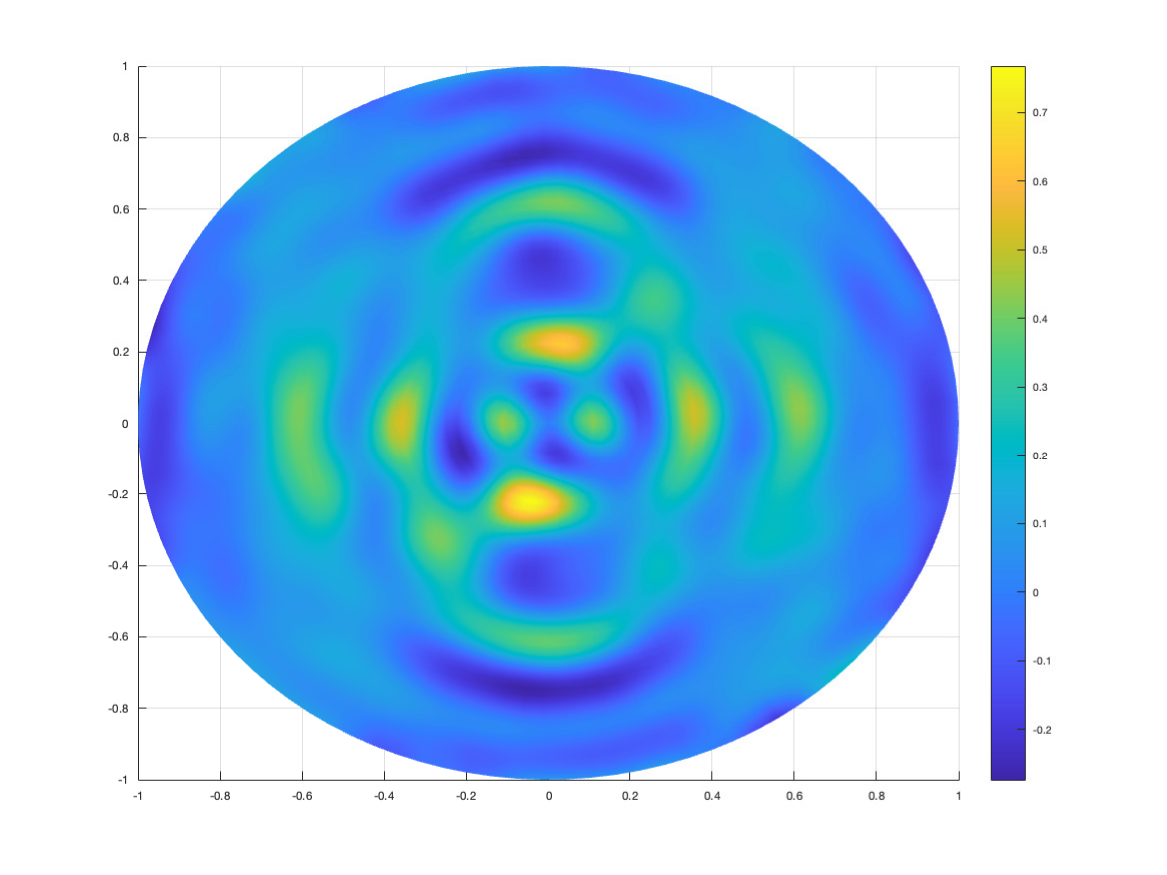}
\includegraphics[width=0.19\linewidth]{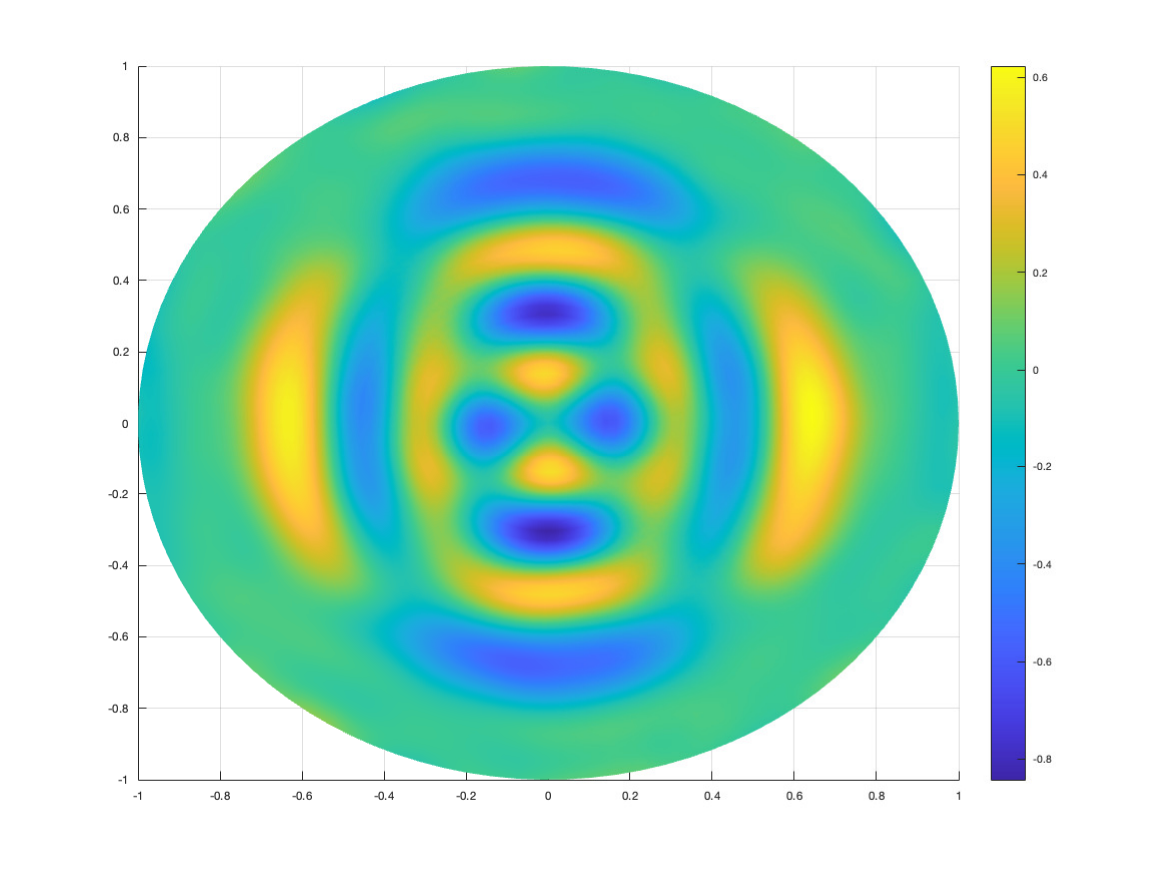}
\includegraphics[width=0.19\linewidth]{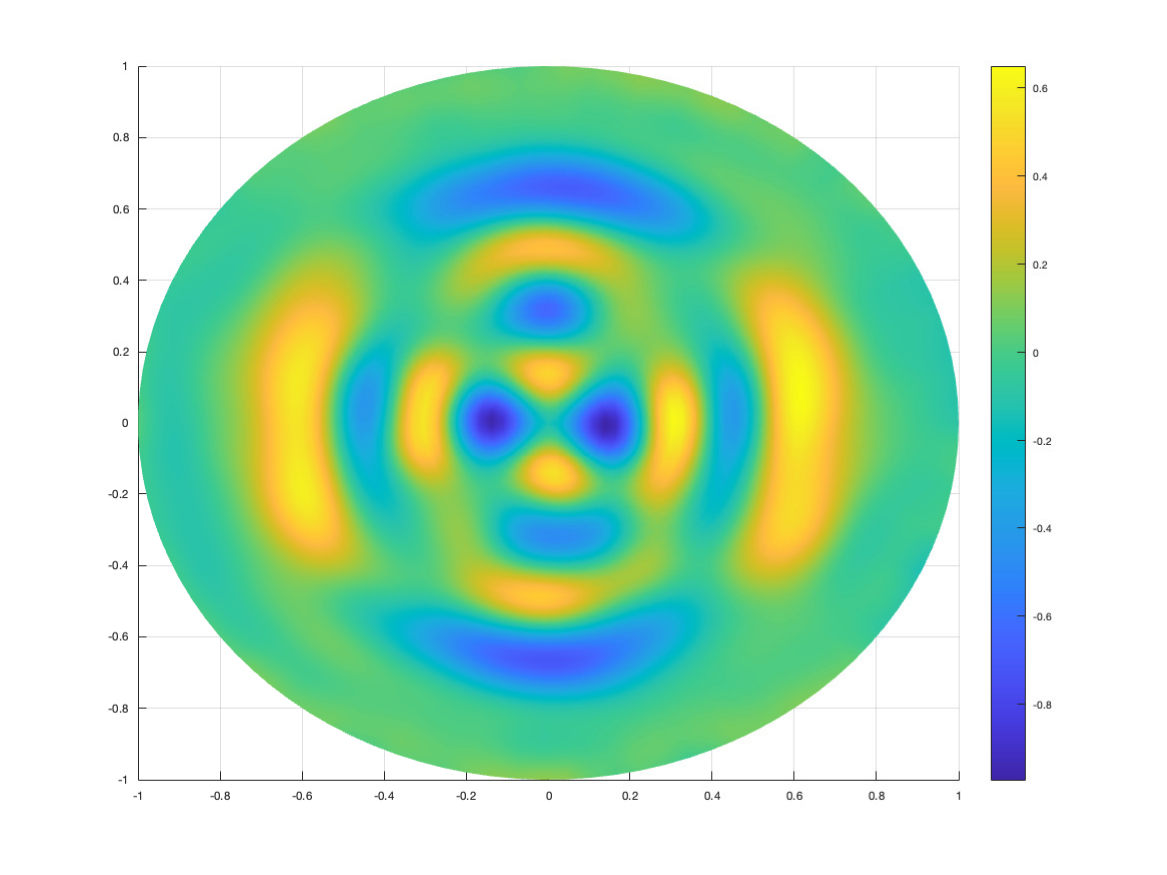}
  \caption{Reconstruction of another set of  out-distribution contrasts. The first column plots the ground truth. The reconstruction from left to right are based on:  dataset-(A) (original training dataset consisting of MNIST and circular samples),  dataset-(A)(B),  dataset-(A)(C), and dataset-(A)(D), respectively.    }
  \label{figure: generalization ability enhancement}
\end{figure}

\section{Extension to limited aperture problems} \label{sec: limited aperture}
In this section, we generalize our method to address limited aperture problems. The goal is to reconstruct the contrast from the limited aperture data
$$\{u^\infty(\hat{x};\hat{\theta}):\hat{x},\hat{\theta}\in \mathbb{S}_L\},$$
where \(\mathbb{S}_L:=\{x\in \mathbb{S}:\arg x\in [-\Theta,\Theta]\}\) (\(0<\Theta<\pi\)). Similar to the full aperture case, the limited aperture data (cf. third column of Figure \ref{figure: rotation-equivariance of limited aperture}) are processed according to  \eqref{eqn: far field to processed}, which yields the processed data \(u_L\) supported in a subset denoted by $D_L$ (cf. second column of Figure \ref{figure: rotation-equivariance of limited aperture}). This limited aperture problem is more challenging due to the fact that only limited data within the disk are available. To extend our strategy in the full aperture case to the limited aperture case, we propose to train another rotation-equivariance-aware neural network to directly correct \(u_L\) to the Born processed data \(u_b\), followed by an inverse Born solver using either a low-rank structure or a U-Net. It is evident that rotation-equivariance still holds for the limited data \(u_L\); that is, a rotation of the aperture \(\mathbb{S}_L\) and contrast corresponds to a rotation of the processed data $u_L$. To illustrate this, we plot in the first row of  \Cref{figure: rotation-equivariance of limited aperture}   the contrast number ``5", the imaginary part of its processed data, and the original limited far-field data, respectively. Here, \(\Theta=\frac{\pi}{2}\), and we set \(u^\infty(\hat{x};\hat{\theta})=0\) if \(\arg \hat{x}, \arg \hat{\theta}\notin [-\Theta,\Theta]\). When the contrast and aperture are rotated clockwise by \(\frac{3\pi}{2}\), we have that \(u^\infty(\hat{x};\hat{\theta})=0\) if \(\arg \hat{x}, \arg \hat{\theta}\notin [-(\Theta-\frac{\pi}{2}),\Theta-\frac{\pi}{2}]\), cf. bottom right of  \Cref{figure: rotation-equivariance of limited aperture}; correspondingly, the processed limited data (indicated by the dashed lines) follow the same rotation, cf. second column of  \Cref{figure: rotation-equivariance of limited aperture}.   Finally we make a remark that the far-field data can be expanded  using the reciprocity relation, however one can check that the processed data remain unchanged as the data processing already takes advantage of the reciprocity relation.

\begin{figure}[htbp]
\centering
 \includegraphics[width=0.30\textwidth]{figures/contrast0_1.pdf} 
 \hspace{0.8em}
\includegraphics[width=0.36\textwidth]{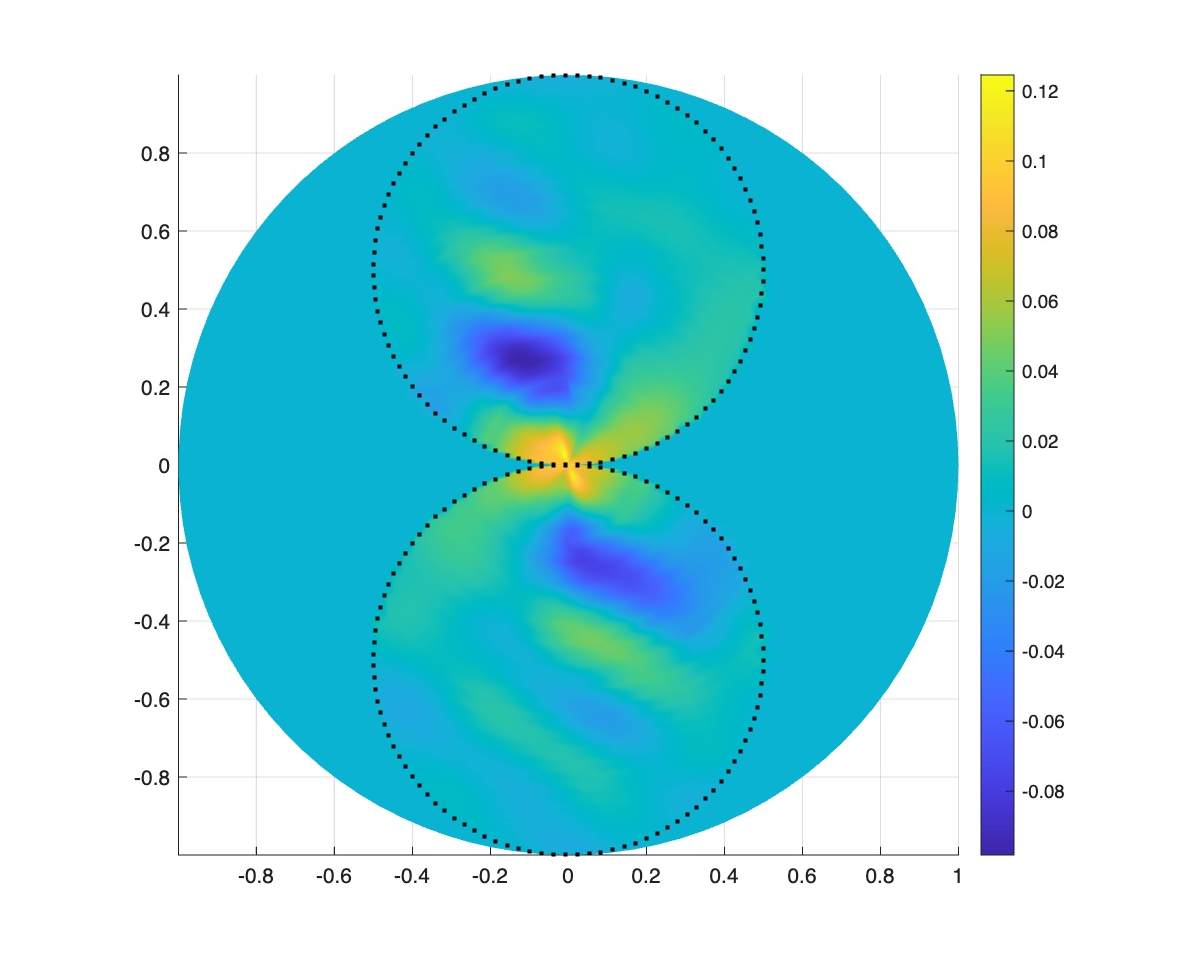} %
 \hspace{0.8em}
\raisebox{0.4cm}{\includegraphics[width=0.27\textwidth]{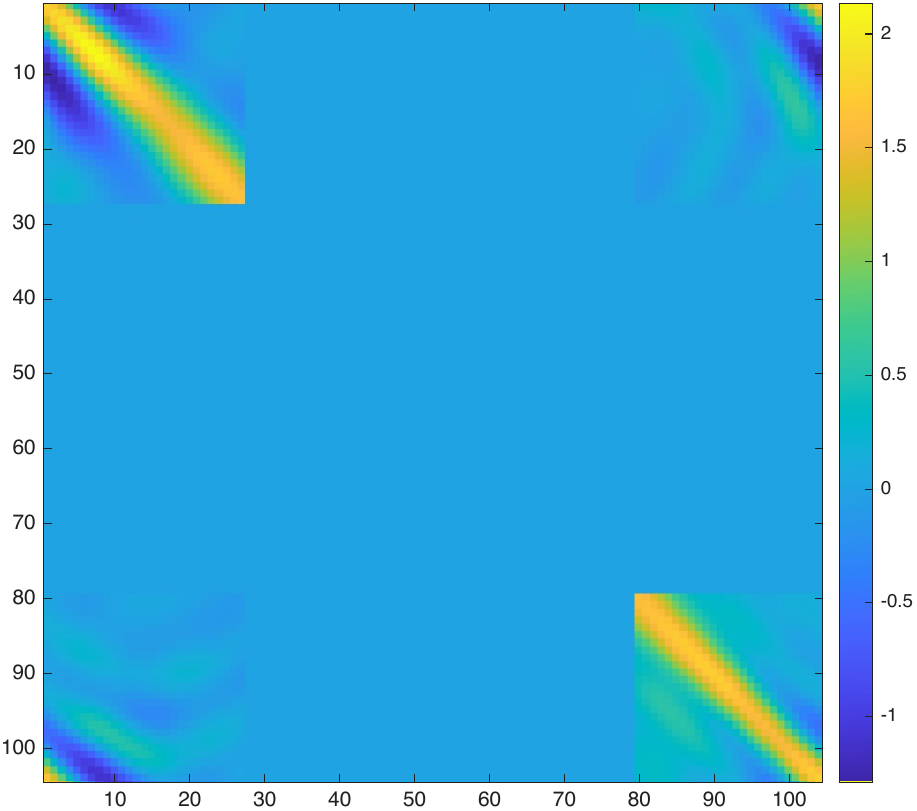}}\\
\includegraphics[width=0.30\textwidth]{figures/contrast_rot90_1.pdf} 
 \hspace{0.8em}
\includegraphics[width=0.36\textwidth]{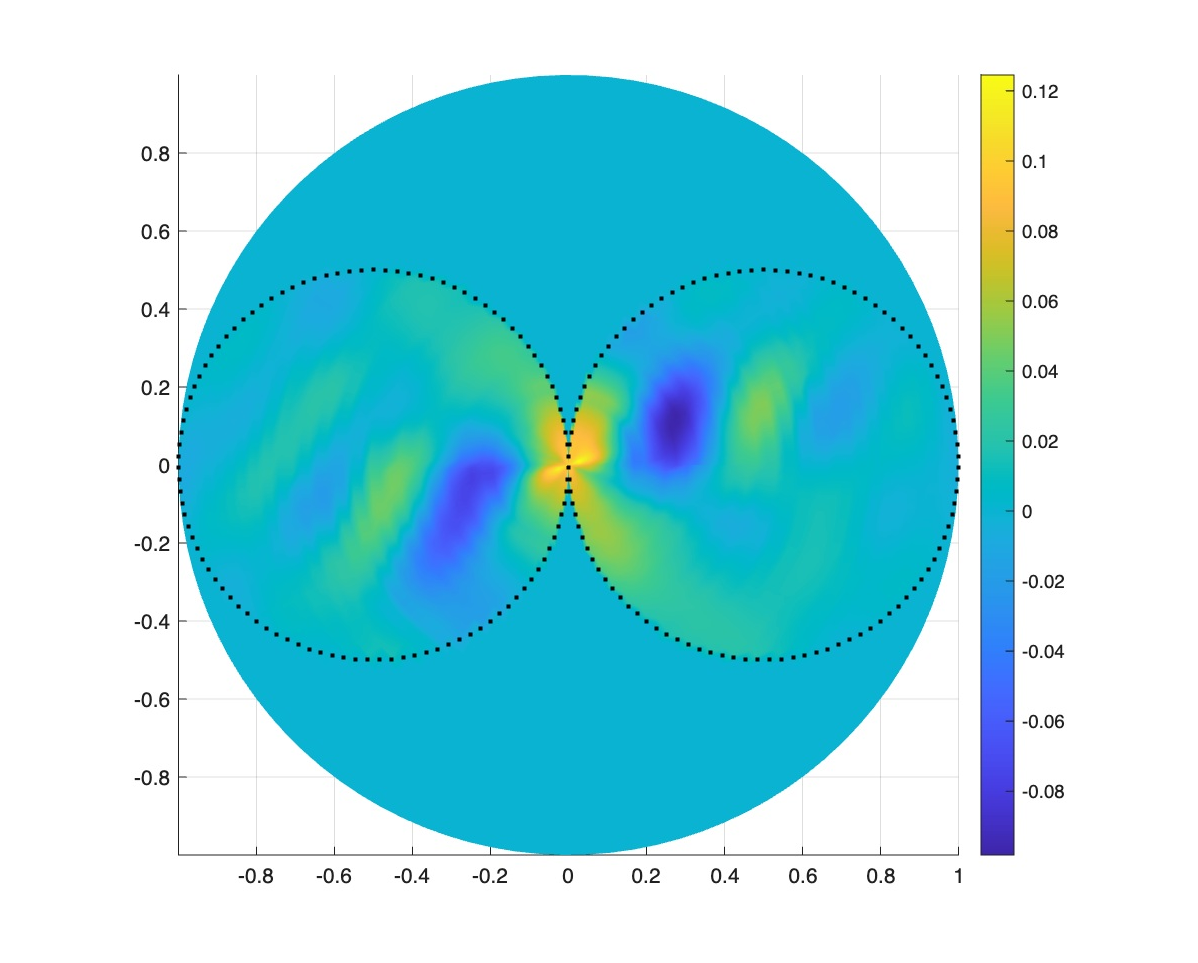} %
 \hspace{0.8em}
\raisebox{0.4cm}{\includegraphics[width=0.27\textwidth]{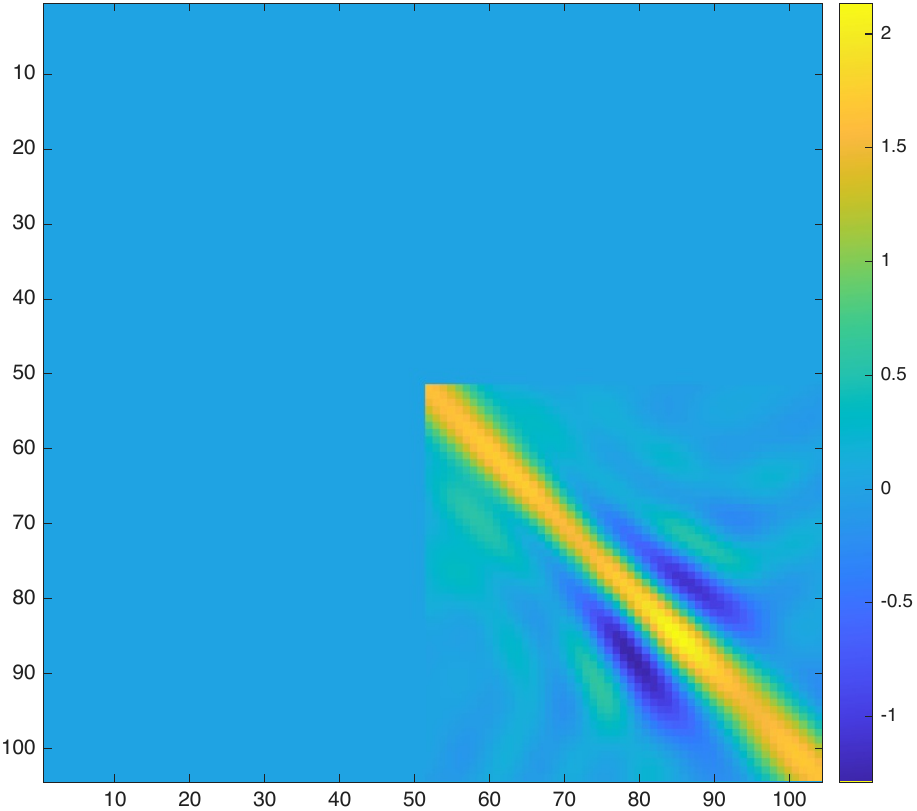}}
   \caption{Illustration of rotation-equivariance with limited aperture data.  First column: contrast $q$; second column, imaginary  part of the processed datum \eqref{eqn: far field to processed}; third column, imaginary part of the far field datum. In the first row, we plot the number ``5'' and the corresponding  data. In the second   row, we plot the    data for a $\frac{3\pi}{2}$ rotated contrast.  }  \label{figure: rotation-equivariance of limited aperture}
\end{figure}

\subsection{Neural network training}
The training follows the full aperture case, after expanding the limited processed data $u_L$ from the subset $D_L \subset B$   to the unit disk $B$ by zero. As an illustration, we plot one sample of the training data in \Cref{figure: training sample 3}(b); clearly this is a subset of the data in the full aperture case in \Cref{figure: training sample 1}(b).   To deal with the limited aperture case, we propose to test only one  dataset category (i.e., MNIST) in this study to impose strong a prior information. Similarly to the full aperture case,  we first  learn a U-net to map the limited processed data \(u_L\) to the full processed Born data \(u_b\), then combine it with the learned U-net for the inverse Born solver or the low-rank solver; this leads to again ULR   and UU. For comparison, we also train  a black-box U-net  to reconstruct the contrast directly from the limited far-field data. 
The aperture size is chosen as \(\Theta=\frac{\pi}{2}\), and a total of 30,000 samples are used for training and validation in the training dataset. We report that 10,000 samples are insufficient for the neural network training in our test.

\begin{figure}[htbp]
\centering 
 \hspace{-1em}
\subfloat[Ground truth $q$]{\includegraphics[width=0.3\linewidth]{figures/contrast-eps-converted-to.pdf}}
 \hspace{0.4em}\subfloat[$\Re u$]{
 \raisebox{0.1cm}{ 
{\includegraphics[width=0.25\linewidth]{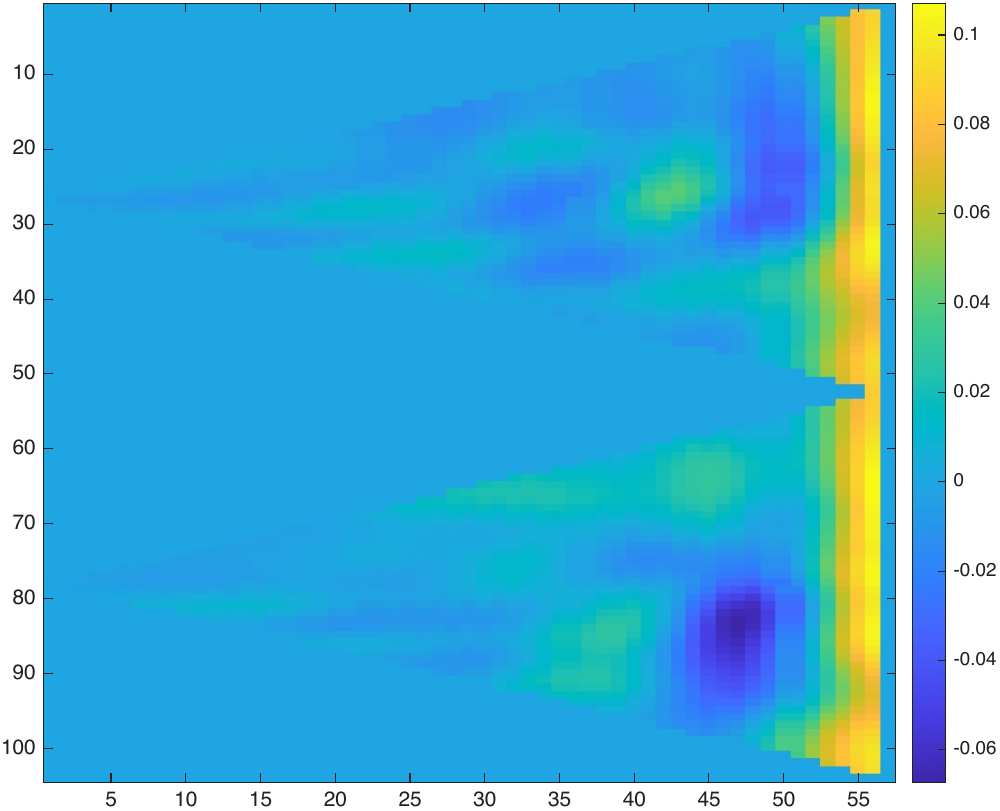}}}}
 \hspace{1.3em}
\subfloat[$\Im u$] {\raisebox{0.1cm}{ {\includegraphics[width=0.25\linewidth]{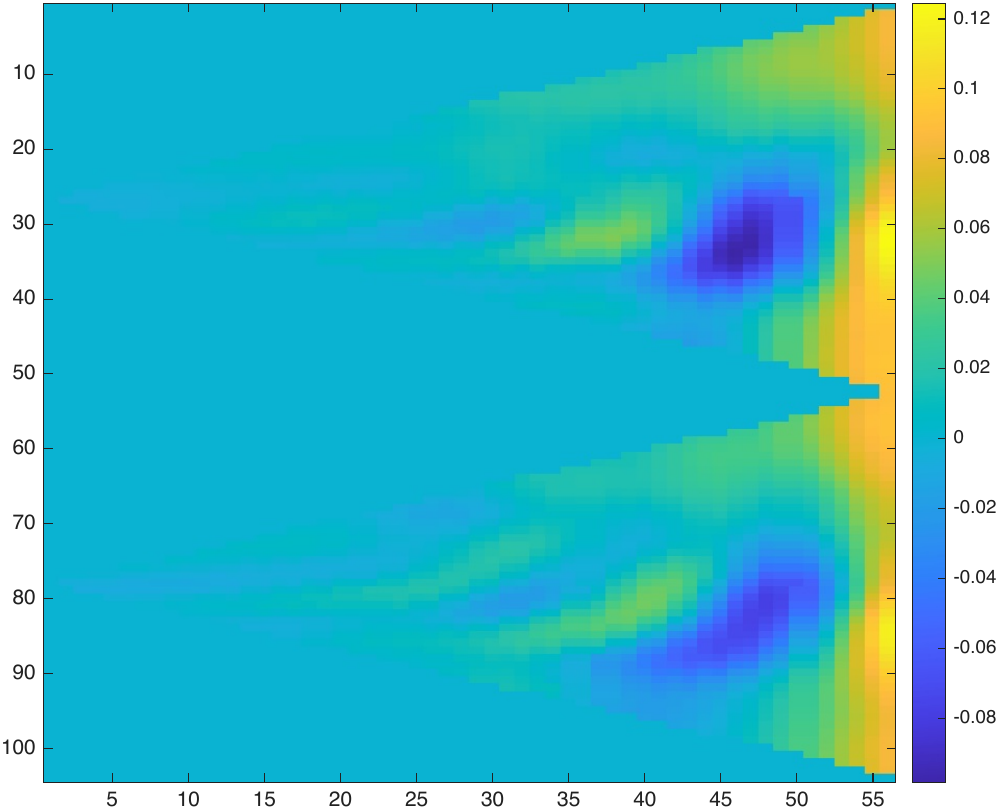}}}}
\\
\subfloat[New rectangular image of $q$]{\includegraphics[width=0.3\linewidth]{figures/contrast_polar-eps-converted-to.pdf}}
\subfloat[$\Re u_b$]{\includegraphics[width=0.3\linewidth]{figures/post_ub_real_polar-eps-converted-to.pdf}}
\subfloat[$\Im u_b$]{\includegraphics[width=0.3\linewidth]{figures/post_ub_imag_polar-eps-converted-to.pdf}}

  \caption{Illustration of  data  in the limited aperture case. Input of neural network $U_1$: (b) and (c); output of neural network $U_1$:(e) and (f).}  \label{figure: training sample 3}
\end{figure}

\subsection{Numerical experiments}
 The testing includes with in-distribution performance and out-distribution performance  with noisy data where noise level $\delta=0.2$.  

We first investigate the rotation-equivariance property. As shown in the first row of \Cref{figure: full rotation-equivariance limited aperture}, where the contrast is in-distribution, the reconstruction results of ULR and UU demonstrate the effectiveness of the data corrector. Specifically, ULR and UU outperform U in terms of reconstruction quality. In the second row, when we rotate both the contrast ``4'' and the aperture \([-\frac{\pi}{2},\frac{\pi}{2}]\) clockwise by \(\frac{\pi}{2}\), we find that ULR and UU can stably reconstruct the rotated contrast, whereas U yields unsatisfactory results. Furthermore,  we consider in the third row the case when only the contrast is rotated but the limited aperture stays the same, ULR and UU again outperform the black-box neural network. 

We also evaluate the out-of-distribution performance of these three methods by \Cref{figure: generalization limited aperture}. It is clear that the black-box neural network U  yields worse  reconstructions.  Comparing \Cref{figure: generalization} and \Cref{figure: generalization limited aperture}, it can be observed that the reconstruction of ULR and UU remain stable, albeit slightly inferior to those obtained with full aperture data; this is expected due to the additional challenge of the limited aperture problem.

\begin{figure}[htbp]
{\includegraphics[width=0.24\linewidth]{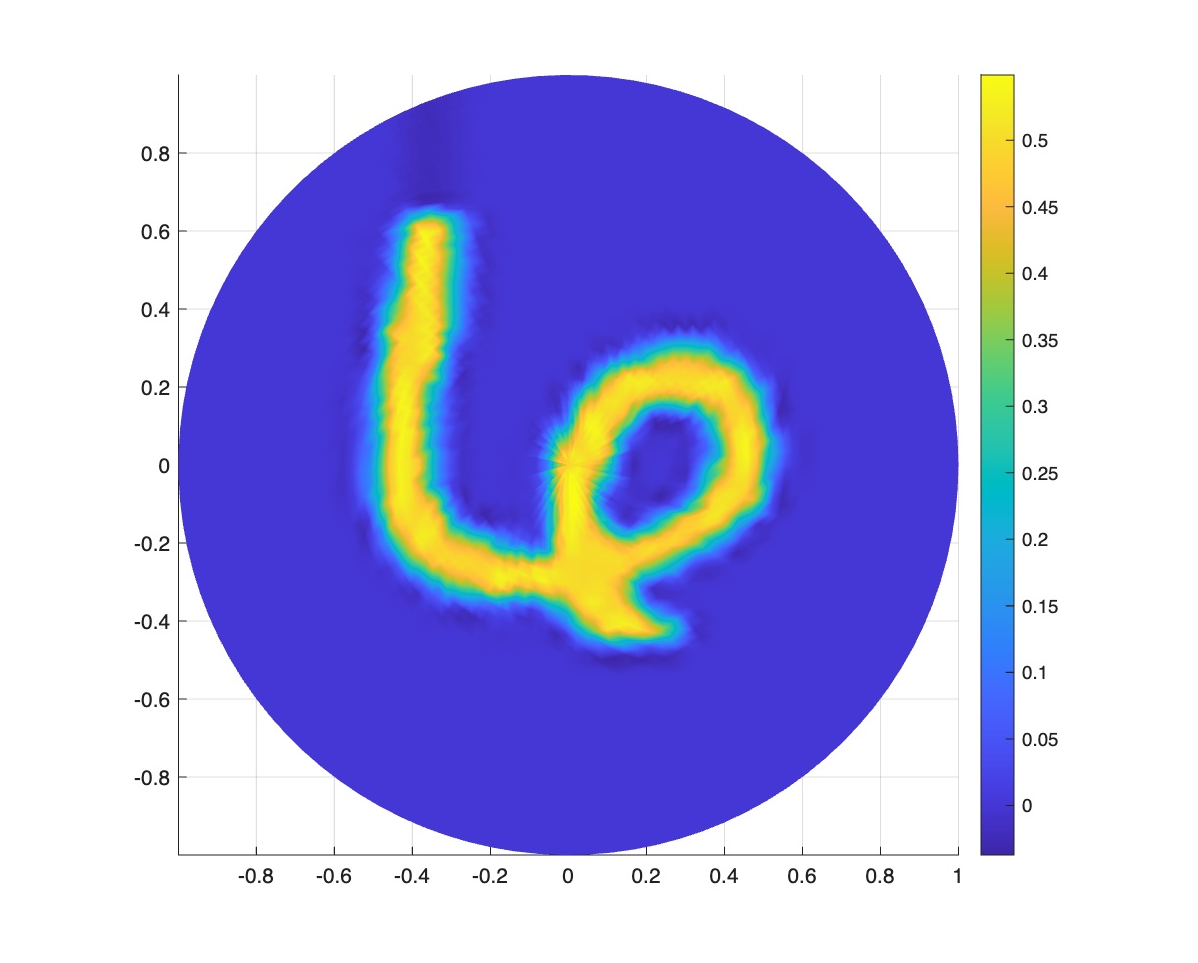}}
{\includegraphics[width=0.24\linewidth]{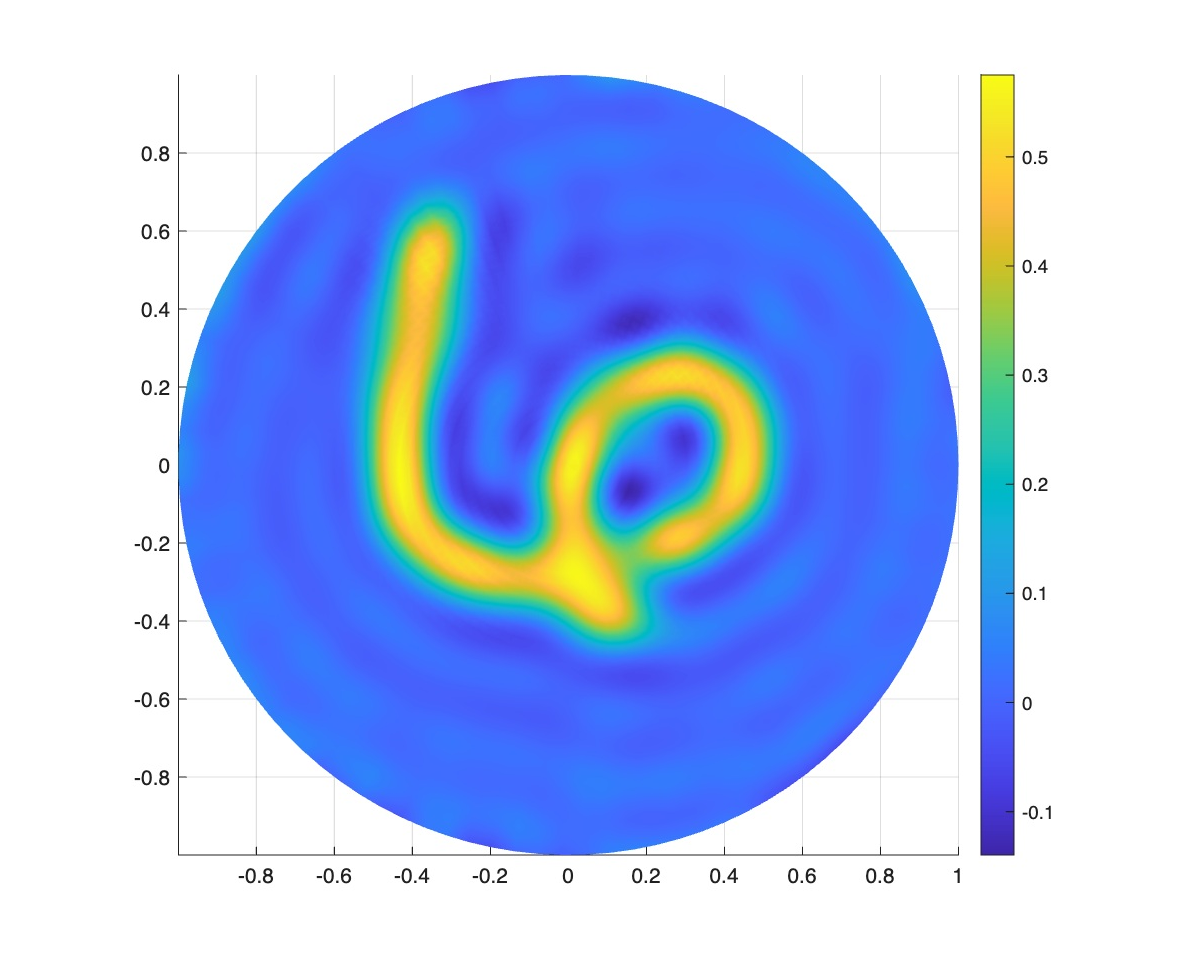}}
{\includegraphics[width=0.24\linewidth]{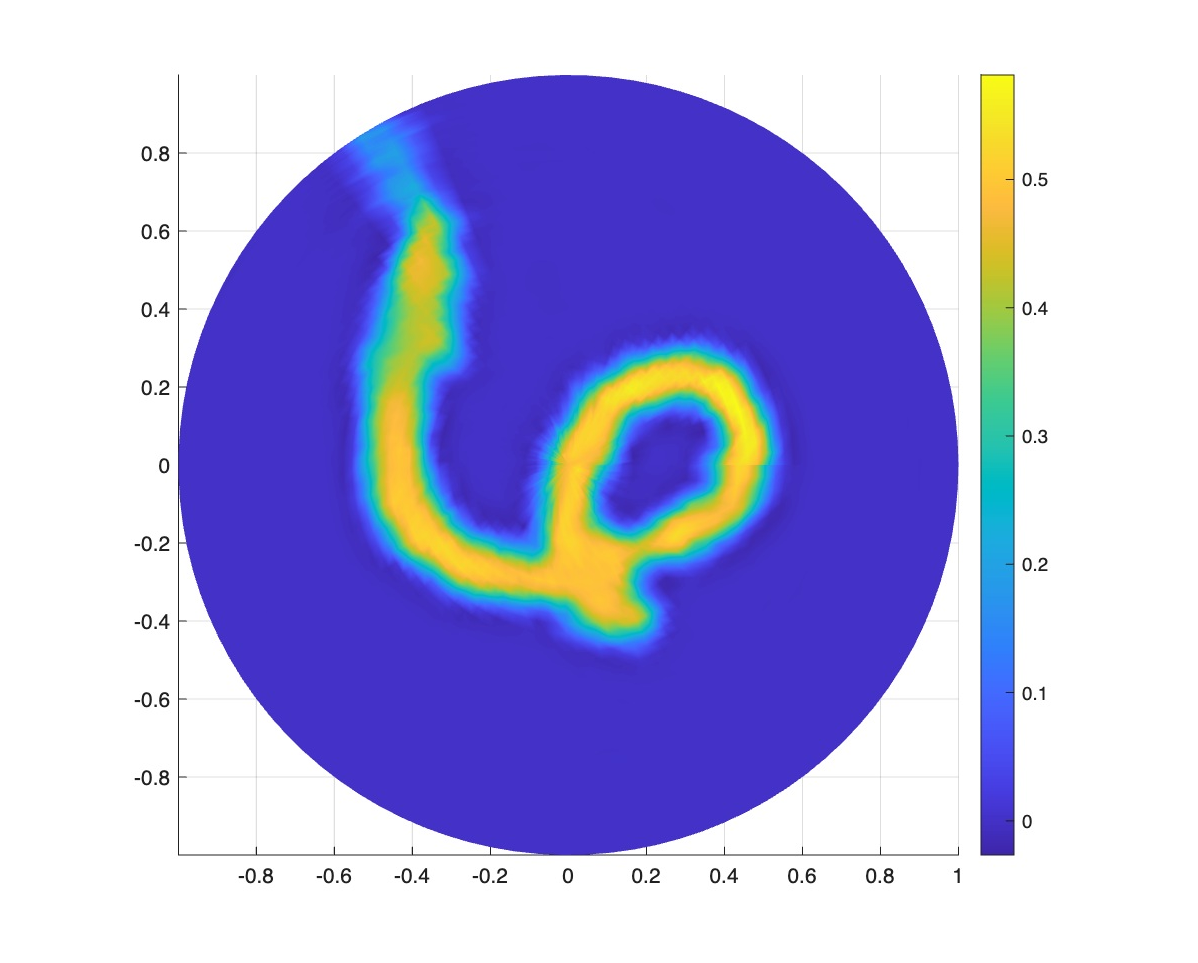}}
{\includegraphics[width=0.24\linewidth]{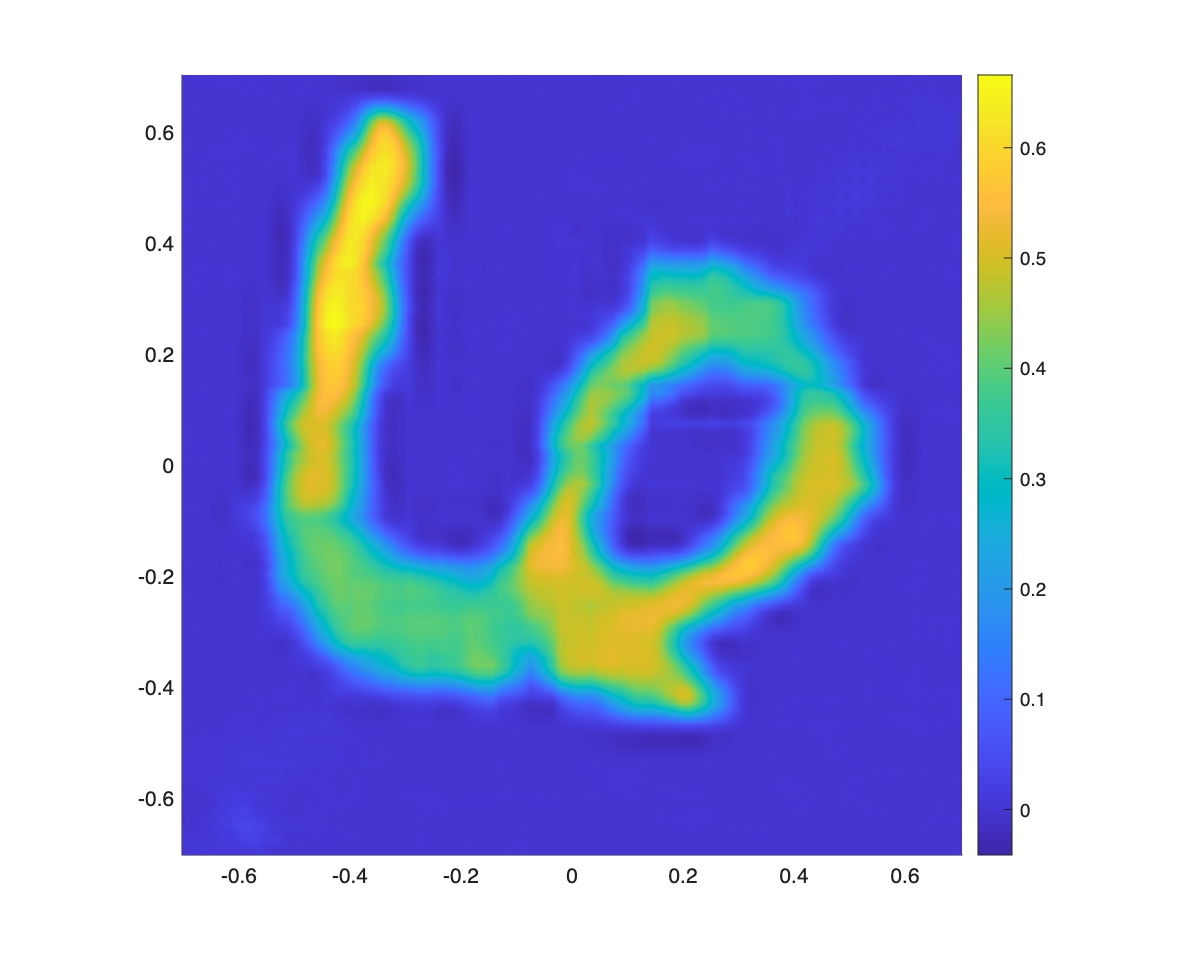}}\\
{\includegraphics[width=0.24\linewidth]{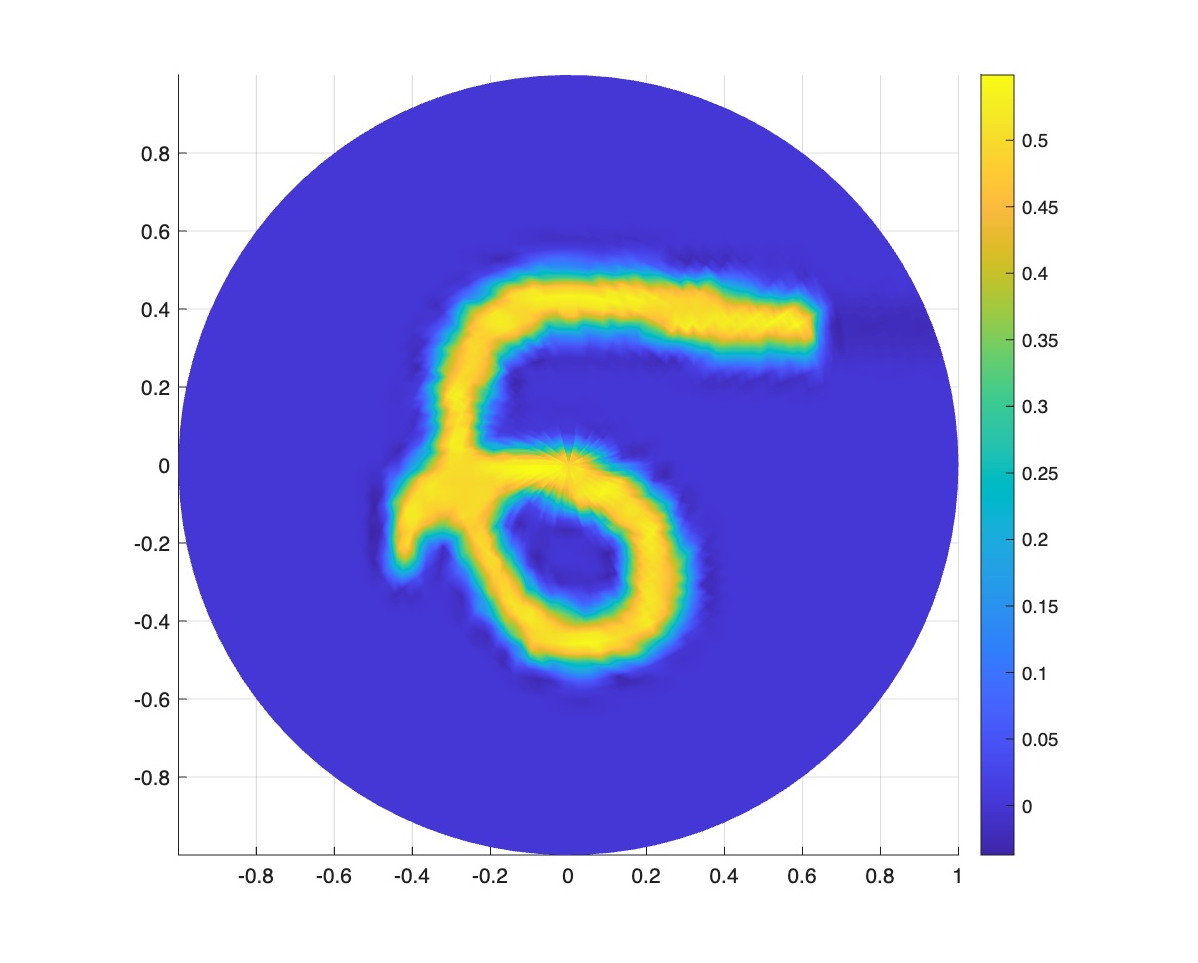}}
{\includegraphics[width=0.24\linewidth]{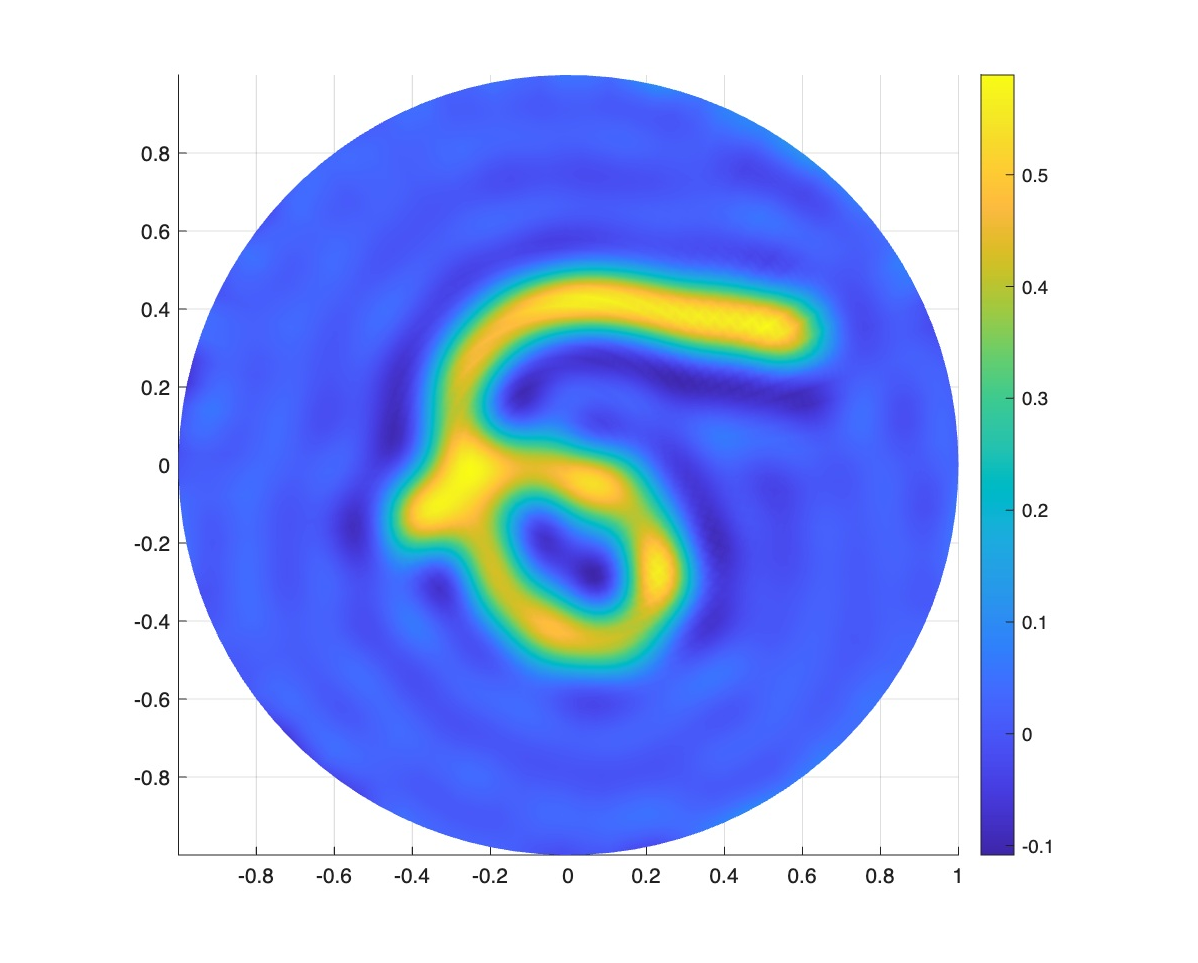}}
{\includegraphics[width=0.24\linewidth]{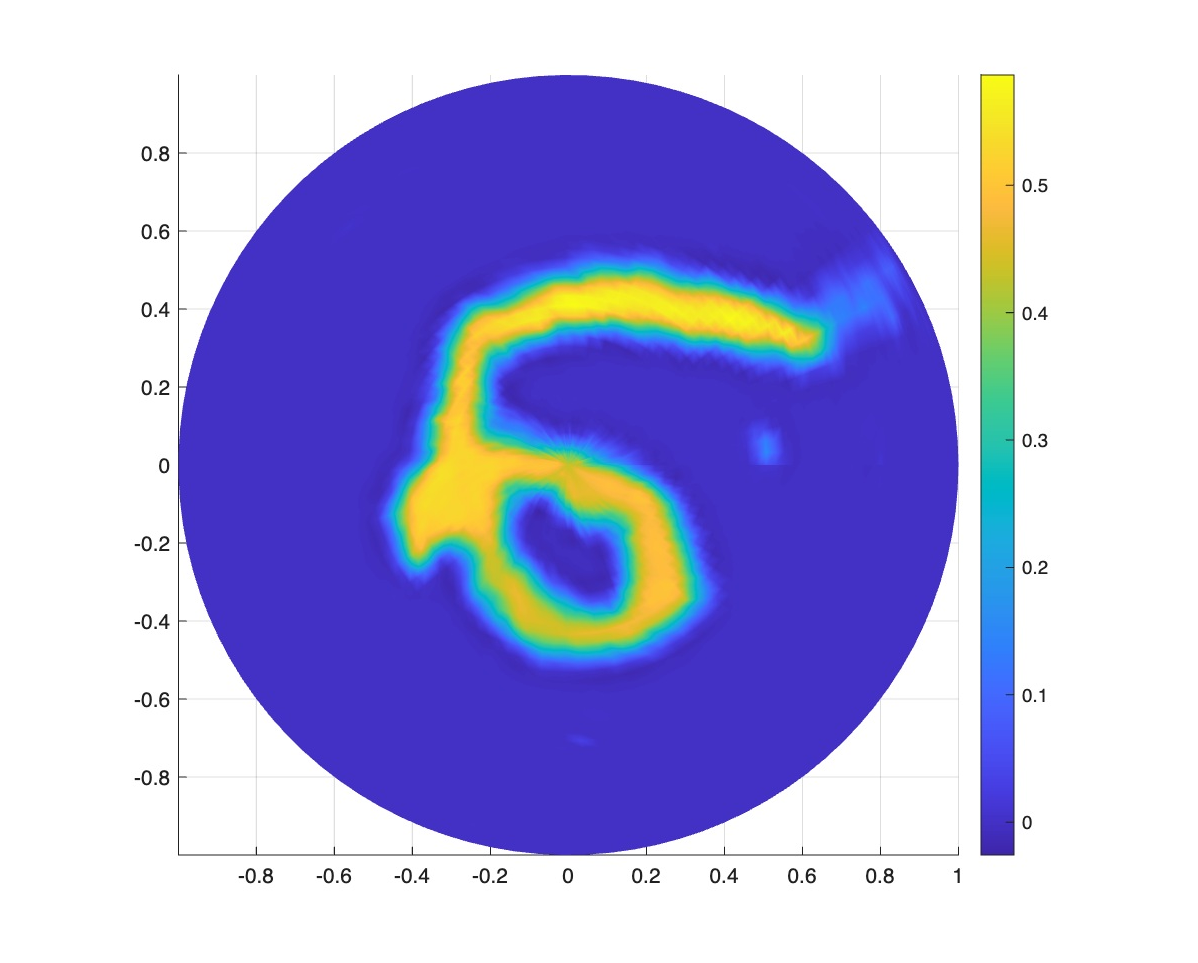}}
{\includegraphics[width=0.24\linewidth]{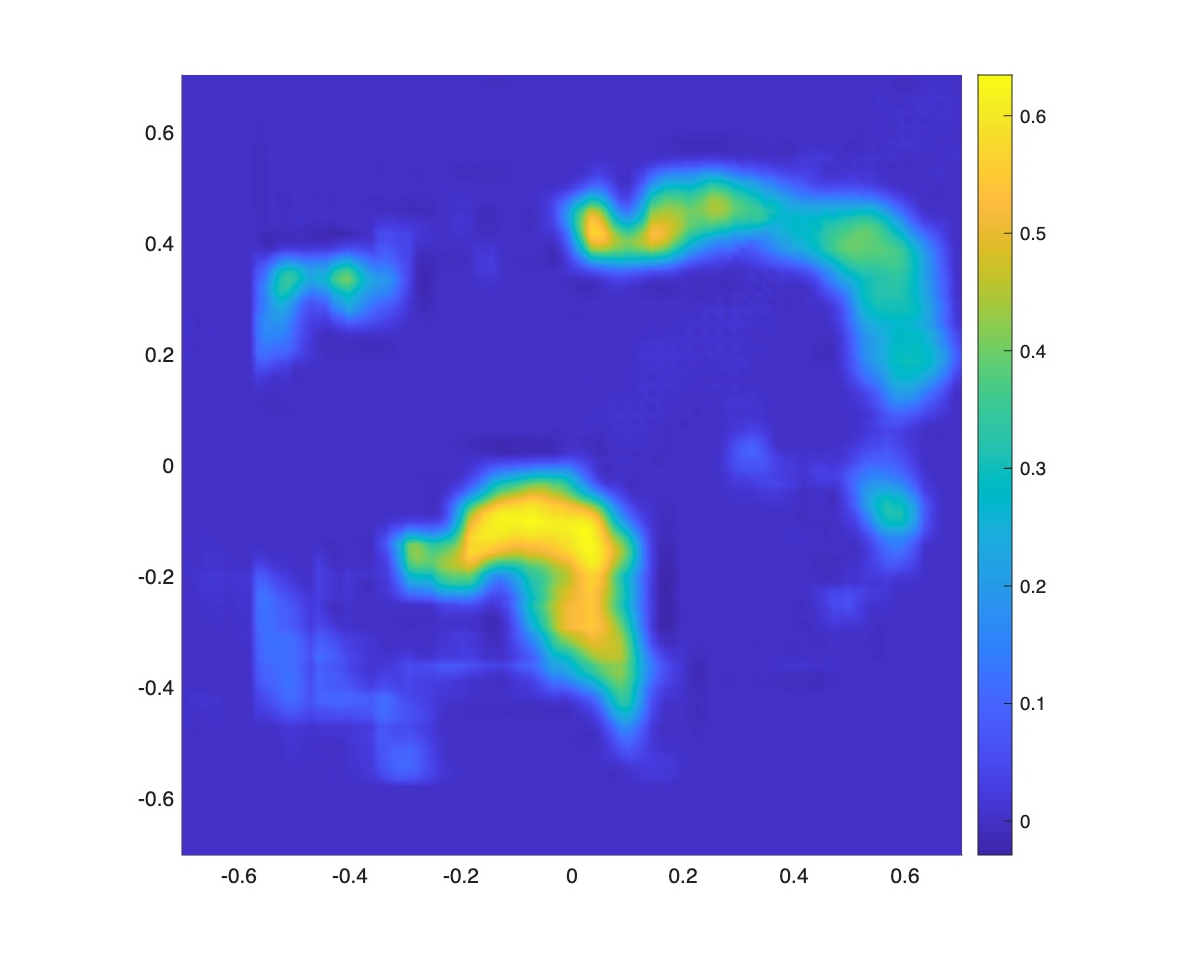}}
\\
{\includegraphics[width=0.24\linewidth]{figures/Groundtruth_4_rot.pdf}
 }
{\includegraphics[width=0.24\linewidth]{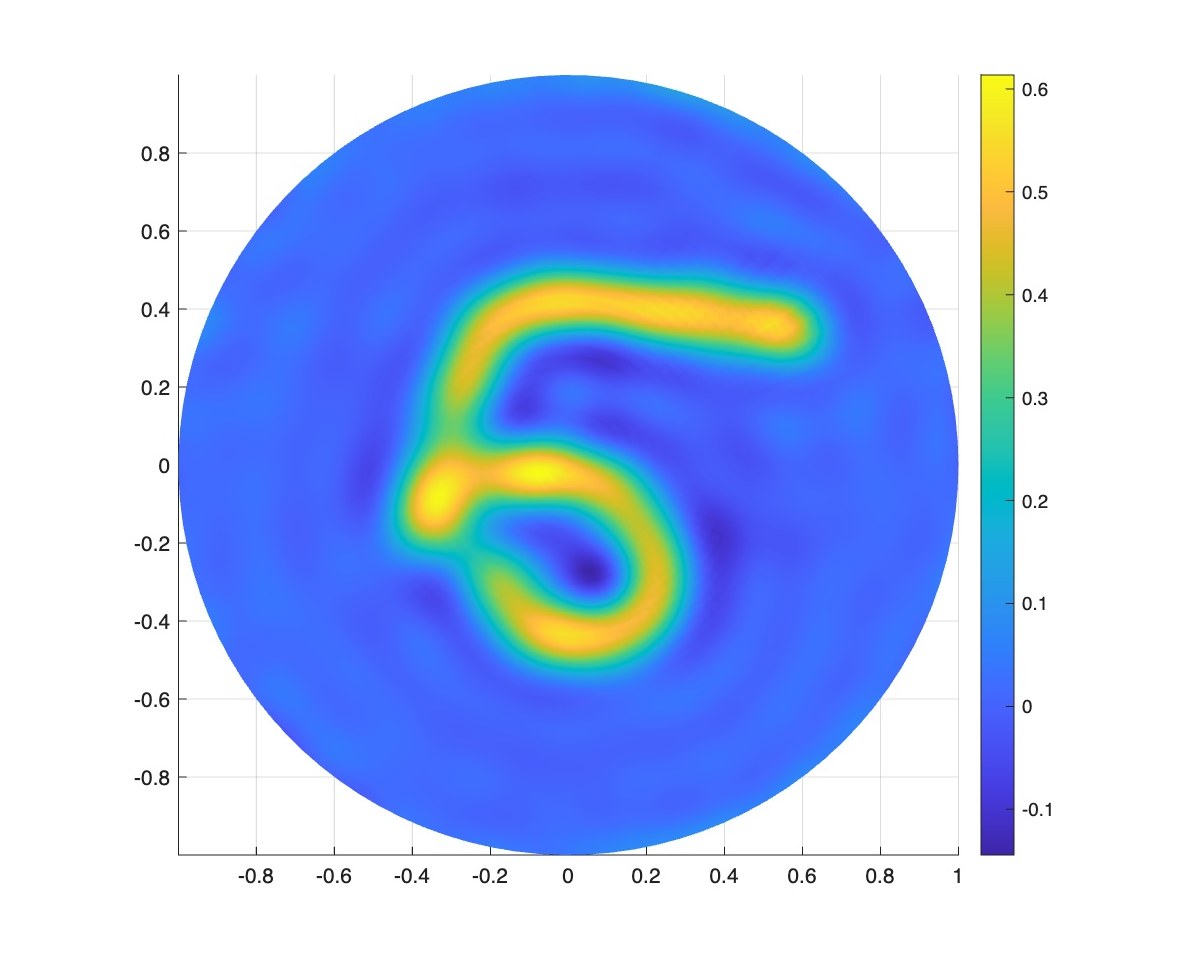}}
{\includegraphics[width=0.24\linewidth]{figures/4rot_eqvar_full_limited_UU.pdf}}
{\includegraphics[width=0.24\linewidth]{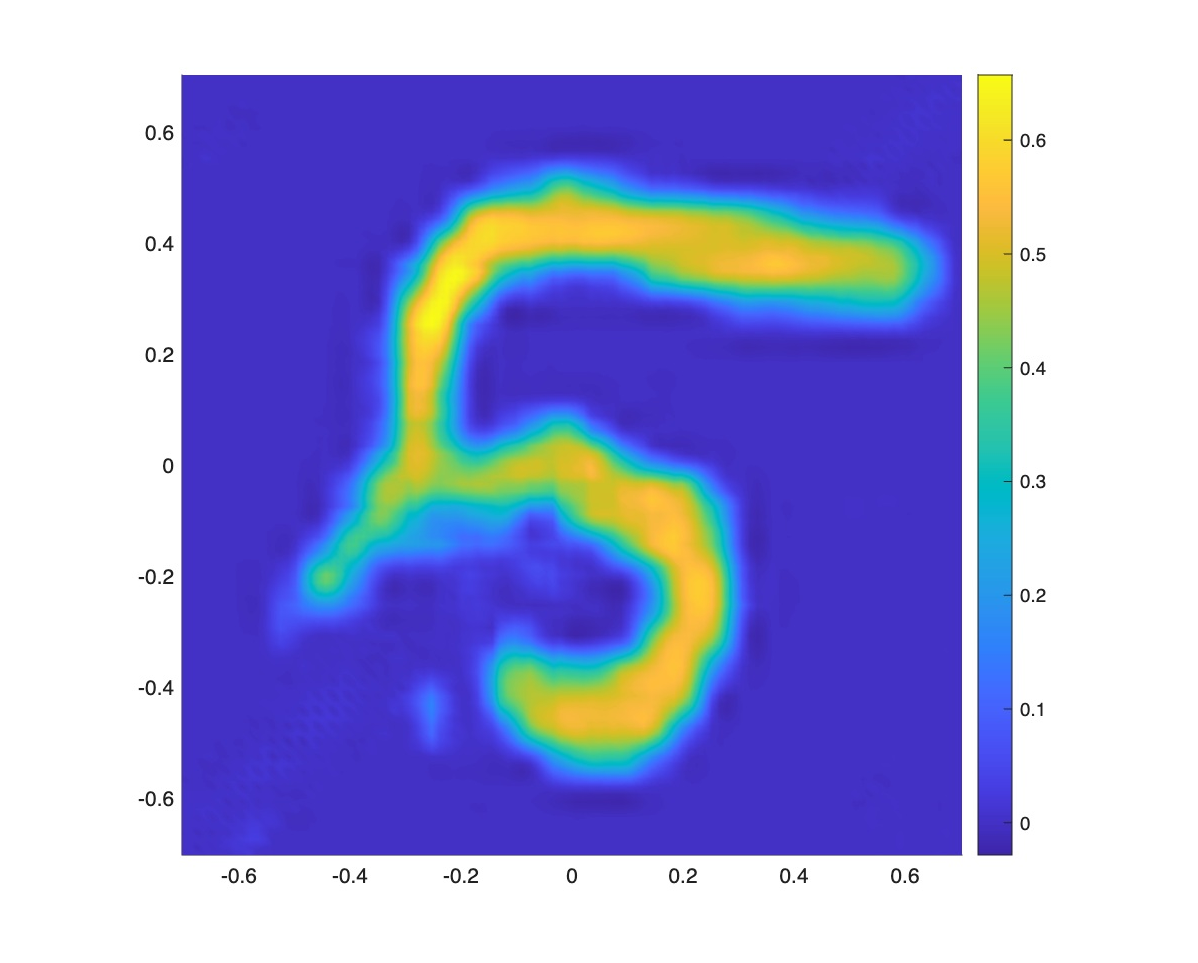}}

  \caption{Reconstruction of number ``4'' and its rotation. Left to right: ground truth, reconstructions by ULR, UU, and U, respectively. In the first row, we test an in-distribution contrast. In the second row, we test the case when  both the contrast ``4'' and the aperture \([-\frac{\pi}{2},\frac{\pi}{2}]\) rotate clockwise by \(\frac{\pi}{2}\). In the third row, we test  the case when only the contrast is rotated and the limited aperture remains the same. }  \label{figure: full rotation-equivariance limited aperture}
\end{figure}

\begin{figure}[htbp]
\includegraphics[width=0.24\linewidth]{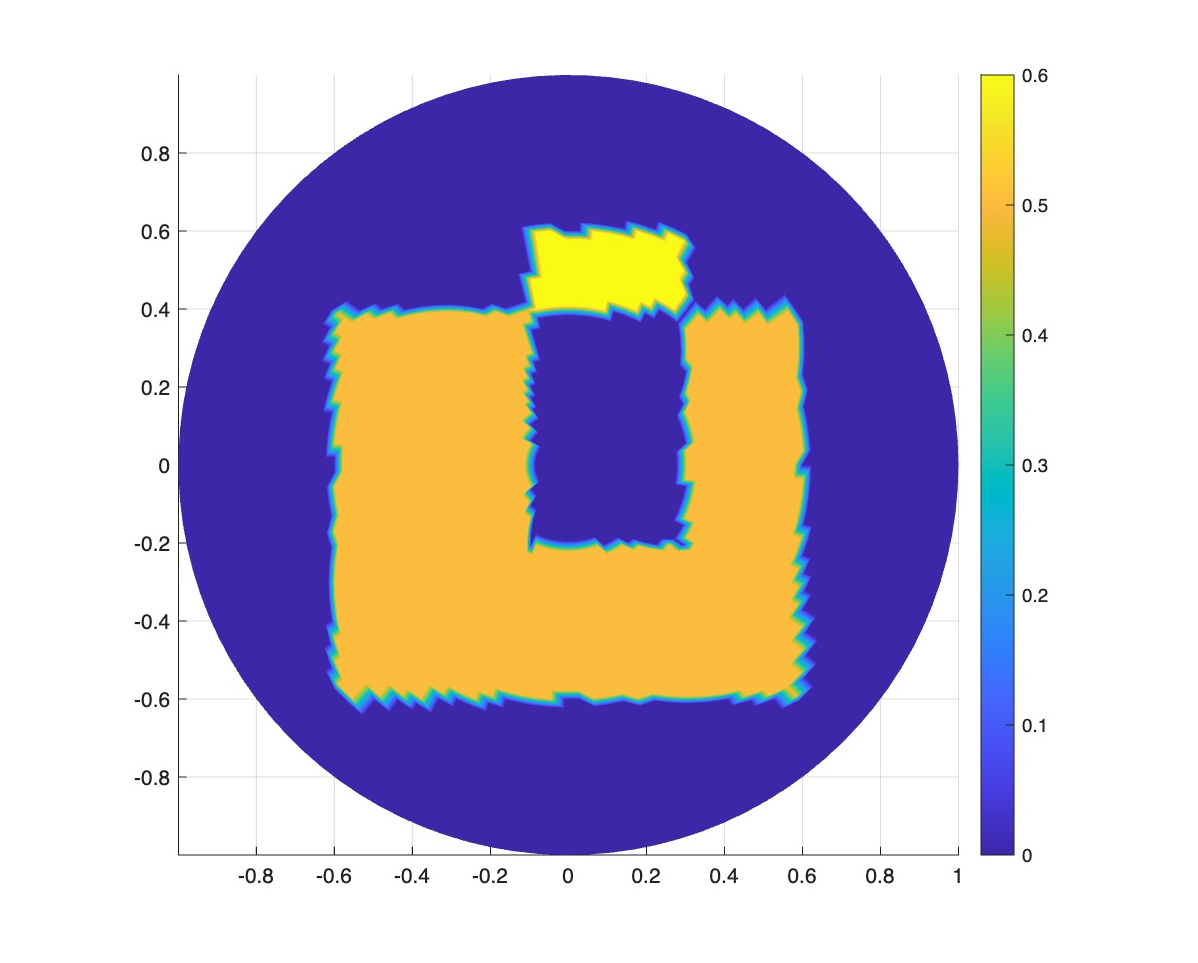}
\includegraphics[width=0.24\linewidth]{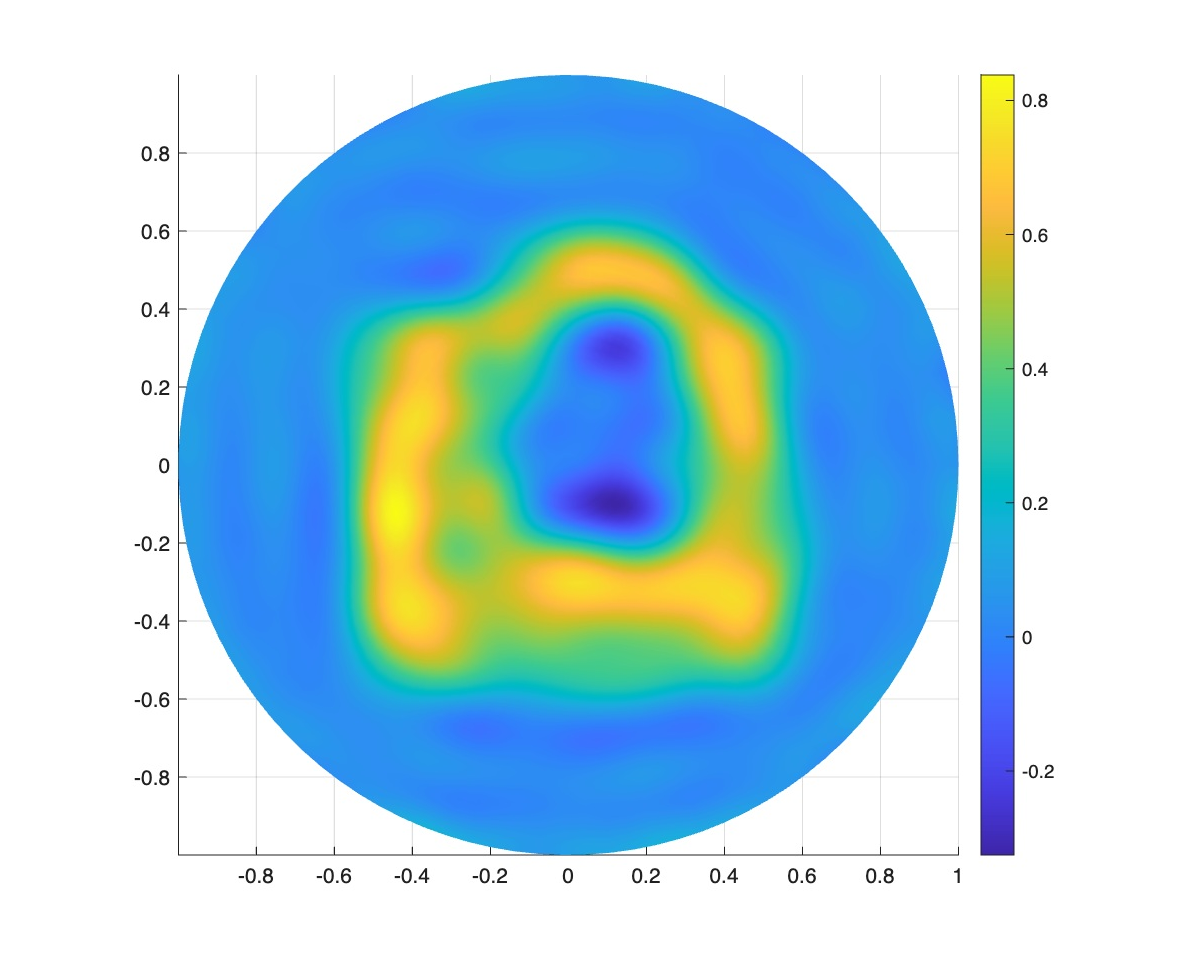}
\includegraphics[width=0.24\linewidth]{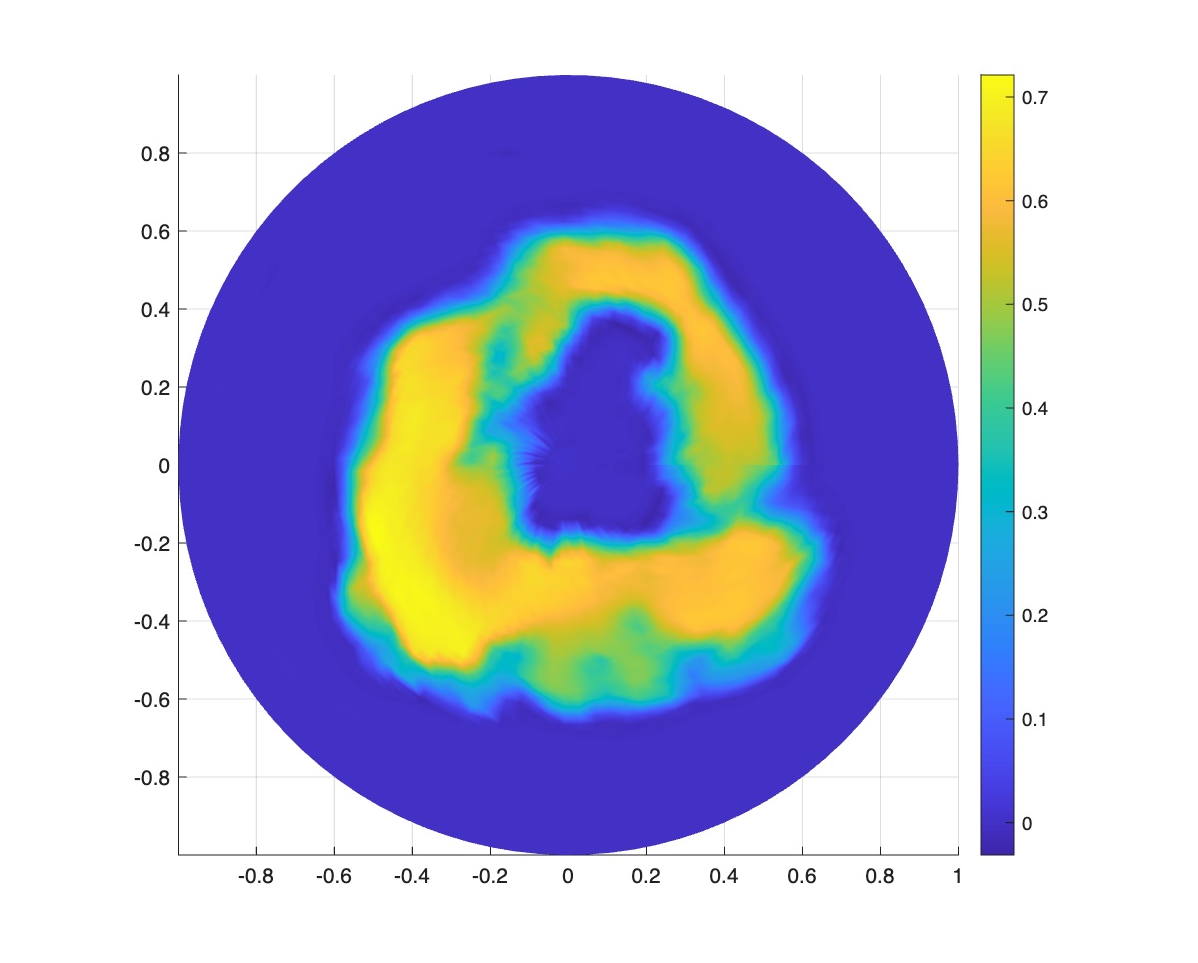}
\includegraphics[width=0.24\linewidth]{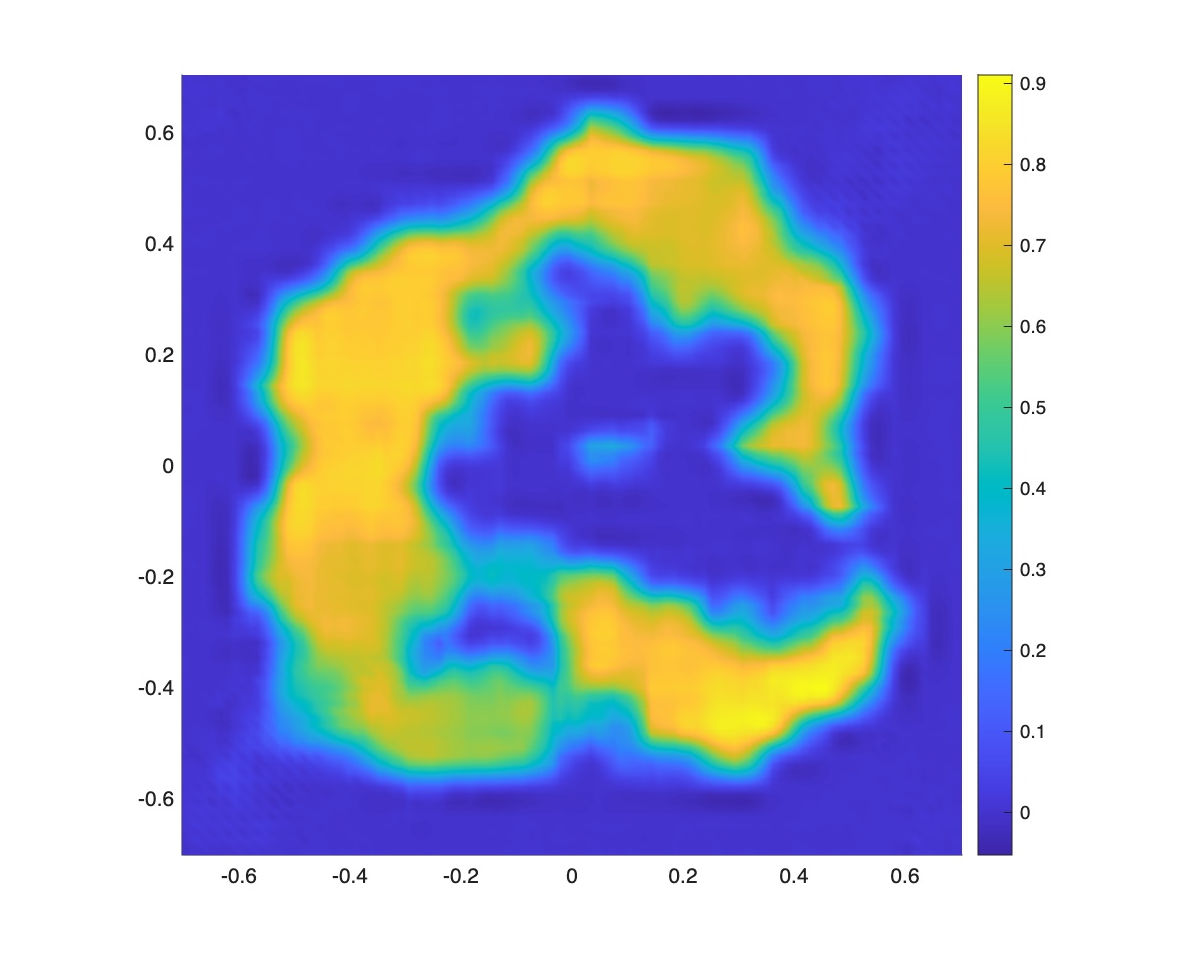}\\
\includegraphics[width=0.24\linewidth]{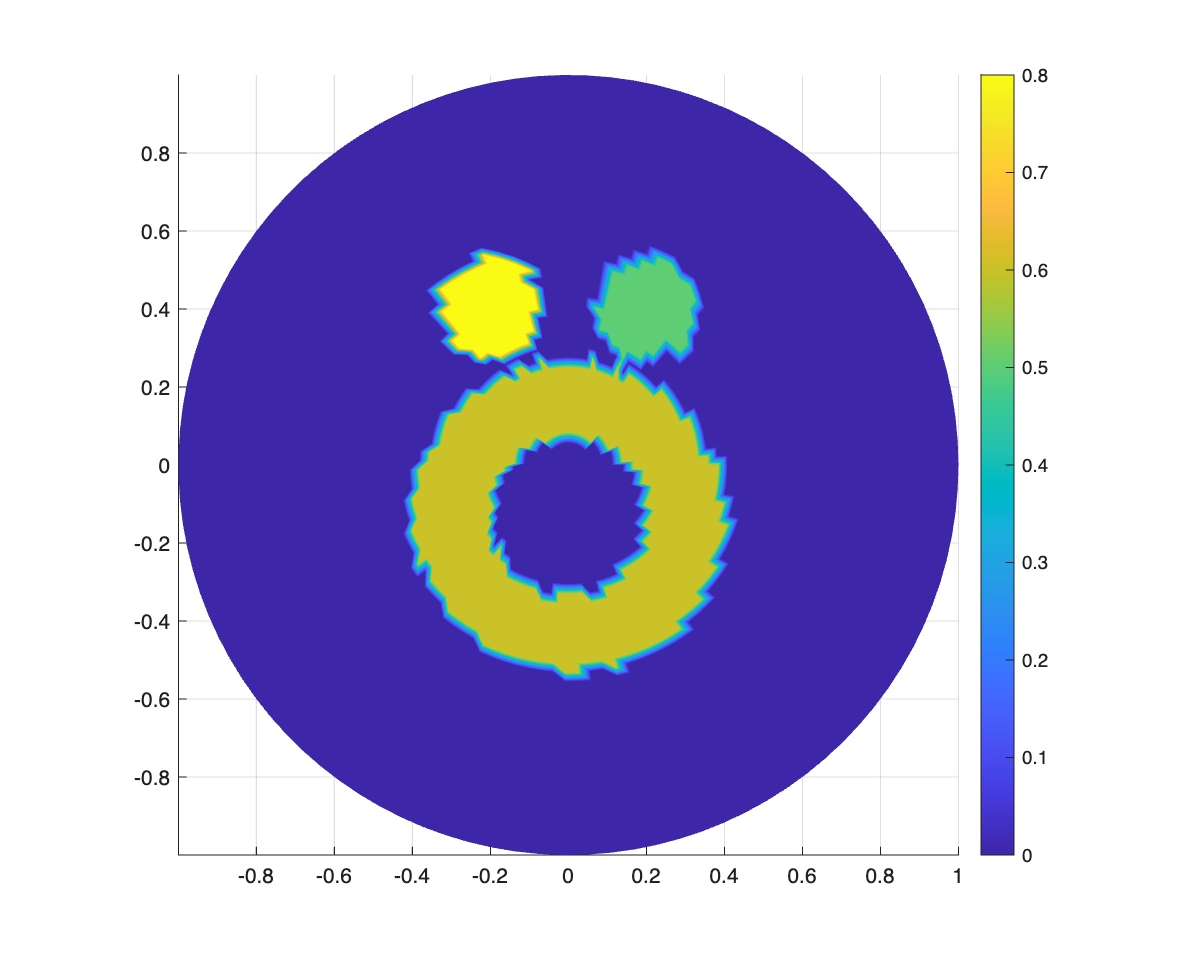}
\includegraphics[width=0.24\linewidth]{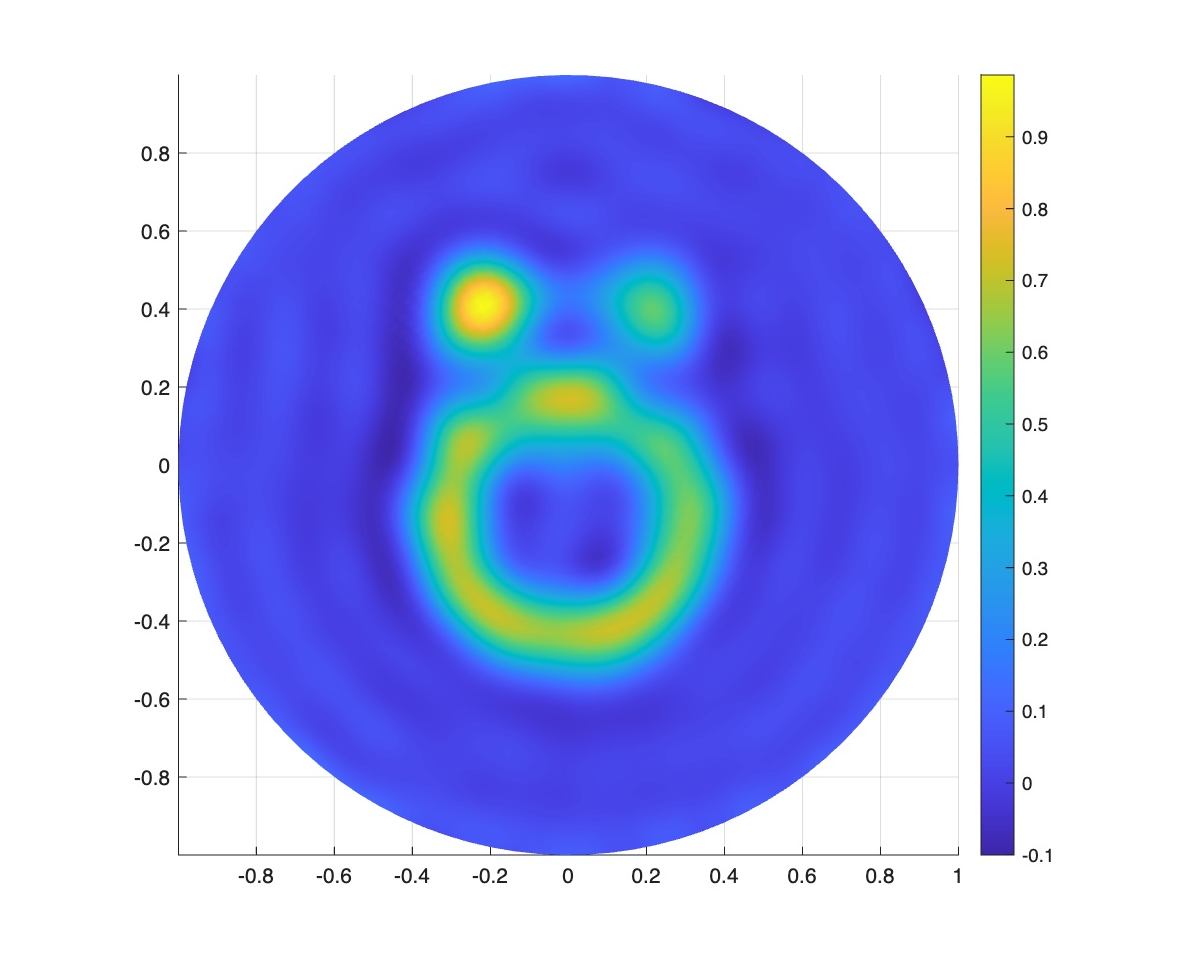}
\includegraphics[width=0.24\linewidth]{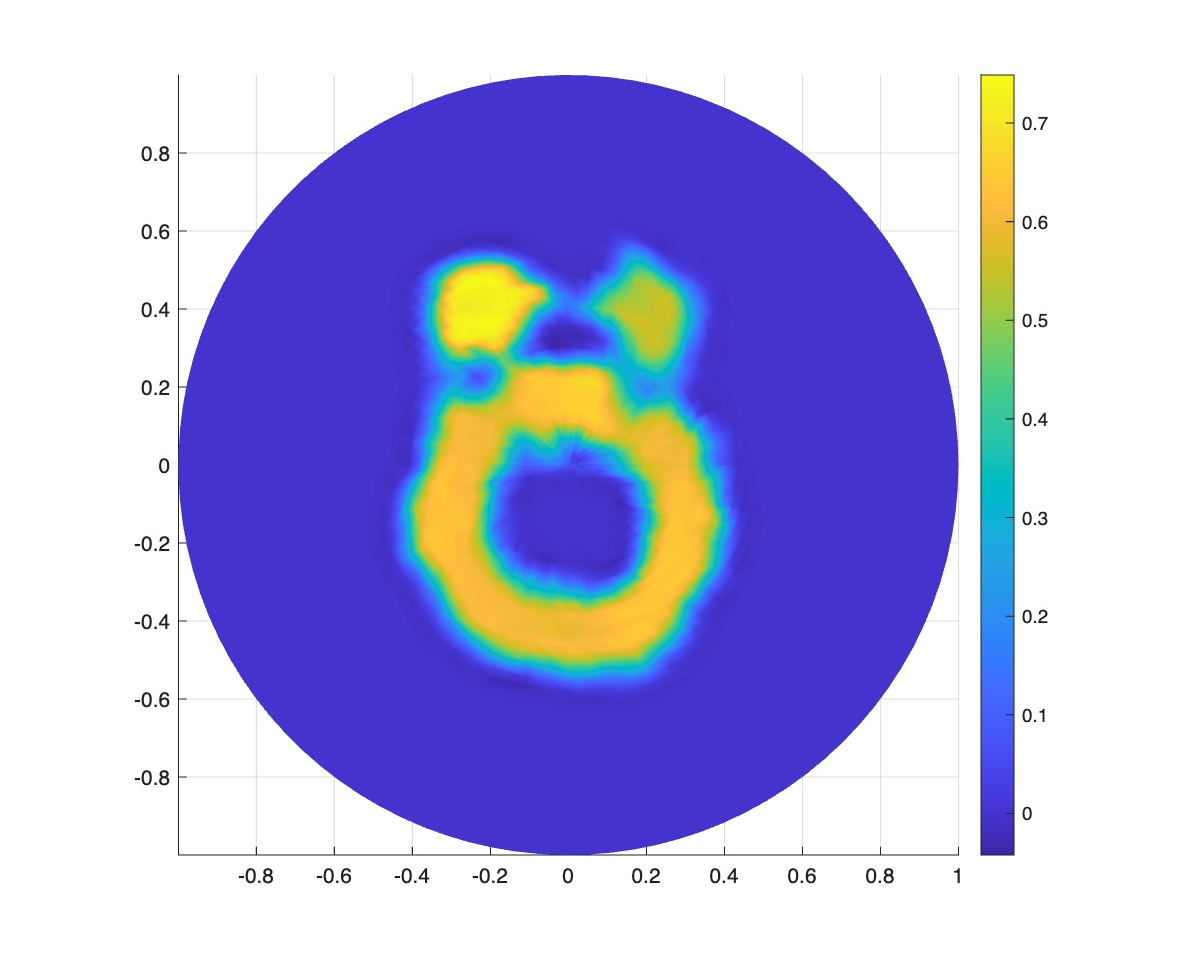}
\includegraphics[width=0.24\linewidth]{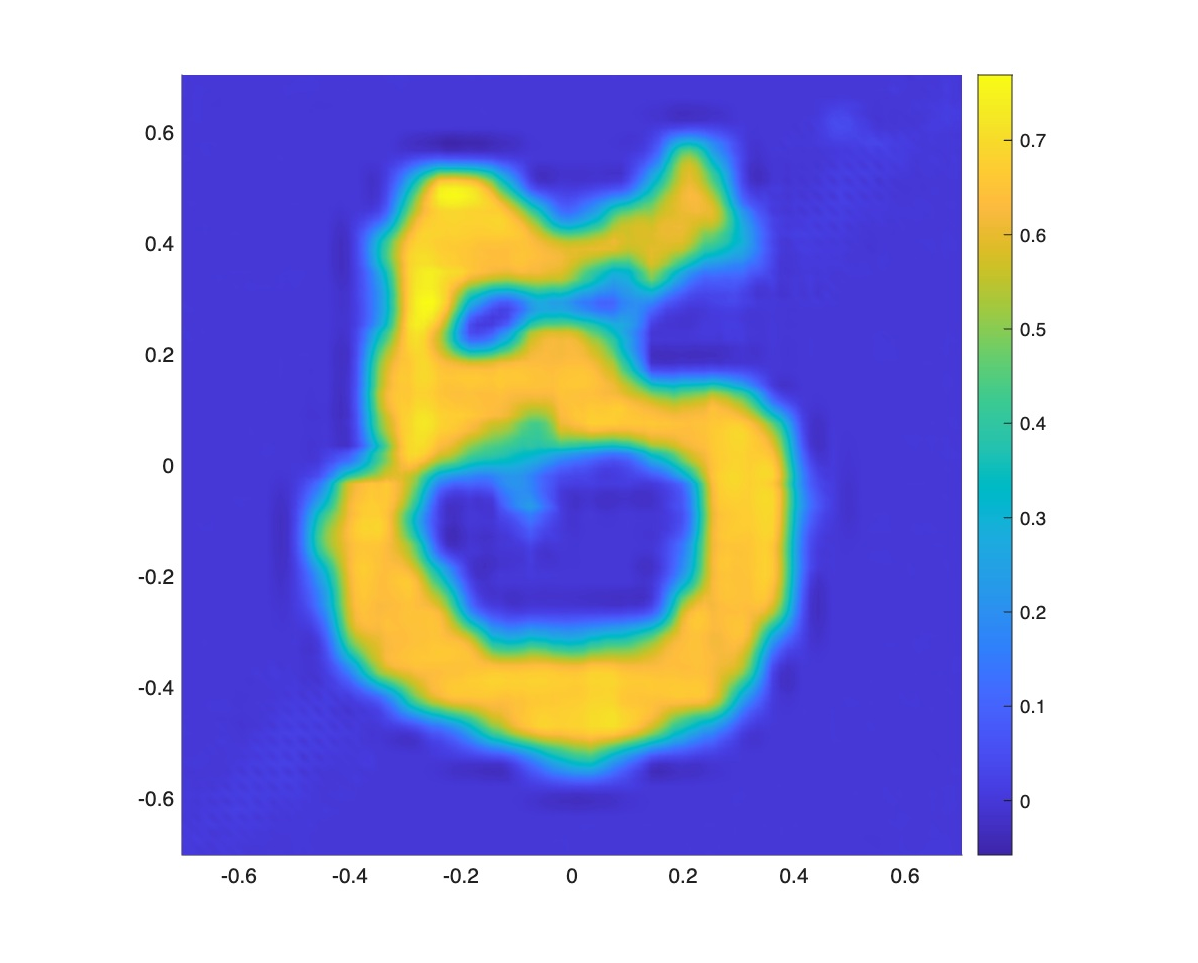}\\
\includegraphics[width=0.24\linewidth]{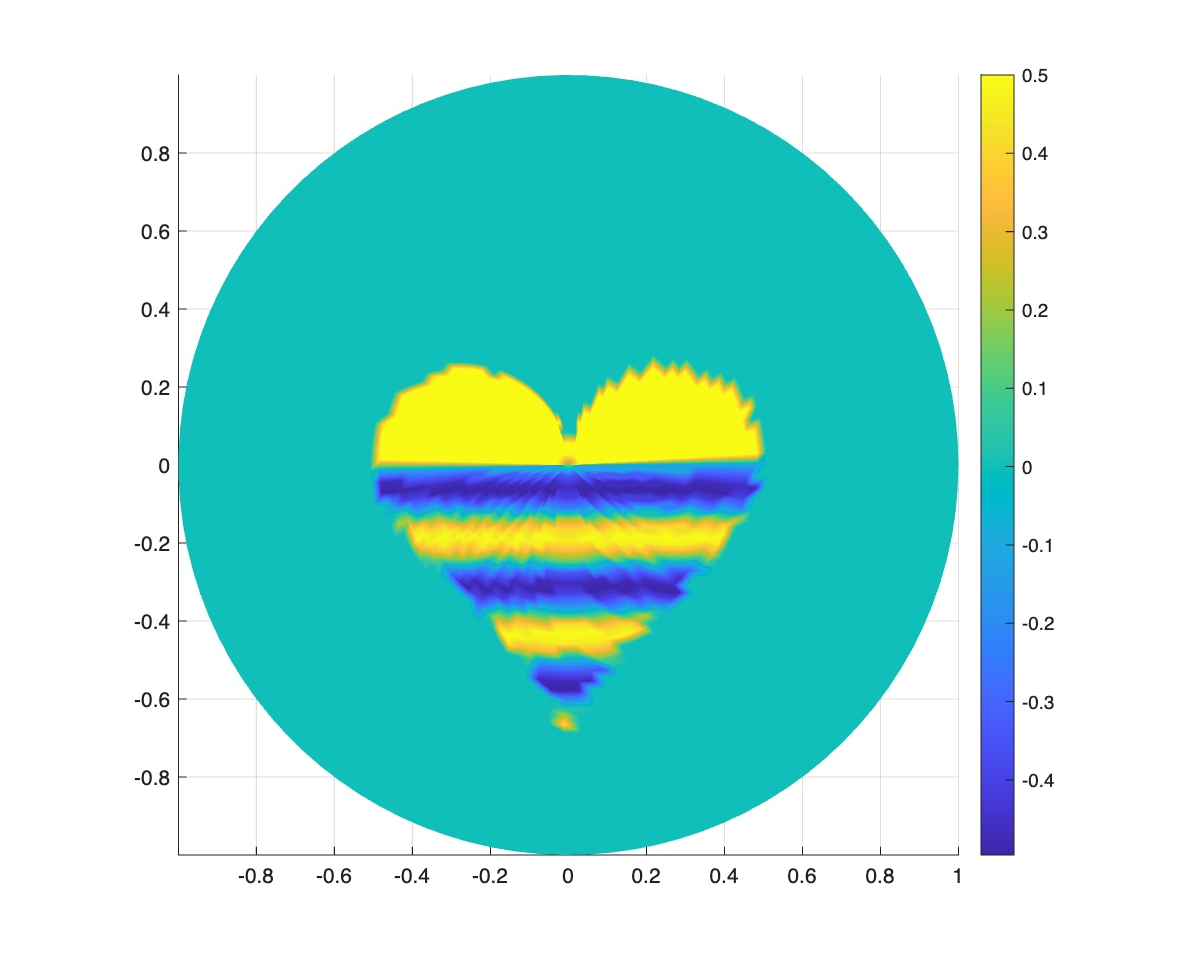}
\includegraphics[width=0.24\linewidth]{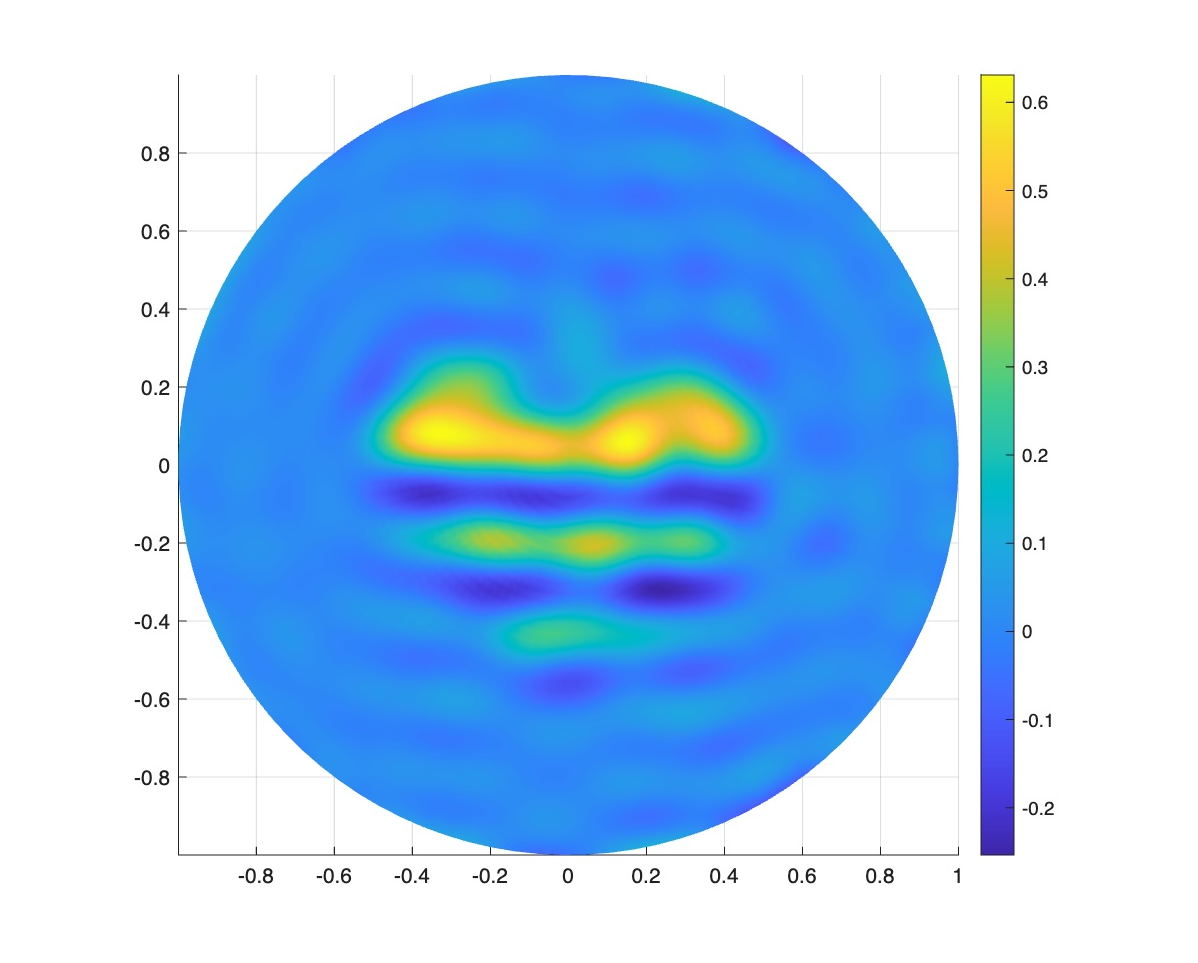}
\includegraphics[width=0.24\linewidth]{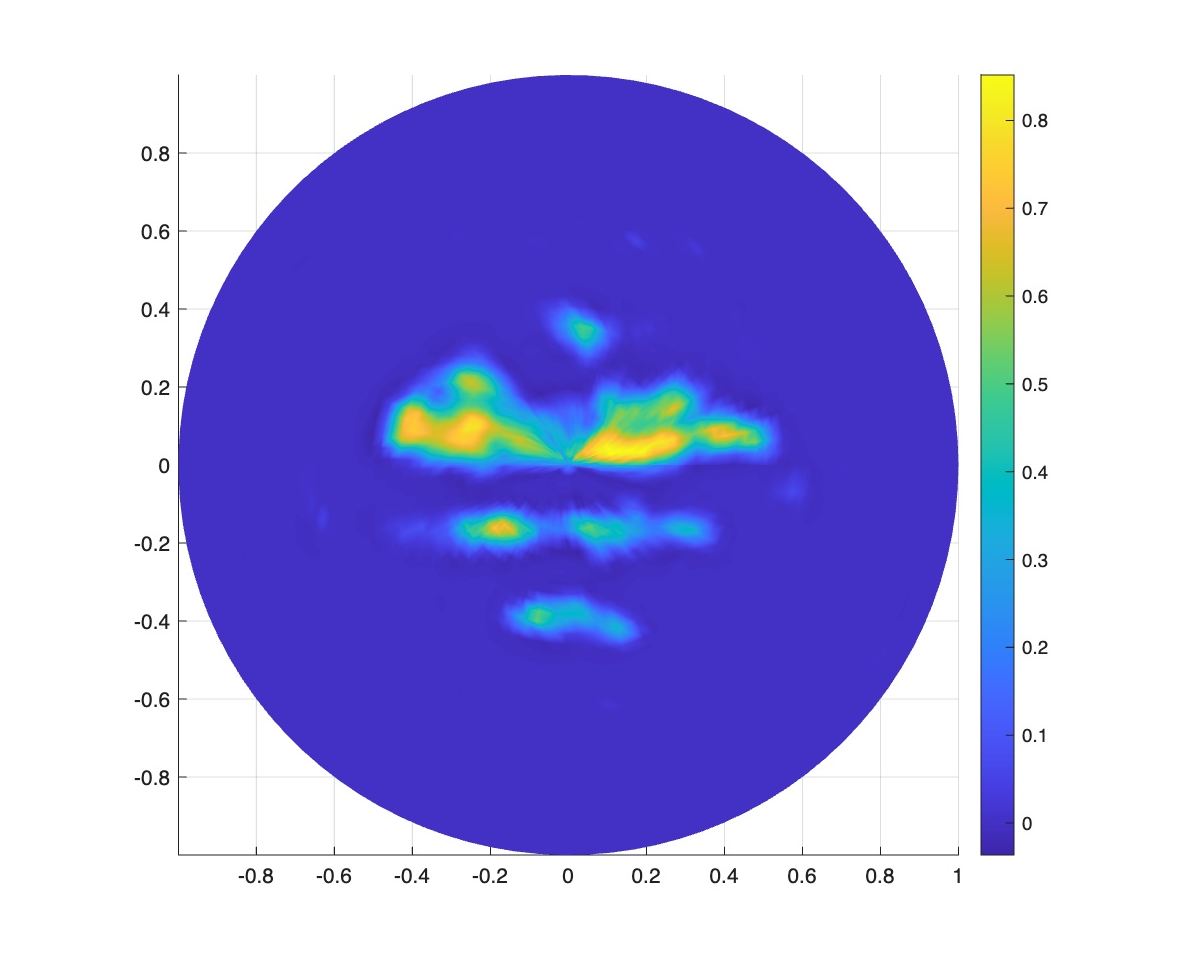}
\includegraphics[width=0.24\linewidth]{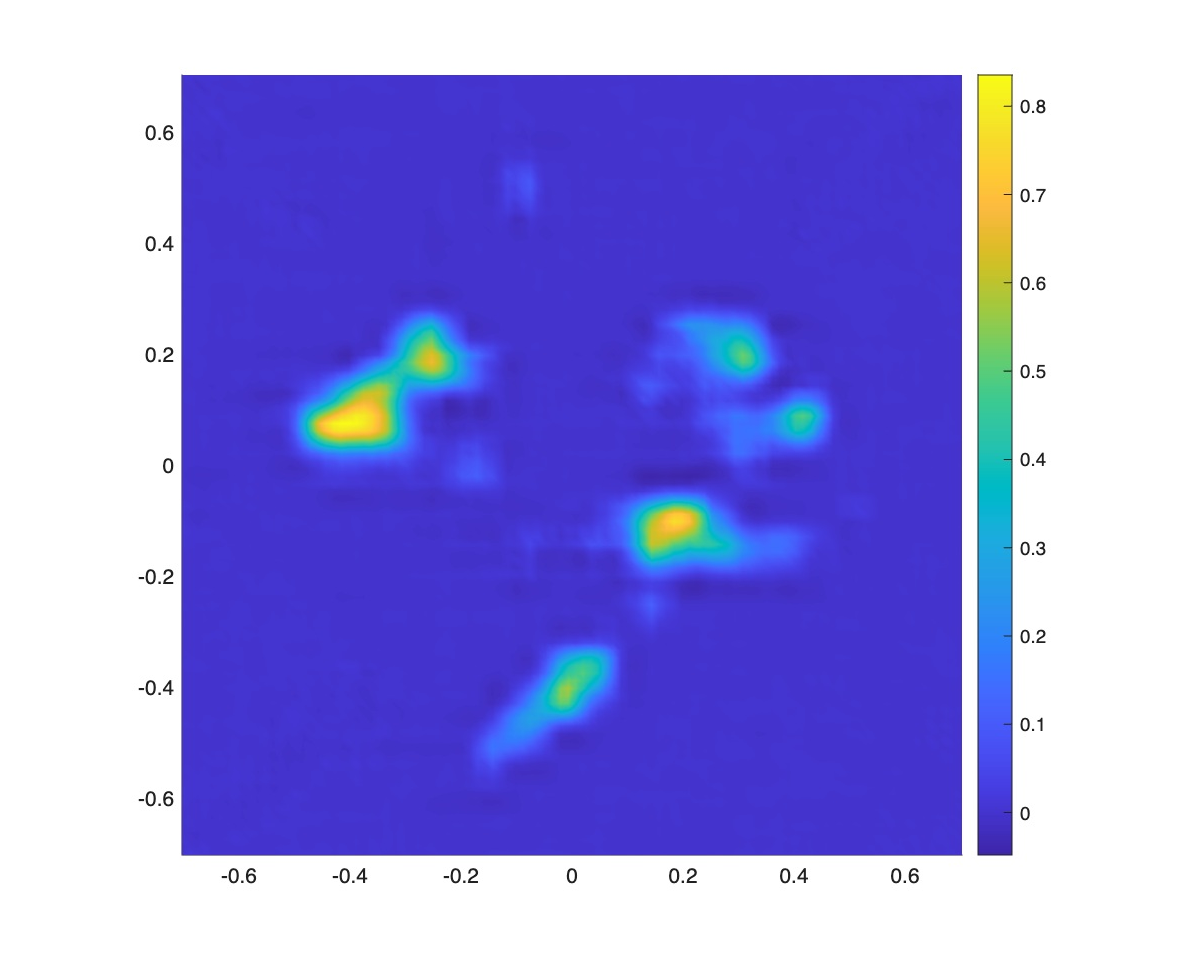}\\
\includegraphics[width=0.24\linewidth]{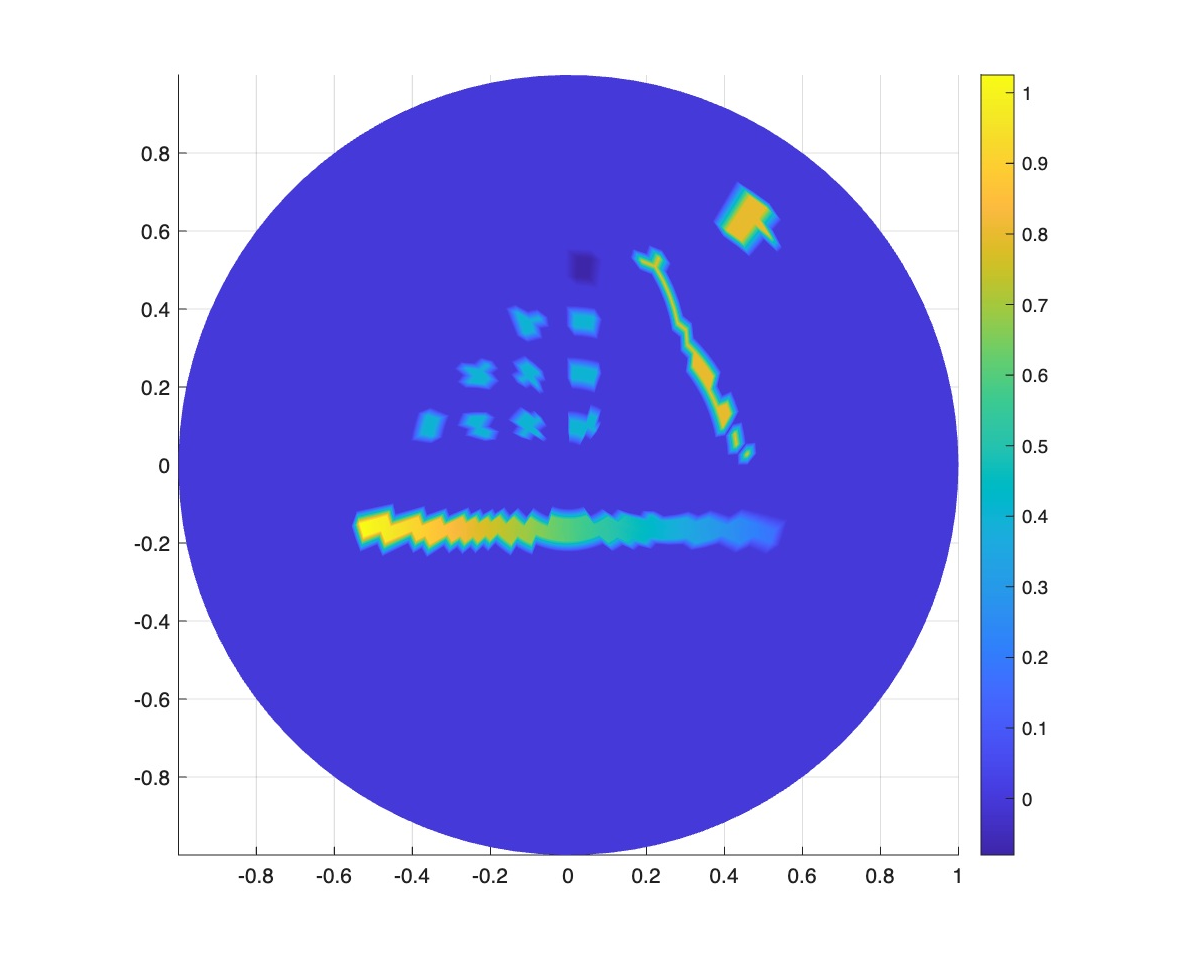}
\includegraphics[width=0.24\linewidth]{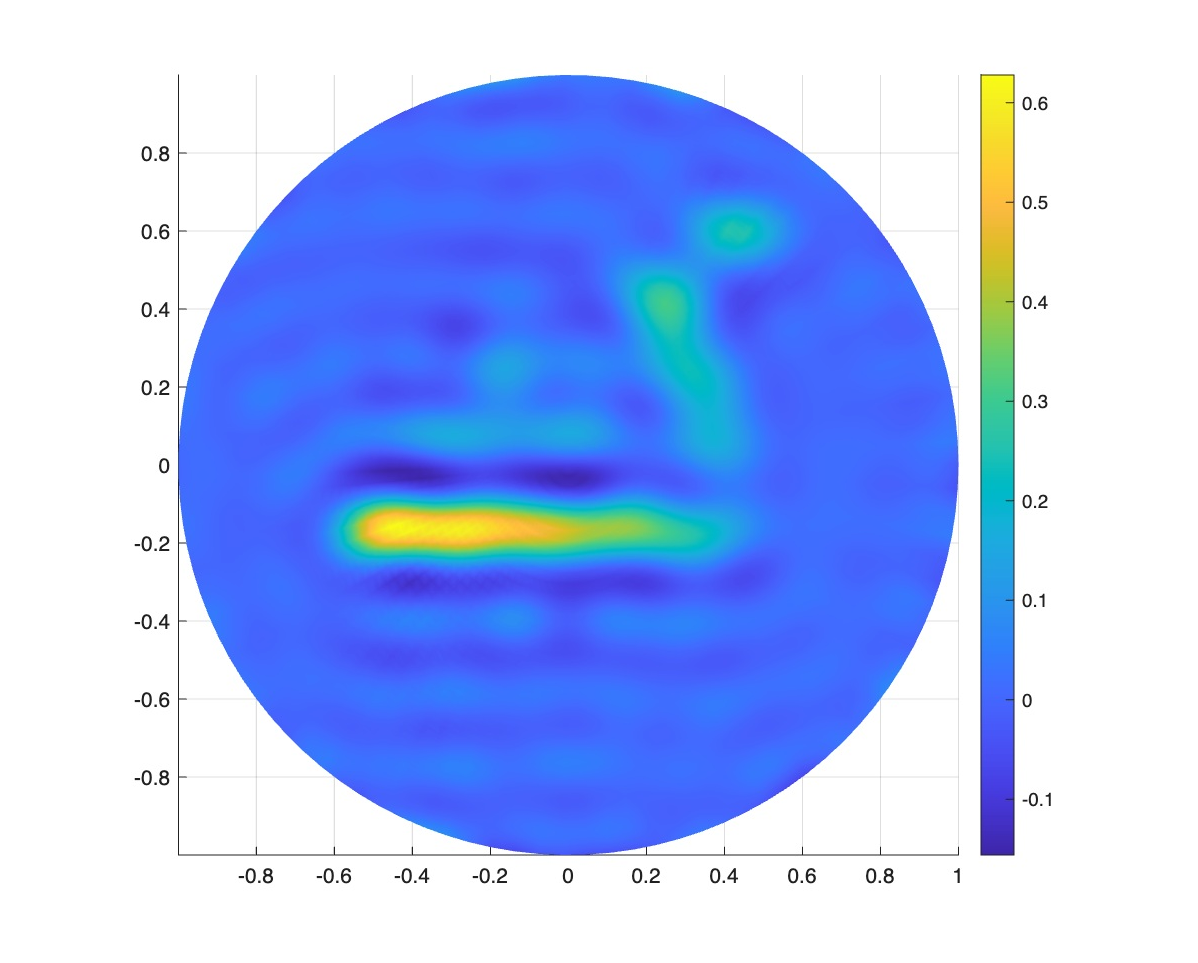}
\includegraphics[width=0.24\linewidth]{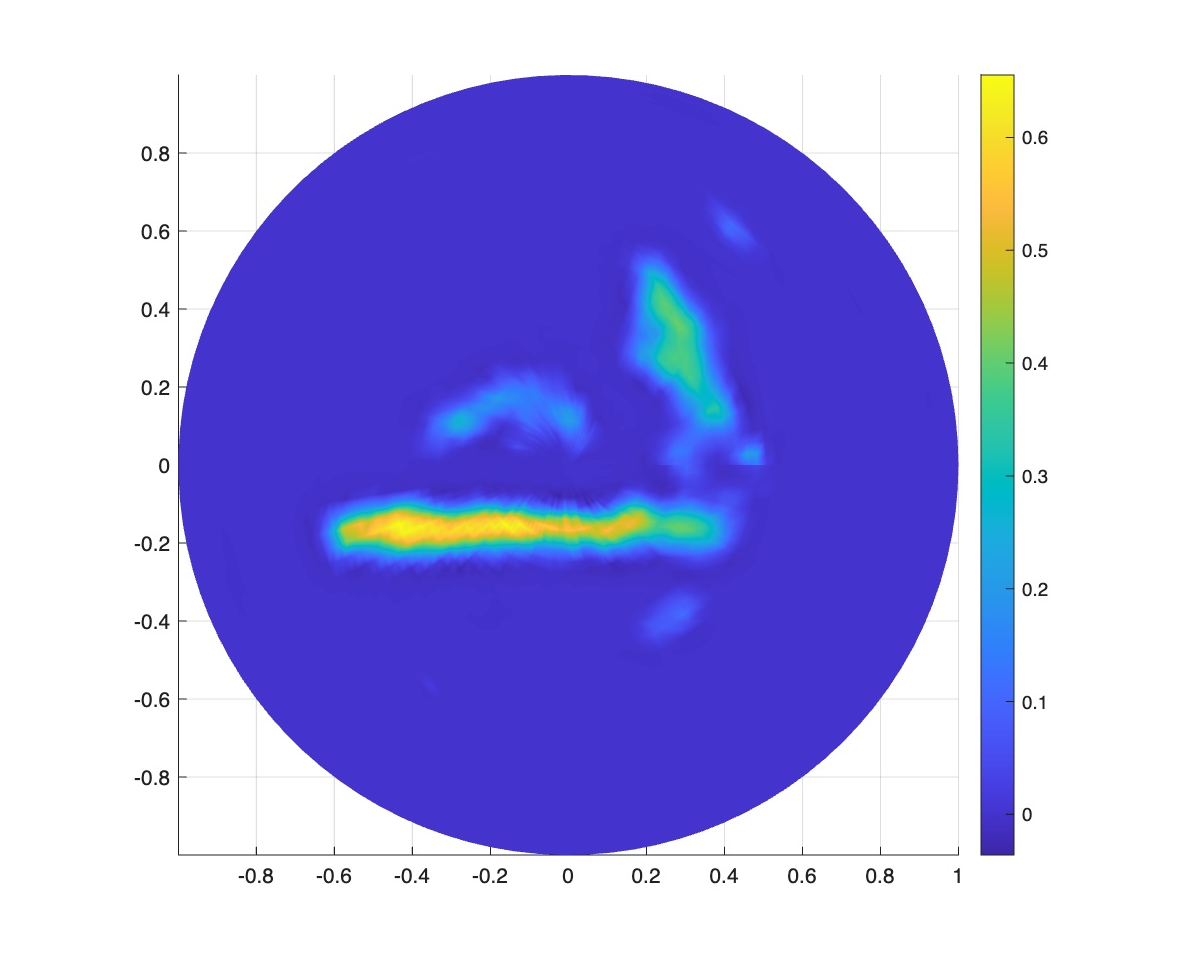}
\includegraphics[width=0.24\linewidth]{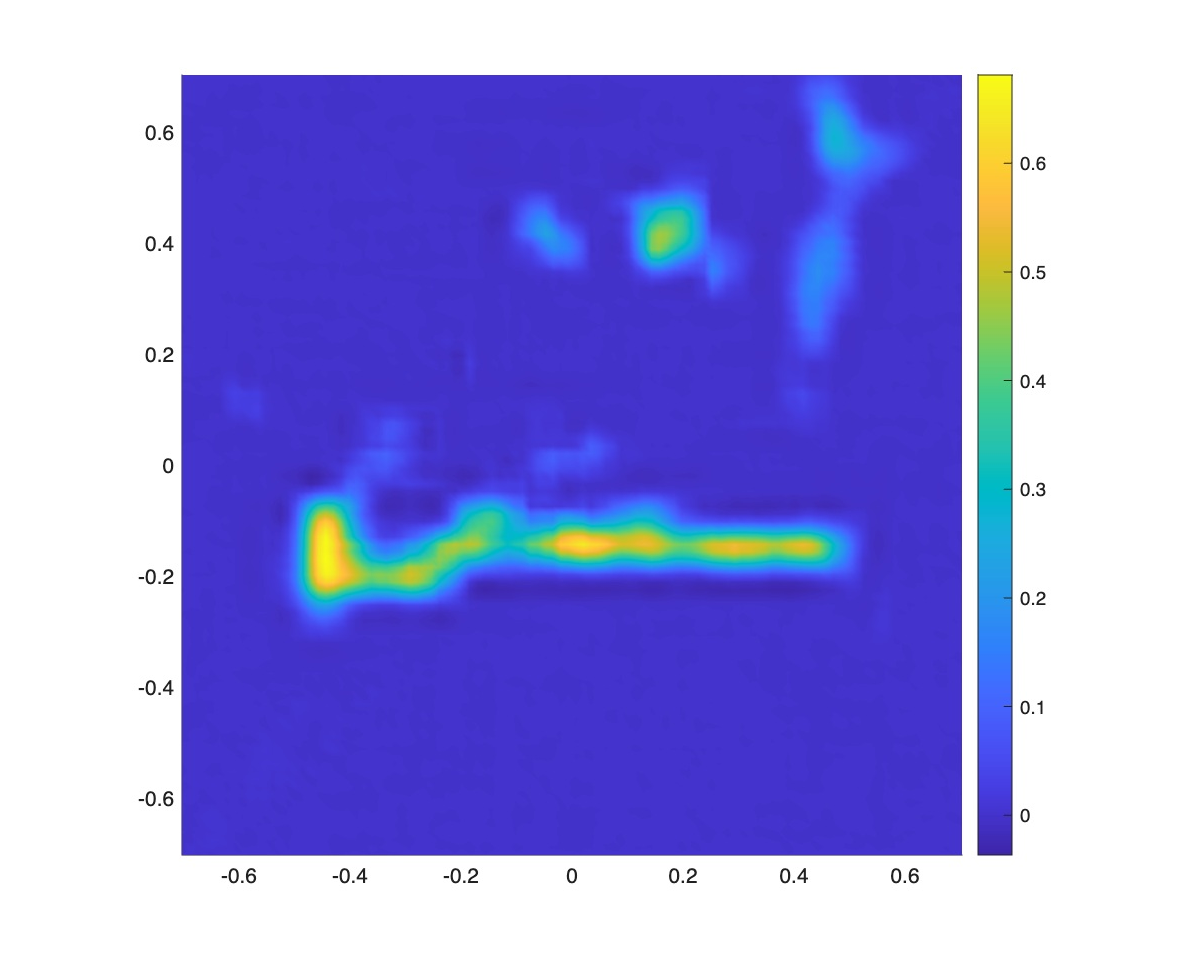}

  \caption{ Reconstruction of out-distributional contrasts with limited aperture data. From left to right: ground truth, reconstructions by ULR, UU, and U, respectively. }  \label{figure: generalization limited aperture}
  
\end{figure}

\bibliographystyle{abbrv}

\end{document}